\documentclass[reqno, 10pt]{article}

\usepackage[utf8]{inputenc}   
\usepackage[T1]{fontenc}      

\usepackage{amsmath,amsthm} 
\usepackage{amssymb,mathrsfs}
\usepackage{amsfonts}
\usepackage{upgreek}
\usepackage{nicefrac}
\usepackage{geometry}
\usepackage{authblk}
\usepackage{graphicx}
\usepackage[caption = false]{subfig} 
\usepackage{color}
\usepackage[colorlinks = true, citecolor = blue]{hyperref}
\usepackage{algpseudocode,algorithm,algorithmicx}
\algnewcommand{\Inputs}[1]{%
  \State \textbf{Inputs:}
  \Statex \hspace*{\algorithmicindent}\parbox[t]{.8\linewidth}{\raggedright #1}
}
\algnewcommand{\Initialize}[1]{%
  \State \textbf{Initialize:}
  \Statex \hspace*{\algorithmicindent}\parbox[t]{.8\linewidth}{\raggedright #1}
}
\algnewcommand{\Outputs}[1]{%
  \State \textbf{Outputs:}
  \Statex \hspace*{\algorithmicindent}\parbox[t]{.8\linewidth}{\raggedright #1}
}

\usepackage{stmaryrd}

\usepackage[inline]{enumitem}
\usepackage{url}
\usepackage{graphicx} 
\usepackage{lmodern} 
\usepackage{xcolor}
\usepackage{bbm}

\usepackage{ifthen}
\usepackage{xargs}

\usepackage[textwidth=1.8cm]{todonotes}

\usepackage{aliascnt}
\usepackage{cleveref}
\usepackage{autonum}
\makeatletter
\newtheorem{theorem}{Theorem}
\crefname{theorem}{theorem}{Theorems}
\Crefname{Theorem}{Theorem}{Theorems}

\newtheorem*{lemma_nonumber*}{Lemma}

\newaliascnt{lemma}{theorem}
\newtheorem{lemma}[lemma]{Lemma}
\aliascntresetthe{lemma}
\crefname{lemma}{lemma}{lemmas}
\Crefname{Lemma}{Lemma}{Lemmas}

\newaliascnt{corollary}{theorem}
\newtheorem{corollary}[corollary]{Corollary}
\aliascntresetthe{corollary}
\crefname{corollary}{corollary}{corollaries}
\Crefname{Corollary}{Corollary}{Corollaries}

\newaliascnt{proposition}{theorem}
\newtheorem{proposition}[proposition]{Proposition}
\aliascntresetthe{proposition}
\crefname{proposition}{proposition}{propositions}
\Crefname{Proposition}{Proposition}{Propositions}

\newaliascnt{definition}{theorem}

\aliascntresetthe{definition}
\crefname{definition}{definition}{definitions}
\Crefname{Definition}{Definition}{Definitions}

\newaliascnt{remark}{theorem}

\aliascntresetthe{remark}
\crefname{remark}{remark}{remarks}
\Crefname{Remark}{Remark}{Remarks}

\crefname{example}{example}{examples}
\Crefname{Example}{Example}{Examples}

\crefname{figure}{figure}{figures}
\Crefname{Figure}{Figure}{Figures}

\newtheorem{assumption}{\textbf{A}\hspace{-3pt}}
\crefformat{assumption}{{\textbf{A}}#2#1#3}

\newtheorem{assumptionB}{\textbf{B}\hspace{-3pt}}
\Crefname{assumptionB}{\textbf{B}\hspace{-3pt}}{\textbf{B}\hspace{-3pt}}
\crefname{assumptionB}{\textbf{B}}{\textbf{B}}

\newtheorem{assumptionC}{\textbf{C}\hspace{-3pt}}
\Crefname{assumptionC}{\textbf{C}\hspace{-3pt}}{\textbf{C}\hspace{-3pt}}
\crefname{assumptionC}{\textbf{C}}{\textbf{C}}

\newtheorem{assumptionH}{\textbf{H}\hspace{-3pt}}
\Crefname{assumptionH}{\textbf{H}\hspace{-3pt}}{\textbf{H}\hspace{-3pt}}
\crefname{assumptionH}{\textbf{H}}{\textbf{H}}

\newtheorem{assumptionT}{\textbf{T}\hspace{-3pt}}
\Crefname{assumptionT}{\textbf{T}\hspace{-3pt}}{\textbf{T}\hspace{-3pt}}
\crefname{assumptionT}{\textbf{T}}{\textbf{T}}

\newtheorem{assumptionD}{\textbf{D}\hspace{-3pt}}
\Crefname{assumptionT}{\textbf{T}\hspace{-3pt}}{\textbf{T}\hspace{-3pt}}
\crefname{assumptionT}{\textbf{T}}{\textbf{T}}

\newtheorem{assumptionL}{\textbf{L}\hspace{-3pt}}
\Crefname{assumptionL}{\textbf{L}\hspace{-3pt}}{\textbf{L}\hspace{-3pt}}
\crefname{assumptionL}{\textbf{L}}{\textbf{L}}

\Crefname{assumptionQ}{\textbf{Q}\hspace{-3pt}}{\textbf{Q}\hspace{-3pt}}
\crefname{assumptionQ}{\textbf{Q}}{\textbf{Q}}

\newtheorem{assumptionAR}{\textbf{AR}\hspace{-3pt}}
\Crefname{assumptionAR}{\textbf{AR}\hspace{-3pt}}{\textbf{AR}\hspace{-3pt}}
\crefname{assumptionAR}{\textbf{AR}}{\textbf{AR}}

\def\mean{\mathpzc{m}}
\def\ratio{\uptheta}
\def\rhoexp{\rho_{\mathrm{exp}}}
\def\xb{\bar{x}}
\def\bR{\bar{R}}
\def\Bdisc{B_{\discrete}}
\def\tBdisc{\tilde{B}_{\discrete}}
\def\dim{d}
\def\En{\tilde{E}_n}
\def\varepsn{\tilde{\vareps}_n}
\def\pow{p}
\def\ntt{\mathtt{n}_0}
\def\tlambda{\tilde{\lambda}}
\newcommand{\tb}{\tilde{b}}
\newcommand{\Time}{T}
\newcommand{\mttun}{\mathtt{k}_1}
\newcommand{\mttdeux}{\mathtt{k}_2}
\newcommand{\mtttrois}{\mtt_3^+}
\newcommand{\bvareps}{\bar{\vareps}}
\newcommand{\transference}{\mathbf{T}}
\newcommand{\esssup}{\mathrm{ess sup}}

\newcommand{\bgM}{b_{\gamma, n}}

\newcommand{\rme}{\mathrm{e}}
\newcommand{\rmE}{\mathrm{E}}

\newcommand{\Lip}{\mathtt{L}}
\newcommand{\tLip}{\tilde{\mathtt{L}}}
\newcommand{\tell}{\tilde{\ell}}
\newcommand{\Lipb}{\mtt_b}
\newcommand{\step}{\ceil{1/\gamma}}
\newcommand{\bstep}{\ceil{1/\bgamma}}
\def\bdisc{b}
\def\bfDd{\mathbf{D}_{\mathrm{d}}}
\def\bfDc{\mathbf{D}_{\mathrm{c}}}
\newcommand{\SDE}{stochastic differential equation}

\newcommand{\measfun}{\mathbb{F}}

\newcommandx{\norm}[2][1=]{\ifthenelse{\equal{#1}{}}{\left\Vert #2 \right\Vert}{\left\Vert #2 \right\Vert^{#1}}}
\newcommandx{\normLigne}[2][1=]{\ifthenelse{\equal{#1}{}}{\Vert #2 \Vert}{\Vert #2\Vert^{#1}}}

\def\bfc{\mathbf{c}}
\def\cbf{\mathbf{c}}

\def\bfX{\mathbf{X}}

\def\bbfX{\tilde{\mathbf{X}}}
\def\bfM{\mathbf{M}}
\def\bfB{\mathbf{B}}

\def\msa{\mathsf{A}}

\def\msk{\mathsf{K}}

\def\msb{\mathsf{B}} 
\def\msc{\mathsf{C}}

\def\msu{\mathsf{U}}

\def\msx{\mathsf{X}}
\def\msy{\mathsf{Y}}

\def\mct{\mathcal{T}}

\def\mcbb{\mathcal{B}}  
\newcommand{\mcb}[1]{\mathcal{B}(#1)}

\def\mcz{\mathcal{Z}}
\def\mcy{\mathcal{Y}}
\def\mcx{\mathcal{X}}

\def\mcf{\mathcal{F}}
\def\mcg{\mathcal{G}}

\def\rset{\mathbb{R}}

\def\nset{\mathbb{N}}
\def\nsets{\mathbb{N}^{\star}}

\def\rmb{\mathrm{b}}

\def\rmw{\mathrm{w}}
\def\rmd{\mathrm{d}}
\def\rmZ{\mathrm{Z}}
\def\rmS{\mathrm{S}}

\def\rme{\mathrm{e}}

\def\rmc{\mathrm{C}}
\def\rmC{\mathrm{C}}

\newcommandx{\functionspace}[2][1=+]{\mathbb{F}_{#1}(#2)}

\newcommand{\argmin}{\operatorname*{arg\,min}}

\newcommandx{\VarDeux}[3][3=]{\operatorname{Var}^{#3}_{#1}\left\{#2 \right\}}

\newcommand{\1}{\mathbbm{1}}

\newcommand{\LeftEqNo}{\let\veqno\@@leqno}

\newcommand{\floor}[1]{\left\lfloor #1 \right\rfloor}
\newcommand{\ceil}[1]{\left\lceil #1 \right\rceil}

\newcommand{\N}{\ensuremath{\mathbb{N}}}

\newcommand{\PE}{\mathbb{E}}
\newcommand{\PP}{\mathbb{P}}

\newcommand{\abs}[1]{\left\vert #1 \right\vert}

\newcommand{\tvnorm}[1]{\| #1 \|_{\mathrm{TV}}}
\newcommand{\tvnormLigne}[1]{\| #1 \|_{\mathrm{TV}}}

\newcommandx{\Vnorm}[2][1=V]{\| #2 \|_{#1}}
\newcommandx{\VnormEq}[2][1=V]{\left\| #2 \right\|_{#1}}

\newcommand{\parenthese}[1]{\left(#1 \right)}
\newcommand{\parentheseLigne}[1]{(#1 )}
\newcommand{\parentheseDeux}[1]{\left[ #1 \right]}
\newcommand{\parentheseDeuxLigne}[1]{[ #1 ]}
\newcommand{\defEns}[1]{\left\lbrace #1 \right\rbrace }
\newcommand{\defEnsLigne}[1]{\lbrace #1 \rbrace }

\newcommand{\ps}[2]{\left\langle#1,#2 \right\rangle}

\newcommand{\proba}[1]{\mathbb{P}\left( #1 \right)}

\newcommand{\probaLigne}[1]{\mathbb{P}( #1 )}
\newcommandx\probaMarkovTilde[2][2=]
{\ifthenelse{\equal{#2}{}}{{\widetilde{\mathbb{P}}_{#1}}}{\widetilde{\mathbb{P}}_{#1}\left[ #2\right]}}
\newcommand{\probaMarkov}[2]{\mathbb{P}_{#1}\left[ #2\right]}
\newcommand{\probaMarkovDD}[1]{\mathbb{P}_{#1}}
\newcommand{\expe}[1]{\PE \left[ #1 \right]}
\newcommand{\expeExpo}[2]{\PE^{#1} \left[ #2 \right]}
\newcommand{\expeLigne}[1]{\PE [ #1 ]}

\newcommand{\expeMarkov}[2]{\PE_{#1} \left[ #2 \right]}

\newcommand{\expeMarkovDD}[1]{\PE_{#1}}
\newcommand{\expeMarkovLigne}[2]{\PE_{#1} [ #2 ]}

\newcommand{\plusinfty}{+\infty}

\newcommand\numberthis{\addtocounter{equation}{1}\tag{\theequation}}

\def\ie{\textit{i.e.}}
\def\as{\textit{a.s}}

\def\eqsp{\;}
\newcommand{\coint}[1]{\left[#1\right)}
\newcommand{\ocint}[1]{\left(#1\right]}
\newcommand{\ooint}[1]{\left(#1\right)}
\newcommand{\ccint}[1]{\left[#1\right]}

\newcommand{\ocintLigne}[1]{(#1]}
\newcommand{\oointLigne}[1]{(#1)}
\newcommand{\ccintLigne}[1]{[#1]}

\newcommandx{\weight}[2][2=n]{\omega_{#1,#2}^N}

\newcommand{\ball}[2]{\operatorname{B}(#1,#2)}

\newcommand{\cball}[2]{\bar{\operatorname{B}}(#1,#2)}
\newcommand{\cballdeux}[2]{\cball{#1}{#2}\times\cball{#1}{#2}}
\newcommand{\diameter}{\operatorname{diam}}

\def\as{\ensuremath{\text{a.s.}}}

\newcommandx\sequence[3][2=,3=]
{\ifthenelse{\equal{#3}{}}{\ensuremath{\{ #1_{#2}\}}}{\ensuremath{\{ #1_{#2}, \eqsp #2 \in #3 \}}}}

\newcommandx\sequenceD[3][2=,3=]
{\ifthenelse{\equal{#3}{}}{\ensuremath{\{ #1_{#2}\}}}{\ensuremath{( #1)_{ #2 \in #3} }}}

\newcommandx{\sequencen}[2][2=n\in\N]{\ensuremath{\{ #1_n, \eqsp #2 \}}}
\newcommandx\sequenceDouble[4][3=,4=]
{\ifthenelse{\equal{#3}{}}{\ensuremath{\{ (#1_{#3},#2_{#3}) \}}}{\ensuremath{\{  (#1_{#3},#2_{#3}), \eqsp #3 \in #4 \}}}}
\newcommandx{\sequencenDouble}[3][3=n\in\N]{\ensuremath{\{ (#1_{n},#2_{n}), \eqsp #3 \}}}

\newcommand{\wrt}{w.r.t.}

\def\iid{i.i.d.}

\def\eg{e.g.}

\newcommand{\opnorm}[1]{{\left\vert\kern-0.25ex\left\vert\kern-0.25ex\left\vert #1 
    \right\vert\kern-0.25ex\right\vert\kern-0.25ex\right\vert}}

\def\generator{\mathcal{A}}

\def\Id{\operatorname{Id}}

\newcommandx{\CPE}[3][1=]{{\mathbb E}_{#1}\left[#2 \middle \vert #3  \right]} 
\newcommandx{\CPVar}[3][1=]{\mathrm{Var}^{#3}_{#1}\left\{ #2 \right\}}
\newcommand{\CPP}[3][]
{\ifthenelse{\equal{#1}{}}{{\mathbb P}\left(\left. #2 \, \right| #3 \right)}{{\mathbb P}_{#1}\left(\left. #2 \, \right | #3 \right)}}

\def\scrA{\mathscr{A}}

\def\measSet{\mathbb{M}}

\newcommandx{\osc}[2][1=]{\mathrm{osc}_{#1}(#2)}

\def\Id{\operatorname{Id}}

\def\transpose{\top}

\def\bD{\bar{D}}
\def\bC{\bar{C}}
\def\brho{\bar{\rho}}
\def\bt{\bar{t}}
\def\bA{\bar{A}}

\def\bc{\bar{c}}
\def\bgamma{\bar{\gamma}}

\def\lambdab{\bar{\lambda}}
\def\blambda{\bar{\lambda}}

\def\tgamma{\tilde{\gamma}}
\def\tC{\tilde{C}}

\def\tc{\tilde{c}}
\def\tvareps{\tilde{\vareps}}
\def\trho{\tilde{\rho}}

\def\Mt{\tilde{M}}
\def\tM{\Mt}

\def\tT{\tilde{T}}

\newcommand{\ensemble}[2]{\left\{#1\,:\eqsp #2\right\}}
\newcommand{\ensembleLigne}[2]{\{#1\,:\eqsp #2\}}

\DeclareMathAlphabet{\mathpzc}{OT1}{pzc}{m}{it}
\def\lyap{\mathpzc{V}}

\newcommand\coupling[2]{\Gamma(\mu,\nu)}

\newcommand{\complementary}{\mathrm{c}}

\def\diam{\mathrm{diam}}

\def\Leb{\lambda}

\def\vareps{\varepsilon}
\def\bvareps{\bar{\varepsilon}}

\def\Phibf{\mathbf{\Phi}}
\def\Psibf{\mathbf{\Psi}}

\def\rker{\mathrm{R}}

\newcommandx{\KL}[2]{\text{KL}\left( #1 | #2 \right)}
\newcommandx{\KLLigne}[2]{\text{KL}( #1 | #2 )}

\def\gaStep
\def\QKer{Q}
\def\Tg{\mathcal{T}_{\gamma}}

\def\distance{\mathbf{d}}
\newcommandx{\wasserstein}[3][1=\distance,3=]{\mathbf{W}_{#1}^{#3}\left(#2\right)}
\newcommandx{\wassersteinLigne}[3][1=\distance,3=]{\mathbf{W}_{#1}^{#3}(#2)}
\newcommandx{\wassersteinD}[1][1=\distance]{\mathbf{W}_{#1}}
\newcommandx{\wassersteinDLigne}[1][1=\distance]{\mathbf{W}_{#1}}

\def\Rcoupling{\mathrm{R}}
\def\Qcoupling{\mathrm{Q}}
\def\Kcoupling{\mathrm{K}}

\def\diagSet{\Delta_{\msx}}
\def\Deltar{\diagSet}
\def\complem{\operatorname{c}}

\def\tildex{\tilde{x}}

\def\sigmaD{\sigma^2}
\def\sigmakD{\sigma^2_k}
\newcommandx{\phibfs}[1][1=]{\pmb{\varphi}_{\sigmaD_{#1}}}
\def\vphibf{\pmb{\varphi}}

\def\funreg{\mct}
\def\kappar{\varpi}

\def\Par{P^{(\mathrm{a})}}

\def\Qar{Q^{(\mathrm{a})}}
\def\eventA{\msa}

\def\transp{\operatorname{T}}

\newcommandx\sequenceg[3][2=,3=]
{\ifthenelse{\equal{#3}{}}{\ensuremath{( #1_{#2})}}{\ensuremath{( #1_{#2})_{ #2 \geq #3}}}}

\def\discrete{\mathrm{d}}

\def\Xar{X^{(\mathrm{a})}}
\def\Yar{Y^{(\mathrm{a})}}
\def\War{W^{(\mathrm{a})}}
\def\Xiar{\Xi^{(\mathrm{a})}}
\def\mcfar{\mcf^{(\mathrm{a})}}

\def\Xart{\tilde{X}^{(\mathrm{a})}}
\def\Yart{\tilde{Y}^{(\mathrm{a})}}

\def\Kker{\Kcoupling}
\def\KkerD{\tilde{\Kcoupling}}
\def\Rker{\Rcoupling}
\def\tRker{\tilde{\Rker}}
\def\Pker{\mathrm{P}}
\def\Qker{\mathrm{Q}}

\def\VlyapD{\lyap}
\def\VlyapDun{\lyap_1}
\def\VlyapDdeux{\lyap_2}
\def\VlyapDtrois{\lyap_3}

\newcommandx{\distV}[1][1=\bfc]{\mathbf{W}_{#1}}
\newcommandx{\distVdeux}[1][1=\lyap_2]{\mathbf{W}_{#1}}

\def\inv{\leftarrow}

\def\mtt{\mathtt{m}}

\def\mttplus{\mathtt{m}^{+}}
\def\mttplusun{\mathtt{m}_1^{+}}
\def\mttplusdeux{\mathtt{m}_2^{+}}
\def\mttplustrois{\mathtt{m}_3^{+}}

\def\cconst{\mathtt{a}}
\def\Run{R_1}
\def\Rdeux{R_2}
\def\Rtrois{R_3}
\def\Rquatre{R_4}
\def\tR{\tilde{R}}
\def\tmttplus{\tilde{\mtt}^+}
\newcommand{\tup}[1]{\textup{#1}}

\newcommand{\stopping}[1]{\T_{\msc,\mathtt{n}_0}^{(#1)}}
\def\wass{\mathcal{W}}
\def\distY{\mathbf{d}}
\def\Xibf{\boldsymbol{\Xi}}
\def\rhomax{\rho_{\rm{max}}}

\def\wasscun{\mathbf{W}_{\bfc_1}}
\def\wasscdeux{\mathbf{W}_{\bfc_2}}
\def\wassctrois{\mathbf{W}_{\bfc_3}}

\title{Convergence of diffusions and their discretizations: from continuous to discrete processes and back}

\author[1]{Valentin De Bortoli \footnote{Email: valentin.debortoli@cmla.ens-cachan.fr}}
\author[1]{Alain Durmus \footnote{Email: alain.durmus@cmla.ens-cachan.fr} }

\affil[1]{CMLA - \'Ecole normale supérieure Paris-Saclay, CNRS, Université Paris-Saclay, 94235 Cachan, France.}

\begin{document}

\maketitle

\begin{abstract}
  In this paper, we establish new quantitative convergence bounds for a class of
  functional autoregressive models in weighted total variation metrics. To
  derive our results, we show that under mild assumptions, explicit minorization
  and Foster-Lyapunov drift conditions hold. 
  The main applications and
  consequences of the bounds we obtain concern the geometric convergence of Euler-Maruyama
  discretizations of diffusions with identity covariance matrix. Second, as a corollary, we
  provide a new approach to establish quantitative convergence of these
  diffusion processes by applying our conclusions in the discrete-time setting
  to a well-suited sequence of discretizations whose associated stepsizes
  decrease towards zero.
\end{abstract}

\section{Introduction}
\label{sec:introduction}

The study of the convergence of Markov processes in general state space is a
very attractive and active field of research motivated by applications in
mathematics, physics and statistics \cite{johndrow2017error}. Among the many
works on the subject, we can mention the pioneering results from
\cite{nummelin1978geometric,nummelin1982geometric,nummelin1983rate} using the
renewal approach. Then, the work of \cite{popov:1977,meyn1993criteria_i} paved
the way for the use of \textit{Foster-Lyapunov drift conditions}
\cite{foster1953stochastic,bremaud1999markov} which, in combination of an
appropriate \textit{minorization condition}, implies $(f,r)$-ergodicity on
general state space, drawing links with control theory, see
\cite{tuominen:tweedie:1994,douc:fort:moulines:soulier:2004,jarner:roberts:2002}. This
approach was successively applied to the study of Markov chains in numerous
papers \cite{chan:1993,chen:tsay:1991,roberts:polson:1994} and was later
extended and used in the case of continuous-time Markov processes in
\cite{khasminskii2011stochastic,meyn1993criteria_ii,meyn1993criteria_iii,down1995exponential,goldys:maslowski:2006,fort:roberts:2005,douc:fort:guillin:2009,veretennikov:1997, devraj2017geometric}.
However, most of these results establish convergence in total variation or in
$V$-norm and are non-quantitative. Explicit convergence bounds in the same metrics
for Markov chains have been established in
\cite{rosenthal:1995,fort:2001,douc:moulines:rosenthal:2004,rosenthal:2002,roberts:tweedie:1999,jones:hobert:2004,lund:tweedie:1996,meyn:tweedie:1994,fort:2002},
driven by the need for stopping rules for Markov Chain Monte Carlo (MCMC)
simulations. To the authors' knowledge, the techniques developed in these papers
have not been adapted to continuous-time Markov processes, except in
\cite{roberts:rosenthal:1996}. One of the main reason is that deriving
quantitative minorization conditions for continuous-time processes seems to be
even more difficult than for their discrete counterparts
\cite{eberle:guillin:zimmer:2018}. Indeed, in most cases, the constants which
appear in minorization conditions are either really pessimistic or hard to
quantify accurately \cite{jones:hobert:2001,qin2018wasserstein}.

Since the last decade, in order to avoid the use of minorization conditions,
other metrics than the total variation distance, or $V$-norm, have been
considered. In particular, Wasserstein metrics have shown to be very interesting
in the study of Markov processes and to derive quantitative bounds of
convergence as well as in the study of perturbation bounds for Markov chains
\cite{rudolf2018perturbation,pillai2014ergodicity}.  For example,
\cite{ollivier:2009,joulin:ollivier:2010,paulin2016mixing} introduced the notion
of Ricci curvature of Markov chains and its use to derive precise bounds on
variance and concentration inequalities for additive functionals.  Following
\cite{hairer2011yet}, \cite{hairer2011asymptotic} generalizes the Harris'
theorem for $V$-norms to handle more general Wasserstein type metrics. In the
same spirit, \cite{butkovsky2014subgeometric} establishes conditions which imply
subgeometric convergence in Wasserstein distance of Markov processes. In
addition, the use of Wasserstein distance has been successively applied to the
study of diffusion processes and MCMC algorithms. In particular,
\cite{eberle2016reflection,eberle:guillin:zimmer:2018} establish explicit
convergence rates for diffusions and McKean Vlasov processes.  Regarding
analysis of MCMC methods, \cite{hairer2014spectral} establishes geometric
convergence of the pre-conditioned Crank-Nicolson algorithm. Besides,
\cite{durmus:moulines:2016,dalalyan2019user,chatterji2018theory,baker2017control}
study the computational complexity in Wasserstein distance to sample from a
log-concave density on $\rset^d$ using appropriate discretizations of the
overdamped Langevin diffusion. One key idea introduced in
\cite{hairer2011asymptotic} and \cite{eberle2016reflection} is the construction
of an appropriate metric designed specifically for the Markov process under
consideration.  The approach of \cite{eberle2016reflection} has then been generalized
in \cite{cheng2018sharp,majka2018non}. While this approach leads to quantitative
results in the case of diffusions or their discretization, we can still wonder
if appropriate minorization conditions can be found to derive similar bounds
using classical results cited above.

In this paper, we show that for a class of functional auto-regressive models, sharp minorization
conditions hold using an iterated Markov coupling kernel. As a result new quantitative convergence
bounds can be obtained combining this conclusion and drift inequalities for well-suited Lyapunov functionals.
We apply them to the study of the Euler-Maruyama discretization of
diffusions with identity covariance matrix under various curvature assumptions on the drift. The rates of convergence we derive in weighted total variation metric in this case
improve the one recently established in
\cite{eberle2018quantitative}. Note that this study is significant to
be able to bound the computational complexity of this scheme when it is applied to the
overdamped Langevin diffusion to sample from a target density $\pi$ on
$\rset^d$. Indeed, while recent papers have established precise bounds
between the $n$-th iterate of the Euler-Maruyama scheme and $\pi$ in
different metrics (\eg~total variation or Wasserstein distances), the convergence of the associated Markov kernel is in general needed to obtain quantitative bounds on the mean square error or concentration inequalities for additive functionals, see \cite{durmus:moulines:2016,joulin:ollivier:2010}.

In the second part of the present paper, we show how the results we derive for
functional auto-regressive models can be used to establish explicit convergence
rates for diffusion processes.  First, we show that, under proper conditions on
a sequence of discretizations, the distance in some metric between the
distributions of the diffusion at time $t$ with different starting points can be
upper bounded by the limit of the distance between the corresponding
discretizations, when the discretization stepsize decreases towards
zero. Similarly, in \cite{kontoyiannis2017approx} general Markov processes are approximated by
  hidden Markov models under a continuous Foster-Lyapunov assumption. Second,
we design appropriate discretizations satisfying the necessary conditions we
obtain and which belong to the class of functional autoregressive models we
study.  Therefore, under the same curvature conditions as in the discrete case,
we get quantitative convergence rates for diffusions by taking the limit in the
bounds we derived for the Euler-Maruyama discretizations.  Finally, the rates we
obtain scale similarly with respect to the parameters of the problem under
consideration to the ones given in
\cite{eberle2016reflection,eberle:guillin:zimmer:2018} for the
Kantorovitch-Rubinstein distance, and improve them in the case of the total
variation norm. Note that in the diffusion case, earlier results were derived in
\cite{chen1997estimation,chen1995estimation,wang1994application}.

The paper is organized as follows. For reader's convenience and to
motivate our results, we begin in \Cref{sec:pres-results-}, with one
of their applications to the specific case of a diffusion over
$\rset^{\dim}$ with identity covariance matrix and its Euler-Maruyama
discretization, in the case where the drift function is strongly
convex at infinity. In \Cref{sec:main-results00}, we present our main
convergence results regarding a class of functional autoregressive
models. We then specialize them to the Euler-Maruyama discretization
of diffusions under various assumptions on the drift function in
\Cref{sec:applications}.  \Cref{sec:quant-conv-bounds} deals with the
convergence of diffusion processes with identity covariance
matrix. More precisely, in \Cref{sec:main-results-3}, we derive
sufficient conditions for the convergence of such processes based on a
sequence of well-suited discretizations. In \Cref{sec:applications-1},
we apply our results to the continuous counterparts of the situations
considered in \Cref{sec:applications}.  For ease of presentation, the
proofs and generalizations of our results are gathered in appendix.
\subsection*{Notation}
\label{sec:notations}
Let $\msa$, $\msb$ and $\msc$ three sets with $\msc \subset \msb$ and $f : \ \msa \to \msb$, we set $f^{\inv}(\msc) = \ensembleLigne{x \in \msa}{ f(x) \in \msc}$. For any $\msa \subset \msb$ and $f: \ \msb \to \msc$ we denote $f|_{\msa}$ the restriction of $f$ to $\msa$.
Let $d \in \nsets$ and $\langle \cdot, \cdot \rangle$ be a scalar product over $\rset^d$, and $\normLigne{\cdot}$ be the corresponding norm.
Let $\msa \subset \rset^d$ and $R \geq 0$, we denote $\diameter(\msa) = \sup_{(x,y) \in \msa} \normLigne{x -y}$ and $\Delta_{\msa, R} = \ensembleLigne{(x,y) \in \msa}{ \norm{x - y} \leq R} \subset \rset^{2d}$ and $\Delta_{\msa} = \Delta_{\msa,0}=  \ensembleLigne{(x,x)}{x \in \msa}$. In this paper,  we consider that $\rset^d$ is endowed with the topology of the norm $\norm{\cdot}$. $\mcb{\rset^d}$ denotes the Borel $\sigma$-field of $\rset^d$ . Let $\msu$ be an open set of $\rset^d$, $n \in \nsets$ and set $\rmC^n(\msu)$ be the set of the $n$-differentiable functions defined over $\msu$. Let $f \in \rmC^1(\msu)$, we denote by $\nabla f$ its gradient. Furthermore, if $f \in \rmC^2(\msu)$ we denote $\nabla^2f$ its Hessian and $\Delta$ its Laplacian.  We also denote $\rmC(\msu)$ the set of continuous functions defined over $\msu$ and for any set $\msa \subset \rset^{\dim}$ and $k \in \nset$ we set $\rmC^k(\msa) = \ensembleLigne{f|_{\msa}}{f \in \rmC^k(U), \ \text{with } \msa \subset \msu \ \text{and $\msu$ open}}$ .Let $f: \msa \to \rset^p$ with $p \in \nsets$. The function $f$ is said to be $L$-Lipschitz with $L \geq 0$ if for any $x,y \in \msa$, $\norm{f(x) - f(y)} \leq L \norm{x-y}$.

Let $\msx \in \mcb{\rset^d}$, $\msx$ is equipped with the trace of $\mcb{\rset^d}$ over $\msx$ defined by $\mcx= \{ \msa \cap \msx \, : \, \msa \in \mcb{\rset^d}\}$. Let $(\msy, \mcy)$ be some measurable space, we denote by $\measfun(\msx, \msy)$ the set of the $\mcx$-measurable functions over $\msx$. For any $f \in \measfun(\msx, \rset)$ we define its essential supremum by $\esssup(f) = \inf \ensembleLigne{a \geq 0}{\lambda(\abs{f}^{\inv}\ooint{a, +\infty}) = 0}$, where $\lambda$ is the Lebesgue measure. Let $\measSet(\mcx)$ be the set of finite signed measures over $\mcx$ and $\mu \in \measSet(\mcx)$. For $f \in \measfun(\msx, \rset)$ a $\mu$-integrable function we denote by $\mu(f)$ the integral of $f$ \wrt \ to $\mu$.  Let $V \in \measfun(\rset^d, \coint{1,+\infty})$. We define the $V$-norm for any $f \in \measfun(\msx, \rset)$ and the $V$-total variation norm for any $\mu \in \measSet(\mcx)$ as follows
\begin{equation}
  \Vnorm{f} = \esssup(\abs{f} / V) \eqsp , \qquad \Vnorm{\mu} = (1/2)\sup_{f \in \measfun(\msx, \rset), \Vnorm{f} \leq 1} \abs{\int_{\rset^d} f(x) \rmd \mu(x)} \eqsp .
\end{equation}
In the case where $V = 1$ this norm is called the total variation norm of $\mu$. Let $\mu, \nu$ be two probability measures over $\mcx$, \ie \ two elements of $\measSet(\mcx)$ such that $\mu(\msx) = \nu(\msx) = 1$.
A probability measure $\zeta$ over $\mcx^{\otimes 2}$ is said to be a transference plan between $\mu$ and $\nu$ if for any $\msa \in \mcx$, $\zeta(\msa \times \mcx) = \mu(\msa)$ and $\zeta(\mcx \times \msa) = \nu(\msa)$. We denote by $\transference(\mu, \nu)$ the set of all transference plans between $\mu$ and $\nu$.
Let $\bfc \in \measfun(\msx \times \msx, \coint{0,+\infty})$. We define the Wasserstein metric/distance $\wassersteinD[\bfc](\mu, \nu)$ between $\mu$ and $\nu$ by
\begin{equation}
  \label{eq:def_distance_wasser}
  \wassersteinD[\bfc](\mu, \nu) = \inf_{\zeta \in \transference(\mu,\nu)} \int_{\msx^2}  \bfc(x,y) \rmd \zeta (x,y) \eqsp. 
\end{equation}
Note that the term Wasserstein metric/distance is an abuse of terminology since $  \wassersteinD[\bfc]$ is only a real metric on a subspace of probability measures on $\msx$ under additional conditions on $\bfc$, \eg~if $\bfc$ is a metric on $\rset^d$, see \cite[Definition 6.1]{villani2009optimal}. 
If $\bfc(x,y) = \norm{x -y}^p$ for $p \geq 1$,  the Wasserstein distance of order $p$ is defined  by $\wassersteinD[p] = \wassersteinD[\bfc]^{1/p}$. Assume that $\bfc(x,y) = \1_{\Deltar^{\complementary}}(x,y)\lyap(x,y)$ with $\lyap \in \measfun(\msx \times \msx, \coint{0,+\infty})$ such that $\lyap$ is symmetric, satisfies the triangle inequality, \ie \ for any $x,y, z \in \msx$, $\lyap(x,z) \leq \lyap(x,y) + \lyap(y,z)$, and for any $x,y \in \msx$, $\lyap(x,y) = 0$ implies $x = y$. Then $\bfc$ is a metric over $\msx^2$ and the associated Wasserstein cost, denoted by $\distV$, is an extended metric. Note that if $\lyap(x,y) = \defEns{V(x) + V(y)}/2$ then $\distV(\mu, \nu) = \Vnorm{\mu - \nu}$
, see \cite[Theorem 19.1.7]{douc:moulines:priouret:soulier:2018}.

Assume that $\mu \ll \nu$ and denote by $\frac{\rmd \mu}{\rmd \nu}$ its Radon-Nikodym derivative. We define the Kullback-Leibler divergence, $\KL{\mu}{\nu}$, between $\mu$ and $\nu$, by
\begin{equation}
  \KL{\mu}{\nu} = \int_{\msx} \log\parenthese{\frac{\rmd \mu}{\rmd \nu}(x)} \rmd \mu(x) \eqsp .
\end{equation}

Let $\mcz$ be a $\sigma$-field. We say that $\Pker : \ \msx \times \mcz \to \coint{0,+\infty}$ is a Markov kernel if for any $x \in \msx$, $\Pker(x, \cdot)$ is a probability measure over $\mcz$ and for any $\msa \in \mcz$, $\Pker(\cdot, \msa) \in \measfun(\msx, \coint{0,+\infty})$. Let $\msy \in \mcb{\rset^d}$ be equipped with $\mcy$ the trace of $\mcb{\rset^d}$ over $\mcy$, $\Pker : \ \msx \times \mcz$ and $\Qker : \ \msy \times \mcz$ be two Markov kernels. We say that $\Kcoupling: \ \msx \times \msy \to \mcz^{\otimes 2}$ is a Markov coupling kernel if for any $(x,y) \in \msx \times \msy$, $\Kcoupling((x,y), \cdot)$ is a transference plan between $\Pker(x, \cdot)$ and $\Qker(y, \cdot)$.


\section{Motivation and illustrative example}
\label{sec:pres-results-}

\subsection{Non-contractive setting}
\label{sec:non-contr-sett}

In this section, we motivate our work with applications of our main results to one
specific example. Let $b : \ \rset^{\dim} \to \rset^{\dim}$ be a drift function, $(\bfB_t)_{t \geq 0}$ be a $d$-dimensional Brownian motion and assume that the \SDE
\begin{equation}
  \label{eq:sde_informal}
  \rmd \bfX_t = b(\bfX_t) \rmd t + \rmd \bfB_t \eqsp ,
\end{equation}
admits a unique strong solution $(\bfX_t)_{t \geq 0}$ on $\rset_+$ for any
starting point $\bfX_0 = x \in \rset^{\dim}$.  We denote by
$(\Pker_t)_{t \geq 0}$ its associated Markov semigroup.  We consider the Euler-Maruyama
discretization of this \SDE, \ie~the homogeneous Markov chain
$(X_k)_{k \in \nset}$, starting from $X_0=x \in \rset^{\dim}$ and defined by the
following recursion: for any $k \in \nset$
\begin{equation}
  \label{eq:euler_informal}
  X_{k+1} = X_k + \gamma b(X_k) + \sqrt{\gamma} Z_{k+1} \eqsp ,
\end{equation}
where $\gamma >0$ is a stepsize and $(Z_k)_{k \in \nsets}$ is a sequence of
\iid~$d$-dimensional Gaussian random variables with zero mean and identity
covariance matrix. We denote by $\Rker_{\gamma}$ its associated Markov kernel.

The first consequence of the results established in the present paper is the
explicit convergence of the Markov chain defined by \eqref{eq:euler_informal} in
a distance which is a mix of the total variation distance and the Wasserstein
distance of order $1$, under the assumption that $b$ is Lipschitz and strongly
convex at infinity.
\begin{theorem}
  \label{thm:informal_1}
  Assume that there exist $\mtt \in \rset$, $\mttplus >0$ and $\Lip, R \geq 0$ such that for any $x,y \in \rset^{\dim}$
  \begin{equation}
    \label{eq:reg_info}
    \norm{b(x) - b(y)} \leq \Lip \norm{x -y} \eqsp , \qquad \langle b(x) - b(y), x - y \rangle \leq -\mtt \norm{x-y}^2 \eqsp ,
  \end{equation}
  and if $\norm{x -y} \geq R$,
  \begin{equation}
    \label{eq:str_cvx_out_info}
    \langle b(x) - b(y), x - y \rangle \leq -\mttplus \norm{x-y}^2 \eqsp .
  \end{equation}
Then there exist $\bgamma >0$, $D_{\bgamma,1}, D_{\bgamma,2}, E_{\bgamma} \geq 0$ and $\lambda_{\bgamma}, \rho_{\bgamma} \in \coint{0, 1}$ with $\lambda_{\bgamma} \leq \rho_{\bgamma}$, which can be explicitly computed, such that for any $\gamma \in \ocint{0, \bgamma}$, $x,y \in \rset^{\dim}$ and $k \in \nset$
\begin{equation}
  \label{eq:true_result_informal}
    \wassersteinD[\bfc](\updelta_x \Rcoupling_{\gamma}^k, \updelta_y \Rcoupling_{\gamma}^k) \leq \lambda_{\bgamma}^{k\gamma/4} [D_{\bgamma,1} \bfc(x,y) + D_{\bgamma,2}\1_{\Delta^{\complementary}}(x,y)] +  E_{\bgamma} \rho_{\bgamma}^{k\gamma/4} \1_{\Delta^{\complementary}}(x,y)  \eqsp ,  
  \end{equation}
  where $\bfc(x,y) = \1_{\Delta^{\complementary}}(x,y) (1 + \norm{x-y}/R)$, $\Delta = \ensembleLigne{(x,x)}{x \in \rset^{\dim}}$ and
  $\Rker_{\gamma}$ is the Markov kernel associated with \eqref{eq:euler_informal}.
\end{theorem}

\begin{proof}
  The result is a direct consequence of \Cref{propo:cvx_outside_bounds} and the corresponding discussion in \Cref{sec:strongly-convex-at}.
\end{proof}

This result is derived as a specific case of a more general theorem for a class
of functional autoregressive models, see \Cref{theo:discrete_contrac_wass_D_v2}
and \Cref{sec:main-results00}. Its proof relies on the use of an extended
Foster-Lyapunov drift assumption as well as a minorization condition on the
Markov chain \eqref{eq:euler_informal}.  As an important consequence, curvature
assumptions on the drift (such as strong convexity at infinity) can be omitted
if we instead assume some Foster-Lyapunov condition, similarly to \cite[Theorem
6.1]{eberle2016reflection} and \cite[Theorem
2.1]{eberle:guillin:zimmer:2018}. 

The result derived in \Cref{thm:informal_1} has several important applications which we
gather in the following corollary.

\begin{corollary}
  \label{prop:collec_result_info}
  Assume that there exist $\mtt \in \rset$, $\mttplus >0$ and $\Lip, R \geq 0$
  such that \eqref{eq:reg_info} and \eqref{eq:str_cvx_out_info} are
  satisfied. Then, there exist $\bgamma >0$, $E_{\bgamma, 1}, E_{\bgamma, 2} \geq 0$ such that for any $\gamma \in \ocint{0, \bgamma}$, $x, y \in \rset^{\dim}$ and $k \in \nset$ we have 
  \begin{align}
    \tvnorm{\updelta_x \Rker_{\gamma}^k - \updelta_y \Rker_{\gamma}^k} & \leq     \wassersteinD[\bfc](\updelta_x \Rker_{\gamma}^k, \updelta_y \Rker_{\gamma}^k) \leq E_{\bgamma, 1} \rho_{\bgamma}^{k \gamma / 4} \bfc(x,y) \eqsp , \label{eq:wc_cv_info} \\  
    \wassersteinD[1](\updelta_x \Rker_{\gamma}^k, \updelta_y \Rker_{\gamma}^k) &\leq E_{\bgamma, 2} \rho_{\bgamma}^{k \gamma / 4} \norm{x-y} \eqsp , \label{eq:w1_cv_info}
  \end{align}
  where $\bfc(x,y) = \1_{\Delta^{\complementary}}(x,y) (1 + \norm{x-y}/R)$, $\Delta = \ensembleLigne{(x,x)}{x \in \rset^{\dim}}$ and $\rho_{\bgamma}$ is given in \eqref{eq:true_result_informal}.
In addition, for any
  $p \in \nset$ and $\upalpha \in \ooint{p, +\infty}$ there exists $E_{\bgamma, \upalpha} \geq 0$ such that for any $\gamma \in \ocint{0, \bgamma}$, $x,y \in \rset^d$ and $k\in \nset$ we have
  \begin{equation}
            \wassersteinD[p](\updelta_x \Rker_{\gamma}^k, \updelta_y \Rker_{\gamma}^k) \leq E_{\bgamma,  \upalpha} \rho_{\bgamma}^{k \gamma / (4\upalpha)} (\norm{x-y} + \norm{x-y}^{1/\upalpha}) \eqsp.   \label{eq:wp_cv_info} 
  \end{equation}
The constants $\bgamma$, $\{E_{\bgamma,i} \, : \, i=1,2,3\}$ and $E_{\bgamma,\upalpha}$ can be explicitly computed. 
\end{corollary}

\begin{proof}
The estimate \eqref{eq:wc_cv_info} is a direct consequences of \Cref{thm:informal_1}. The two inequalities \eqref{eq:w1_cv_info} and \eqref{eq:wp_cv_info} follow from \Cref{coro:w1_wp}.
\end{proof}

Note that the same rate $\rho_{\bgamma}$ appears in the inequalities \eqref{eq:true_result_informal}, \eqref{eq:wc_cv_info},  \eqref{eq:w1_cv_info} and \eqref{eq:wp_cv_info}. \Cref{sec:quant-conv-bounds} is devoted to the extension of our discrete-time results to their continuous-time counterparts.
Note also that  \Cref{thm:informal_2} and its consequences still hold if we only assume a local Lipschitz assumption, see the condition  \Cref{ass:loc_lip_bb}.

\begin{theorem}
  \label{thm:informal_2}
    Assume that there exist $\mtt\in \rset$, $\mttplus >0$ and $\Lip, R \geq 0$
  such that \eqref{eq:reg_info} and \eqref{eq:str_cvx_out_info} are
  satisfied.
  Then there exist $D_{1}, D_{2}, E \geq 0$ and $\lambda, \rho \in \coint{0, 1}$ with $\lambda \leq \rho$ such that for any $x,y \in \rset^{\dim}$ and $t \geq 0$
  \begin{equation}
    \label{eq:thm:informal_2}
\tvnorm{\updelta_x \Pker_t- \updelta_y \Pker_t} \leq    \wassersteinD[\bfc](\updelta_x \Pker_t, \updelta_y \Pker_t) \leq \lambda^{t/4} [D_1 \bfc(x,y) + D_2 \1_{\Delta^{\complementary}}(x,y)] +  E \rho^{t/4} \1_{\Delta^{\complementary}}(x,y)  \eqsp ,  
  \end{equation}
  where $\bfc(x,y) = \1_{\Delta^{\complementary}}(x,y) (1 + \norm{x-y}/R)$, $\Delta = \ensembleLigne{(x,x)}{x \in \rset^{\dim}}$, $(\Pker_t)_{t \geq 0}$ is the Markov semigroup associated with \eqref{eq:sde_informal} and 
\begin{align}
  &D_{1} = \lim_{\bgamma \to 0} D_{\bgamma, 1}\eqsp , \quad D_{2} = \lim_{\bgamma \to 0} D_{\bgamma, 2}\eqsp , \quad E = \lim_{\bgamma \to 0}E_{\bgamma}\eqsp , \quad \lambda = \lim_{\bgamma \to 0}\lambda_{\bgamma} \eqsp , \quad \rho = \lim_{\bgamma \to 0}\rho_{\bgamma} \eqsp,
\end{align}
and $D_{\bgamma,1}, D_{\bgamma,2}, E_{\bgamma}, \lambda_{\bgamma}, \rho_{\bgamma}$ are given in \Cref{thm:informal_1}.
\end{theorem}
\begin{proof}
  This result follows from \Cref{propo:cv_wass_continuous}.  
\end{proof}
Note that the constants $D_1,D_2,E,\lambda$ and $\rho$ have explicit expressions,
see the corresponding discussion in \Cref{sec:applications-1} after
\Cref{propo:cv_wass_continuous}. In addition, the rate $\rho$ and $\lambda$ in \eqref{eq:thm:informal_2} are independent of the dimension $d$. This is a significant improvement compared to the convergence results in total variation derived in \cite[Theorem 2.1]{eberle:guillin:zimmer:2018} which imply a convergence rate which scales exponentially in the dimension, under the setting we consider. 
Similarly, we derive a continuous counterpart of \Cref{prop:collec_result_info} in the continuous time setting, see \Cref{coro:cont_w1_wp}.

As stated before, the convergence rates $\rho_{\bgamma}$, $\rho$, $\lambda_{\bgamma}$, $\lambda$,
given in \Cref{thm:informal_1} and \Cref{thm:informal_2}
can be explicitly computed. More precisely, we obtain the following expressions (up to logarithmic terms) with respect to the parameters $\mtt$, $\Lip$ and $R$ in the case $-\mtt R^2 \gg 1$, see \Cref{sec:drift-cond-conv}, \Cref{propo:cvx_outside_bounds}, Equations \eqref{eq:rho_bgamma_wass_majo_b}  and  \eqref{eq:bound_rho_a_bgamma_0_comp}:
    \begin{align}
      & \hspace{-0.9cm}\log(\log^{-1}(\rho_{\bgamma}^{-1})) \simeq -(\mtt R^2/4)  \sup_{\gamma \in \ocint{0, \bgamma}} \defEns{ \parenthese{1 - \frac{\gamma \Lip^2}{2\mtt}}\parenthese{1 - \exp\parentheseDeux{\frac{R^2(2\mtt - \gamma \Lip^2)}{1 - 2\mtt \gamma +\gamma^2 \Lip^2}}}^{-1} }    \label{eq:log_rho_gamma} \\ & \hspace{-0.9cm} \log(\log^{-1}(\rho^{-1})) \simeq -(\mtt R^2 / 4) \times (1-\rme^{2\mtt R^2})^{-1} \eqsp , \label{eq:log_rho}
                    \\ & \hspace{-0.9cm} \log(\lambda_{\bgamma}) = -\mttplus/2 + \bgamma \Lip^2/4                                                             \eqsp, \qquad  \log(\lambda) = -\mttplus/2 \eqsp.  \qquad
      \end{align}
    where $\simeq$ denotes equality up to logarithmic factors.

It is sensible to obtain two different convergence rates
    $\lambda_{\bgamma},\rho_{\bgamma}$ (resp. $\lambda,\rho$) in
    \Cref{thm:informal_1} (resp. in \Cref{thm:informal_2}), one
    characterizing the forgetting of the initial distance between the
    two starting points $x,y \in \rset^d$, corresponding to a burn-in
    period, and the other one characterizing the effective
    convergence. In addition, note that
    $\lambda_{\bgamma} \ll \rho_{\bgamma}$ and $\lambda \ll \rho$ if $-\mtt R^2 \gg 1$.

  We now compare these results and the rates obtained in
\eqref{eq:log_rho_gamma}-\eqref{eq:log_rho} with recent works studying the convergence of
  the Markov chain defined by \eqref{eq:euler_informal} and/or the
  corresponding diffusion process \eqref{eq:sde_informal} in the same framework, \ie~under the conditions \eqref{eq:reg_info} and \eqref{eq:str_cvx_out_info}.
Note that the same conclusions hold under more general Foster-Lyapunov drift conditions but it would make the comparison more involved. Note also we are still able to derive convergence results under weaker  curvature assumptions on the drift. The discussion  is postponed to \Cref{sec:other-curv-cond_disc} for the discrete setting and \Cref{sec:other-curv-cond_cont} for the continuous setting.
  
First, a major difference between our work and the ones mentioned below is that
we use a completely different technique to establish our results. Indeed,  all of them  follow the approach initiated in \cite{eberle2011reflection},
designing a suitable coupling and distance function of the form
$\cbf(x,y) = f(\norm{x-y})$, for any $x,y \in \rset^d$, with
$f: \rset_+ \to \rset_+$, to obtain a geometric contraction in
$\wassersteinDLigne[\bfc]$ for either the Markov chain \eqref{eq:euler_informal}
or the diffusion \eqref{eq:sde_informal} under the conditions
\eqref{eq:reg_info}-\eqref{eq:str_cvx_out_info}.  In this paper, we follow a
different path and derive convergence estimates using minorization and
Foster-Lyapunov drift conditions, adapting the technique used in
\cite{douc:moulines:priouret:soulier:2018} and the references therein.  It has
been thought for a long time that such an approach only gives very pessimistic
convergence bounds \cite{eberle:guillin:zimmer:2018}. We now compare more
specifically our results with the ones obtained following the work of
\cite{eberle2011reflection} and show that in fact our technique inspired by
classical methods to establish geometric convergence of Markov chains gives very
sharp estimates, improving and simplifying the results obtained in the existing
literature. This discussion and its conclusion are summarized in
\Cref{tab:cv_res}.   In the rest of this section, $C \geq 0$ stands for a  positive constant which may
  be different at each occurrence.



First we compare our work with the results of \cite{eberle2018quantitative}
which extend to the discrete setting the estimates of \cite{eberle2016reflection}.
The authors use the following cost function defined for any $x,y \in \rset^{\dim}$ by
\begin{equation}
  \label{eq:bca}
  \bfc_a(x,y) = a \1_{\Delta^{\complementary}}(x,y) + f_a(\norm{x-y}) \eqsp ,
\end{equation}
where $a \geq 0$ and $f_a$ is given in \cite[Equation
(2.53)]{eberle2018quantitative}.
Note that the cost $\bfc_a$ is close to the one introduced in
\Cref{thm:informal_1}. Then, 
\cite[Theorem 2.10]{eberle2018quantitative} states that  if
$a \in \ccintLigne{2 \gamma^{1/2}, \Phibf_{E}(R)}$ where $\Phibf_{E}$ is given in
\cite[Theorem 2.10]{eberle2018quantitative}, then there exist $\bgamma_{a} >0$ and $\rho_{a} \in \coint{0,1}$ such that for any $\gamma \in \ocint{0, \bgamma_{a}}$, $x,y \in \rset^{\dim}$ and $k \in \nset$,
\begin{equation}
  \label{eq:contrac_markov_eberle_1}
  \wassersteinD[\bfc_a](\updelta_x \Rker_{\gamma}^k, \updelta_y \Rker_{\gamma}^k) \leq \rho^{k \gamma}_{a} \bfc_a(x,y) \eqsp .
\end{equation}
Compared to our results \Cref{thm:informal_1},
\eqref{eq:contrac_markov_eberle_1} only gives one convergence rate $\rho_{a}$
and does not dissociate the forgetting of the initial distance between the
starting points $x,y \in \rset^d$ from the long-term behavior. In addition, $a$
may depend on $\gamma$, since it is required that
$a \in \ccintLigne{2 \gamma^{1/2}, \Phibf(R)}$ and $\bgamma_{a} < \bgamma$ where
$\bgamma$ is given by \Cref{thm:informal_1}. Omitting the dependency of $a$ and
$\rho_{a}$ with respect to $\gamma$ for the sake of simplicity, and applying
\cite[Theorem 2.10]{eberle2018quantitative} yield
\begin{equation}
  \log(\log^{-1}(\rho^{-1}_{a})) \simeq -\mtt R^2 / c_1 \eqsp , \text{ with }
  \label{eq:c1}
  c_1 = 16^{-1} \int_{1/4}^{3/8} (1-\rme^{u-1/2}) \vphibf(u) \rmd u \leq 0.00051 \eqsp ,
\end{equation}
where for any $t \in \rset$, $\vphibf(t) = (2 \uppi)^{-1/2}\exp(-t^2/2)$. It is worth noticing that in the case we are interested in, $-\mtt R^2 \gg 1$, we obtain that our rate given by \eqref{eq:log_rho_gamma} satisfies $\rho_{\bgamma} \ll \rho_{a}$ (also omitting dependency of $\rho_{\bgamma}$ with respect to $\gamma$).

Let $\bfc_b$ be defined
for any $x,y \in \rset^{\dim}$ by $\bfc_b(x,y) = f_b(\norm{x-y})$ with $f_b$
given in \cite[Equation (2.68)]{eberle2018quantitative}. Then, 
\cite[Theorem 2.12]{eberle2018quantitative} implies that there exist $\bgamma_b >0$ and $\rho_{b} \in \coint{0,1}$
such that for any $\gamma \in \ocint{0, \bgamma_b}$, $x,y \in \rset^{\dim}$ and $k \in \nset$,
\begin{equation}
  \label{eq:contrac_markov_eberle_2}
  \wassersteinD[\bfc_b](\updelta_x \Rker_{\gamma}^k, \updelta_y \Rker_{\gamma}^k) \leq \rho^{k \gamma}_{b} \bfc_b(x,y) \eqsp .
\end{equation}
Note that \eqref{eq:contrac_markov_eberle_1} implies convergence bounds both
with respect to $\wassersteinD[1]$ and the total variation distance whereas
\eqref{eq:contrac_markov_eberle_2} implies convergence bounds with respect to
$\wassersteinD[1]$ only. Once again, omitting the dependency with respect to $\gamma$,
we obtain that the rate satisfies
\begin{equation}
  \label{eq:eberle2}
  \log(\log^{-1}(\rho^{-1}_{b})) \simeq -49 \mtt R^2 / (6 c_2) \eqsp , 
\end{equation}
with
\begin{equation}
  \label{eq:c2}
  c_2 = 4 \min \left( \int_0^{1/2}u^2(1-\rme^{u-1/2}) \vphibf(u) \rmd u, (1- \rme^{-1}) \int_0^{1/2}u^3 \vphibf(u) \rmd u \right) \leq  0.0072 \eqsp ,
\end{equation}
and we obtain that our rate given by \eqref{eq:log_rho_gamma} satisfies $\rho_{\bgamma} \ll \rho_{a}$ when $-\mtt R^2 \gg 1$.

  We now compare our results with the ones derived in \cite{majka2018non}. For fair comparison,
  since \cite{majka2018non} does not assume a one-sided
  Lipschitz condition but only a global Lipschitz condition we set $\mtt = -\Lip$ in the next paragraph.  This paper
extends the techniques of \cite{eberle2018quantitative,eberle2016reflection} to
deal with $\wassersteinD[2]$. It is shown in
\cite[Theorem 2.1]{majka2018non} that there exist $\bgamma_{c} >0$ and $\rho_{c} \in \coint{0,1}$ such that for
any $\gamma \in \ocint{0, \bgamma_{c}}$, $x,y \in \rset^{\dim}$ and $k \in \nset$,
\begin{equation}
  \label{eq:contrac_markov_majka_1}
  \wassersteinD[\bfc_c](\updelta_x \Rker_{\gamma}^k, \updelta_y \Rker_{\gamma}^k) \leq \rho^{k \gamma}_{c} \bfc_c(x,y) \eqsp ,
\end{equation}
with $\bfc_c$ given for any $x,y \in \rset^{\dim}$ by $\bfc_c(x,y) = f_c(\norm{x-y})$ and $f_c$ given in \cite[Equation (2.11)]{eberle2018quantitative}. Note that this result implies convergence bounds with respect to $\wassersteinD[1]$
and $\wassersteinD[2]$. In particular, we have for any $\gamma \in \ocint{0, \bgamma}$, $x,y \in \rset^{\dim}$ and $k \in \nset$,
\begin{equation}
  \label{eq:contrac_markov_majka_1}
  \wassersteinD[2](\updelta_x \Rker_{\gamma}^k, \updelta_y \Rker_{\gamma}^k) \leq C \rho^{k \gamma/2}_{c} \bfc_c^{1/2}(x,y) \leq C \rho_c^{k \gamma/2} (\norm{x-y} + \norm{x-y}^{1/2}) \eqsp .
\end{equation}
In addition, it holds that
\begin{equation}
  \label{eq:majka}
  \log(\log^{-1}(\rho^{-1}_c)) \simeq \Lip R^2 / (6c_2) \eqsp \text{ where $c_2$ is defined by \eqref{eq:c2}} \eqsp,
\end{equation}
and therefore our rate also satisfies $\rho_{\bgamma} \ll \rho_{c}$ when $\Lip R^2 \gg 1$.

The results of \cite{eberle2018quantitative,majka2018non} both extend, and
generalize, in the discrete-time setting the techniques used in
\cite{eberle2016reflection}.  In the latter, contraction results for the
semigroup $(\Pker_t)_{t \geq 0}$ are obtained with respect to
$\wassersteinD[\bfc_e]$, where for any $x,y \in \rset^{\dim}$,
$\bfc_e(x,y) = f_e(\norm{x-y})$ and $f_e$ is defined by \cite[Equation
(2.6)]{eberle2016reflection}. In particular, in \cite[Corollary
2.3]{eberle2016reflection}, it is shown that there exists
$\rho_e \in \coint{0, 1}$ such that for any $x,y \in \rset^{\dim}$ and
$t \geq 0$
\begin{equation}
  \wassersteinD[\bfc_e](\updelta_x \Pker_t, \updelta_y \Pker_t) \leq \rho^t_e \bfc_e(x,y) \eqsp .
\end{equation}
Note that this result implies convergence bounds in $\wassersteinD[1]$, see \cite[Corollary 2.3]{eberle2016reflection}.
The rate is given \cite[Lemma 2.9]{eberle2016reflection} and, in the case $-\mtt R^2 \gg 1$, we have 
\begin{equation}
  \label{eq:rate_eberle_correct}
  \log(\log^{-1}(\rho^{-1}_e)) \simeq -\mtt R^2 / 4 \eqsp ,
\end{equation}
which is better than our rate in the continuous-time case\footnote{Note that in \cite[Lemma 2.9, Equation (2.18)]{eberle2016reflection} the stated result implies that $\log(\log^{-1}(\rho^{-1})) \simeq \Lip R^2 /8$ if $\kappa(r) \geq -\Lip r$ for any $r \geq 0$, where $\kappa$ is defined in \cite[p.5]{eberle2016reflection}. However,
note that if $b$ is $\Lip$-Lipschitz then $\kappa(r) \geq -2 \Lip$ and \eqref{eq:rate_eberle_correct} follows.}. However, note that we
derive our results in $\wassersteinD[1]$ from our estimates with respect to
$\wassersteinD[\bfc]$ with $\bfc$ given in \Cref{thm:informal_1}, which controls
both $\wassersteinD[1]$ and the total variation norm. Also, the discrepancy
between our rate and the one of \eqref{eq:rate_eberle_correct} is controlled by $(\rme^{-2 \mtt R^2} - 1)^{-1}$ which is small when $-\mtt R^2$ is large.

Finally we compare our continuous-time results with the ones of
\cite{luo2016exponential}.  It is shown in \cite[Theorem
1.3]{luo2016exponential} that for any $p>1$ there exist $\rho_f \in \coint{0, 1}$ and $C \geq 0$
such that for any $x,y \in \rset^{\dim}$ and $t \geq 0$
\begin{equation}
  \label{eq:rate_luo_correct}
  \wassersteinD[p](\updelta_x \Pker_t, \updelta_y \Pker_t) \leq C\rho^t_f \defEns{\norm{x-y} + \norm{x-y}^{1/p}} \eqsp ,
\end{equation}
and the rate is given in \cite[Theorem
1.3]{luo2016exponential} by
\begin{equation}
  \log(\log^{-1}(\rho^{-1}_f)) = (-\mtt + \mttplus) R^2 / 4 \eqsp .
\end{equation}
The additional term $\mttplus R^2 / 4$ does not appear in our rates\footnote{Similarly to \cite{eberle2016reflection},
in \cite[Theorem 1.3]{luo2016exponential} the stated result implies that $\log(\log^{-1}(\rho^{-1})) \simeq \Lip R^2 /2$ if $\kappa(r) \leq \Lip r$ for any $r \geq 0$ and $\kappa(r) \leq -\mttplus r$ for $r \geq R$, where $\kappa$ is defined in \cite[Equation (1.4)]{luo2016exponential}. However, note that if $b$ is $\Lip$-Lipschitz and $\mttplus$ strongly convex outside of $\cball{0}{R}$ we have $\kappa(r) \leq \Lip r/2$ for any $r \geq 0$ and $\kappa(r) \leq -\mttplus r/2$ for any $r \geq R$ and \eqref{eq:rate_luo_correct} follows.}.
As a consequence our rate is better as soon as
\begin{equation}
  \label{eq:16}
\mttplus \geq -\mtt / (\rme^{-2\mtt R^2}-1) \eqsp .
\end{equation}
\Cref{tab:cv_res} gives a summary of the comparisons we made above.  

\begin{table}[h]
  \centering
  \begin{tabular}{|c|c|c|c|c|c|}
    \hline
    Reference & Wasserstein distance & distance bound & (D) & (C) & (NR)\\
    \hline \hline
    \cite{eberle2018quantitative} & $\tvnorm{\cdot}$  & $\1_{\Delta^{\complementary}}(x,y) + \norm{x-y}$ & \checkmark &  &  
    $7840$ \\ 
              & $\wassersteinD[1]$ & $\norm{x-y}$ & \checkmark & & 
                                                                   $4536$ \\
    \hline \cite{majka2018non}  & $\wassersteinD[2]$ & $\norm{x-y} + \norm{x-y}^{1/2}$ &\checkmark & & 
                                                                                                       $332$ \\    
    \hline \cite{eberle2016reflection} & $\wassersteinD[1]$ & $\norm{x-y}$ &  & \checkmark & $1$ \\
    \hline \cite{luo2016exponential} & $\wassersteinD[p]$ & $\norm{x-y} + \norm{x-y}^{1/p}$ &  & \checkmark & $1 - \mttplus /  \mtt$ \\
    \hline  & $\tvnorm{\cdot}$  & $\1_{\Delta^{\complementary}}(x,y) + \norm{x-y}$  & \checkmark & \checkmark  & $(1 - \rme^{2\mtt R^2})^{-1}$ \\
           Ours   & $\wassersteinD[1]$  & $\norm{x-y}$  & \checkmark & \checkmark  & idem \\
              & $\wassersteinD[p]$  & $\norm{x-y} + \norm{x-y}^{1/\upalpha}$  & \checkmark & \checkmark  &  idem \\
    \hline
  \end{tabular}
  \caption{
    Every line of the table reads as follows. Suppose ``Wasserstein distance'' reads $\wassersteinD[\bfc_1]$ and ``distance bound'' reads $\bfc_2(x,y)$ then: if (D) is checked, there exist $C \geq 0$ and $\rho \in \coint{0,1}$ such that for any $x,y \in \rset^{\dim}$ and $k \in \nset$, $\wassersteinD[\bfc_1](\updelta_x \Rker_{\gamma}^{k}, \updelta_y \Rker_{\gamma}^{k}) \leq C \rho^{k \gamma}\bfc_2(x,y)$ for $\gamma$ small enough. If (C) is checked, there exist $C \geq 0$ and $\rho \in \coint{0,1}$ such that for any $x,y \in \rset^{\dim}$ and $t \geq 0$,  $\wassersteinD[\bfc_1](\updelta_x \Pker_{t}, \updelta_y \Pker_t) \leq C \rho^{t} \bfc_2(x,y)$. In addition, if the normalized rate ``(NR)'' reads $\beta$ we have $-4 \log(\log^{-1}(\rho^{-1})) / (\mtt R^2)  \simeq \beta$ (with $\mtt$ replaced by $-\Lip$ in the case of \cite{majka2018non}). Note that for the sake of simplicity we omit the dependency with respect to $\bgamma$ in the present analysis. The exact distances used in papers with which we compare our results, are given in \cite[Equation (2.53)]{eberle2018quantitative}, \cite[Equation (2.11)]{majka2018non}, \cite[Equation (2.6)]{eberle2016reflection} and \cite[Equation (2.4)]{luo2016exponential}. Note that $p \in \nset$ and $\upalpha \in \oointLigne{p, +\infty}$. }
  \label{tab:cv_res}
\end{table}

\subsection{An illustrative example}
\label{sec:an-illustr-example}

We now consider a toy example to justify the setting under study in the previous section.
Consider the following Gaussian mixture distribution $\pi$ whose Radon-Nikodym density with respect to
the Lebesgue measure $\lambda$ is given for any $x \in \rset$ by
\begin{equation}(\rmd \pi / \rmd \Leb)(x) = (2\sqrt{2\uppi \sigma^2})^{-1} \exp[-x^2/ (2\sigma^2)] + (2\sqrt{2\uppi \sigma^2})^{-1} \exp[-(x-\mean)^2/ (2\sigma^2)] \eqsp ,
\end{equation}
where $\sigma >0$ and $\mean \geq 0$. For any $x \in \rset$, we have $(\rmd \pi / \rmd \Leb)(x) \propto \rme^{-U(x-\mean/2)}$ and for any $\xb \in \rset$
\begin{equation}
  U(\xb) = \xb^2 / (2\sigma^2) - \log\parentheseDeux{\mathrm{cosh}(\mean \xb / (2\sigma^2))} \eqsp ,
\end{equation}
Note that $U'$ is $\Lip$-Lipschitz with $\Lip = \sigma^{-2} \max\defEnsLigne{1, (\mean/(2\sigma))^2 - 1}$ and that $U$ is convex if and only if $\mean \leq 2 \sigma$. Also, we obtain that $b = -U'$ satisfies \eqref{eq:reg_info} with $  \Lip = \sigma^{-2} \max\defEnsLigne{1, (\mean/(2\sigma))^2 - 1}$, $ R = 2 \mean$, $\mttplus = 1/(2\sigma^2)$. 

We now consider the Markov chain \eqref{eq:euler_informal} with $b = - U'$ and
its associated Markov kernel $\Rker_{\gamma}$ for $\gamma > 0$. Let
$x_0 \in \rset$ and we define $\log(\rhoexp)$ to be the slope of the function
$n \mapsto \log(\tvnorm{\updelta_{x_0} \Rker_{\gamma}^n - \pi})$. Note that this
slope is computed only until
$\log(\tvnorm{\updelta_{x_0} \Rker_{\gamma}^n - \pi})$ reaches a given
precision, since for $\gamma >0$ small enough there exists a probability measure
$\pi_{\gamma}$ such that
$\tvnorm{\updelta_{x_0} \Rker_{\gamma}^n - \pi_{\gamma}} \to 0$ and
$\pi_{\gamma} \neq \pi$. In what follows we compare $\log(\rhoexp)$ with our
estimates.

\begin{figure}
  \centering
  \subfloat[]{\includegraphics[width=0.33\linewidth]{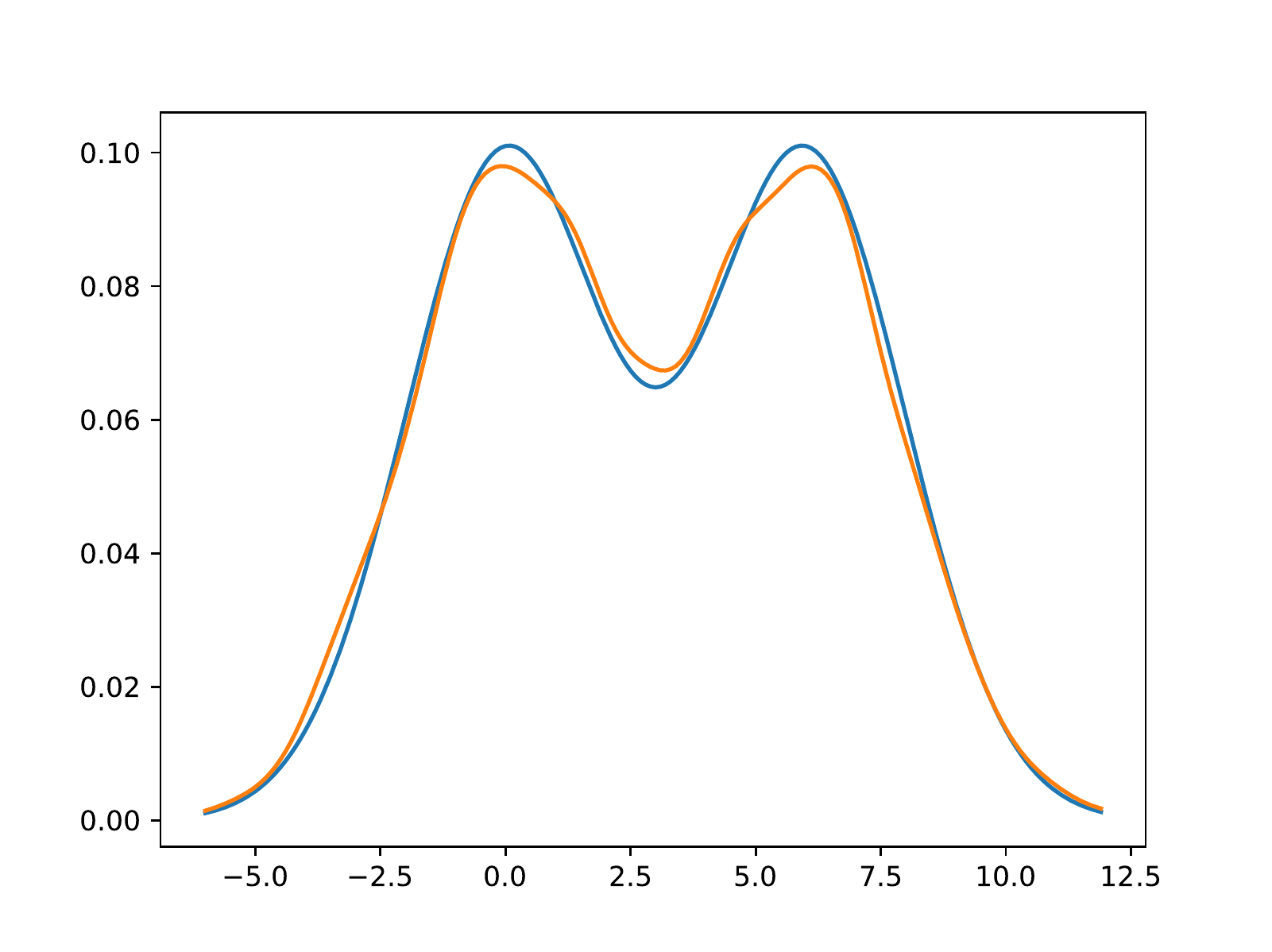}} \hfill
  \subfloat[]{\includegraphics[width=0.33\linewidth]{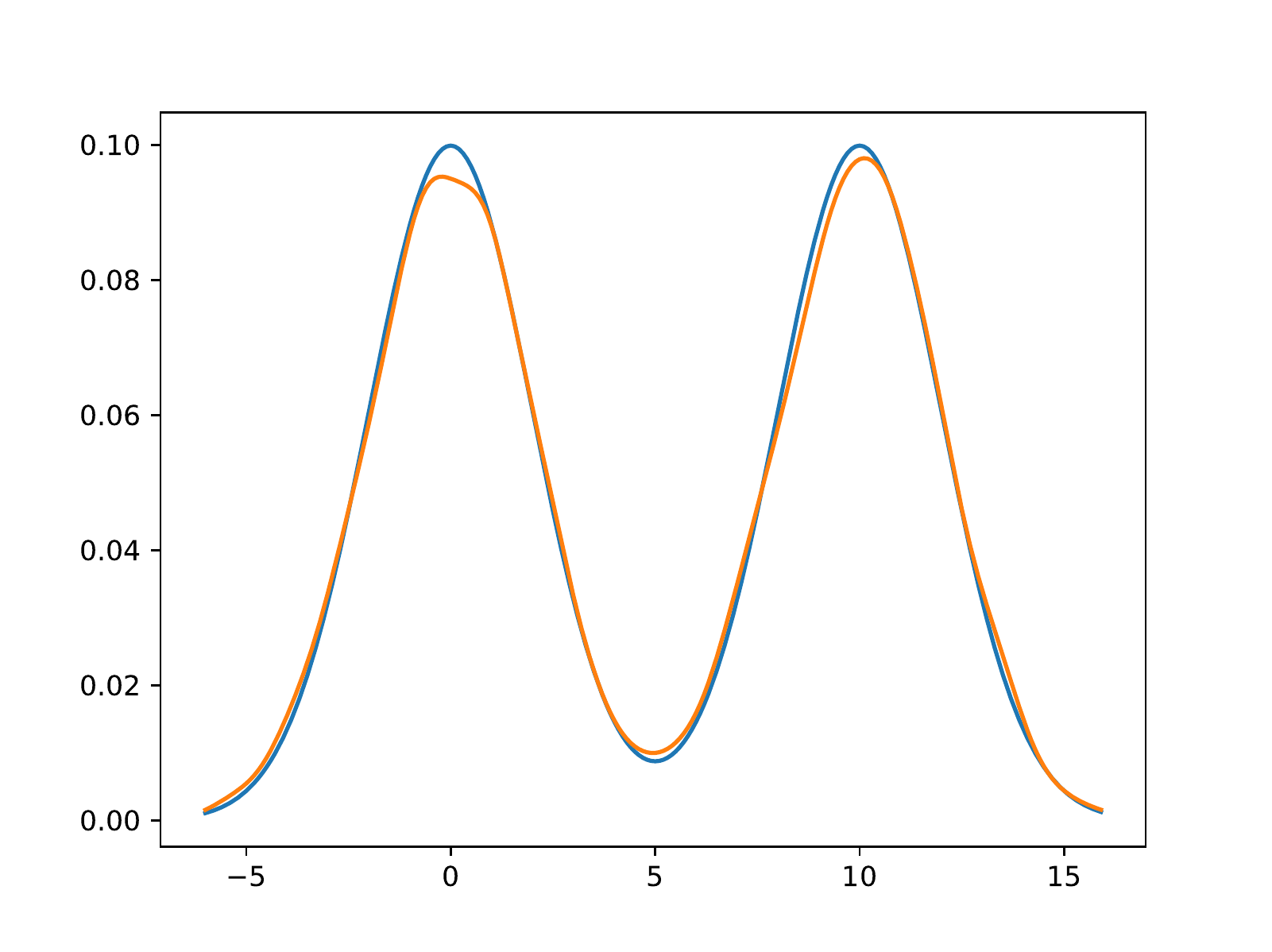}} \hfill
  \subfloat[]{\includegraphics[width=0.33\linewidth]{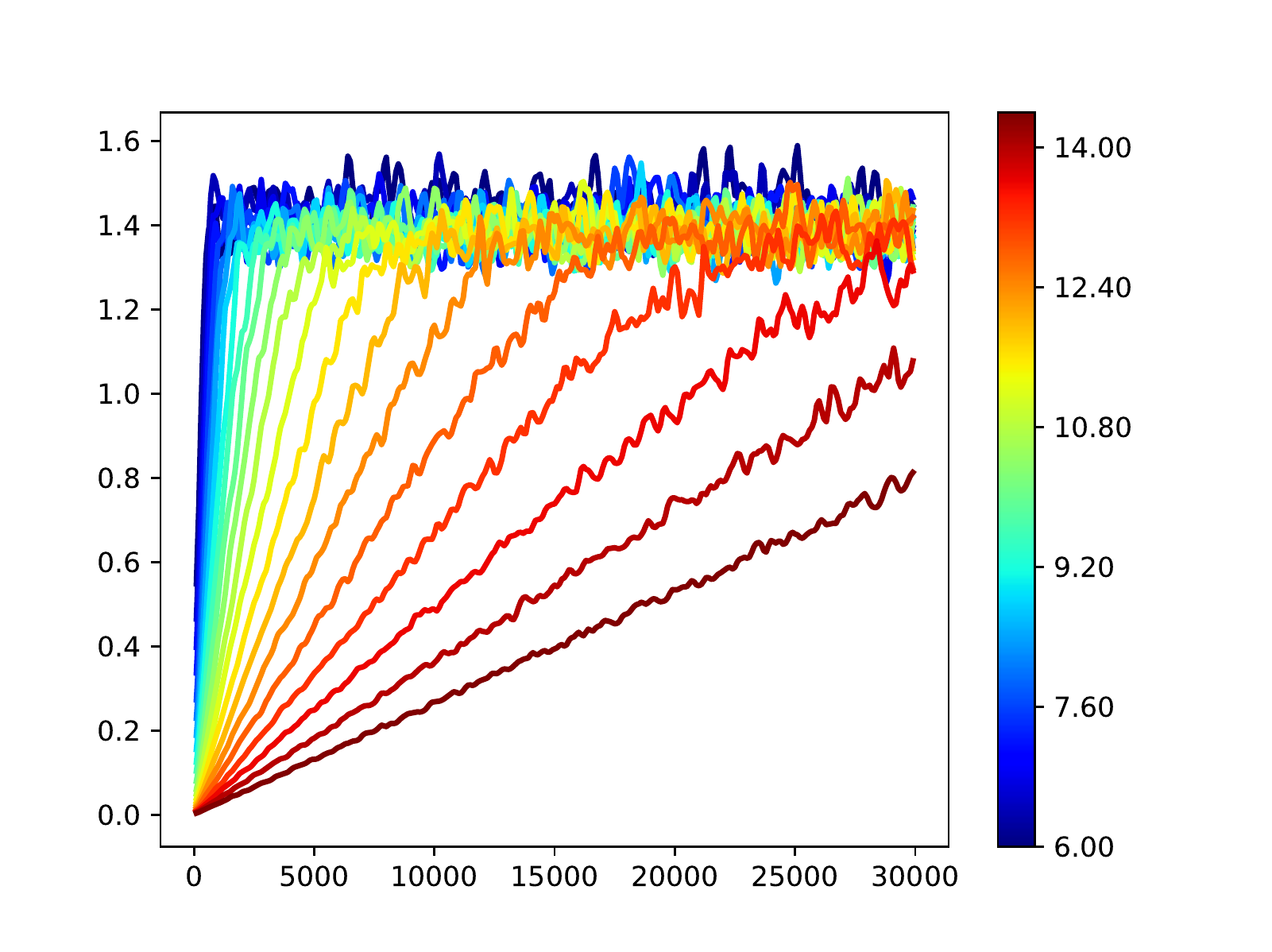}} \hfill
  \caption{In (a) and (b), the blue curve is the theoretical log-partition and in orange the estimated log-partition of $\updelta_{x_0} \Rker_{\gamma}^n$ at iteration $n = 10000$ with $\gamma = 0.1$. The estimation of the log-partition is performed using Gaussian kernels and $1000$ points sampled from $1000000$ points using a bootstrap procedure. In (a), $\mean = 6$ and $\sigma = 2$ and in (b) $\mean= 10$ and $\sigma = 2$. In (c) we illustrate the behavior of $-\log_{10}(\tvnorm{\updelta_{x_0} \Rker_{\gamma}^n - \pi})$ for $\sigma = 2$ and $\mean$ between $6$ and $14$ (color blue to red). Note that the precision saturates since $\pi \neq \pi_{\gamma}$.}
  \label{fig:cv_rate}  
\end{figure}

Let $\ratio = \mean / (2\sigma)$ and assume that $\ratio \geq \sqrt{2}$. Note that in this case $\Lip R^2 = 16 \ratio^2 (\ratio^2 - 1)$. Let $\rho$ be the rate we identify in \eqref{eq:log_rho}. Up to logarithmic terms we have $\log(\log^{-1}(\rho^{-1})) \simeq 4 \ratio^2  (\ratio^2 - 1) / (1 - \rme^{-32 \ratio^2  (\ratio^2 - 1)})$.
In \Cref{fig:cv_rate} and \Cref{fig:quad}, we fix $\sigma = 2$ and study the behavior of $\log(\rhoexp)$ and $\log(\rho)$ \wrt \ $\mean$. In particular, \Cref{fig:quad}-(b) illustrates that the rates we obtain are much closer to the ones estimated by our numerical simulations.

\begin{figure}
  \centering
  \subfloat[]{\includegraphics[width=0.4\linewidth]{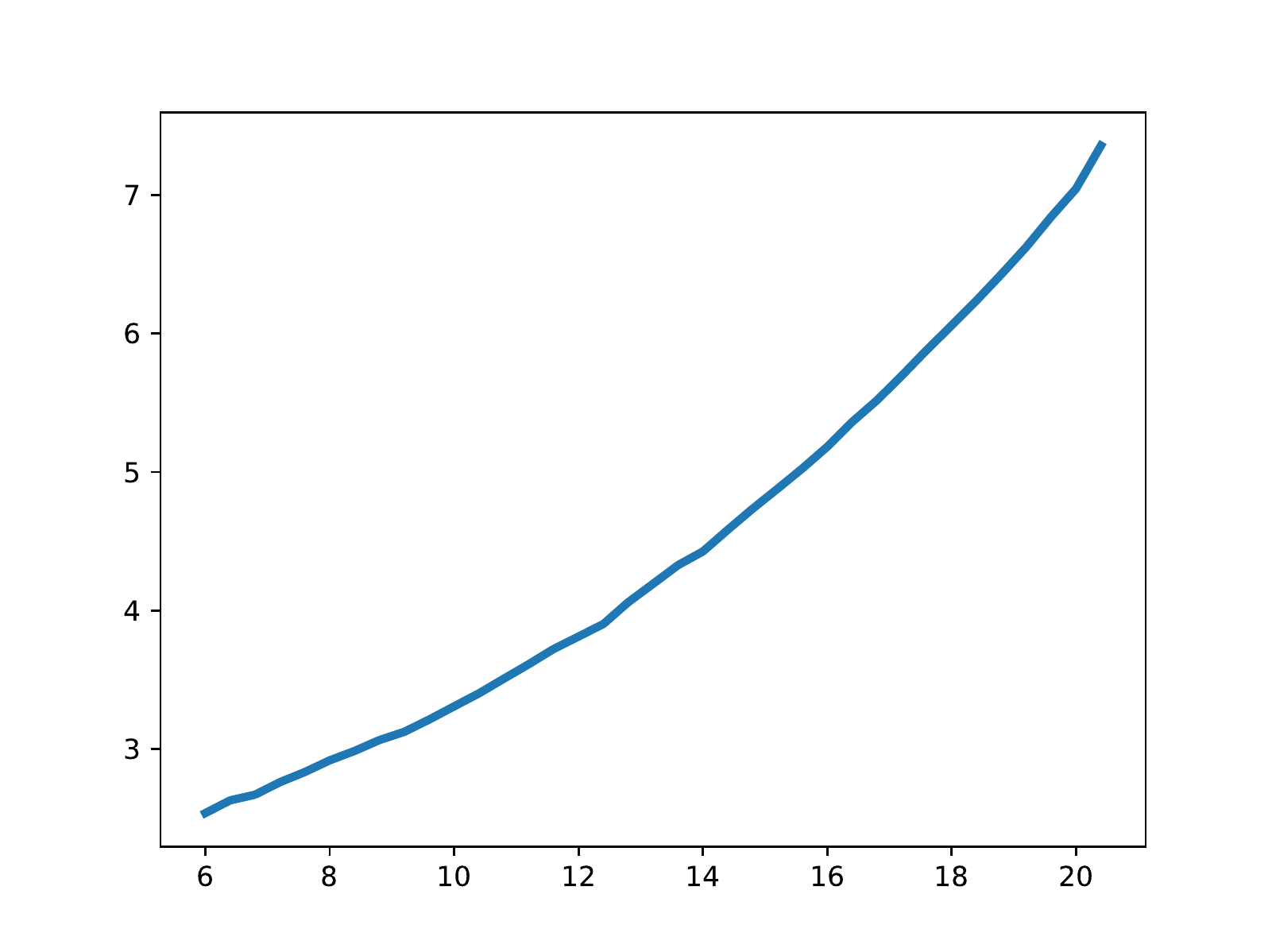}} \qquad 
  \subfloat[]{\includegraphics[width=0.4\linewidth]{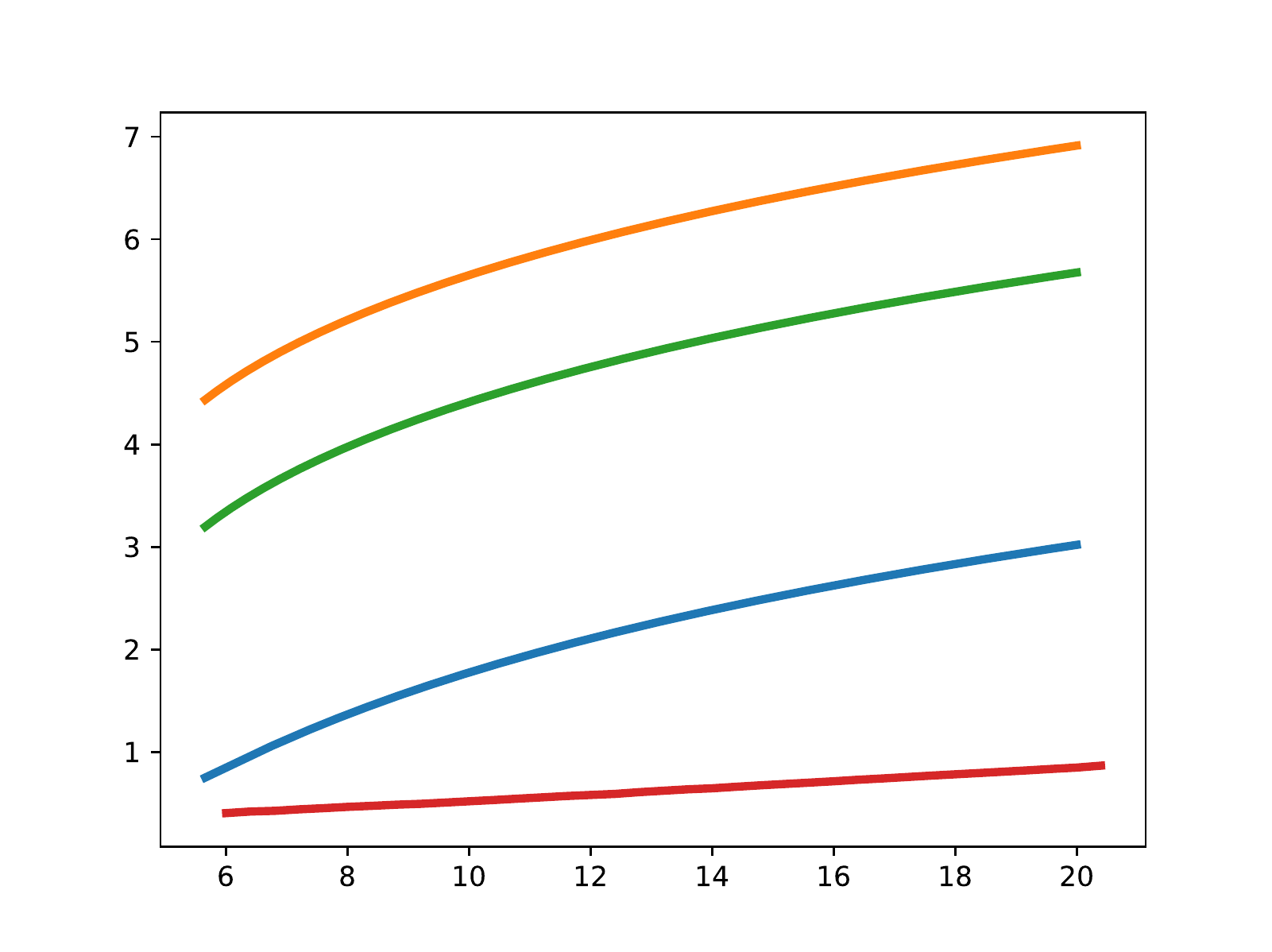}} 
  \caption{In (a) we present $\log(\log^{-1}(\rhoexp^{-1}))$ as $\mean$
    varies. In (b) we present $\log_{10}(\log(\log^{-1}(\rho^{-1})))$ with
    $\rho \leftarrow \rhoexp$ (red), $\rho$ given by \eqref{eq:log_rho} (blue),
    $\rho$ given by \eqref{eq:majka} (green) and $\rho$ given by \eqref{eq:c1}
    (orange). 
  }
  \label{fig:quad}  
\end{figure}


\section{Quantitative convergence bounds  for a class of functional autoregressive models}
\label{sec:presentation}

\label{sec:main-results00}
Let $\msx \in \mathcal{B}(\rset^d)$ endowed with the trace of $\mathcal{B}(\rset^d)$ on $\msx$ denoted by $\mcx= \{ \msa \cap \msx \, : \, \msa \in \mcbb(\rset^d)\}$. 
In this section we consider the Markov chain $(X_k)_{k \in \N}$ defined by $X_0 \in \msx$ and the following recursion: for any $k \in \nset$
\begin{equation}
  \label{eq:langevin_discrete}
  X_{k+1} = \Pi\parenthese{\Tg(X_k)  + \sqrt{ \gamma} \, Z_{k+1}} \eqsp ,
\end{equation}
where $\{\mathcal{T}_{\gamma} : \gamma \in \ocint{0, \bgamma} \}$ is a family of measurable functions from $\msx$ to $\rset^d$ with $\bgamma > 0$, $\gamma \in \ocint{0,\bgamma}$ is a stepsize, $(Z_k)_{k \in \nsets}$ is a sequence of i.i.d $d$-dimensional
zero mean Gaussian random variables with covariance identity and $\Pi: \rset^d \to \msx$ is a measurable function. The Markov chain $(X_k)_{k \in \nset}$ defined by  \eqref{eq:langevin_discrete} is associated with the Markov kernel $\Rcoupling_{\gamma}$ defined on $\msx \times \mcb{\rset^d}$ for any $\gamma \in \ocint{0, \bgamma}$, $x \in \rset^d$ and $\msa \in \mathcal{B}(\rset^d)$ by
\begin{equation}
  \Rcoupling_{\gamma}(x, \msa) = (2 \uppi \gamma)^{-d/2} \int_{\Pi^{\inv}(\msa)}  \exp\parentheseDeux{-(2\gamma)^{-1} \| y - \Tg(x)  \|^2} \rmd y \eqsp .   \label{eq:kernel_langevin}
\end{equation}
Note that for any $x \in \msx$, $\Rker_{\gamma}(x, \msx) = 1$ and therefore, $\Rker_{\gamma}$ given in \eqref{eq:kernel_langevin} is also a Markov kernel over $\msx \times \mcx$.

In this section we state explicit convergence results for $\Rker_{\gamma}$ for some Wasserstein distances
and discuss the rates we obtain. These results rely on appropriate minorization and  Foster-Lyapunov
drift conditions. We first derive the minorization condition for the $n$-th iterate of $\Rker_{\gamma}$. To do so, we consider 
a Markov coupling kernel $\Kcoupling_{\gamma}$ for $\Rcoupling_{\gamma}$ for any $\gamma \in \ocint{0, \bgamma}$, \ie \ for any $x,y \in \rset^d$, $\Kcoupling_{\gamma}((x,y), \cdot)$ is a transference plan between $\Rcoupling_{\gamma}(x, \cdot)$ and $\Rcoupling_{\gamma}(y, \cdot)$. Indeed, in that case, by \cite[Theorem 19.1.6]{douc:moulines:priouret:soulier:2018},  we have  for any $x,y \in \msx$,  $\gamma \in \ocint{0, \bgamma}$ and $n \in \nsets$,
\begin{equation}
  \label{eq:distrib_coupling}
  \tvnorm{\updelta_x \Rker_{\gamma}^n - \updelta_y \Rker_{\gamma}^n } \leq \Kcoupling_{\gamma}^n((x,y), \Deltar^{\complementary}) \eqsp,
\end{equation}
where $\Delta_{\msx}= \ensemble{(x,x)}{x \in \msx}$.  In this paper,
  we consider a projected version of the discrete reflection coupling
  \cite{bubley:dyer:jerrum:1998,durmus:moulines:2016} which is the discrete
  counterpart of the coupling introduced in \cite{lindvall1986coupling}. For any
$x, y,z \in \rset^d$, $\gamma \in \ocint{0, \bgamma}$, let
\begin{equation} \label{eq:e_def}
  \rme(x,y) = \begin{cases} \rmE(x,y) / \| \rmE(x,y) \| & \text{if } \Tg(x) \neq \Tg(y) \\
    0 & \text{otherwise}\end{cases}\eqsp , \qquad \rmE(x,y) = \Tg(y) - \Tg(x) \eqsp ,
\end{equation}
and
\begin{align}
  \mathcal{S}_{\gamma}(x,y,z)& = \Tg(y) +  (\Id - 2\rme(x,y)\rme(x,y)^{\top})z \eqsp , \
   p_{\gamma}(x,y,z)= 1 \wedge \frac{\vphibf_{\gamma}(\| \rmE(x,y) \| - \langle \rme(x,y),z \rangle)}{\vphibf_{\gamma}(\langle \rme(x,y), z \rangle)} \eqsp , \label{eq:success_prob}
\end{align}
where $\vphibf_{\gamma}$ is the one dimensional zero mean Gaussian distribution function with variance $\gamma$. Let $(U_k)_{k \in \nsets}$ be a sequence of \iid \ uniform random variables on $\ccint{0,1}$ independent of $(Z_k)_{k \in \nsets}$. Define the Markov chain $(X_k, Y_k)_{k \in \nset}$ starting from $(X_0, Y_0) \in \msx^2$ by the  recursion: for any $k \in \nset$,
\begin{equation}
  \label{eq:def_project_x_Y_1}
\begin{aligned}
  &\tilde{X}_{k+1} = \Tg(X_k) + \sqrt{\gamma} Z_{k+1} \eqsp , \\
  &\tilde{Y}_{k+1} =  \begin{cases} \tilde{X}_{k+1} & \text{if} \ \Tg(X_{k}) = \Tg(Y_{k}) \eqsp, \\ W_{k+1}\tilde{X}_{k+1} + (1-W_{k+1})\mathcal{S}_{\gamma}(X_k,Y_k, \sqrt{\gamma} Z_{k+1}) & \text{otherwise} \eqsp , \end{cases} 
\end{aligned}
\end{equation}
where $  W_{k+1} = \1_{\ocint{-\infty, 0}}(U_{k+1} - p(X_k, Y_k, \sqrt{\gamma}  Z_{k+1}))$ and finally set 
\begin{equation}(X_{k+1}, Y_{k+1}) = (\Pi(\tilde{X}_{k+1}), \Pi(\tilde{Y}_{k+1})) \eqsp . \label{eq:mc_def}\end{equation}
The Markov chain $(X_k, Y_k)_{k \in \nset}$ is associated with the Markov kernel $\Kcoupling_{\gamma}$ on $\msx^2 \times \mathcal{X}^{\otimes 2}$ given for all $\gamma \in \ocint{0, \bgamma}$, $x, y \in \msx$ and $\eventA \in \mathcal{X}^{\otimes 2}$ by
\begin{align}
&
\Kcoupling_{\gamma}((x,y) , \eventA) = \frac{\1_{\Delta_{\rset^d}}(\Tg(x),\Tg(y))} {(2 \uppi \gamma )^{d/2}}\int_{\rset^d} \1_{\Pi_{\msa}}(\tildex,\tildex) \rme^{-\frac{\norm{\tildex - \Tg(x)}^2}{2\gamma}} \rmd \tildex \\
\nonumber
&+ \frac{\1_{\Delta_{\rset^d}^{\complem}}(\Tg(x),\Tg(y))}{ (2 \uppi \gamma )^{d/2}} \left[ \int_{\rset^d} \1_{\Pi_{\msa}}(\tildex,\tildex)
p_{\gamma}\parenthese{x,y,\tildex - \Tg(x)} \rme^{-\frac{\norm{\tildex - \Tg(x)}^2}{2\gamma}} \rmd \tildex \right.
\\
&\left. +  \int_{\rset^d} \1_{\Pi_{\msa}}\parenthese{\tildex,\mathcal{S}_{\gamma}\parenthese{x,y,\tildex - \Tg(x)}}
\defEns{1-p_{\gamma}\parenthese{x,y,\tildex - \Tg(x)}} \rme^{-\frac{\norm{\tildex - \Tg(x)}^2}{2\gamma}} \rmd \tildex \right] \eqsp, \label{eq:coupling_form}
\end{align}
where $\Pi_{\msa} = (\Pi, \Pi)^{\inv}(\eventA)$ and $ \Delta_{\rset^d} = \ensembleLigne{(x,x)}{x \in \rset^d}$.  Note that marginally, by definition, the distribution of $X_{k+1}$ given $X_k$ is $\Rcoupling_{\gamma}(X_k, \cdot)$. It is well-know (see \eg \  \cite[Section 3.3]{bubley:dyer:jerrum:1998}) that $\tilde{Y}_{k+1}$ and $\Tg(Y_k) + \sqrt{\gamma} Z_{k+1}$ have the same distribution given $Y_k$, and therefore the distribution of $Y_{k+1}$ given $Y_k$ is  $\Rcoupling_{\gamma}(Y_k, \cdot)$. As a result, for any $\gamma \in \ocint{0, \bgamma}$, $x,y \in \msx$, $\Kcoupling_{\gamma}((x,y), \cdot)$ is a transference plan between $\Rcoupling_{\gamma}(x, \cdot)$ and $\Rcoupling_{\gamma}(y, \cdot)$.


As emphasized previously, based on  \eqref{eq:distrib_coupling}, to study convergence of $\rker_{\gamma}$ for $\gamma \in \ocint{0,\bgamma}$<, we first give upper bounds for $\Kcoupling_{\gamma}^{n}((x,y), \Deltar^{\complementary})$
for any  $x, y \in \msx$ and $n \in \nsets$ under appropriate conditions on $\Tg$ and $\Pi$.

\begin{assumption}
  \label{ass:non_expansive_Pi}
  The function $\Pi : \rset^d \to \msx$ is non expansive: \ie~for any $x,y \in \rset^d$, $\norm{\Pi(x) - \Pi(y)} \leq \norm{x-y}$. 
\end{assumption}
Note that \Cref{ass:non_expansive_Pi} is satisfied if $\Pi$ is the
proximal operator \cite[Proposition 12.27]{bauschke:combettes:2011}
associated with a convex lower semi-continuous function
$\mathrm{f} : \rset^d \to \ocint{-\infty,\plusinfty}$. For example, if
$\mathrm{f}(x) = \sum_{i=1}^d \abs{x_i}$, the associated proximal operator is the soft
thresholding operator \cite[Section 6.5.2]{parikh:boyd:2013}. If $\mathrm{f}$ is the convex indicator of a closed convex set $\msc \subset \rset^d$, defined by 
$\mathrm{f}(x) = 0$ for $x \in \msc$,
$\mathrm{f}(x) = \plusinfty$ otherwise, the proximal operator is simply the orthogonal projection onto $\msc$ by \cite[Example 12.21]{bauschke:combettes:2011} and we define for any $x \in \rset^{\dim}$
\begin{equation}
  \label{eq:def_proj}
  \Pi_{\msc}(x) = \argmin_{y \in \msc} \norm{y-x} \eqsp.
\end{equation}

First, the class of  Markov chains defined by \eqref{eq:langevin_discrete} contains Euler-Maruyama discretizations of diffusion processes with identity diffusion matrix and for which $\Pi = \Id$. Our results will be  specified for this particular case in \Cref{sec:applications}.
Second, for the
applications that we have in mind,
the use of Markov chains defined by \eqref{eq:langevin_discrete} with $\Pi \not = \Id$ satisfying \Cref{ass:non_expansive_Pi}, has been proposed based on optimization literature to sample non-smooth log-concave densities \cite{durmus2016efficient,Bubeck:2015,durmus2018analysis,bernton2018langevin}. Finally, we will also make use of \eqref{eq:langevin_discrete} with $\Pi = \Pi_{\msk_n}$, where $\Pi_{\msk_n}$ is defined by \eqref{eq:def_proj} with $\msc \leftarrow \msk_n$, and $(\msk_n)_{n \in \nsets}$ is a sequence of increasing compact sets of $\rset^d$, to derive our results on diffusion processes  in \Cref{sec:applications-1}. 

We now consider the following assumption on
$\{\mct_{\gamma} \, : \, \gamma \in \ocint{0,\bgamma}\}$. Let $\msa \in \mcbb(\rset^{2d})$. 
\begin{assumption}[$\msa$]
  \label{assum:lip_op}
  There exists
  $\kappa : \ocint{0,\bgamma} \to \rset$ such that for any $\gamma \in \ocint{0,\bgamma}$ 
  and $(x, y) \in \msa \cap \msx^2$ 
  \begin{equation}
    \label{eq:lip_op}
    \| \Tg(x) - \Tg(y) \|^2 \le (1 + \gamma \kappa(\gamma)) \| x - y\|^2 \eqsp .
  \end{equation}
  Further, one of the following conditions holds for any $\gamma \in \ocint{0,\bgamma}$: 
  \begin{enumerate*}[label=(\roman*)]
      \item  \label{assum:lip_op_str_convex}
        $\kappa(\gamma) <0$;
              \item  \label{assum:lip_op_convex}
 $\kappa(\gamma)  \leq 0$;
\item \label{assum:lip_op_non_convex}
 $\kappa(\gamma)  > 0$.
  \end{enumerate*}
\end{assumption}
If $\Tg(x) = x + \gamma b(x)$ and $b$ is $\Lip$-Lipschitz we have that \Cref{assum:lip_op}($\rset^{\dim}$) holds for any with $\kappa(\gamma) = \Lip(2 + \gamma \Lip)$. Note that \Cref{assum:lip_op}($\msx^2$)-\ref{assum:lip_op_str_convex} or  \Cref{assum:lip_op}($\msx^2$)-\ref{assum:lip_op_convex} imply that for any $\gamma \in \ocint{0,\bgamma}$, $\Tg$ is non-expansive itself (see \Cref{ass:non_expansive_Pi}). For $\upkappa : \ocint{0,\bgamma} \to \rset$ and $\ell \in \nset^{\star}$, $\gamma \in \ocint{0, \bgamma}$ such that $\gamma \upkappa(\gamma) \in \ooint{-1, +\infty}$, define 
  \begin{equation}
\label{eq:def_Xi}
\Xi_{n}(\upkappa) = \gamma \sum_{k=1}^{n}(1 + \gamma \upkappa(\gamma))^{-k} \eqsp .
\end{equation}
The following theorem gives a generalization of a minorization condition on autoregressive models \cite[Section 6]{durmus:moulines:2016}.

    \begin{theorem}
      \label{theo:minorization_general}
Let $\msa \in \mcbb(\rset^{2d})$ and assume \tup{\Cref{ass:non_expansive_Pi}} and \tup{\Cref{assum:lip_op}($\msa$)}.  Let  $(X_k, Y_k)_{k \in \nset}$ be defined by \eqref{eq:mc_def} with $(X_0, Y_0)= (x,y) \in \msa \cap \msx^2$ and $\gamma \in \ocint{0, \bgamma}$. Then for any $n \in \nsets$
  \begin{multline}
\proba{X_n \neq Y_n \text{ and for any $k \in \{1,\ldots,n-1\}$,} \, (X_k,  Y_k) \in \msa}  \leq  \1_{\Deltar^{\complem}}(x,y) \defEns{1-2\Phibf\parenthese{-\frac{\norm{x-y} }{2\Xi_n^{1/2}(\kappa)}}} \eqsp,
  \end{multline}
where $\Phibf$ is the cumulative distribution function of the Gaussian distribution with zero mean and unit variance on $\rset$. 
\end{theorem}
\begin{proof}
 The proof is a simple application of \Cref{theo:strict_convergence_AR}.
\end{proof}
Based on \Cref{theo:minorization_general}, since $\proba{X_n \neq Y_n} = \Kcoupling^n((x,y),\Deltar^{\complementary})$ where $(X_k, Y_k)_{k \in \nset}$ is  defined by \eqref{eq:mc_def} with $(X_0, Y_0)= (x,y) \in  \msx^2$,  we can derive minorization conditions for the Markov kernel $\Rker_{\gamma}^n$ with $n \in \nsets$ for any $\gamma \in \ocintLigne{0,\bgamma}$ depending on the assumption we make on $\kappa$ in \Cref{assum:lip_op}($\msx^2$).
More precisely, these minorization conditions are derived using $\Kker_{\gamma}^{\ell \step}$ with $\ell \in \nsets$. This is 
a requirement to obtain sharp bounds in the limit $\gamma \to 0$. Indeed, for any $x,y \in \msx$, based only on the results of \Cref{theo:minorization_general}, we get that for any $\ell \in \nsets$,
$\lim_{\gamma \to 0} \tvnorm{\updelta_x \Rker_{\gamma}^{\ell} - \updelta_y \Rker_{\gamma}^{\ell}} \leq 1$, whereas the following proposition implies that for any $\ell \in \nsets$, $\lim_{\gamma \to 0} \tvnorm{\updelta_x \Rker_{\gamma}^{\ell \step} - \updelta_y \Rker_{\gamma}^{\ell \step}} < 1$.

\begin{proposition}
  \label{propo:doeb}
Let $\msa \in \mcbb(\rset^{2d})$ and assume \tup{\Cref{ass:non_expansive_Pi}} and \tup{\Cref{assum:lip_op}($\msa$)} hold.  Let  $(X_k, Y_k)_{k \in \nset}$ be defined by \eqref{eq:mc_def} with $(X_0,Y_0) = (x,y) \in \msa \cap \msx^2$ and $\gamma \in \ocint{0, \bgamma}$. Then for any  $\ell \in \N^{*}$ and $\gamma \in \ocint{0,\bgamma}$,
  \begin{align}
    &\proba{X_{\ell \ceil{\gamma}} \neq Y_{\ell \ceil{\gamma}} { \text{ and for any $k \in \{1,\ldots,n-1\}$,} \, (X_k,  Y_k) \in \msa}}
       \leq 1- 2 \Phibf\parenthese{ - \alpha^{-1/2}(\kappa,\gamma, \ell) \| x- y \|/2} \eqsp ,     \label{eq:doeblin_condition_1}
  \end{align}
  where
  \begin{enumerate}[label= (\alph*),  wide, labelwidth=!, labelindent=0pt]
  \item  $\alpha(\upkappa, \gamma, \ell) = -\upkappa^{-1}(\gamma) \parentheseDeux{ \exp(-\ell \upkappa(\gamma)) - 1}$  if  \tup{\Cref{assum:lip_op}($\msa$)-\ref{assum:lip_op_str_convex}} holds ;
    \label{item:kappa_neg}
  \item  $\alpha(\upkappa, \gamma, \ell) = \ell$  if \tup{\Cref{assum:lip_op}($\msa$)-\ref{assum:lip_op_convex}} holds ; \label{item:kappa_0}    
  \item  $\alpha(\upkappa, \gamma, \ell) =  \upkappa^{-1}(\gamma) \parentheseDeux{ 1 - \exp \defEns{-\ell \upkappa(\gamma)/(1 + \gamma  \upkappa(\gamma))}}$ if \tup{\Cref{assum:lip_op}($\msa$)-\ref{assum:lip_op_non_convex}} holds. \label{item:kappa_pos}
  \end{enumerate}
\end{proposition}

\begin{proof}
    The proof is postponed to \Cref{sec:proof-crefl}.
  \end{proof}
  Depending on the conditions imposed on $\kappa$ defined in
  \Cref{assum:lip_op}($\msx^2$), we obtain the following consequences of \Cref{propo:doeb} which establish, either  an explicit
  convergence bound in total variation for $\Rcoupling_{\gamma}$, or a quantitative
  minorization condition satisfied by this kernel. 
\begin{corollary}
\label{coro:doeblin_lemme_1}
Assume \tup{\Cref{ass:non_expansive_Pi}} and \tup{\Cref{assum:lip_op}($\msx^2$)}.
  \begin{enumerate}[label=(\alph*)]
  \item \label{coro:doeblin_lemme_1_a}
If \tup{\Cref{assum:lip_op}($\msx^2$)-\ref{assum:lip_op_str_convex}} holds and $\kappa_- = \sup_{\gamma \in \ocint{0, \bgamma}} \kappa(\gamma) <0$. Then, for any $\gamma \in \ocint{0,\bgamma}$, $\Rcoupling_{\gamma}$ admits a unique invariant probability measure $\pi_{\gamma}$ and we have for any $\gamma \in \ocint{0, \bgamma}$, $x \in \rset^d$ and $\ell \in \nsets$,
  \begin{align}
    \label{eq:coro_strongly_convex}
    &\tvnorm{ \updelta_x \Rcoupling_{\gamma}^{\ell\ceil{1/\gamma}} -  \pi_{\gamma} }  \leq 1- 2  \int_{\rset^d}\Phibf\defEns{ -(-\kappa_-)^{1/2} \| x- y \|/\{2(\exp(-\ell \kappa_-)-1)^{1/2}\}} \rmd \pi_{\gamma}(y) \eqsp.
  \end{align}
\item \label{coro:doeblin_lemme_1_b}
If \tup{\Cref{assum:lip_op}($\msx^2$)-\ref{assum:lip_op_convex}} holds and, in addition, assume that for any $\gamma \in \ocint{0, \bgamma},$ $\Rker_{\gamma}$ admits an invariant probability measure $\pi_{\gamma}$, then we have for any $\gamma \in \ocint{0, \bgamma}$, $x \in \rset^d$ and $\ell \in \nsets$,
  \begin{equation}
    \label{eq:coro_strongly_convex}
    \tvnorm{ \updelta_x \Rcoupling_{\gamma}^{\ell\ceil{1/\gamma}} -  \pi_{\gamma} } \leq 1- 2  \int_{\rset^d}\Phibf\defEns{ - \| x- y \|/(2\ell^{1/2})} \rmd \pi_{\gamma}(y) \eqsp.
  \end{equation}
  \end{enumerate}
\end{corollary}
\begin{proof}
  The proof is postponed to \Cref{sec:proof-crefc}.
\end{proof}

In other words, if $\Tg$ is a contractive mapping, see \Cref{assum:lip_op}($\msx^2$)-\ref{assum:lip_op_str_convex}, then for $x\in \rset^d$ the convergence of $(\updelta_x \Rker_{\gamma}^{\ell \step})_{\ell \in \nsets}$ to $\pi_{\gamma}$ in total variation is exponential in $\ell$. If $\Tg$ is non expansive, see \Cref{assum:lip_op}($\msx^2$)-\ref{assum:lip_op_convex}, and $\Rker_{\gamma}$ admits an invariant probability measure $\pi_{\gamma}$, for any $x \in \rset^d$,  the convergence of $(\updelta_x \Rker_{\gamma}^{\ell \step})_{\ell \in \nsets}$ to $\pi_{\gamma}$ in total variation is linear in $\ell^{1/2}$. In the case where $\Tg$ is non expansive, see \Cref{assum:lip_op}($\msx^2$)-\ref{assum:lip_op_convex}, or simply Lipschitz, see \Cref{assum:lip_op}($\msx^2$)-\ref{assum:lip_op_non_convex} and no additional assumption is made, we do not directly obtain contraction in total variation but only minorization conditions.

\begin{corollary}
\label{coro:doeblin_lemme_2}
Assume \tup{\Cref{ass:non_expansive_Pi}} and \tup{\Cref{assum:lip_op}($\msx^2$)}.
  Then,  for any $\gamma \in \ocint{0,\bgamma}$, 
  \begin{enumerate}[label=(\alph*)]
  \item \label{coro:doeblin_lemme_2_a} if \tup{\Cref{assum:lip_op}($\msx^2$)-\ref{assum:lip_op_convex}} holds, for any  $x,y \in \msx$ with $\norm{x-y} \leq M$ with $M \geq 0$ and $\ell \in \nsets$ with $\ell \geq \ceil{M^2}$,
    \begin{equation}
      \label{eq:doeblin_condition_2_a}  
       \Kcoupling_{\gamma}^{\ell \step}((x,y), \Deltar^{\complementary})\leq 1 - 2 \Phibf\parenthese{ -1/2}  \eqsp;
    \end{equation}
  \item \label{coro:doeblin_lemme_2_b} if  \tup{\Cref{assum:lip_op}($\msx^2$)-\ref{assum:lip_op_non_convex}} holds, for any $x,y \in \msx $ and $\ell \in \nsets$,
        \begin{equation}
    \label{eq:doeblin_condition_2}    
 \Kcoupling_{\gamma}^{\ell \step}((x,y), \Deltar^{\complementary}) \leq 1- 2 \Phibf\defEns{ -(1 + \bgamma)^{1/2} (1 + \kappa_+)^{1/2} \norm{x-y}/2  } \eqsp ,
    \end{equation}
where $\kappa_+ = \sup_{\gamma \in \ocint{0, \bgamma}} \kappa(\gamma)$.
  \end{enumerate}
\end{corollary}

\begin{proof}
    The proof is postponed to \Cref{coro:doeblin_lemme_2:proof}.
\end{proof}

In our application below, we are mainly interested in the case where $\Rker_{\gamma}$ satisfies a geometric drift condition. Let $(\msy,\mcy)$ be a measurable space,  $\lambda \in (0,1)$, $A\geq0$, $V: \msy \to \coint{1,+\infty}$ be a measurable function and $\msc \in \mcy$.
\begin{assumptionD}[\hypertarget{ass:drift_discrete}{$\bfDd(V,\lambda,A,\msc)$}]
  A Markov kernel $\Rcoupling$ on $\msy\times \mcy$ satisfies the discrete Foster-Lyapunov drift condition if  for all $y \in \msy$
\begin{equation}
  \label{eq:discrete_drift}
  \Rcoupling V(y) \leq \lambda V(y) + A \1_{\msc}(y) \eqsp.
\end{equation}
\end{assumptionD}
The index $\mathrm{d}$ in $\bfDd$ stands for ``discrete'' as we will introduce the continuous-time counterpart of this drift condition, denoted by $\bfDc$, in \Cref{sec:main-results-3}.
Note that this drift condition implies the existence of an invariant probability measure if $\Rker$ is a Feller kernel and the level sets of $V$ are compact, see \cite[Theorem 12.3.3]{douc:moulines:priouret:soulier:2018}. In the sequel, 
we are interested in establishing convergence results in  the Wasserstein metric $\distV$ associated with the
cost
\begin{equation}
  \label{eq:def_wbf}
  \bfc : (x,y) \mapsto \1_{\Deltar^{\complementary}}(x,y)\VlyapD(x,y)
\end{equation}
where $\lyap: \ \msx \times \msx \to \coint{0,+\infty}$ satisfies for any $x,y, z \in \msx$,
$\lyap(x,y) = \lyap(y,x)$, $\lyap(x,z) \leq \lyap(x,y) + \lyap(y,z)$ and $\lyap(x,y) = 0$ implies that $x = y$. Note that under these conditions on $\lyap$, $\bfc$ defines a metric on $\rset^d$. Let $\mu, \nu$ be two probability measures over $\mcx$, we highlight three cases.
\begin{itemize}
\item total variation: if $\lyap = 1$ then $\distV(\mu,\nu) = \tvnorm{\mu - \nu}$ ;
\item $V$-norm: if $\lyap(x,y) = \defEnsLigne{V(x) + V(y)}/2$ where $V: \ \rset^d \to \coint{1,+\infty}$ is measurable then $\distV(\mu,\nu) = \Vnorm[V]{\mu - \nu}$ ;
\item total variation + Kantorovitch-Rubinstein metric: if $\lyap(x,y) = 1 + \vartheta \norm{x-y}$ with $\vartheta >0$, then by definition of Wasserstein metrics, $\distV(\mu, \nu) \geq \tvnorm{\mu - \nu} + \vartheta \wassersteinD[1](\mu, \nu) $.
\end{itemize}
We now state convergence bounds for Markov kernels which satisfy one of the conclusions of \Cref{coro:doeblin_lemme_2}. Indeed, in order to deal with the two assumptions  \Cref{assum:lip_op}($\msx^2$)-\ref{assum:lip_op_convex} and \Cref{assum:lip_op}($\msx^2$)-\ref{assum:lip_op_non_convex} together, we provide a general result regarding the contraction of $\Rker_{\gamma}$ in the metric $\distV$ for some cost function $\bfc$ on $\msx^2$. This result is based on an abstract condition on $\KkerD_{\gamma}^{\step}\1_{\Deltar^{\complementary}}$, which is satisfied under \Cref{assum:lip_op}($\msx^2$)-\ref{assum:lip_op_convex} or \Cref{assum:lip_op}($\msx^2$)-\ref{assum:lip_op_non_convex} by  \Cref{coro:doeblin_lemme_2} with $\KkerD_{\gamma} \leftarrow \Kker_{\gamma}$,  and a drift condition for $\KkerD_{\gamma}$, where $\KkerD_{\gamma}$ is a Markov coupling kernel for $\Rker_{\gamma}$.  We recall that for any $M \geq 0$,
\begin{equation}
  \label{eq:def_delta_M}
  \Delta_{\msx,M} = \{(x,y) \in \msx \, : \, \norm{x-y} \leq M\} \eqsp. 
\end{equation}

\begin{theorem}
  \label{theo:discrete_contrac_wass_D_v2}
  Assume that there exist  $\lambda \in (0,1)$, $A \geq 0$, $\tM_{\discrete}>0$, a measurable function $\VlyapD : \msx \times \msx \to \coint{1,\plusinfty}$, $\msc \in \mcx^{\otimes 2}$  with $\msc \subset \Delta_{\msx,\tM_{\discrete}}$ and for any $\gamma \in \ocint{0,\bgamma}$, $\KkerD_{\gamma}$ a Markov coupling kernel for $\Rcoupling_{\gamma}$ satisfying  \hyperlink{ass:drift_discrete}{$\bfDd(\VlyapD,\lambda^{\gamma}, A\gamma, \msc)$}.
  Further, assume that for any $\gamma \in \ocint{0, \bgamma}$, $\Delta_{\msx}$ is absorbing for $\KkerD_{\gamma}$, \ie \ for any $x \in \msx$, $\KkerD_{\gamma} \1_{\Delta_{\msx}}(x,x) = 1$, 
 and that there exists $\Psibf : \ocint{0,\bgamma} \times \nsets \times \rset_+ \to \ccint{0,1}$ such that for any $\gamma \in \ocint{0, \bgamma}$, $\ell \in \nsets$ and $x,y \in \msx$
  \begin{equation}
  \label{eq:minorization_condition_v2}
   \KkerD_{\gamma}^{\ell \step} ((x,y), \Deltar^{\complementary}) \leq 1 - \Psibf(\gamma, \ell, \norm{x-y})  \eqsp, 
 \text{ and for any $M \geq 0$, $\inf_{(x,y) \in \Delta_{\msx,M}} \Psibf(\gamma,\ell, \norm{x-y}) >0$} \eqsp.
\end{equation}
Then, for any $\gamma \in \ocint{0, \bgamma}$, $\ell \in \nsets$ and $x,y \in \msx$
\begin{equation}
  \label{eq:theo:discrete_contrac_wass_D_v2_a}
  \distV(\updelta_x \Rcoupling_{\gamma}^k, \updelta_y \Rcoupling_{\gamma}^k) \leq \KkerD_{\gamma}^k \bfc(x,y) \leq \lambda^{k\gamma/4} [\bD_1
 \bfc(x,y) + \bD_2\1_{\Deltar^{\complementary}}(x,y)] +  \bC_1 \brho_1^{k\gamma/4}\1_{\Deltar^{\complementary}}(x,y)  \eqsp,
  \end{equation}
 where 
  \begin{equation}
    \begin{aligned}
      \bD_1 &= 1+ 4 A  \log^{-1}(1/\lambda)/\lambda^{\bgamma} \eqsp, \quad \bD_2  = \bD_1 A \lambda^{-(1+\bgamma)\ell} (1 + \bgamma)\ell  \eqsp, \quad \bC_{1}  = 8 A \log^{-1}(1/\brho_1)/ \brho_1^{\bgamma}\eqsp, \\
      \log(\brho_1) &= \defEns{\log(\lambda) \log(1 - \bvareps_{\discrete, 1}) } / \defEns{-\log(\bc_{1})  + \log(1 - \bvareps_{\discrete, 1}) } \eqsp , \quad 
     \bc_{1} = \tBdisc +  A \lambda^{-(1 + \bgamma)\ell} (1 + \bgamma)\ell \eqsp , \\
      \bvareps_{\discrete, 1}&  =  \inf_{\gamma \in \ocint{0, \bgamma}, \ (x,y) \in \Delta_{\msx,\tM_{\discrete}}} \Psibf(\gamma, \ell, \norm{x-y}) \qquad \tBdisc = \sup_{(x,y) \in \msc }\lyap(x,y) \eqsp.
  \end{aligned}  
\end{equation}  
In addition, if $\bgamma \leq 1$ 
       and $\bvareps_{\discrete, 1} \leq 1 - \rme^{-1}$, then
       \begin{equation}
         \label{eq:rho_upper}
        \log^{-1}(\brho_1^{-1}) \leq \left .\parentheseDeux{1 + \log\parentheseLigne{\tBdisc} + \log(1 + 2A\ell) +2\ell \log(\lambda^{-1})}  \middle/  \parentheseDeux{\log(\lambda^{-1}) \bvareps_{\discrete, 1}} \right.\eqsp . \label{eq:majo_rho} 
      \end{equation}

\end{theorem}

\begin{proof}
  The proof is postponed to \Cref{sec:theo:discrete_contrac_wass_D_v2:proof}.
\end{proof}
We emphasize that \eqref{eq:minorization_condition_v2} is satisfied under \Cref{assum:lip_op}($\msx^2$)-\ref{assum:lip_op_convex} or \Cref{assum:lip_op}($\msx^2$)-\ref{assum:lip_op_non_convex} by  \Cref{coro:doeblin_lemme_2} with $\KkerD_{\gamma} \leftarrow \Kker_{\gamma}$.

Further, note that in \eqref{eq:theo:discrete_contrac_wass_D_v2_a}, the leading term, $\bC_1 \brho_1^{k\gamma/4}$, does not depend on $x,y \in \msx$. Indeed, the rate in front of the initial conditions $\lyap(x,y)$ is given by $\lambda^{\gamma/4}$ which is always smaller than $\brho_1^{\gamma/4}$. Therefore, \Cref{theo:discrete_contrac_wass_D_v2} implies in particular that for any $\gamma \in \ocint{0, \bgamma}$, $x,y \in  \msk$ and $k \in \nset$
\begin{equation}
  \label{eq:convergence_c}
    \distV(\updelta_x \Rcoupling_{\gamma}^k, \updelta_y \Rcoupling_{\gamma}^k) \leq \brho_1^{k\gamma/4} [\bD_{1}  + \bD_{2} +  \bC_{1}] \bfc(x,y)  \eqsp.
\end{equation}

We conclude this section with two propositions which highlight the usefulness of
 of the conclusions of  \Cref{theo:discrete_contrac_wass_D_v2} to establish convergence
estimates with respect to different metrics.  First, in
\Cref{prop:w1_contrac_iterees}, under additional conditions on $\Psibf$ and on
$\lyap$ (which will be satisfied in our applications, see
\Cref{coro:w1_wp}) we get a similar result to \eqref{eq:convergence_c}
replacing $\bfc$ by $(x,y) \mapsto \norm{x-y}$, \ie \ replacing
$\wassersteinD[\bfc]$ by $\wassersteinD[1]$.

\begin{proposition}
  \label{prop:w1_contrac_iterees}
  Assume that the conditions of \Cref{theo:discrete_contrac_wass_D_v2} are satisfied with for any $x,y \in \msx$, $\lyap(x,y) = 1 + \vartheta \norm{x-y}$, where $\vartheta >0$.
  In addition, assume that the following conditions hold.
  \begin{enumerate}[label=(\roman*)]
  \item For any $\gamma \in \ocint{0, \bgamma}$, $t \mapsto \Psibf(\gamma, 1, t)$ is convex on $\rset_+$, admits a right-derivative at $0$, denoted by $\Psibf'(\gamma, 1, 0)$, and $\mathbf{a} = \inf_{\gamma \in \ocint{0, \bgamma}} \Psibf'(\gamma, 1, 0) > -\infty$. \label{prop:w1_contrac_iterees:1}
  \item There exists $\varkappa \geq 0$ such that for any $x,y \in \msx$,  
    $\KkerD_{\gamma}\norm{x-y} \leq (1 + \gamma \varkappa) \norm{x-y}$. \label{prop:w1_contrac_iterees:3}
  \end{enumerate}
  Then there exist $\bD_3 \geq 0$ and $\brho_1 \in \coint{0,1}$ such that for any
  $\gamma \in \ocint{0, \bgamma}$, $x,y \in \msx$ and $k \in \nset$
  \begin{equation}
    \label{eq:wass_un}
    \wassersteinD[1](\updelta_x \Rcoupling_{\gamma}^k, \updelta_y \Rcoupling_{\gamma}^k) \leq \KkerD_{\gamma}^k \norm{x-y} \leq \bD_3  \brho_1^{k \gamma/4} \norm{x-y} \eqsp ,
  \end{equation}
  with $\brho_1$ given in \Cref{theo:discrete_contrac_wass_D_v2} and $\bD_3$ explicit in the proof.
\end{proposition}

\begin{proof}
  The proof is postponed to \Cref{sec:prop:w1_contrac_iterees:proof}.
\end{proof}

As a consequence, if $\msx$ is closed, the Markov kernel
$\Rker_{\gamma}$ admits a unique invariant probability measure
$\pi_{\gamma}$, \ie \ $\pi_{\gamma} = \pi_{\gamma} \Rker_{\gamma}$,
using \cite[Chapter 1, 6, A.1]{granas2003fixed}, and since we have
that
$\mathscr{P}_1(\msx) = \{\mu \text{ probability measure on }
(\rset^d,\mcbb(\rset^d)) \, : \, \int_{\rset^d} \norm{x} \rmd \mu(x) <
\plusinfty\}$, endowed with $\wassersteinD[1]$ is complete, see
\cite[Theorem 6.18]{villani2009optimal}. Further, for any
$\gamma \in \ocint{0, \bgamma}$, using \cite[Theorem
1]{meyers1967converse}, there exists a distance $\mathbf{d}_{\gamma}$
on $\mathscr{P}_1(\msx)$, topologically equivalent to
$\wassersteinD[1]$, such that $\mathscr{P}_1$ is complete and for any
$x, y \in \rset^{\dim}$ and $k \in \nset$
\begin{equation}
  \mathbf{d}_{\gamma}(\updelta_x \Rker_{\gamma}^k, \updelta_y \Rker_{\gamma}^k) \leq \rho^{k \gamma/4} \mathbf{d}_{\gamma}(\updelta_x,\updelta_y) \eqsp .
\end{equation}

A similar result to \eqref{eq:wass_un} in \Cref{prop:w1_contrac_iterees} can be derived when replacing $\wassersteinD[1]$ by $\wassersteinD[p]$
with $p \in \nset$, if we assume some Foster-Lyapunov condition with respect to $(x,y) \mapsto \norm{x-y}^p$.

\begin{proposition}
  \label{prop:from_1_to_p}
  Assume that there exist $\brho \in \ocint{0, 1}$, $\bD \geq 0$ and for any $\gamma \in \ocint{0, \bgamma}$,  $\KkerD_{\gamma}$ a Markov coupling kernel for $\Rcoupling_{\gamma}$  satisfying for any $x,y \in \msx$ and $k \in \nset$
  \begin{equation}
    \label{eq:wass_un_simpl}
    \KkerD_{\gamma}^k \norm{x-y} \leq \bD \brho^{k \gamma} \norm{x-y} \eqsp .
  \end{equation}
  In addition, assume that for any $\gamma \in \ocint{0, \bgamma}$ and $q \in \nset$,
  $\KkerD_{\gamma}$  satisfies
  \hyperlink{ass:drift_discrete}{$\bfDd((x,y) \mapsto \normLigne{x-y}^q,\tilde{\lambda}_q^{\gamma}, \tilde{A}_q\gamma)$} with $\tilde{\lambda}_q \in \ocint{0,1}$ and $\tilde{A}_q \geq 0$.
  Then, for any $p \geq 1$ and $\upalpha \in \ooint{p, +\infty}$ there exists $\bD_{4, \upalpha} \geq 0$ such that
  \begin{equation}
    \wassersteinD[p](\updelta_x \Rker_{\gamma}^k, \updelta_x \Rker_{\gamma}^k) \leq \parenthese{\KkerD_{\gamma}^k \norm{x-y}^p}^{1/p} \leq \bD_{4, \upalpha} \brho^{k \gamma / \upalpha} \defEns{\norm{x-y} + \norm{x-y}^{1/\upalpha}} \eqsp ,
  \end{equation}
  with $\bD_{4, \upalpha}$ explicit in the proof.
\end{proposition}

\begin{proof}
  The proof is postponed to \Cref{sec:prop:from_1_to_p:proof}.
\end{proof}


\section{Application to the projected Euler-Maruyama discretization}
\label{sec:applications}

Here we consider the case in which the operator $\Tg$ in
\eqref{eq:langevin_discrete} is given by the discretization of a diffusion
\eqref{eq:sde_informal}. More precisely, for $b : \ \rset^d \to \rset^d$, we
study the projected Euler-Maruyama discretization associated to the diffusion
with drift function $b$ and diffusion coefficient $\Id$, \ie \ we consider the
following assumption for $\msx \subset \rset^d$.
  \begin{assumptionB}[$\msx$]
    \label{ass:def_Euler}
    $\msx$ is assumed to be a closed convex (non-empty) subset of $\rset^d$,
    $\Pi = \Pi_\msx$ is the orthogonal projection onto $\msx$ defined in
    \eqref{eq:def_proj} and
  \begin{equation}
    \Tg(x) = x + \gamma b(x)    \text{ for any $\gamma >0$ and $x \in \msx$} \eqsp,   
    \label{eq:def_proj_euler}
  \end{equation}
  where $b :\rset^d \to \rset^d$ is continuous. 
\end{assumptionB} 
Note that if $\msx = \rset^d$ and $\Pi = \Id$, then this scheme is the classical
Euler-Maruyama discretization of a diffusion with drift $b$ and diffusion
coefficient $\Id$.  The application to the tamed Euler-Maruyama discretization
of the results of \Cref{sec:main-results00} is given in \Cref{sec:tamed-unadj-lang}.  In what
follows, we show the convergence in weighted total variation for the projected
Euler-Maruyama discretization and discuss the dependency of the constants
appearing in the bounds we obtain with respect to the properties we assume on
the drift $b$. We first derive minorization conditions or convergence in total
variation depending on the regularity/curvature assumption on the drift $b$ in
\Cref{sec:minor-cond}.  Drift conditions and the ensuing convergence when
combined with the minorization assumption are studied in
\Cref{sec:drift-cond-conv}.

\subsection{Minorization condition}
\label{sec:minor-cond}

First, we show that some regularity/curvature conditions on the drift $b$ imply condition  \Cref{assum:lip_op}$(\msx^2)$ for $\Tg$ given by \eqref{eq:def_proj_euler}. Let $\mtt \in \rset$. 

\begin{assumptionB}
  \label{as:item:lip}
    There exists $\Lip \geq 0$ such that  $b$ is $\Lip$-Lipschitz, \ie~for any $x,y \in \msx$, $\| b(x) - b(y) \| \leq \Lip \| x - y \|$ and $b(0) = 0$.
  \end{assumptionB}

  \begin{assumptionB}[$\mtt$]
  \label{as:b_min}
For any $x,y \in \msx$,
  \begin{equation}
    \langle b(x) - b(y) , x -y \rangle \leq -\mtt \, \| x - y \|^2 \eqsp .
  \end{equation}
  \end{assumptionB}
Note that \Cref{as:item:lip} implies \Cref{as:b_min}($-\Lip$). However, we are interested in the case where $\abs{\mtt}$ is possibly strictly smaller than $\Lip$.
If there exists $U \in \rmC^1(\msx)$ such that for any $x \in \msx$, $b(x) = -\nabla U(x)$ and \Cref{as:b_min}($\mtt$) holds with $\mtt =0$, respectively $\mtt >0$ then $U$ is convex, respectively strongly convex. Note that \Cref{as:b_min}$(0)$ does not imply that $\Tg$ given by \eqref{eq:def_proj_euler} is non-expansive, therefore we consider the following assumption.

  \begin{assumptionB}
  \label{as:b_lemme_descente}
There exists $\Lipb >0$ such that for any $x,y \in \msx$,
  \begin{equation}
\langle b(x) - b(y) , x -y \rangle \leq -\Lipb \| b(x) - b(y) \|^2 \eqsp . \label{eq:descent_lemma}
  \end{equation}
  \end{assumptionB}
  Note that \Cref{as:b_lemme_descente} implies that \Cref{as:item:lip}
  with $\Lip = \Lipb^{-1}$ and \Cref{as:b_min}($0$) hold. Conversely, in the case
  where $\msx = \rset^d$ and there exists
  $U \in \rmC^1(\rset^d)$ such that for any $x \in \rset^d$,
  $b(x) = -\nabla U(x)$, \cite[Theorem
  2.1.5]{nesterov:2004} implies that under \Cref{as:item:lip} and
  \Cref{as:b_min}($0$), \Cref{as:b_lemme_descente} holds with
  $\Lipb = \Lip^{-1}$. Based on \Cref{prop:a1_type} and assuming
  \Cref{ass:def_Euler}, we obtain the following results on the Markov
  kernel $\Rker_{\gamma}$ defined by \eqref{eq:kernel_langevin} with
  $\gamma >0$.

\begin{proposition}
  \label{prop:a1_type}
Assume \tup{\Cref{ass:def_Euler}($\msx$)} holds for $\msx \subset \rset^d$.
  \begin{enumerate}[label= (\alph*)]
  \item \label{item:a1_1}
    \tup{If \Cref{as:item:lip}} and \tup{\Cref{as:b_min}($\mtt$)} hold with $\mtt \in \rset$. Then \eqref{eq:lip_op} in \tup{\Cref{assum:lip_op}($\msx^2$)} holds for any $\gamma >0$ with $\kappa(\gamma) = -2 \mtt + \Lip^2 \gamma$. In particular, if $\mtt > 0 $ then \tup{\Cref{assum:lip_op}($\msx^2$)-\ref{assum:lip_op_str_convex}} holds for any $\bgamma < 2 \mtt/ \Lip^2$ and if  $\mtt \leq  0 $ then \tup{\Cref{assum:lip_op}($\msx^2$)-\ref{assum:lip_op_non_convex}} holds for any $\bgamma>0$ ;
\item \label{item:a1_2}
If   \tup{\Cref{as:b_lemme_descente}} holds, 
  then \tup{\Cref{assum:lip_op}($\msx^2$)-\ref{assum:lip_op_convex}} holds with $\kappa(\gamma) = 0$ for any $\bgamma \leq 2\Lipb$.
  \end{enumerate}
\end{proposition}
\begin{proof}
    The proof is postponed to \Cref{prop:a1_type:proof}.
\end{proof}

Combining \Cref{prop:a1_type} and \Cref{propo:doeb} and/or \Cref{coro:doeblin_lemme_1}, we can draw the following conclusions. 

If \Cref{as:item:lip} and \Cref{as:b_min}($\mtt$) hold with $\mtt >0$, then we obtain, by \Cref{prop:a1_type}-\ref{item:a1_1} and \Cref{propo:doeb}-\ref{item:kappa_neg}, that for any $ \gamma \in \ooint{0, 2\mtt/\Lip^2 }$ and $\ell \in \nsets$, \eqref{eq:doeblin_condition_1} holds with $\alpha = \alpha_-$ given by 
\begin{equation}
  \alpha_-(\kappa,\gamma, \ell) = - \frac{\exp(-\ell( -2\mtt + \Lip^2 \gamma )) - 1}{-2 \mtt + \Lip^2 \gamma } \eqsp.
\end{equation}
In addition, \Cref{coro:doeblin_lemme_1}-\ref{coro:doeblin_lemme_1_a} implies  that for any $\gamma \in \ooint{0,2\mtt/\Lip^2}$ and $x \in \msx$, $\parentheseLigne{\updelta_x \Rcoupling_\gamma^{\ceil{1/\gamma} \ell}}_{\ell \in \nset}$ converges exponentially fast to its invariant probability measure $\pi_{\gamma}$ in total variation, with a rate which does not depend on $\gamma$, but only on $\mtt$.

  Under \Cref{as:b_lemme_descente},  combining \Cref{prop:a1_type}-\ref{item:a1_2} and \Cref{coro:doeblin_lemme_2}-\ref{coro:doeblin_lemme_2_a} we obtain that on any compact set $\msk \subset \msx$, $\Rker_{\gamma}^{\step \ell}$ satisfies the minorization condition \eqref{eq:doeblin_condition_2_a} with $\ell \geq \diameter(\msk)^2$. In addition, if $\Rker_{\gamma}$ admits an invariant probability measure $\pi_{\gamma}$, then \Cref{coro:doeblin_lemme_1}-\ref{coro:doeblin_lemme_1_b} implies that for any $\gamma \in \ocint{0, 2\mtt_b}$ and $x \in \msx$, $(\updelta_x \Rker_{\gamma}^{\step \ell})_{\ell \in \nset}$ converges linearly in $\ell^{1/2}$ to $\pi_{\gamma}$ in total variation.

  In the case where \Cref{as:item:lip} and \Cref{as:b_min}($\mtt$) are satisfied with $\mtt \in \rset_-$, we obtain that for any $\gamma >0$ and $\ell \in \nsets$, \eqref{eq:doeblin_condition_1} holds with $\alpha = \alpha_+$ given by 
  \begin{align}
  \alpha_+(\kappa,\gamma, \ell) &= (-2\mtt + \Lip^2 \gamma )^{-1}\defEns{1 - \exp\parentheseDeux{- \ell( -2 \mtt + \Lip^2 \gamma  )/(1 + \gamma (-2 \mtt + \Lip^2 \gamma))}} \leq  (-2\mtt + \Lip^2 \gamma)^{-1} \eqsp,     \label{eq:alpha_plus}
\end{align}
which implies that the bound given by \Cref{propo:doeb}-\ref{item:kappa_pos} does not go to $0$ when $\ell$ goes to infinity. Therefore we cannot directly conclude that the Markov chain converges in total variation. However, by \Cref{prop:a1_type}-\ref{item:a1_1}, \Cref{coro:doeblin_lemme_2}-\ref{coro:doeblin_lemme_2_b} shows that for any $\gamma \in \ocint{0, \bgamma}$ with $\bgamma >0$ and $\ell \in \nsets$, $\Rcoupling_{\gamma}^{\ceil{1/\gamma}\ell}$ satisfies the minorization condition  \eqref{eq:doeblin_condition_2}, with constants which only depend on $\mtt$ and $\Lip$. Note however that in \eqref{eq:alpha_plus} the influence of $\mtt$ is different than the one of $\Lip$ and this result justifies the two assumptions \Cref{as:item:lip} and \Cref{as:b_min}($\mtt$).

\subsection{Drift conditions and convergence}
\label{sec:drift-cond-conv}

In the sequel of this section, we consider several assumptions on the
drift function $b$ which imply Foster-Lyapunov drift conditions on the
Markov coupling kernel $\Kcoupling_{\gamma}$ defined in
\eqref{eq:coupling_form}. These results in combination with
\Cref{prop:a1_type} will allow us to use
\Cref{theo:discrete_contrac_wass_D_v2}, see also
\Cref{theo:discrete_contrac_wass}.

 \subsubsection{Strongly convex at infinity}
 \label{sec:strongly-convex-at}

 First, we consider conditions
 on $b$ which imply that $\Rker_{\gamma}$ for
 $\gamma \in \ocint{0,\bgamma}$, is geometrically
 convergent in a metric which dominates the total variation distance
 and the Wasserstein distance of order $1$. This result will be an application of   \Cref{theo:discrete_contrac_wass_D_v2} and the constants we end up with are independent of the dimension $d$. To do so, we establish that there exists a Lyapunov
 function $\lyap$ for which $\Kcoupling_{\gamma}$ satisfies for $\gamma \in \ocint{0,\bgamma}$, 
 \hyperlink{ass:drift_discrete}{$\bfDd(\lyap,\lambda^{\gamma},A
   \gamma, \Delta_{\msx,M_{\discrete}})$} where  $\Delta_{\msx,M_{\discrete}}$ is given by \eqref{eq:def_delta_M} and 
 $M_{\discrete} \geq 0$  which do not depend on the dimension. 
\begin{assumptionC}
  \label{assum:strong_convex_outside_ball}
There exist $\Run > 0$ and $\mttplusun >0$ such that for any $x,y \in \msx$ with $\norm{x-y} \geq \Run$,
  \begin{equation}
    \ps{b(x) - b(y)}{x-y} \leq      -\mttplusun \norm[2]{x-y} \eqsp.
  \end{equation}
\end{assumptionC}
This assumption has been considered in  \cite{eberle2018quantitative,eberle2016reflection,luo2016exponential,majka2018non} and is sometimes referred to as  strong convexity of the drift $b$ outside of the ball $\ball{0}{\Run}$, see \Cref{sec:an-illustr-example} for an example of such a setting.
In the next proposition,  we derive the announced drift for $\VlyapDun: \ \msx^2 \to \coint{1,+\infty}$ defined for any $x,y  \in \msx$ by
  \begin{equation}
    \label{eq:def_Wun}
    \VlyapDun(x,y) = 1 + \norm{x-y}/\Run \eqsp.
  \end{equation}
\begin{proposition}  
  \label{propo:drift_strong_convex_wass}
Assume \tup{\Cref{ass:def_Euler}$(\msx)$} for $\msx \subset \rset^d$, \tup{\Cref{as:item:lip}}, \tup{\Cref{as:b_min}($\mtt$)} for $\mtt \in \rset_-$ and
\tup{\Cref{assum:strong_convex_outside_ball}}.
Let $\Kker_{\gamma}$ be defined by \eqref{eq:coupling_form} and $\bgamma \in (0, 2\mttplusun/\Lip^2)$. Then the following hold:
\begin{enumerate}[label= (\alph*)]
 \item for any $\gamma \in \ocint{0, \bgamma}$, we have
   \begin{equation}
     \label{eq:majo_norme_reflex}
     \Kker_{\gamma} \norm{x-y} \leq \norm{\Tg(x) - \Tg(y)} \leq (1 + \gamma (-\mtt + \bgamma\Lip^2/2)) \norm{x-y} \eqsp .
   \end{equation} \label{item:reflex_ineq_a}
\item for any $\gamma \in \ocint{0,\bgamma}$, $\Kker_{\gamma}$ satisfies \hyperlink{ass:drift_discrete}{$\bfDd(\VlyapDun,\lambda^{\gamma},A \gamma, \Delta_{\msx,\Run})$}  where $\Delta_{\msx,\Run}$ is given by \eqref{eq:def_delta_M}  and  
   \begin{equation} \lambda = \exp\parentheseDeux{ -(\mttplusun - \bgamma \Lip^2/2)/2} \eqsp , \quad A= \mttplusun - \mtt 
     \eqsp . \label{eq:const_drift_wass}
   \end{equation} \label{item:reflex_ineq_b}
 \item for any $p \in \nset$ with $p \geq 2$, there  exist $\lambda_p \in \ocint{0, 1}$, $A_p \geq 0$ and $\gamma \in \ocint{0,\bgamma}$, such that $\Kker_{\gamma}$ satisfies \hyperlink{ass:drift_discrete}{$\bfDd((x,y) \mapsto \norm{x-y}^p,\lambda_p^{\gamma},A_p \gamma)$}, with explicit constants given in the proof. 
 \label{item:reflex_ineq_c}
\end{enumerate}
 \end{proposition}

 \begin{proof}
   The proof is postponed to \Cref{propo:drift_strong_convex_wass:proof}.
\end{proof}

 \begin{theorem}
  \label{propo:cvx_outside_bounds}
Assume \tup{\Cref{ass:def_Euler}$(\msx)$} for $\msx \subset \rset^d$, \tup{\Cref{as:item:lip}} and 
  \tup{\Cref{assum:strong_convex_outside_ball}}.  Assume in addition either 
  \tup{\Cref{as:b_min}($\mtt$)} for $\mtt \in \rset_-$ or \tup{\Cref{as:b_lemme_descente}}.
  Then the conditions and the  conclusions of \Cref{theo:discrete_contrac_wass_D_v2} hold with $\bgamma$, $\lambda$ and $A$ given by \Cref{propo:drift_strong_convex_wass}-\ref{item:reflex_ineq_b}, $\tM_{\discrete} = R_1$, $\Kker_{\gamma}$ given by \eqref{eq:coupling_form} for any $\gamma \in \ocint{0, \bgamma}$, $\lyap = \VlyapDun$ defined in \eqref{eq:def_Wun}, and for any $\gamma \in \ocint{0, \bgamma}$, $\ell \in \nsets$ and $t >0$,
  \begin{align}
\qquad \text{ under \tup{\Cref{as:b_min}($\mtt$)}} \eqsp,    \Psibf(\gamma, \ell, t) & =
                                                                                2\Phibf\defEnsLigne{- t / (2\Xi_{\ell \step}^{1/2}(\kappa)) } \eqsp,
                                                                                \label{eq:Psi_def}\\
                                                                                    \label{eq:Psi_def_2}
     \text{ under \tup{\Cref{as:b_lemme_descente}} } \eqsp,                                       \Psibf(\gamma, \ell, t) &= \begin{cases} 2\Phibf \{-1/2
    \} \qquad  \text{if } \ell \geq \ceil{\Run}^2 \text{ and } t \leq \Run
    \eqsp ,\\ 2\Phibf\defEnsLigne{-t /(2\Xi_{\ell
        \step}^{1/2}(\kappa)) } \qquad  \text{otherwise} \eqsp, \end{cases}                                                                                                                
    \end{align}
    where $\kappa$ is given in \Cref{prop:a1_type}-\ref{item:a1_1} and $\Xi_{\ell \step}$ in \eqref{eq:def_Xi}.
  \end{theorem}
  \begin{proof}
    First, note that for any $\gamma >0$, $\Delta_{\msx}$ is absorbing for $\Kker_{\gamma}$  by definition of the reflection coupling, see \eqref{eq:coupling_form}.
   We assume that \Cref{as:b_min}($\mtt$) holds.  Let $\bgamma \in \oointLigne{0, 2\mttplusun/\Lip^2}$. Using \Cref{propo:drift_strong_convex_wass}-\ref{item:reflex_ineq_b} we obtain that $\VlyapDun$ given by \eqref{eq:def_Wun} satisfies \hyperlink{ass:drift_discrete}{$\bfDd(\VlyapDun, \lambda^{\gamma}, A\gamma, \Delta_{\msx,\Run})$} for any $\gamma \in \ocint{0, \bgamma}$ with $\lambda$ and $A$ given in \eqref{eq:const_drift_wass}. Using \Cref{theo:minorization_general}, \Cref{prop:a1_type}-\ref{item:a1_1}, we have for any $\gamma \in \ocint{0, \bgamma}$, $\ell \in \nsets$ and $x,y \in \msx$
  \begin{equation}
    \Kcoupling_{\gamma}^{\ell \step}((x,y), \Deltar^{\complementary}) \leq 1 - 2 \Phibf(-\Xi_{\ell \step}^{-1/2}(\kappa)\norm{x-y}/2) \eqsp , \label{eq:mino_complete}
  \end{equation}
  where $\kappa(\gamma) = -2\mtt + \gamma \Lip^2$, which concludes the proof.

   The proof under \Cref{as:b_lemme_descente} follows the same lines upon noting that \Cref{as:b_lemme_descente} implies that  \Cref{as:b_min}($0$) holds and using \Cref{prop:a1_type}-\ref{item:a1_2} instead of \Cref{prop:a1_type}-\ref{item:a1_1}.
\end{proof}
  Let $\bgamma \in \oointLigne{0, \max(2\mttplusun/\Lip^2,1)}$, $\ell \in \nsets$ specified below, $\lambda_{\bgamma, a}, \rho_{\bgamma, a} \in \ooint{0,1}$ and $D_{\bgamma, 1, a}$, $D_{\bgamma, 2, a}$, $C_{\bgamma, a} \geq 0$ the constants given by \Cref{propo:cvx_outside_bounds}, such that for any $k \in \nset$, $\gamma \in \ocint{0, \bgamma}$ and $x,y \in \msx$
  \begin{equation}
    \wasscun(\updelta_x \Rcoupling_{\gamma}^k, \updelta_y \Rcoupling_{\gamma}^k) \leq \Kker_{\gamma}^k \bfc_1(x,y) \leq \lambda_{\bgamma, a}^{k\gamma/4} [D_{\bgamma,1,a} \bfc_1(x,y) + D_{\bgamma,2,a}\1_{\Deltar^{\complementary}}] +  C_{\bgamma, a} \rho_{\bgamma, a}^{k\gamma/4}  \eqsp ,  \label{eq:wc1_convergence}\end{equation}
  with $\bfc_1(x,y) = \1_{\Deltar^{\complementary}}(x,y) (1 + \norm{x-y}/R_1)$ for any $x,y \in \msx$.
Note that by \eqref{eq:convergence_c},  this result implies that for any $k \in \nset$, $\gamma \in \ocint{0, \bgamma}$ and $x,y \in \msx$
  \begin{equation}
    \label{eq:wc1_convergence_true}
    \wasscun(\updelta_x \Rcoupling_{\gamma}^k, \updelta_y \Rcoupling_{\gamma}^k) \leq \defEns{D_{\bgamma, 1, a} + D_{\bgamma, 2, a} + C_{\bgamma, a} } \rho_{\bgamma, a}^{k \gamma} \bfc_1(x,y) \eqsp .
  \end{equation}
  We now give upper-bounds on $\rho_{\bgamma, a}$.
Note that using \Cref{theo:discrete_contrac_wass_D_v2}, we obtain that the following limits exist and do not depend on~$\Lip$
\begin{equation}
 D_{1,a} = \lim_{\bgamma \to 0} D_{\bgamma, 1, a}\eqsp , \quad D_{2,a} = \lim_{\bgamma \to 0} D_{\bgamma, 2, a}\eqsp , \quad C_a = \lim_{\bgamma \to 0}C_{\bgamma, a}\eqsp , \quad \lambda_a = \lim_{\bgamma \to 0}\lambda_{\bgamma, a} \eqsp , \quad \rho_a = \lim_{\bgamma \to 0}\rho_{\bgamma, a} \eqsp .\label{eq:inde_1}\end{equation}
Once again, we point out that $\lambda_{\bgamma, a} \leq \rho_{\bgamma, a}$ in \Cref{theo:discrete_contrac_wass_D_v2}.
In the following discussion we assume that \tup{\Cref{ass:def_Euler}$(\msx)$} for $\msx \subset \rset^d$, \tup{\Cref{as:item:lip}} and \tup{\Cref{assum:strong_convex_outside_ball}} hold.
We now give upper bounds on the rate
$\rho_{\bgamma, a}$ and $\rho_a$ using
\Cref{theo:discrete_contrac_wass_D_v2} depending on the assumptions in \Cref{propo:cvx_outside_bounds}.
  
\begin{enumerate}[label=(\alph*),wide, labelwidth=!, labelindent=0pt]
\item If \Cref{as:b_lemme_descente} holds,  set $\ell = \ceil{\Run^2}$. Using that $2 \Phibf(-1/2) \leq 1 - \rme^{-1}$ and choosing $\mttplusun$ sufficiently small such that the conditions of \Cref{theo:discrete_contrac_wass_D_v2} hold, we have
    \begin{align}
      &\log^{-1}(\rho_{\bgamma, a}^{-1}) \leq \left[1 + \log(2) + \log\parenthese{1 + 2(1 + \Run^2)\mttplusun} \right . \\
     & \qquad \qquad \qquad \qquad \left .  \left . + 2(1 + \Run^2)(\mttplusun - \bgamma \Lip^2/2)\right] \middle/ \parentheseDeux{(\mttplusun - \bgamma \Lip^2/2) \Phibf\defEnsLigne{-1/2} } \right. \eqsp . \label{eq:rho_bgamma_wass_a}
\end{align}
  Taking the limit $\bgamma \to 0$ in \eqref{eq:rho_bgamma_wass_a} and using that for any $t \geq 0$, $\log(1+t)\leq t$, we get that
    \begin{equation}
      \log^{-1}(\rho_{a}^{-1}) \leq (1+ \log(2)) / \parentheseLigne{\mttplusun \Phibf\defEns{-1/2}} + 4(1+\Run^2)/\Phibf\defEns{-1/2}\eqsp . \label{eq:rho_bgamma_wass_majo_a_cont_disc}
\end{equation}
The leading term in \eqref{eq:rho_bgamma_wass_majo_a_cont_disc} is of order $\max(\Run^2, 1/\mttplusun)$,
which corresponds to the one identified in \cite[Theorem 2.8]{eberle2018quantitative} and is optimal, see \cite[Remark 2.10]{eberle2016reflection}.
\item If \Cref{as:b_min}($\mtt$) holds with $\mtt \in \rset_-$, set $\ell = \ceil{R_1^2}$. Choosing $\mttplusun >0$ sufficiently small and $\Run, |\mtt|$ sufficiently large such that the conditions of \Cref{theo:discrete_contrac_wass_D_v2} hold, we have
  \begin{multline}
    \log^{-1}(\rho_{\bgamma, a}^{-1}) \leq  \left[1 + \log(2) + \log\parenthese{1 + 2(1+ R_1^2)\defEnsLigne{\mttplusun - \mtt}} \right . \\
    \left .  \left . + 2(1 + R_1^2)(\mttplusun - \bgamma \Lip^2/2)\right] \middle/ \parentheseDeux{(\mttplusun - \bgamma \Lip^2/2) \Phibf\defEnsLigne{-\Xi_{\bstep \ceil{\Run^2}}^{-1/2}(\kappa)\Run/2}} \right. \eqsp . \label{eq:rho_bgamma_wass_majo_b}
  \end{multline}  
  Taking the limit $\bgamma \to 0$ in this result and using \eqref{eq:def_Xi}, we get that
  \begin{align}
    &\log^{-1}(\rho_{a}^{-1}) \leq  \left[1 + \log(2) + \log\parentheseLigne{1 + 2\defEnsLigne{\mttplusun - \mtt}}  + 2\mttplusun\right]    \left .   \vphantom{1 + \log(2) + \log\parentheseLigne{1 + 4\defEnsLigne{\mttplusun - \mtt}}  + 3\mttplusun} \middle/ \parentheseDeux{\mttplusun \Phibf\defEnsLigne{-(-\mtt)^{1/2}\Run/(2 - 2\rme^{2\mtt \Run^2})^{1/2}}} \right. \eqsp . \label{eq:bound_rho_a_bgamma_0_comp}
  \end{align}
  The comparison between this rate and the ones derived in recent works is conducted in \Cref{sec:non-contr-sett}. We extend our result to other Wasserstein metrics in the following proposition.
  \begin{corollary}
    \label{coro:w1_wp}
    Assume \tup{\Cref{ass:def_Euler}$(\msx)$} for $\msx \subset \rset^d$, \tup{\Cref{as:item:lip}} and 
  \tup{\Cref{assum:strong_convex_outside_ball}}.  Assume in addition either 
  \tup{\Cref{as:b_min}($\mtt$)} for $\mtt \in \rset_-$ or \tup{\Cref{as:b_lemme_descente}}. Then for any $p \in \nset$, $\upalpha \in \ooint{p, +\infty}$, $\gamma \in \ocint{0, \bgamma}$, $x,y \in \msx$ and $k \in \nset$ we have
    \begin{align}
      & \wassersteinD[1](\updelta_x \Rker_{\gamma}^k, \updelta_y \Rker_{\gamma}^k ) \leq D_{3, \bgamma,a} \rho_{\bgamma, a}^{k \gamma/4} \norm{x-y} \eqsp , \label{eq:w1_contrac_a}\\
      & \wassersteinD[p](\updelta_x \Rker_{\gamma}^k, \updelta_y \Rker_{\gamma}^k ) \leq D_{\upalpha, \bgamma,a} \rho_{\bgamma, a}^{k \gamma/(4\upalpha)} \defEns{ \norm{x-y} + \norm{x-y}^{1/\upalpha}} \label{eq:wp_contrac_a} \eqsp ,
    \end{align}
  where $\rho_{\bgamma, a}$, $D_{3, \bgamma, a}$ and $D_{\upalpha, \bgamma, a}$  are given in \eqref{eq:wc1_convergence},  \Cref{prop:w1_contrac_iterees} and  \Cref{prop:from_1_to_p} respectively.
  \end{corollary}

  \begin{proof}
    The proof is postponed to \Cref{coro:w1_wp:proof}.
  \end{proof}

\end{enumerate}

\subsubsection{Other curvature conditions}
\label{sec:other-curv-cond_disc}

We now derive uniform ergodic convergence in $V$-norm under weaker conditions than \Cref{assum:strong_convex_outside_ball}.
The following assumption ensures that the radial part of $b$ decrease faster
than a linear function with slope $-\mttplusdeux < 0$.
\begin{assumptionC}
  \label{assum:drift_strong}
There exist $\Rdeux \geq 0$ and $\mttplusdeux >0$ such that for any $x \in \cball{0}{\Rdeux}^{\complementary} \cap \msx$,
    \begin{equation}\ps{b(x)}{x} \leq      -\mttplusdeux \norm[2]{x} \eqsp . \end{equation}
  \end{assumptionC}
  In the next proposition we derive a Foster-Lyapunov drift condition for $\VlyapDdeux: \ \msx^2 \to \coint{1,+\infty}$ defined for any $x, y \in \msx$ by
  \begin{equation}
    \VlyapDdeux(x,y) = 1 + \norm{x}^2/2 + \norm{y}^2/2 \eqsp , \qquad \bfc_2(x,y) = \1_{\Delta_{\msx}}(x,y) \VlyapDdeux(x,y) \eqsp .    \label{eq:def_V_1}
  \end{equation}
  Note that for any $x,y \in \msx$, $    \VlyapDdeux(x,y) = \defEns{V(x) +V(y)}/2$ with $V(x) = 1 + \norm[2]{x}$.
  
\begin{proposition}
  \label{propo:drift_strong_convex}
    Assume \tup{\Cref{ass:def_Euler}$(\msx)$} for $\msx \subset \rset^d$, \tup{\Cref{as:item:lip}}, \tup{\Cref{as:b_min}($\mtt$)} for $\mtt \in \rset_-$ and \tup{\Cref{assum:drift_strong}}.
  Then $\Kcoupling_{\gamma}$ defined by \eqref{eq:coupling_form} satisfies \hyperlink{ass:drift_discrete}{$\bfDd(    \VlyapDdeux,\lambda^{\gamma},A \gamma, \cball{0}{R} \times \cball{0}{R})$} for any $\gamma \in \ocint{0,\bgamma}$ where $\bgamma \in \oointLigne{0, 2\mttplusdeux/\Lip^2}$ and
  \begin{equation}
\lambda = \exp \parentheseDeuxLigne{ - (\mttplusdeux - \bgamma \Lip^2/2)} \eqsp , \quad  A= d + 2\Rdeux^2(\mttplusdeux - \mtt) + 2 \mttplusdeux \eqsp , \quad R = \sqrt{2}\lambda^{-\bgamma} A^{1/2} \log^{-1/2}(1/\lambda) \eqsp . \label{eq:lambda_A_drift_strong}
 \end{equation}
\end{proposition}

\begin{proof}
The proof is postponed to \Cref{propo:drift_strong_convex:proof}.
\end{proof}

\begin{theorem}
  \label{propo:drift_strong_convex_bounds}
      Assume \tup{\Cref{ass:def_Euler}$(\msx)$} for $\msx \subset \rset^d$, \tup{\Cref{as:item:lip}} and \tup{\Cref{assum:drift_strong}}. 
      Assume in addition either
      \tup{\Cref{as:b_min}($\mtt$)} for $\mtt \in \rset_-$ or \tup{\Cref{as:b_lemme_descente}}. 
      Then the conditions and  conclusions of \Cref{theo:discrete_contrac_wass_D_v2} hold with $\lyap = \VlyapDdeux$ defined in \eqref{eq:def_V_1},  $\bgamma$, $\lambda$, $A$ and $\tM_{\discrete} = 2R$ given by \Cref{propo:drift_strong_convex},
  and $\Psibf$ given by \eqref{eq:Psi_def} or \eqref{eq:Psi_def_2}.
\end{theorem}

\begin{proof}
  The proof is similar to the one of \Cref{propo:cvx_outside_bounds}.
\end{proof}
  Let $\bgamma \in \oointLigne{0, \max(2\mttplusdeux/\Lip^2,1)}$, $\ell \in \nsets$ specified below, $\lambda_{\bgamma, b}, \rho_{\bgamma, b} \in \ooint{0,1}$ and $D_{\bgamma, 1, b}$, $D_{\bgamma, 2, b}$, $C_{\bgamma, b} \geq 0$ the constants given by \Cref{propo:drift_strong_convex_bounds}, such that for any $k \in \nset$, $\gamma \in \ocint{0, \bgamma}$ and $x,y \in \msx$
  \begin{equation}
    \wasscdeux(\updelta_x \Rcoupling_{\gamma}^k, \updelta_y \Rcoupling_{\gamma}^k) \leq \Kker_{\gamma}^k \bfc_2(x,y) \leq \lambda_{\bgamma, b}^{k\gamma/4} [D_{\bgamma,1,b} \bfc_2(x,y) + D_{\bgamma,2,b}\1_{\Deltar^{\complementary}}] +  C_{\bgamma, b} \rho_{\bgamma, b}^{k\gamma/4}  \eqsp ,  \label{eq:wc1_convergence}\end{equation}
  with $\bfc_2(x,y) = \1_{\Deltar^{\complementary}}(x,y) \defEnsLigne{V(x) + V(y)}/2$ for any $x,y \in \msx$.
Note that by \eqref{eq:def_V_1},  this result implies that for any $k \in \nset$, $\gamma \in \ocint{0, \bgamma}$ and $x,y \in \msx$
  \begin{equation}
    \label{eq:wc2_convergence_true}
    \Vnorm{\updelta_x \Rcoupling_{\gamma}^k - \updelta_y \Rcoupling_{\gamma}^k} \leq \defEns{D_{\bgamma, 1, b} + D_{\bgamma, 2, b} + C_{\bgamma, b} } \rho_{\bgamma, b}^{k \gamma} \bfc_2(x,y) \eqsp .
  \end{equation}

  Note that using \Cref{theo:discrete_contrac_wass_D_v2}, we obtain that the following limits exist and do not depend on~$\Lip$
\begin{equation}
D_{1,b} = \lim_{\bgamma \to 0} D_{\bgamma, 1, b}\eqsp , \quad D_{2,b} = \lim_{\bgamma \to 0} D_{\bgamma, 2, b}\eqsp , \quad C_b = \lim_{\bgamma \to 0}C_{\bgamma, b}\eqsp , \quad \lambda_b = \lim_{\bgamma \to 0}\lambda_{\bgamma, b} \eqsp , \quad \rho_b = \lim_{\bgamma \to 0}\rho_{\bgamma, b} \eqsp . \label{eq:inde_2}\end{equation}
We now discuss the dependency of $\rho_b$ with respect to the introduced parameters, depending on the sign of $\mtt$ and  based on \Cref{theo:discrete_contrac_wass_D_v2}. 
\begin{enumerate}[label= (\alph*), wide, labelwidth=!, labelindent=0pt]
\item If \Cref{as:b_lemme_descente} holds, set $\ell = \ceil{\tM_{\discrete}^2}$. Then, if  we consider $\mttplusdeux$ sufficiently small and $\abs{\mtt}$ and $\Rdeux$ sufficiently large such that the conditions of \Cref{theo:discrete_contrac_wass_D_v2} hold, we have
    \begin{align}
      &\log^{-1}(\rho_b^{-1}) \\
      & \qquad \leq \left. \parentheseDeux{1 + 2\log(1+ R^2) + \log(1+2A) + 2(1 + 4R^2) \mttplusdeux} \middle/ \parentheseDeux{\mttplusdeux\Phibf(-1/2)}  \right. \label{eq:rho_bgamma_1_majo_lim} \eqsp.
  \end{align}
    Note that the leading term on the right hand side of this equation is of order $R^2$, \ie \ of order $\max(\Rdeux^2, d/\mttplusdeux)$.
\item If \Cref{as:b_min}($\mtt$) with $\mtt \in \rset_-$, set
  $\ell = \ceil{\tM_{\discrete}^2}$. Then,  if  we consider $\mttplusdeux$ sufficiently small and $\abs{\mtt}$ and $\Rdeux$ sufficiently large such that the conditions of \Cref{theo:discrete_contrac_wass_D_v2} hold, we have 
  \begin{align}
    &\log^{-1}(\rho_b^{-1}) \leq \left. \parentheseDeux{1 + 2\log(1+ R^2) + \log(1+2A) + 2(1 + 4R^2) \mttplusdeux} \right . \\ & \qquad \qquad \qquad \left . \middle/ \parentheseDeux{\mttplusdeux \Phibf\defEnsLigne{-2(-\mtt)^{1/2}R/(2 - 2\rme^{2\mtt R^2})^{1/2}}}  \label{eq:rho_bgamma_2_majo_lim} \right. \eqsp ,
  \end{align}
Note that the right hand side of \eqref{eq:rho_bgamma_2_majo_lim}  is exponential in $-\mtt R^2$, \ie \ exponential  in $-\mtt d / \mttplusdeux$ and $-\Rdeux^2(\mttplusdeux - \mtt)\mtt / \mttplusdeux$.
  \end{enumerate}
  We now consider a condition which enforces weak curvature outside of a compact set.
\begin{assumptionC}
  \label{assum:curvature}
  There exist $\Rtrois, \cconst \geq 0$, $\mttun, \mttdeux >0$, such that for any $x \in \rset^d$
  \[ \langle b(x), x \rangle \leq - \mttun \| x \| \1_{\cball{0}{\Rtrois}^{\complementary}}(x) -\mttdeux \| b (x) \|^2 + \cconst/2 \eqsp.\]
\end{assumptionC}
In the case where $\msx = \rset^d$, $\Pi_{\msx} = \Id$ and there exists $U \in \rmC^1(\rset^d, \rset)$ such that \tup{\Cref{as:item:lip}} and  \tup{\Cref{as:b_min}($0$)} hold with $b = - \nabla U$ and $\int_{\rset^d} \rme^{-U(x)} \rmd x < +\infty$, then there exist $\Rtrois \geq 0$ and $\mttun >0$ such that \Cref{assum:curvature} holds with $\mttdeux = \cconst = 0$, see \cite[Lemma 2.2]{barky2008simple}.
Define $V : \ \msx \to \coint{1,+\infty}$ for any $x \in \msx$ by
  \begin{equation}
    \label{eq:def_V_2}
    V(x) = \exp(\mtttrois\phi(x)) \eqsp , \qquad \phi(x) =  \sqrt{1 + \norm{x}^2} \eqsp , \qquad \mtttrois \in \ocint{0, \mttun / 2} \eqsp .
  \end{equation}
  We also define for any $x,y \in \msx$,
  \begin{equation} \VlyapDtrois(x,y) = \defEns{V(x) + V(y)}/2 \eqsp , \qquad \bfc_3(x,y) = \1_{\Delta_{\msx}}(x,y) \VlyapDtrois(x,y) \eqsp . \label{eq:def_c3} \end{equation}
\begin{proposition}
  \label{propo:drift_convex}
   Assume \tup{\Cref{ass:def_Euler}$(\msx)$} for $\msx \subset \rset^d$ and \tup{\Cref{assum:curvature}}.
  Then for any $\gamma \in \ocint{0,\bgamma}$, $\Kcoupling_{\gamma}$ defined by \eqref{eq:coupling_form} satisfies \hyperlink{ass:drift_discrete}{$\bfDd(\VlyapDtrois,\lambda^{\gamma},A \gamma, \cballdeux{0}{R})$} where $\bgamma \in (0, 2\mttdeux)$, $\Rquatre = \max(1, \Rtrois, (d+\cconst)/\mttun)$ and
  \begin{equation}
    \label{eq:trois}
    \begin{aligned}
      & \lambda = \rme^{-(\mtttrois)^2/2} \eqsp , \\
      &A = \exp\parentheseDeux{\bgamma (\mtttrois (d+\cconst)+ (\mtttrois)^2)/2 + \mtttrois (1 + \Rquatre^2)^{1/2}}(\mtttrois (d+\cconst)/2 + (\mtttrois)^2) \eqsp , \quad R = \log(2\lambda^{-2\bgamma} A \log^{-1}(1/\lambda)) \eqsp .
      \end{aligned} \label{eq:const_drift_vexp}
 \end{equation}
\end{proposition}

\begin{proof}
  The proof is postponed to \Cref{propo:drift_convex:proof}.
\end{proof}
 \begin{theorem}
   \label{propo:weak_outside_bounds}
      Assume \tup{\Cref{ass:def_Euler}$(\msx)$} for $\msx \subset \rset^d$, \tup{\Cref{as:item:lip}} and \tup{\Cref{assum:curvature}}. 
      Assume in addition either
      \tup{\Cref{as:b_min}($\mtt$)} for $\mtt \in \rset_-$ or \tup{\Cref{as:b_lemme_descente}}. 
      Then the conditions and  conclusions of \Cref{theo:discrete_contrac_wass_D_v2} hold with $\lyap = \VlyapDdeux$ defined in \eqref{eq:def_V_1},  $\bgamma$, $\lambda$, $A$ and $\tM_{\discrete} = 2R$ given by \Cref{propo:drift_convex},
  and $\Psibf$ given by \eqref{eq:Psi_def} or \eqref{eq:Psi_def_2}.   
\end{theorem}
\begin{proof}
  The proof is similar to the one of \Cref{propo:cvx_outside_bounds}.
\end{proof}    
The dependency of the rate given by \Cref{propo:weak_outside_bounds}
with respect to the constants is discussed in \Cref{sec:rates-crefpr}.


\section{Quantitative convergence bounds  for diffusions}
\label{sec:quant-conv-bounds}
\subsection{Main results}
\label{sec:main-results-3}
In this section, we aim at deriving quantitative convergence bounds with respect to some Wasserstein metrics
for diffusion processes under regularity and curvature assumptions on the drift $b$. Consider 
the following \SDE 
\begin{equation}
  \rmd \bfX_t = b(\bfX_t) \rmd t + \rmd \bfB_t \eqsp ,\label{eq:diff}
\end{equation}
where $(\bfB_t)_{t \geq 0}$ is a $d$-dimensional Brownian motion and $b : \ \rset^d \to \rset^d$ is a continuous drift. 

  When there exists a unique strong solution $(\bfX_t)_{t \geq 0}$  of \eqref{eq:diff} for any starting point $\bfX_0 = x$, with $x \in \rset^d$, we define the semi-group $(\Pker_t)_{t\geq0}$ for any
  $\msa \in \mcb{\rset^d}$, $x \in \rset^d$ and $t\geq0$ by
  $\Pker_t(x, \msa) = \proba{\bfX_t \in \msa}$. We now
  turn to establishing that $(\Pker_t)_{t \geq 0}$ converges for some
  Wasserstein metrics.  In order to prove this result we will rely on
  discretizations of the \SDE \ \eqref{eq:diff}.  If the conditions of
  \Cref{theo:discrete_contrac_wass_D_v2} are satisfied, these discretized
  processes are uniformly geometrically ergodic and taking the limit when
  the discretization stepsize goes to zero, we obtain the convergence of
  the associated diffusion processes.  

  First, assume that $b$ is Lipschitz regular. We establish in \Cref{thm:limit_lip} that for
  any $T \geq 0$ and $x, y \in \rset^{\dim}$, the Wasserstein distance
  $\wassersteinD[\bfc](\updelta_x \Pker_T, \updelta_y \Pker_T)$ is upper-bounded
  by the upper limit when $m \to +\infty$ of
  $\wassersteinD[\bfc](\updelta_x \Rker_{T/m}^m, \updelta_y \Rker_{T/m}^m)$,
  where $\Rker_{\gamma}$ is given for any $\gamma >0$ in
  \eqref{eq:kernel_langevin} and $\bfc$ is given in \eqref{eq:def_wbf}.

  Second, this result is extended in \Cref{thm:limit_loc_lip} to cover the case
  where $b$ is no longer Lipschitz regular but only locally Lipschitz regular,
  see \Cref{ass:loc_lip_bb}.  \Cref{thm:limit_lip} and \Cref{thm:limit_loc_lip}
  are applications of a more general theory developed in
  \Cref{sec:proof-crefs-results-3}. Let $M \geq 0$, we consider for any $x \in \rset^{\dim}$
  \begin{equation}
    \label{eq:lyap_m}
    V_M(x,y) = \exp[M\phi(x)] \eqsp , \qquad \phi(x) = (1 + \norm{x})^{1/2} \eqsp .
  \end{equation}

  \begin{theorem}
    \label{thm:limit_lip}
    Assume \tup{\Cref{as:item:lip}} and
    $\sup_{x \in \rset^{\dim}} \langle x, b(x) \rangle < +\infty$. Then, for any
    starting point $\bfX_0 = x$, with $x \in \rset^d$, there exists a unique
    strong solution to \eqref{eq:diff}. In addition, for any
    $\lyap: \ \rset^{\dim}\times \rset^{\dim} \to \coint{1,+\infty}$ satisfying
    $\sup_{(x,y) \in \rset^{\dim} \times \rset^{\dim}} \defEnsLigne{\lyap(x,y) (V_M(x) + V_M(y))^{-1}} < +\infty$ with
    $M \geq 0$ and $V_M$ given in \eqref{eq:lyap_m}, we get that for any
    $x,y \in \rset^{\dim}$ and $T \geq 0$
          \begin{equation}
        \distV(\updelta_x \Pker_{\Time}, \updelta_y \Pker_{\Time}) \leq \limsup_{m \to +\infty} \distV(\updelta_x \Rker_{\Time/m}^{m}, \updelta_y \Rker_{\Time/m}^{m}) \eqsp, \label{eq:conclu_1}
      \end{equation}
      where $\bfc$ is given by \eqref{eq:def_wbf}, $(\Pker_t)_{t \geq 0}$ is the semigroup associated with \eqref{eq:diff} and  for any $\gamma\in \ocint{0, \bgamma}$, $\Rker_{\gamma}$ is the Markov kernel associated with \eqref{eq:langevin_discrete} where $\Tg(x) = x + \gamma b(x)$, $\msx = \rset^{\dim}$ and $\Pi = \Id$.
    \end{theorem}

    \begin{proof}
      The proof is postponed to \Cref{thm:limit_lip:proof}.
    \end{proof}
    
We now weaken the Lipschitz regularity assumption and consider the following condition on the drift $b$.
\begin{assumptionB}
  \label{ass:loc_lip_bb}
      $b$ is locally Lipschitz, \ie~for any $M \geq 0$, there exists $\Lip_M \geq 0$ such that for any
    $x,y \in \cball{0}{M}$, $\norm{b(x)-b(y)} \leq \Lip_M \norm{x-y}$ and $b(0) = 0$.
  \end{assumptionB}
  As a consequence, under a mild integrability assumption, which will be satisfied in all of our applications,
  we obtain the following generalization of \Cref{thm:limit_lip}.

  \begin{theorem}
    \label{thm:limit_loc_lip}
    Assume \tup{\Cref{as:b_min}($\mtt$)}, \tup{\Cref{ass:loc_lip_bb}} and that
    $\sup_{x \in \rset^{\dim}} \langle x, b(x) \rangle < +\infty$.  Then, for
    any starting point $\bfX_0 = x$, with $x \in \rset^d$, there exists a unique
    strong solution of \eqref{eq:diff}.  In addition assume that for any
    $x \in \rset^{\dim}$ and $T \geq 0$ there exists $\vareps_b >0$
    such that
    \begin{equation}
      \label{eq:pre_L2}
  \sup_{s \in \ccint{0,\Time}} \defEns{ \updelta_x \Pker_s \norm[2(1+\vareps_b)]{b(x)}} < \plusinfty \eqsp .
\end{equation}
Then, for any
    $\lyap: \ \rset^{\dim}\times \rset^{\dim} \to \coint{1,+\infty}$ satisfying
    $\sup_{(x,y) \in \rset^{\dim} \times \rset^{\dim}} \defEnsLigne{\lyap(x,y) (V_M(x) + V_M(y))^{-1}} < +\infty$ with
    $M \geq 0$ and $V_M$ given in \eqref{eq:lyap_m}, we get that for any
$x,y \in \rset^{\dim}$ and $T \geq 0$
      \begin{equation}
        \distV(\updelta_x \Pker_{\Time}, \updelta_y \Pker_{\Time}) \leq \limsup_{n \to +\infty} \limsup_{m \to +\infty} \distV(\updelta_x \Rker_{\Time/m, n}^{m}, \updelta_y \Rker_{\Time/m, n}^{m}) \eqsp, \label{eq:conclu_1}
      \end{equation}
where for any $x,y \in \rset^d$, $\bfc(x,y) = \1_{\Deltar^{\complementary}}(x,y) \lyap(x,y)$, $(\Pker_t)_{t \geq 0}$ is the semigroup associated with \eqref{eq:diff} and  for any $\gamma\in \ocint{0, \bgamma}$, $n \in \nset$,  $\Rker_{\gamma, n}$ is the Markov kernel associated with \eqref{eq:langevin_discrete} where $\Tg(x) = x + \gamma b(x)$, $\msx = \cball{0}{n}$ and $\Pi = \Pi_{\cball{0}{n}}$.
\end{theorem}

\begin{proof}
  The proof is postponed to \Cref{thm:limit_loc_lip:proof}.
\end{proof}

Note that \eqref{eq:pre_L2} holds under mild conditions on the drift function, see \Cref{prop:existence_integr}.
In the next section we apply these results to diffusion processes and derive
sharp convergence bounds in the case where $b$ satisfies some curvature
assumption, similarly to \Cref{sec:applications}.

\subsection{Applications}
\label{sec:applications-1}
In this section, we combine the results of \Cref{thm:limit_loc_lip} with the convergence bounds for discrete processes derived in \Cref{sec:applications}, in order to obtain convergence bounds for continuous processes that are solutions of \eqref{eq:diff}.

\subsubsection{Strongly convex at infinity}

\begin{theorem}
  \label{propo:cv_wass_continuous}
  Assume either \tup{\Cref{as:b_min}($\mtt$)} for $\mtt \in \rset_-$
  or \tup{\Cref{as:b_lemme_descente}}.  Assume
  \tup{\Cref{assum:strong_convex_outside_ball}},
  \tup{\Cref{ass:loc_lip_bb}}. In addition, assume
  $\sup_{x \in \rset^d}\{ \norm{b(x)}^{2(1+\vareps_b)}
  \rme^{-\mttplusun\norm{x}^2}\} < +\infty$ for some $\vareps_b >0$.
  Then, for any $\Time \geq 0$, and $x,y \in \rset^d$
  \begin{equation}
    \label{eq:wass_c1_cont}
    \distV[\bfc_1](\updelta_x \Pker_{\Time}, \updelta_y \Pker_{\Time}) \leq \lambda_a^{T/4}(D_{1,a} \bfc_1(x,y) + D_{2,a} \1_{\Delta^{\complementary}}(x,y)) + C_a \rho_a^{T/4} \1_{\Delta^{\complementary}}(x,y) \eqsp , 
  \end{equation}
  with $D_{1,a}, D_{2,a}, C_a \geq 0$, $\lambda_a, \rho_a \in \ooint{0,1}$ given by \eqref{eq:inde_1} and
  for any $x,y \in \rset^d$, $\bfc_1(x,y) = \1_{\Deltar^{\complementary}}(x,y) \VlyapDun(x,y)$ with $\VlyapDun(x,y) = 1 + \norm{x-y}/R_1$.
\end{theorem}

\begin{proof}
  Le $\Time \geq 0$ and $x,y \in \rset^{\dim}$. Using \Cref{thm:limit_lip} or
  \Cref{prop:existence_integr} and \Cref{thm:limit_loc_lip} we have
        \begin{equation}
          \distV[\bfc_1](\updelta_x \Pker_{\Time}, \updelta_y \Pker_{\Time}) \leq \limsup_{n \to +\infty} \limsup_{m \to +\infty} \distV[\bfc_1](\updelta_x \Rker_{\Time/m, n}^{m}, \updelta_y \Rker_{\Time/m, n}^{m}) \eqsp .
      \end{equation}
      Let $n \in \nset$ and $m \in \nsets$ such that $x,y \in \cball{0}{n}$ and
      $T/m \leq 2\mttplusun / \Lip_n^2$.  Since
      \Cref{ass:def_Euler}$(\cball{0}{n})$ holds and \Cref{ass:loc_lip_bb}
      implies \Cref{as:item:lip} on $\cball{0}{n}$, we can apply
      \Cref{propo:cvx_outside_bounds} and we get
      \begin{multline}
        \distV[\bfc_1](\updelta_x \Rker_{\Time/m, n}^{m}, \updelta_y
        \Rker_{\Time/m, n}^{m}) \leq \lambda_{\Time/m, a}^{\Time/4}
        (D_{\Time/m, 1, a} \bfc_1(x,y) + D_{\Time/m, 2,
          a}\1_{\Delta^{\complementary}}(x,y)) + C_{\Time /m, a} \rho_{\Time/m,
          a}^{\Time /4}\1_{\Delta^{\complementary}}(x,y) \eqsp ,
      \end{multline}
      where $D_{\Time/m, 1, a}, D_{\Time/m, 2, a}, C_{\Time/m, a}, \lambda_{\Time/m, a}$ and $\rho_{\Time/m, a}$ are given in \eqref{eq:wc1_convergence}. In addition, these quantities admit limits $D_{1,a}, D_{2,a}, C_a \geq 0$ and $\lambda_a, \rho_a\in \ooint{0,1}$ when $m \to +\infty$ which do not depend on $\Lip_n$, hence on $n$, see \eqref{eq:inde_1}.
    \end{proof}
    Note that \Cref{as:item:lip} implies \Cref{ass:loc_lip_bb} and
    $\sup_{x \in \rset^d}\{ \norm{b(x)}^{2(1+\vareps_b)}
    \rme^{-\mttplusun\norm{x}^2}\} < +\infty$ for some $\vareps_b >0$.
    \begin{corollary}
      \label{coro:cont_w1_wp}
  Assume either \tup{\Cref{as:b_min}($\mtt$)} for $\mtt \in \rset_-$ or
  \tup{\Cref{as:b_lemme_descente}}.  Assume 
  \tup{\Cref{assum:strong_convex_outside_ball}}, 
  \tup{\Cref{ass:loc_lip_bb}} and also that 
  $\sup_{x \in \rset^d}\{ \norm{b(x)}^{2(1+\vareps_b)}
  \rme^{-\mttplusun\norm{x}^2}\} < +\infty$ for some $\vareps_b >0$.  Then, for any $p \in \nset$, $\upalpha \in \ooint{p, +\infty}$, $\Time \geq 0$, and $x,y \in \rset^d$ we have
    \begin{align}
      & \wassersteinD[1](\updelta_x \Pker_T, \updelta_y \Pker_T) \leq D_{3, a} \rho_{a}^{T/4} \norm{x-y} \eqsp , \label{eq:w1_contrac_a_cont}\\
      & \wassersteinD[p](\updelta_x \Pker_T, \updelta_y \Pker_T) \leq D_{\upalpha, a} \rho_{a}^{T/(4\upalpha)} \defEns{ \norm{x-y} + \norm{x-y}^{1/\upalpha}} \label{eq:wp_contrac_a_cont} \eqsp ,
    \end{align}
  where $\rho_{a}$ is given in \eqref{eq:inde_1}, $D_{3, a} = \lim_{\bgamma \to 0} D_{3, \bgamma, a}$  and $D_{\upalpha, a} = \lim_{\bgamma \to 0} D_{\upalpha, \bgamma, a}$ with $D_{3, \bgamma, a}$ and  $D_{\upalpha, \bgamma, a}$ given in \Cref{coro:w1_wp}.
    \end{corollary}
    \begin{proof}
      The proof is similar to the one of \Cref{propo:cv_wass_continuous}.
    \end{proof}

    The discussion on the dependency of $\rho_a$ with respect to the parameters
    of the problem conducted in \Cref{sec:applications} still holds. We
    distinguish the following cases, assuming that the conditions of
    \Cref{theo:discrete_contrac_wass_D_v2}
    are satisfied.
\begin{enumerate}[label=(\alph*)]
\item If \Cref{as:b_lemme_descente} holds, we have
    \begin{equation}
      \log^{-1}(\rho_{a}^{-1}) \leq (1+ \log(2)) / (\Phibf\defEnsLigne{-1/2}\mttplusun) + 4\Run^2/\Phibf\defEnsLigne{-1/2}\eqsp . \label{eq:rho_bgamma_wass_majo_a_cont}
\end{equation}

The leading term in \eqref{eq:rho_bgamma_wass_majo_a_cont} is of order $\max(\Run^2, 1/\mttplusun)$,
which corresponds to the one identified in \cite[Lemma 2.9]{eberle2016reflection} and is optimal, see \cite[Remark 2.10]{eberle2016reflection}.
\item If \Cref{as:b_min}($\mtt$) holds with $\mtt \in \rset_-$, we have
  \begin{align}
    &\log^{-1}(\rho_{a}^{-1}) \leq  \left[1 + \log(2) + \log\parentheseLigne{1 + 2\defEnsLigne{\mttplusun - \mtt}\defEnsLigne{1 + \Run^2}}  + 2\mttplusun (1+\Run^2)\right]   \\
    & \qquad \qquad \qquad \qquad \qquad \left .   \vphantom{1 + \log(2) + \log\parentheseLigne{1 + 4\defEnsLigne{\mttplusun - \mtt}}  + 3\mttplusun} \middle/ \parentheseDeux{\mttplusun \Phibf\defEnsLigne{-(-\mtt)^{1/2}\Run/(2 - 2\rme^{2\mtt \Run^2})^{1/2}}} \right. \eqsp . \label{eq:rho_bgamma_wass_majo_b_cont}
  \end{align}  
We now give an upper-bound for  \eqref{eq:rho_bgamma_wass_majo_b_cont} when both $R$ and $\mtt$ are large. For any $t \geq C$ with $C \geq 0$ we have
\begin{equation}
  \Phibf(-t)^{-1} \leq \sqrt{2\uppi}(1 + C^{-2}) t\rme^{t^2/2} \eqsp . \label{eq:majo_vareps}
\end{equation}
As a consequence if we also have $\Run \geq 2$, $1 \leq -\mtt \Run^2$ and using that for any $t \in (0,1)$, $-\log(1-t)\leq t$ as well as \eqref{eq:majo_vareps} we get that $\log^{-1}(\rho_a^{-1}) \leq \log^{-1}(\rhomax^{-1})$
\begin{align}
  &\log^{-1}(\rhomax^{-1}) = C\parentheseDeux{1 + \log(1 + 2\defEnsLigne{\mttplusun - \mtt}\defEnsLigne{1 + \Run^2}) + 2\mttplusun(1 + \Run^2)}\Run(-\mtt)^{1/2}\\ & \qquad \qquad \qquad \qquad \qquad \times \left. \exp\parentheseDeux{-\mtt\Run^2/(4 - 4\rme^{2\mtt \Run^2})} \middle/ \parentheseDeux{\mttplusun(1 - \rme^{2\mtt \Run^2})^{1/2}} \right. \eqsp , \label{eq:eberle_rate}
\end{align}
with $C = 2(1+\log(2))\sqrt{\uppi} \approx 6.00$.
\end{enumerate}
For a comparison of our results with recent works, see \Cref{sec:non-contr-sett}.

\subsubsection{Other curvature conditions}
\label{sec:other-curv-cond_cont}
\begin{theorem}
  \label{thm:curv_b_cont}
  Assume either \tup{\Cref{as:b_min}($\mtt$)} for $\mtt \in \rset_-$
  or \tup{\Cref{as:b_lemme_descente}}.  Assume
  \tup{\Cref{assum:drift_strong}}, \tup{\Cref{ass:loc_lip_bb}}. In
  addition, assume
  $\sup_{x \in \rset^d} \{\norm{b(x)}^{2(1+\vareps_b)}
  \rme^{-\mttplusdeux\norm{x}^2} \}< +\infty$ for some $\vareps_b >0$.
  Then for any $\Time \geq 0$ and $x,y \in \rset^d$
  \begin{equation}
    \Vnorm{\updelta_x \Pker_{\Time} -\updelta_y \Pker_{\Time}} \leq (D_{1, b} + D_{2,b} + C_b) \rho_b^{\Time} \bfc_2(x,y) \eqsp , \label{eq:c_ergo}
  \end{equation}
  with $D_{1,b}, D_{2,b}, C_b \geq 0$ and $\rho_b \in \ooint{0,1}$ given by \eqref{eq:inde_2} and $\bfc_2$ defined in \eqref{eq:def_V_1}.
\end{theorem}

\begin{proof}
The proof is identical to the one of \Cref{propo:cv_wass_continuous} upon replacing \Cref{propo:cvx_outside_bounds} by \Cref{propo:drift_strong_convex_bounds}.
\end{proof}

\begin{theorem}
  \label{theo:c_ergo_v_norm}
  Assume either \tup{\Cref{as:b_min}($\mtt$)} for $\mtt \in \rset_-$
  or \tup{\Cref{as:b_lemme_descente}}.  Assume
  \tup{\Cref{assum:curvature}}, \tup{\Cref{ass:loc_lip_bb}}. In
  addition, assume
  $\sup_{x\in \rset^d} \norm{b(x)}^{2(1 + \vareps_b)} \rme^{-\mttun (1
    + \norm{x})^{1/2}} <+\infty$ for some $\vareps_b >0$.  Then for
  any $T \geq 0$ and $x, y \in \rset^{\dim}$
  \begin{equation}
    \Vnorm{\updelta_x \Pker_{\Time} -\updelta_y \Pker_{\Time}} \leq (D_{1, c} + D_{2,c} + C_c) \rho_c^{\Time} \bfc_3(x,y) \eqsp , \label{eq:c_ergo}
  \end{equation}
  with $D_{1,c}, D_{2,c}, C_c \geq 0$ and $\rho_c \in \ooint{0,1}$
  given by \Cref{sec:rates-crefpr} and $\bfc_3$ defined in
  \eqref{eq:def_c3}.
\end{theorem}

\begin{proof}
  The proof is postponed to \Cref{sec:theo:c_ergo_v_norm:proof}.
\end{proof}
The rates we obtain  in \Cref{thm:curv_b_cont}, respectively \Cref{theo:c_ergo_v_norm}, are identical to the ones derived taking the limit $\bgamma \to 0$ in \Cref{propo:drift_strong_convex_bounds}, respectively \Cref{propo:weak_outside_bounds}. An upper bound on $\rho_b$, respectively $\rho_c$,  is provided in \eqref{eq:rho_bgamma_1_majo_lim} and \eqref{eq:rho_bgamma_2_majo_lim}, respectively \Cref{sec:rates-crefpr}.


\subsection{From discrete to continuous processes}
\label{sec:proof-crefs-results-3}

In this section we present the general theory which leads to \Cref{thm:limit_lip} and \Cref{thm:limit_loc_lip}.
First, we derive bounds between the discrete and continuous process given a family of approximating drift functions in
\Cref{sec:appr-family-drift}. Second, we show in \Cref{sec:regul-appr} that under mild regularity assumptions
on $b$ such families can be explicitly constructed.

\subsubsection{Quantitative convergence bounds for diffusion processes}
\label{sec:appr-family-drift}

We recall that the \SDE \ under study is given by 
\begin{equation}
  \rmd \bfX_t = b(\bfX_t) \rmd t + \rmd \bfB_t \eqsp ,
\end{equation}
where $(\bfB_t)_{t \geq 0}$ is a $d$-dimensional Brownian motion and $b : \ \rset^d \to \rset^d$ is a continuous drift. 
In the sequel we will always consider the following assumption.
\begin{assumptionL}
    \label{ass:loc_lip_b}
    There exists a unique strong solution of \eqref{eq:diff} for any starting point $\bfX_0 = x$, with $x \in \rset^d$.
  \end{assumptionL}
  Under \Cref{ass:loc_lip_b}, the Markov semigroup $\Pker_t$, whose definition is given in \Cref{sec:main-results-3}, exists for any time $t \geq 0$. Consider the extended infinitesimal generator $\generator$ associated with $(\Pker_t)_{t \geq 0}$ and defined for any $f \in \rmC^2(\rset^d, \rset)$ by 
  \begin{equation} \generator f  = (1/2)\Delta f + \langle \nabla f, b \rangle \eqsp . \label{eq:def_generator}
  \end{equation}
  Let $V \in \rmC^2(\rset^d, \coint{1,+\infty})$, $\zeta \in \rset $ and $B \geq 0$
   \begin{assumptionD}[\hypertarget{ass:drift_continuous}{$\bfDc(V,\zeta,B)$}]
    The extended infinitesimal generator $\generator$  satisfies the continuous Foster-Lyapunov drift condition  if  for all $x \in \rset^d$
\begin{equation}
  \label{eq:continuous_drift}
  \mathcal{A}V(x) \leq - \zeta V(x) + B \eqsp .
\end{equation}
\end{assumptionD}
This assumption is the continuous counterpart of \hyperlink{ass:drift_discrete}{$\bfDd(V,\lambda,A,\rset^d)$}. We start by drawing a link between the continuous drift condition \hyperlink{ass:drift_continuous}{$\bfDc(V,\zeta,B)$} and the discrete drift condition \hyperlink{ass:drift_discrete}{$\bfDd(V,\lambda, A, \rset^d)$}. The result and its proof are standard \cite[Theorem 2.1]{meyn1993criteria_iii} but are given here for completeness. Denote by $(\mcf_t)_{t \geq 0}$ the filtration associated with $(\bfB_t)_{t \geq 0}$ satisfying the usual conditions \cite[Chapter I, Section 5]{ikeda1989sto}.
  \begin{lemma}
  \label{lemma:disclyapfromc}
Let $\zeta \in \rset$, $B \geq 0$ and $V \in \rmC^2(\rset^d, \coint{1,+\infty})$ such that $\lim_{\norm{x} \to \plusinfty} V(x) = \plusinfty$. 
Assume \tup{\Cref{ass:loc_lip_b}} and \hyperlink{ass:drift_continuous}{$\bfDc(V,\zeta,B)$}.  
\begin{enumerate}[label=(\alph*)]
\item \label{lemma:disclyapfromc_a} If $B=0$, then for any $x \in \rset^d$, $(V(\bfX_t) \rme^{\zeta t})_{t \geq 0}$ is a $(\mcf_t)_{t \geq 0}$-supermartingale where $(\bfX_t)_{t \geq 0}$ is the solution of \eqref{eq:diff} starting from $\bfX_0 = x$.   \item \label{lemma:disclyapfromc_b} For any $t_0 >0$, $\Pker_{t_0}$ satisfies \hyperlink{ass:drift_discrete}{$\bfDd(V,\exp(-\zeta t_0), B(1-\exp(-\zeta t_0))/\zeta, \rset^d)$}.
\end{enumerate}
 \end{lemma}
 \begin{proof}
      The proof is postponed to \Cref{sec:proof-crefl-1}.
   \end{proof}

Consider a family of drifts
 $\ensembleLigne{\bgM:  \ \rset^d \to \rset^d}{\gamma \in \ocint{0, \bgamma}, n \in \nset}$ for some $\bgamma >0$.
 For all $\gamma\in \ocint{0, \bgamma}$ and $n \in \nset$, we denote by $\tRker_{\gamma, n}$ the Markov kernel associated with \eqref{eq:langevin_discrete} where $\Tg(x) = x + \gamma \bgM(x)$, $\msx = \rset^d$ and $\Pi = \Id$. 
We will show that under the following assumptions the family $\ensembleLigne{\tRker_{\gamma,n}^{\ceil{\Time/\gamma}}:  \ \rset^d \to \rset^d}{\gamma \in \ocint{0, \bgamma}, n \in \nset}$ approximates $\Pker_{\Time}$ for $\Time \geq 0$ as $\gamma \to 0$ and $n \to \plusinfty$.

\begin{assumptionL}
  \label{ass:diff_disc}
  There exist $\beta > 0$ and $C_1  \geq 0$ such that for any $\gamma \in \ocint{0, \bgamma}$, $n \in \nset$, $b_{\gamma, n} \in \rmc(\rset^d, \rset^d)$ and for any $x \in \rset^d$, 
    \begin{equation}
      \| b(x) - \bgM(x) \|^2 \leq C_1 \gamma^{\beta} \norm[2]{b(x)}   \eqsp .
    \end{equation}
\end{assumptionL}
 The following assumption is mainly technical and is satisfied in our applications.
 \begin{assumptionL}
   \label{assum:unif_integr}
   There exists $\vareps_b >0$ such that  
$\sup_{s \in \ccint{0,\Time}} \defEnsLigne{ \updelta_x \Pker_s  \norm[2(1+\vareps_b)]{b(x)}} < \plusinfty$, for any $x \in \rset^d$ and $\Time \geq 0$.
\end{assumptionL}
By \Cref{lemma:disclyapfromc}-\ref{lemma:disclyapfromc_a}, if \hyperlink{ass:drift_continuous}{$\bfDc(V,\zeta,0)$} is satisfied with $\zeta \in \rset$, it holds that for any starting point $x \in \rset^d$,  $\sup_{t \in \ccint{0,\Time}} \expeLigne{V(\bfX_t)} \leq \rme^{-\zeta \Time} V(x)$, where $(\bfX_t)_{t \geq 0}$ is solution of \eqref{eq:diff} starting from $x$. Therefore, if $\normLigne[2(1+\vareps_b)]{b(x)} \leq V(x)$ for any $x \in \rset^d$,  \Cref{assum:unif_integr} is satisfied. 

The proof of the next result relies on the combination of the Girsanov theorem with estimates on the drift functions, adapting \cite[Theorem 10]{durmus2017nonasymp}.  Similar strategies have also been used in \cite{dalalyan2017theoretical, fang2019multilevel, raginsky2017nonconvex}. 
 \begin{proposition}
  \label{lemma:V-norm_control}
      Assume 
      \tup{\Cref{ass:loc_lip_b}}, \tup{\Cref{ass:diff_disc}} and \tup{\Cref{assum:unif_integr}}. Let $V : \rset^d \to \coint{1,+\infty}$. In addition, assume that for any $n \in \nset$, $T \geq 0$ and $x \in \rset^d$ 
      \begin{equation} \Pker_{\Time} V^2(x) < +\infty \eqsp , \qquad \limsup_{m \to +\infty} \tRker_{\Time/m, n}^{m} V^2(x) < +\infty \eqsp . \label{eq:integr_condition}\end{equation}     
      Then for any $n \in \nset$, $\Time \geq 0$ and $x \in \rset^d$
      \begin{equation}
   \label{eq:girsanov}
 \lim_{m \to +\infty} \Vnorm{\updelta_x \Pker_{\Time} - \updelta_x \tRker_{\Time/m, n}^{m} } = 0  \eqsp ,
\end{equation}
where $(\Pker_t)_{t \geq 0}$ is the semigroup associated with \eqref{eq:diff} and  for any $\gamma\in \ocint{0, \bgamma}$ and $n \in \nset$,  $\tRker_{\gamma, n}$ is the Markov kernel associated with \eqref{eq:langevin_discrete} where $\Tg(x) = x + \gamma \bgM(x)$ and $\Pi = \Id$.
\end{proposition}
\begin{proof}
  The proof is postponed to \Cref{lemma:V-norm_control:proof}.
\end{proof}
If $V = 1$, \Cref{lemma:V-norm_control} implies that $
\lim_{m \to +\infty} \tvnorm{\updelta_x \Pker_{\Time} - \updelta_x \tRker_{\Time/m, n}^{m} } = 0$.
Let $V : \rset^d \to \coint{1,\plusinfty}$ and $\bfc : \ \rset^d \times \rset^d \to \coint{1,+\infty}$ a distance such that for any $x,y \in \rset^d$, $\bfc(x,y) \leq \defEns{V(x) + V(y)}/2$. Then, under the conditions of \Cref{lemma:V-norm_control}, we obtain that for any $\Time \geq 0$, $n \in \nset$ and $x,y \in \rset^d$
  \begin{equation}
    \distV(\updelta_x \Pker_{\Time}, \updelta_y \Pker_{\Time}) \leq \limsup_{m \to +\infty} \distV(\updelta_x \tRker_{\Time/m,n}^{ m}, \updelta_y \tRker_{\Time/m,n}^{m }) \eqsp ,\label{eq:theo_preli}
  \end{equation}
  Therefore, if for any $\Time \geq 0$, $\distV(\updelta_x \tRker_{\Time/m,n}^{m }, \updelta_y \tRker_{\Time/m,n}^{m})$ can be bounded uniformly in $m$ using \Cref{theo:discrete_contrac_wass_D_v2}, we obtain an explicit bound for $    \distV(\updelta_x \Pker_{\Time}, \updelta_y \Pker_{\Time})$ for any  $\Time \geq 0$. As a consequence, this result  easily implies non-asymptotic convergence bounds of $(\Pker_{t})_{t \geq 0}$ to its invariant measure if it exists.  However, in our applications, global Lipschitz regularity on $b_{\Time/m, n}: \ \rset^d \to \rset^d$ is needed in order to apply  \Cref{theo:discrete_contrac_wass_D_v2} to  $\tRker_{\Time/m,n}$ for $\Time \geq 0$, $m \in \nsets$ and $n \in \nset$. To be able to deal with the fact that $b_{\Time/m, n}$ is non necessarily globally Lipschitz, we consider an appropriate sequence of projected Euler-Maruyama schemes associated to a sequence of subsets of $\rset^d$, $(\msk_n)_{n \in \nset}$ satisfying the following assumption.
  \begin{assumptionL}
    \label{ass:seq_k_n}
    For any $n \in \nset$, $\msk_n$ is convex and  closed, and $\cball{0}{n} \subset \msk_n$. 
  \end{assumptionL}

Consider for any $\gamma \in \ocint{0, \bgamma}$ and $n \in \nset$ the Markov chain associated \eqref{eq:langevin_discrete}, where for any $x \in \rset^d$, $\Tg(x) = x + \gamma b_{\gamma,n}(x)$, $\msx = \msk_n$ and $\Pi = \Pi_{\msk_n}$, the projection on $\msk_n$. The Markov kernel associated with this chain is denoted $\Rker_{\gamma, n}$ for any $\gamma \in \ocint{0, \bgamma}$ and $n \in \nset$. Assuming only local Lipschitz regularity we can apply \Cref{theo:discrete_contrac_wass_D_v2} to the projected version of the Markov chain associated with $\Rker_{\Time/m,n}$. Therefore we want to replace $\tRker_{\Time/m,n}$ by $\Rker_{\Time/m,n}$ in \eqref{eq:theo_preli}. In order to do so we consider the following assumption on the family of drifts $\{b_{\gamma, n} \, ; \, \gamma \in \ocint{0,\bgamma}, n \in \nset\}$.
  \begin{assumptionL}
   \label{ass:lyap_majo}
 There exist $\tilde{A} > 0$ and $\tilde{V}:\rset^d \to [1,+\infty)$ such that for any $n \in \nset$ there exist $\En \geq 0$,  $\varepsn >0$ and $\bgamma_n \in \ocint{0, \bgamma}$ satisfying for any $\gamma \in \ocint{0,\bgamma_n}$  and $x \in \rset^d$,
 \begin{equation} \tRker_{\gamma, n} \tilde{V}(x) \leq \exp\parentheseDeux{\log(\tilde{A})\gamma (1  + \En\gamma^{\varepsn})} \tilde{V}(x) \eqsp , \qquad \sup_{x \in \rset^d}\defEns{  \norm{x} /\tilde{V}(x) } \leq 1 \eqsp,
 \end{equation}
 where for any $\gamma\in \ocint{0, \bgamma}$ and $n \in \nset$,  $\tRker_{\gamma, n}$ is the Markov kernel associated with \eqref{eq:langevin_discrete} where $\Tg(x) = x + \gamma \bgM(x)$ and $\Pi = \Id$.
 \end{assumptionL}

\begin{proposition}
  \label{propo:compare}
  Let $V : \rset^d \to \coint{1,+\infty}$. Assume 
      \tup{\Cref{ass:loc_lip_b}}, \tup{\Cref{ass:seq_k_n}}, \tup{\Cref{ass:lyap_majo}} and that for any $\Time \geq 0$, $x \in \rset^d$ 
      \begin{equation}\limsup_{n \to +\infty} \limsup_{m \to +\infty} \parenthese{\Rker_{\Time/m, n}^{m} + \tRker_{\Time/m, n}^{m}} V^2(x) < +\infty \eqsp .
        \end{equation}
      Then for any $\Time \geq 0$ and $x \in \rset^d$
      \begin{equation}
   \label{eq:compare_proj}
 \lim_{n \to +\infty} \limsup_{m \to +\infty} \Vnorm{\updelta_x \Rker_{\Time/m, n}^{mk} - \updelta_x \tRker_{\Time/m, n}^{mk} } = 0  \eqsp ,
\end{equation}
\end{proposition}

\begin{proof}
  The proof is postponed to \Cref{propo:compare:proof}.
\end{proof}

Based on \Cref{lemma:V-norm_control} and \Cref{propo:compare}, we have the following result which establishes a  clear link between the convergence of the family of the projected Euler-Maruyama scheme $\{\Rker_{\gamma,n} \, : \, \gamma \in \ocint{0,\bgamma}, n \in \nset\}$ and the semigroup $(\Pker_t)_{t \geq 0}$ associated with \eqref{eq:diff}.

\begin{theorem}
  \label{prop:general_bound}
      Let $\lyap : \rset^d \times \rset^d \to \coint{1,+\infty}$ and $V : \rset^d \to \coint{1,\plusinfty}$ satisfying for any $x,y \in \rset^d$, $\sup_{(x,y) \in \rset^d\times \rset^d} \lyap(x,y) \defEns{V(x) + V(y)}^{-1}< +\infty$. Assume \tup{\Cref{ass:loc_lip_b}}, 
\tup{\Cref{ass:diff_disc}},        \tup{\Cref{assum:unif_integr}}, \tup{\Cref{ass:seq_k_n}} and  \tup{\Cref{ass:lyap_majo}}. 
 In addition, assume that for any $\Time \geq 0$ and $x \in \rset^d$ 
      \begin{equation} \Pker_{\Time} V^2(x) < +\infty \eqsp , \quad \limsup_{n \to +\infty} \limsup_{m \to +\infty} \parenthese{\Rker_{\Time/m, n}^{m} + \tRker_{\Time/m, n}^{m}}  V^2(x) < +\infty \eqsp . \label{eq:integr_condition_3}\end{equation}      
      Then,
      \begin{equation}
        \distV(\updelta_x \Pker_{\Time}, \updelta_y \Pker_{\Time}) \leq \limsup_{n \to +\infty} \limsup_{m \to +\infty} \distV(\updelta_x \Rker_{\Time/m, n}^{m}, \updelta_y \Rker_{\Time/m, n}^{m}) \eqsp, \label{eq:conclu_1}
      \end{equation}
where for any $x,y \in \rset^d$, $\bfc(x,y) = \1_{\Deltar^{\complementary}}(x,y) \lyap(x,y)$, $(\Pker_t)_{t \geq 0}$ is the semigroup associated with \eqref{eq:diff} and  for any $\gamma\in \ocint{0, \bgamma}$, $n \in \nset$,  $\Rker_{\gamma, n}$ is the Markov kernel associated with \eqref{eq:langevin_discrete} where $\Tg(x) = x + \gamma \bgM(x)$, $\msx = \msk_n$ and $\Pi = \Pi_{\msk_n}$, $\tRker_{\gamma, n}$ is the Markov kernel associated with \eqref{eq:langevin_discrete} where $\Tg(x) = x + \gamma \bgM(x)$, $\msx = \rset^d$ and $\Pi = \Id$.
\end{theorem}

\begin{proof}
  Let $\Time \geq 0$, $x,y \in \rset^d$ and \begin{equation}
    C_V = 2 \sup_{(x,y) \in \rset^d\times \rset^d} \lyap(x,y) \defEns{V(x) + V(y)}^{-1} < +\infty \eqsp . \end{equation}
    We have for any $n \in \nset$ and $m \in \nsets$ such that $\Time/m \leq \bgamma$
  \begin{multline}
    \distV(\updelta_x \Pker_{\Time}, \updelta_y \Pker_{\Time}) \leq C_V\Vnorm{\updelta_x \Pker_{\Time}- \updelta_x \tRker_{\Time/m,n}^{m}} + C_V\Vnorm{\updelta_x \Rker_{\Time/m, n}^{m}- \updelta_x \tRker_{\Time/m,n}^{m}}  \\ + \distV(\updelta_x \Rker_{\Time/m, n}^{m}, \updelta_y \Rker_{\Time/m, n}^{m})  + C_V \Vnorm{\updelta_y \Pker_{\Time}- \updelta_y \tRker_{\Time/m,n}^{m}} + C_V \Vnorm{\updelta_y \Rker_{\Time/m, n}^{m}- \updelta_y \tRker_{\Time/m,n}^{m}} \eqsp ,
  \end{multline}
  which concludes the proof upon combining \Cref{lemma:V-norm_control} and \Cref{propo:compare}.
\end{proof}

\subsubsection{Explicit approximating family of drifts }
\label{sec:regul-appr}

In this section we show that under regularity and curvature assumptions on the
drift function $b$ we can construct explicit families of approximating drift
functions satisfying the assumptions of \Cref{prop:general_bound}. The section is
divided into two parts. First, we show under regularity conditions
\tup{\Cref{ass:loc_lip_b}}, \tup{\Cref{ass:diff_disc}},
\tup{\Cref{assum:unif_integr}}, \tup{\Cref{ass:seq_k_n}} and
\tup{\Cref{ass:lyap_majo}} are satisfied.  Second, we show, under similar, that
the summability assumptions \eqref{eq:integr_condition_3} in
\Cref{prop:general_bound} hold for $V \leftarrow V_M$ with
$V_M : \ \rset^{\dim} \to \coint{1, +\infty}$ given by \eqref{eq:lyap_m} for $M \geq 0$.
We start with the case where $b$ satisfies \Cref{as:item:lip}. 

\begin{proposition}
  \label{prop:L_ass_check_lip}
  Assume \tup{\Cref{as:item:lip}}. Let $\{\rmb_{\gamma, n}  \, : \, \gamma \in \ocint{0, \bgamma}, n \in \nset\}$ be given  for any $\gamma >0$, $n \in \nset$ and $x \in \rset^d$ by $\rmb_{\gamma, n}(x) = b(x)$. Let $\msk_n = \rset^{\dim}$ for any $n \in \nset$.
  Then, \tup{\Cref{ass:loc_lip_b}}, 
  \tup{\Cref{ass:diff_disc}},        \tup{\Cref{assum:unif_integr}}, \tup{\Cref{ass:seq_k_n}} and  \tup{\Cref{ass:lyap_majo}}
  are satisfied.
\end{proposition}

\begin{proof}
  The proof is postponed to \Cref{prop:L_ass_check_lip:proof}.
\end{proof}

We now consider the more challenging case where \Cref{as:item:lip} does not hold and is replaced by the weaker condition \Cref{ass:loc_lip_bb}. 
  In this setting, by \cite[Chapter 4, Theorem 2.3]{ikeda1989sto}, \eqref{eq:diff} admits a unique solution $(\bfX_t)_{t \in \coint{0,+\infty}}$ with $\bfX_0 = x \in \rset^d$ and let 
  $e = \inf \ensemble{s \geq 0}{\norm{\bfX_s} = +\infty}$. 
  In particular, the condition $e = +\infty$ is met \as \ if we assume that $b$ is sub-linear \cite[Chapter 4, Theorem 2.3]{ikeda1989sto} or that the condition \hyperlink{ass:drift_continuous}{$\bfDc(V,\zeta,0)$} holds with $\zeta \in \rset$ and $\lim_{\norm{x}\to +\infty} V(x) = +\infty$ \cite[Theorem 3.5]{khasminskii2011stochastic}. This last condition is satisfied for all the applications we consider in \Cref{sec:applications-1}.

\begin{proposition}
  \label{prop:existence_integr}
  Assume \tup{\Cref{as:b_min}($\mtt$)} with $\mtt \in \rset$ and  \tup{\Cref{ass:loc_lip_bb}}, then \tup{\Cref{ass:loc_lip_b}} holds. In addition:
  \begin{enumerate}[label=(\alph*), wide, labelwidth=!, labelindent=0pt]
  \item if there exists $\vareps_b >0$ and $p \in \nsets$ such that $\sup_{x \in \rset^d}  \{\norm{b(x)}^{2(1+\vareps_b)} (1 + \norm{x}^{2p})^{-1}\} < +\infty$ then \tup{\Cref{assum:unif_integr}} holds ;
  \item 
    assume that \tup{\Cref{assum:drift_strong}
    } holds and $\sup_{x \in \rset^d} \{ \norm{b(x)}^{2(1+\vareps_b)} \rme^{-\mttplusdeux\norm{x}^2}\}<+\infty$ for some $\vareps_b >0$ satisfying  then \tup{\Cref{assum:unif_integr}} holds.
  \end{enumerate}
\end{proposition}

\begin{proof}
  The proof is postponed to \Cref{prop:existence_integr:proof}.
\end{proof}

\Cref{prop:existence_integr} gives conditions under which \tup{\Cref{ass:loc_lip_b}} and \tup{\Cref{assum:unif_integr}} hold. In addition, \Cref{ass:seq_k_n} is satisfied if we take for any $n \in \nset$, $\msk_n = \cball{0}{n}$. 
Therefore, it only remains to find a family of drift functions which satisfies \Cref{ass:diff_disc} and \Cref{ass:lyap_majo}. 
To this end, consider the following family of drift functions $\{\rmb_{\gamma, n}  \, : \, \gamma \in \ocint{0, \bgamma}, n \in \nset\}$ defined for any $\gamma >0$, $n \in \nset$ and $x \in \rset^d$ by
\begin{equation}
  \rmb_{\gamma, n}(x) = \varphi_n(x) b(x) + (1 - \varphi_n(x)) \frac{b(x)}{1 + \gamma^{\alpha} \norm{b(x)}} \eqsp , \label{eq:def_fam}
\end{equation}
with $\alpha <1/2$ and $\varphi_n \in \rmC(\rset^d, \rset)$ such that for any $n \in \nset$ and $x \in \rset^d$,
\begin{equation}
  \label{eq:varphi_condition}
 \varphi_n(x) \in \ccint{0,1} \quad \text{ and } \quad \varphi_n(x)
  =
  \begin{cases}
    1 & \text{ if $x \in \cball{0}{n}$}, \\
    0 & \text{ if $x \in \cball{0}{n+1}^{\complementary}$ } \eqsp.
  \end{cases}
\end{equation}
An example of such a family is displayed in \Cref{fig:sinus_drift}. 
\begin{figure}[t]
  \centering
  \subfloat[]{\includegraphics[width=0.495\linewidth]{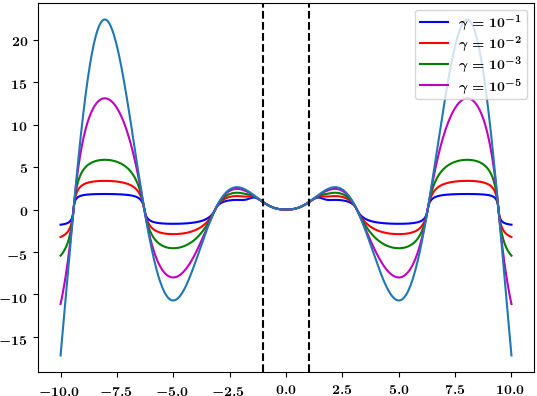}} \hfill
  \subfloat[]{\includegraphics[width=0.495\linewidth]{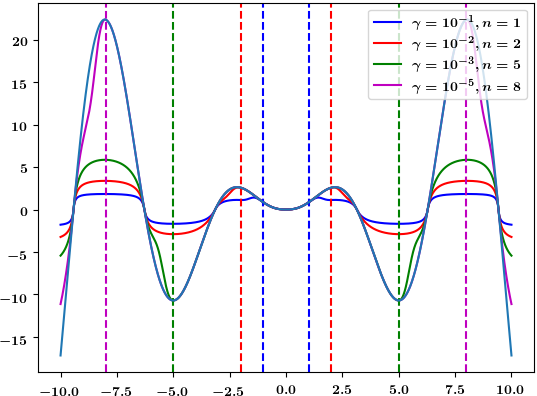}} \hfill
  \caption{In this figure we illustrate the approximation properties of the family of drift functions defined by \eqref{eq:def_fam}. Let $b(x) = |x|^{1.5}\sin(x)$ and for any $n \in \nset$, $\varphi_n(x) = d(x, \cball{0}{+1}^{\complementary})^2/(d(x, \cball{0}{n})^2+ d(x, \cball{0}{n+1}^{\complementary})^2)$. In both figures the original drift is displayed in cyan and we fix $\alpha = 0.3$. In (a), we fix $n = 1$, represented by the black dashed lines, and observe the behavior of the drift functions for different values of $\gamma>0$. In (b), we plot the drift for different $\gamma >0$ and $n \in \nset$.}
  \label{fig:sinus_drift}
\end{figure}
\begin{proposition}
  \label{prop:fam_prop}
      Assume \tup{\Cref{as:b_min}($\mtt$)} for $\mtt \in \rset$ and \tup{\Cref{ass:loc_lip_bb}}, then
      \tup{\Cref{ass:diff_disc}} and \tup{\Cref{ass:lyap_majo}} hold for the family $\{\rmb_{\gamma, n}  \, : \, \gamma \in \ocint{0, \bgamma}, n \in \nset\}$ defined by \eqref{eq:def_fam}.
\end{proposition}

\begin{proof}
  The proof is postponed to \Cref{prop:fam_prop:proof}.
\end{proof}

The following proposition is a generalization of \Cref{prop:L_ass_check_lip}.
\begin{proposition}
  \label{prop:L_ass_check_loc_lip}
  Assume \tup{\Cref{as:b_min}($\mtt$)} with $\mtt \in \rset$ and  \tup{\Cref{ass:loc_lip_bb}}.
Let $\{\rmb_{\gamma, n}  \, : \, \gamma \in \ocint{0, \bgamma}, n \in \nset\}$ be given  for any $\gamma >0$, $n \in \nset$ and $x \in \rset^d$ by \eqref{eq:def_fam}. Let $\msk_n = \cball{0}{n}$ for any $n \in \nset$.
  Then, \tup{\Cref{ass:loc_lip_b}}, 
  \tup{\Cref{ass:diff_disc}}, \tup{\Cref{ass:seq_k_n}} and  \tup{\Cref{ass:lyap_majo}}
  are satisfied. 
\end{proposition}

\begin{proof}
  The proof is a straightforward combination of \Cref{prop:existence_integr} and \Cref{prop:fam_prop}.
\end{proof}

In \Cref{prop:integrability_R_lip} and \Cref{prop:integrability_R_loc_lip} we show that the second part of \eqref{eq:integr_condition_3} holds under regularity assumptions on the drift function $b$.

\begin{proposition}
  \label{prop:integrability_R_lip}
      Assume \tup{\Cref{as:b_min}($\mtt$)} for $\mtt \in \rset$ and \tup{\Cref{as:item:lip}}, then
 for any $\Time, M \geq 0$ and $x \in \rset^d$
      \begin{equation}
\limsup_{m \to +\infty} \Rker_{\Time/m}^{m}   V_M(x) < +\infty \eqsp , 
      \end{equation}
      with $V_M$ given in \eqref{eq:lyap_m} and where for any $\gamma\in \ocint{0, \bgamma}$, $\Rker_{\gamma}$ is the Markov kernel associated with \eqref{eq:langevin_discrete} where $\Tg(x) = x + \gamma b(x)$, $\msx = \rset^{\dim}$ and $\Pi = \Id$.
    \end{proposition}

    \begin{proof}
      The proof is postponed to \Cref{sec:prop:integrability_R_lip:proof}.
    \end{proof}

\begin{proposition}
  \label{prop:integrability_R_loc_lip}
      Assume \tup{\Cref{as:b_min}($\mtt$)} for $\mtt \in \rset$ and \tup{\Cref{ass:loc_lip_bb}}, then
 for any $\Time, M \geq 0$ and $x \in \rset^d$
      \begin{equation}
        \limsup_{n \to +\infty} \limsup_{m \to +\infty} \parenthese{\Rker_{\Time/m, n}^{m} + \tRker_{\Time/m, n}^{m}}  V_M(x) < +\infty \eqsp , 
      \end{equation}
      with $V_M$ given in \eqref{eq:lyap_m} and where for any $\gamma\in \ocint{0, \bgamma}$, $n \in \nset$,  $\Rker_{\gamma, n}$ is the Markov kernel associated with \eqref{eq:langevin_discrete} where $\Tg(x) = x + \gamma \rmb_{\gamma,n}(x)$, $\msx = \cball{0}{n}$ and $\Pi = \Pi_{\cball{0}{n}}$, $\tRker_{\gamma, n}$ is the Markov kernel associated with \eqref{eq:langevin_discrete} where $\Tg(x) = x + \gamma \rmb_{\gamma,n}(x)$, $\msx = \rset^d$ and $\Pi = \Id$.
    \end{proposition}

    \begin{proof}
      The proof is postponed to \Cref{sec:prop:integrability_R_loc_lip:proof}.
    \end{proof}
    
Finally, we show that under mild curvature assumption on the drift function $b$, the first part of \eqref{eq:integr_condition_3} holds.

\begin{proposition}
  \label{propo:majo_cont}
  Assume \tup{\Cref{ass:loc_lip_b}} and that $\sup_{x \in \rset^d} \langle b(x), x \rangle < +\infty$.
  Then for any $M \geq 0$, there exists $\zeta \in \rset$ such that \hyperlink{ass:drift_continuous}{$\bfDc(V_M,\zeta,0)$} holds  with $V_M$ given in \eqref{eq:lyap_m}. In particular, for any $\Time, M \geq 0$, $\Pker_{\Time}V_M(x) < +\infty$.
\end{proposition}
\begin{proof}
  The proof is postponed to \Cref{propo:majo_cont:proof}.
\end{proof}


\bibliographystyle{plainnat}
\bibliography{main.bbl}

\begin{thebibliography}{10}

\bibitem{baker2017control}
J.~Baker, P.~Fearnhead, E.~B. Fox, and C.~Nemeth.
\newblock Control variates for stochastic gradient mcmc.
\newblock {\em Statistics and Computing}, pages 1--17.

\bibitem{barky2008simple}
D.~Bakry, F.~Barthe, P.~Cattiaux, and A.~Guillin.
\newblock A simple proof of the {P}oincar\'{e} inequality for a large class of
  probability measures including the log-concave case.
\newblock {\em Electron. Commun. Probab.}, 13:60--66, 2008.

\bibitem{bauschke:combettes:2011}
H.~H. Bauschke and P.~L. Combettes.
\newblock {\em Convex Analysis and Monotone Operator Theory in {H}ilbert
  Spaces}.
\newblock Springer Publishing Company, Incorporated, 1st edition, 2011.

\bibitem{bernton2018langevin}
E.~Bernton.
\newblock Langevin monte carlo and jko splitting.
\newblock {\em arXiv preprint arXiv:1802.08671}, 2018.

\bibitem{boucheron2013concentration}
S.~Boucheron, G.~Lugosi, and P.~Massart.
\newblock {\em Concentration inequalities}.
\newblock Oxford University Press, Oxford, 2013.
\newblock A nonasymptotic theory of independence, With a foreword by Michel
  Ledoux.

\bibitem{bremaud1999markov}
P.~Br\'{e}maud.
\newblock {\em Markov chains}, volume~31 of {\em Texts in Applied Mathematics}.
\newblock Springer-Verlag, New York, 1999.
\newblock Gibbs fields, Monte Carlo simulation, and queues.

\bibitem{brosse2018tamed}
N.~Brosse, A.~Durmus, Éric Moulines, and S.~Sabanis.
\newblock The tamed unadjusted langevin algorithm.
\newblock {\em Stochastic Processes and their Applications}, 2018.

\bibitem{Bubeck:2015}
S.~Bubeck, R.~Eldan, and J.~Lehec.
\newblock Finite-time analysis of projected langevin monte carlo.
\newblock In {\em Proceedings of the 28th International Conference on Neural
  Information Processing Systems}, NIPS'15, pages 1243--1251, Cambridge, MA,
  USA, 2015. MIT Press.

\bibitem{bubley:dyer:jerrum:1998}
R.~Bubley, M.~Dyer, and M.~Jerrum.
\newblock An elementary analysis of a procedure for sampling points in a convex
  body.
\newblock {\em Random Structures Algorithms}, 12(3):213--235, 1998.

\bibitem{butkovsky2014subgeometric}
O.~Butkovsky.
\newblock Subgeometric rates of convergence of {M}arkov processes in the
  {W}asserstein metric.
\newblock {\em Ann. Appl. Probab.}, 24(2):526--552, 2014.

\bibitem{chan:1993}
K.~S. Chan.
\newblock Asymptotic behavior of the {G}ibbs sampler.
\newblock {\em J. Amer. Statist. Assoc.}, 88(421):320--326, 1993.

\bibitem{chatterji2018theory}
N.~Chatterji, N.~Flammarion, Y.~Ma, P.~Bartlett, and M.~Jordan.
\newblock On the theory of variance reduction for stochastic gradient monte
  carlo.
\newblock In {\em International Conference on Machine Learning}, pages
  763--772, 2018.

\bibitem{chen1995estimation}
M.~F. Chen and F.~Y. Wang.
\newblock Estimation of the first eigenvalue of second order elliptic
  operators.
\newblock {\em J. Funct. Anal.}, 131(2):345--363, 1995.

\bibitem{chen1997estimation}
M.-F. Chen and F.-Y. Wang.
\newblock Estimation of spectral gap for elliptic operators.
\newblock {\em Trans. Amer. Math. Soc.}, 349(3):1239--1267, 1997.

\bibitem{chen:tsay:1991}
R.~Chen and R.~S. Tsay.
\newblock On the ergodicity of {${\rm TAR}(1)$} processes.
\newblock {\em Ann. Appl. Probab.}, 1(4):613--634, 1991.

\bibitem{cheng2018sharp}
X.~Cheng, N.~S. Chatterji, Y.~Abbasi-Yadkori, P.~L. Bartlett, and M.~I. Jordan.
\newblock Sharp convergence rates for langevin dynamics in the nonconvex
  setting.
\newblock {\em arXiv preprint arXiv:1805.01648}, 2018.

\bibitem{dalalyan2017theoretical}
A.~S. Dalalyan.
\newblock Theoretical guarantees for approximate sampling from smooth and
  log-concave densities.
\newblock {\em J. R. Stat. Soc. Ser. B. Stat. Methodol.}, 79(3):651--676, 2017.

\bibitem{dalalyan2019user}
A.~S. Dalalyan and A.~Karagulyan.
\newblock User-friendly guarantees for the langevin monte carlo with inaccurate
  gradient.
\newblock {\em Stochastic Processes and their Applications}, 2019.

\bibitem{devraj2017geometric}
A.~Devraj, I.~Kontoyiannis, and S.~Meyn.
\newblock Geometric ergodicity in a weighted sobolev space, 2017.

\bibitem{douc:fort:guillin:2009}
R.~Douc, G.~Fort, and A.~Guillin.
\newblock Subgeometric rates of convergence of {$f$}-ergodic strong {M}arkov
  processes.
\newblock {\em Stochastic Process. Appl.}, 119(3):897--923, 2009.

\bibitem{douc:fort:moulines:soulier:2004}
R.~Douc, G.~Fort, E.~Moulines, and P.~Soulier.
\newblock Practical drift conditions for subgeometric rates of convergence.
\newblock {\em Ann. Appl. Probab.}, 14(3):1353--1377, 2004.

\bibitem{douc:moulines:priouret:soulier:2018}
R.~Douc, E.~Moulines, P.~Priouret, and P.~Soulier.
\newblock {\em Markov Chains}.
\newblock Springer, 2019.

\bibitem{douc:moulines:rosenthal:2004}
R.~Douc, E.~Moulines, and J.~S. Rosenthal.
\newblock Quantitative bounds on convergence of time-inhomogeneous {M}arkov
  chains.
\newblock {\em Ann. Appl. Probab.}, 14(4):1643--1665, 2004.

\bibitem{down1995exponential}
D.~Down, S.~P. Meyn, and R.~L. Tweedie.
\newblock Exponential and uniform ergodicity of {M}arkov processes.
\newblock {\em Ann. Probab.}, 23(4):1671--1691, 1995.

\bibitem{durmus2016subgeometric}
A.~Durmus, G.~Fort, and E.~Moulines.
\newblock Subgeometric rates of convergence in {W}asserstein distance for
  {M}arkov chains.
\newblock {\em Ann. Inst. Henri Poincar\'{e} Probab. Stat.}, 52(4):1799--1822,
  2016.

\bibitem{durmus2018analysis}
A.~Durmus, S.~Majewski, and B.~Miasojedow.
\newblock Analysis of langevin monte carlo via convex optimization.
\newblock {\em arXiv preprint arXiv:1802.09188}, 2018.

\bibitem{durmus2015quantitative}
A.~Durmus and {\'E}.~Moulines.
\newblock Quantitative bounds of convergence for geometrically ergodic markov
  chain in the wasserstein distance with application to the metropolis adjusted
  langevin algorithm.
\newblock {\em Statistics and Computing}, 25(1):5--19, 2015.

\bibitem{durmus:moulines:2016}
A.~{Durmus} and E.~{Moulines}.
\newblock {High-dimensional Bayesian inference via the Unadjusted Langevin
  Algorithm}.
\newblock {\em ArXiv e-prints}, May 2016.

\bibitem{durmus2017nonasymp}
A.~Durmus and E.~Moulines.
\newblock Nonasymptotic convergence analysis for the unadjusted {L}angevin
  algorithm.
\newblock {\em Ann. Appl. Probab.}, 27(3):1551--1587, 2017.

\bibitem{durmus2016efficient}
A.~Durmus, E.~Moulines, and M.~Pereyra.
\newblock Efficient bayesian computation by proximal markov chain monte carlo:
  when langevin meets moreau.
\newblock {\em SIAM Journal on Imaging Sciences}, 11(1):473--506, 2018.

\bibitem{eberle2011reflection}
A.~Eberle.
\newblock Reflection coupling and {W}asserstein contractivity without
  convexity.
\newblock {\em C. R. Math. Acad. Sci. Paris}, 349(19-20):1101--1104, 2011.

\bibitem{eberle2016reflection}
A.~Eberle.
\newblock Reflection couplings and contraction rates for diffusions.
\newblock {\em Probab. Theory Related Fields}, 166(3-4):851--886, 2016.

\bibitem{eberle:guillin:zimmer:2018}
A.~Eberle, A.~Guillin, and R.~Zimmer.
\newblock Quantitative harris-type theorems for diffusions and mckean--vlasov
  processes.
\newblock {\em Transactions of the American Mathematical Society}, 2018.

\bibitem{eberle2018quantitative}
A.~Eberle and M.~B. Majka.
\newblock Quantitative contraction rates for markov chains on general state
  spaces.
\newblock {\em arXiv preprint arXiv:1808.07033}, 2018.

\bibitem{fang2019multilevel}
W.~Fang and M.~B. Giles.
\newblock Multilevel {M}onte {C}arlo method for ergodic {SDE}s without
  contractivity.
\newblock {\em J. Math. Anal. Appl.}, 476(1):149--176, 2019.

\bibitem{fort:2001}
G.~Fort.
\newblock {\em Contrôle explicite d'ergodicité de chaînes de {M}arkov :
  {A}pplications à l'analyse de convergence de l'algorithme {Monte-Carlo EM}}.
\newblock PhD thesis, Université Pierre et Marie Curie, Paris, Paris, 2001.

\bibitem{fort:2002}
G.~Fort.
\newblock Computable bounds for {V}-geometric ergodicity of {M}arkov transition
  kernels.
\newblock {\em Rapport de Recherche, Univ. J. Fourier, RR 1047-M.,
  \url{https://www.math.univ-toulouse.fr/%7Egfort/Preprints/fort:2002.pdf}},
  2002.

\bibitem{fort:roberts:2005}
G.~Fort and G.~O. Roberts.
\newblock Subgeometric ergodicity of strong {M}arkov processes.
\newblock {\em Ann. Appl. Probab.}, 15(2):1565--1589, 2005.

\bibitem{foster1953stochastic}
F.~G. Foster.
\newblock On the stochastic matrices associated with certain queuing processes.
\newblock {\em Ann. Math. Statistics}, 24:355--360, 1953.

\bibitem{goldys:maslowski:2006}
B.~Goldys and B.~Maslowski.
\newblock Lower estimates of transition densities and bounds on exponential
  ergodicity for stochastic {PDE}'s.
\newblock {\em Ann. Probab.}, 34(4):1451--1496, 2006.

\bibitem{granas2003fixed}
A.~Granas and J.~Dugundji.
\newblock {\em Fixed point theory}.
\newblock Springer Monographs in Mathematics. Springer-Verlag, New York, 2003.

\bibitem{hairer2011yet}
M.~Hairer and J.~C. Mattingly.
\newblock Yet another look at harris’ ergodic theorem for markov chains.
\newblock In {\em Seminar on Stochastic Analysis, Random Fields and
  Applications VI}, pages 109--117. Springer, 2011.

\bibitem{hairer2011asymptotic}
M.~Hairer, J.~C. Mattingly, and M.~Scheutzow.
\newblock Asymptotic coupling and a general form of {H}arris' theorem with
  applications to stochastic delay equations.
\newblock {\em Probab. Theory Related Fields}, 149(1-2):223--259, 2011.

\bibitem{hairer2014spectral}
M.~Hairer, A.~M. Stuart, and S.~J. Vollmer.
\newblock Spectral gaps for a {M}etropolis-{H}astings algorithm in infinite
  dimensions.
\newblock {\em Ann. Appl. Probab.}, 24(6):2455--2490, 2014.

\bibitem{ikeda1989sto}
N.~Ikeda and S.~Watanabe.
\newblock {\em Stochastic differential equations and diffusion processes},
  volume~24 of {\em North-Holland Mathematical Library}.
\newblock North-Holland Publishing Co., Amsterdam; Kodansha, Ltd., Tokyo,
  second edition, 1989.

\bibitem{jarner:roberts:2002}
S.~F. Jarner and G.~O. Roberts.
\newblock Polynomial convergence rates of {M}arkov chains.
\newblock {\em Ann. Appl. Probab.}, 12(1):224--247, 2002.

\bibitem{jarner2001locally}
S.~F. Jarner and R.~L. Tweedie.
\newblock Locally contracting iterated functions and stability of {M}arkov
  chains.
\newblock {\em J. Appl. Probab.}, 38(2):494--507, 2001.

\bibitem{johndrow2017error}
J.~E. Johndrow and J.~C. Mattingly.
\newblock Error bounds for approximations of markov chains used in bayesian
  sampling.
\newblock {\em arXiv preprint arXiv:1711.05382}, 2017.

\bibitem{jones:hobert:2001}
G.~L. Jones and J.~P. Hobert.
\newblock Honest exploration of intractable probability distributions via
  {M}arkov chain {M}onte {C}arlo.
\newblock {\em Statist. Sci.}, 16(4):312--334, 2001.

\bibitem{jones:hobert:2004}
G.~L. Jones and J.~P. Hobert.
\newblock Sufficient burn-in for {G}ibbs samplers for a hierarchical random
  effects model.
\newblock {\em Ann. Statist.}, 32(2):784--817, 2004.

\bibitem{joulin:ollivier:2010}
A.~Joulin and Y.~Ollivier.
\newblock Curvature, concentration and error estimates for markov chain monte
  carlo.
\newblock {\em The Annals of Probability}, 38(6):2418--2442, 2010.

\bibitem{khasminskii2011stochastic}
R.~Khasminskii.
\newblock {\em Stochastic stability of differential equations}, volume~66.
\newblock Springer Science \& Business Media, 2011.

\bibitem{kontoyiannis2017approx}
I.~Kontoyiannis and S.~P. Meyn.
\newblock Approximating a diffusion by a finite-state hidden {M}arkov model.
\newblock {\em Stochastic Process. Appl.}, 127(8):2482--2507, 2017.

\bibitem{lindvall1986coupling}
T.~Lindvall and L.~C.~G. Rogers.
\newblock Coupling of multidimensional diffusions by reflection.
\newblock {\em Ann. Probab.}, 14(3):860--872, 1986.

\bibitem{liptser2013statistics}
R.~Liptser and A.~N. Shiryaev.
\newblock {\em Statistics of random Processes: I. general Theory}, volume~5.
\newblock Springer Science \& Business Media, 2013.

\bibitem{lund:tweedie:1996}
R.~B. Lund and R.~L. Tweedie.
\newblock Geometric convergence rates for stochastically ordered {M}arkov
  chains.
\newblock {\em Math. Oper. Res.}, 21(1):182--194, 1996.

\bibitem{luo2016exponential}
D.~Luo and J.~Wang.
\newblock Exponential convergence in {$L^p$}-{W}asserstein distance for
  diffusion processes without uniformly dissipative drift.
\newblock {\em Math. Nachr.}, 289(14-15):1909--1926, 2016.

\bibitem{majka2018non}
M.~B. Majka, A.~Mijatovi{\'c}, and L.~Szpruch.
\newblock Non-asymptotic bounds for sampling algorithms without log-concavity.
\newblock {\em arXiv preprint arXiv:1808.07105}, 2018.

\bibitem{meyers1967converse}
P.~R. Meyers.
\newblock A converse to {B}anach's contraction theorem.
\newblock {\em J. Res. Nat. Bur. Standards Sect. B}, 71B:73--76, 1967.

\bibitem{meyn1993criteria_i}
S.~P. Meyn and R.~L. Tweedie.
\newblock Stability of {M}arkovian processes. {I}. {C}riteria for discrete-time
  chains.
\newblock {\em Adv. in Appl. Probab.}, 24(3):542--574, 1992.

\bibitem{meyn1993criteria_ii}
S.~P. Meyn and R.~L. Tweedie.
\newblock Stability of {M}arkovian processes. {II}. {C}ontinuous-time processes
  and sampled chains.
\newblock {\em Adv. in Appl. Probab.}, 25(3):487--517, 1993.

\bibitem{meyn1993criteria_iii}
S.~P. Meyn and R.~L. Tweedie.
\newblock Stability of {M}arkovian processes. {III}. {F}oster-{L}yapunov
  criteria for continuous-time processes.
\newblock {\em Adv. in Appl. Probab.}, 25(3):518--548, 1993.

\bibitem{meyn:tweedie:1994}
S.~P. Meyn and R.~L. Tweedie.
\newblock Computable bounds for geometric convergence rates of {M}arkov chains.
\newblock {\em Ann. Appl. Probab.}, 4(4):981--1011, 1994.

\bibitem{nesterov:2004}
Y.~Nesterov.
\newblock {\em Introductory Lectures on Convex Optimization: A Basic Course}.
\newblock Applied Optimization. Springer, 2004.

\bibitem{neveu1975discrete}
J.~Neveu.
\newblock {\em Discrete-parameter martingales}.
\newblock North-Holland Publishing Co., Amsterdam-Oxford; American Elsevier
  Publishing Co., Inc., New York, revised edition, 1975.
\newblock Translated from the French by T. P. Speed, North-Holland Mathematical
  Library, Vol. 10.

\bibitem{nummelin1982geometric}
E.~Nummelin and P.~Tuominen.
\newblock Geometric ergodicity of {H}arris recurrent {M}arkov chains with
  applications to renewal theory.
\newblock {\em Stochastic Process. Appl.}, 12(2):187--202, 1982.

\bibitem{nummelin1983rate}
E.~Nummelin and P.~Tuominen.
\newblock The rate of convergence in {O}rey's theorem for {H}arris recurrent
  {M}arkov chains with applications to renewal theory.
\newblock {\em Stochastic Process. Appl.}, 15(3):295--311, 1983.

\bibitem{nummelin1978geometric}
E.~Nummelin and R.~L. Tweedie.
\newblock Geometric ergodicity and {$R$}-positivity for general {M}arkov
  chains.
\newblock {\em Ann. Probability}, 6(3):404--420, 1978.

\bibitem{ollivier:2009}
Y.~Ollivier.
\newblock Ricci curvature of markov chains on metric spaces.
\newblock {\em Journal of Functional Analysis}, 256(3):810--864, 2009.

\bibitem{parikh:boyd:2013}
N.~Parikh and S.~Boyd.
\newblock {\em Proximal Algorithms}.
\newblock Foundations and Trends(r) in Optimization. Now Publishers, 2013.

\bibitem{paulin2016mixing}
D.~Paulin.
\newblock Mixing and concentration by {R}icci curvature.
\newblock {\em J. Funct. Anal.}, 270(5):1623--1662, 2016.

\bibitem{pillai2014ergodicity}
N.~S. Pillai and A.~Smith.
\newblock Ergodicity of approximate mcmc chains with applications to large data
  sets.
\newblock {\em arXiv preprint arXiv:1405.0182}, 2014.

\bibitem{popov:1977}
N.~N. Popov.
\newblock Geometric ergodicity conditions for countable {M}arkov chains.
\newblock {\em Dokl. Akad. Nauk SSSR}, 234(2):316--319, 1977.

\bibitem{qin2018wasserstein}
Q.~Qin and J.~P. Hobert.
\newblock Wasserstein-based methods for convergence complexity analysis of mcmc
  with application to albert and chib's algorithm.
\newblock {\em arXiv preprint arXiv:1810.08826}, 2018.

\bibitem{qin:hobert:2019}
Q.~Qin and J.~P. Hobert.
\newblock Geometric convergence bounds for markov chains in wasserstein
  distance based on generalized drift and contraction conditions.
\newblock {\em arXiv preprint arXiv:1902.02964}, 2019.

\bibitem{raginsky2017nonconvex}
M.~Raginsky, A.~Rakhlin, and M.~Telgarsky.
\newblock Non-convex learning via stochastic gradient langevin dynamics: a
  nonasymptotic analysis, 2017.

\bibitem{roberts:polson:1994}
G.~O. Roberts and N.~G. Polson.
\newblock On the geometric convergence of the gibbs sampler.
\newblock {\em Journal of the Royal Statistical Society, Series B},
  56:377--384, 1994.

\bibitem{roberts:rosenthal:1996}
G.~O. Roberts and J.~S. Rosenthal.
\newblock Quantitative bounds for convergence rates of continuous time {M}arkov
  processes.
\newblock {\em Electron. J. Probab.}, 1:no. 9, approx. 21 pp., 1996.

\bibitem{roberts:tweedie:1999}
G.~O. Roberts and R.~L. Tweedie.
\newblock Bounds on regeneration times and convergence rates for {M}arkov
  chains.
\newblock {\em Stochastic Process. Appl.}, 80(2):211--229, 1999.

\bibitem{rosenthal:1995}
J.~S. Rosenthal.
\newblock Minorization conditions and convergence rates for {M}arkov chain
  {M}onte {C}arlo.
\newblock {\em J. Amer. Statist. Assoc.}, 90(430):558--566, 1995.

\bibitem{rosenthal:2002}
J.~S. Rosenthal.
\newblock Quantitative convergence rates of {M}arkov chains: a simple account.
\newblock {\em Electron. Comm. Probab.}, 7:123--128, 2002.

\bibitem{rudolf2018perturbation}
D.~Rudolf and N.~Schweizer.
\newblock Perturbation theory for {M}arkov chains via {W}asserstein distance.
\newblock {\em Bernoulli}, 24(4A):2610--2639, 2018.

\bibitem{tuominen:tweedie:1994}
P.~Tuominen and R.~L. Tweedie.
\newblock Subgeometric rates of convergence of {$f$}-ergodic {M}arkov chains.
\newblock {\em Adv. in Appl. Probab.}, 26(3):775--798, 1994.

\bibitem{veretennikov:1997}
A.~Y. Veretennikov.
\newblock On polynomial mixing bounds for stochastic differential equations.
\newblock {\em Stochastic Process. Appl.}, 70(1):115--127, 1997.

\bibitem{villani2009optimal}
C.~Villani.
\newblock {\em Optimal transport}, volume 338 of {\em Grundlehren der
  Mathematischen Wissenschaften [Fundamental Principles of Mathematical
  Sciences]}.
\newblock Springer-Verlag, Berlin, 2009.
\newblock Old and new.

\bibitem{wang1994application}
F.~Y. Wang.
\newblock Application of coupling methods to the {N}eumann eigenvalue problem.
\newblock {\em Probab. Theory Related Fields}, 98(3):299--306, 1994.

\end{thebibliography}
\appendix
\section{Proofs of \Cref{sec:presentation} }
\label{sec:presentation_proof}

\subsection{Proof of \Cref{propo:doeb} }
\label{sec:proof-crefl}

First, we prove the following technical lemma.

\begin{lemma}
  \label{lemma:technical_doeb_0}
  Let $\bgamma >0$ and $\upkappa : \ocint{0,\bgamma} \to \rset$, with $\upkappa(\gamma)\gamma \in \ooint{-1,+\infty}$ for any $\gamma \in \ocint{0, \bgamma}$.
    We have for any $\gamma \in \ocint{0, \bgamma}$ such that $\upkappa(\gamma) \neq 0$ and $\ell \in \nsets$
    \begin{equation}
      \Xi_{\ell \ceil{1/\gamma}}(\upkappa) =  - \upkappa^{-1}(\gamma) \defEns{ \exp\parentheseDeux{-\ell\ceil{1/\gamma} \log \defEns{1 + \gamma  \upkappa(\gamma)}} -1} \eqsp, \label{eq:Xi_equality}
    \end{equation}
    where $      \Xi_{\ell \ceil{1/\gamma}}(\upkappa)$ is defined by \eqref{eq:def_Xi}. 
  In addition, for any $\ell \in \nsets$ and $\gamma \in \ocint{0, \bgamma}$
  \begin{enumerate}[label= (\alph*),  wide, labelwidth=!, labelindent=0pt]
  \item $    \Xi_{\ell \step}(\upkappa) \geq \alpha_{-}(\upkappa,\gamma, \ell) =   -\upkappa^{-1}(\gamma) \parentheseDeux{ \exp(-\ell \upkappa(\gamma)) - 1}$ if for any $\gamma \in \ocint{0,\bgamma}$, $ \upkappa(\gamma) <  0$ ; \label{item:r_neg}
  \item $\Xi_{\ell \step}(\upkappa) \geq \alpha_0(\upkappa,\gamma, \ell) = \ell$ if for any $\gamma \in \ocint{0,\bgamma}$, $  \upkappa(\gamma) \leq  0$ ; \label{item:r_0}    
  \item $\displaystyle \Xi_{\ell \step}(\upkappa) \geq \alpha_{+}(\upkappa,\gamma, \ell) =  \upkappa^{-1}(\gamma) \parentheseDeux{ 1 - \exp \defEns{-\frac{\ell \upkappa(\gamma)}{1 + \gamma  \upkappa(\gamma)}}}$ if for any $\gamma \in \ocint{0,\bgamma}$, $  \upkappa(\gamma) > 0$. \label{item:r_pos}
  \end{enumerate}
  \end{lemma}

  \begin{proof}
    Let $\ell \in \nsets$ and $\gamma \in \ocint{0,\bgamma}$. First note that the  following equalities hold if $\upkappa(\gamma) \neq 0$
    \begin{align}
    \Xi_{\ell \ceil{1/\gamma}}(\upkappa) &=  \gamma \sum_{i=1}^{\ell \ceil{1/\gamma}} (1+ \gamma  \upkappa(\gamma))^{-i}\\
    &= \gamma (1 + \gamma  \upkappa(\gamma))^{-1}\frac{1 - (1+\gamma  \upkappa(\gamma))^{-\ell \ceil{1/\gamma}}} {1 - (1 + \gamma  \upkappa(\gamma))^{-1}} \\
      &= - \upkappa^{-1}(\gamma) \defEns{\parentheseDeux{1 + \gamma  \upkappa(\gamma)}^{-\ell \step} -1} \\
      &= - \upkappa^{-1}(\gamma) \defEns{ \exp\parentheseDeux{-\ell\ceil{1/\gamma} \log \defEns{1 + \gamma  \upkappa(\gamma)}} -1} \eqsp . \numberthis \label{eq:Xi_majo_2}
    \end{align}
  We now give a lower-bound on $\Xi_{\ell \step}(\upkappa)$ depending on the condition satisfied by $\gamma \mapsto \upkappa(\gamma)$.
  \begin{enumerate}[label= (\alph*),  wide, labelwidth=!, labelindent=0pt]
  \item 
    Assume that for any $\tgamma \in \ocint{0,\bgamma}$ $,  \upkappa (\tgamma) < 0$. 
    Using that $\log(1-t) \le -t$ for $t \in \ooint{0,1}$, we obtain that
  \begin{equation}
    \exp\parentheseDeux{-\ell\ceil{1/\gamma} \log \defEns{1 + \gamma  \upkappa(\gamma)}} \geq \exp (-\ell\ceil{1/\gamma} \gamma  \upkappa(\gamma) ) \geq \exp(-\ell \upkappa(\gamma)) \eqsp ,
  \end{equation}
  which together with \eqref{eq:Xi_majo_2} concludes the proof for \Cref{propo:doeb}-\ref{item:kappa_neg}.
\item     Assume that for any $\tgamma \in \ocint{0,\bgamma}$, $  \upkappa (\tgamma) \leq 0$. Then,
  \begin{equation}
    \Xi_{\ell \step}(\upkappa) =  \gamma \sum_{i=1}^{\ell \ceil{1/\gamma}} (1+ \gamma  \upkappa(\gamma))^{-i} \geq   \gamma \step \ell \geq   \ell \eqsp .
  \end{equation}
  \item      Assume that for any $\tgamma \in \ocint{0,\bgamma}$ $,  \upkappa (\tgamma) >  0$. Using that $\log(1+t) \geq t / (1+t)$ for $t >0$, we obtain that
  \begin{align}
    \exp\parentheseDeux{-\ell\ceil{1/\gamma} \log \defEns{1 + \gamma  \upkappa(\gamma)}} &\leq \exp \parentheseDeux{-(\ell/\gamma) \log\defEns{1 + \gamma  \upkappa(\gamma)} } 
    \\ &\leq \exp \parentheseDeux{-\ell \upkappa(\gamma) /(1+\gamma  \upkappa(\gamma))} \eqsp ,
  \end{align}
  which concludes the proof for \Cref{propo:doeb}-\ref{item:kappa_neg}.
\end{enumerate}
\end{proof}

\begin{proof}[Proof of \Cref{propo:doeb}]
  The proof is a direct application of
  \Cref{theo:minorization_general} and \Cref{lemma:technical_doeb_0}
  with $\upkappa(\gamma) = \kappa(\gamma)$.
\end{proof}

\subsection{Proof of \Cref{coro:doeblin_lemme_1}}
\label{sec:proof-crefc}
  \begin{enumerate}[label= (\alph*),  wide, labelwidth=!, labelindent=0pt]
  \item Consider $V : \msx \to \ccint{1,\plusinfty}$ given for any
    $x \in \msx$ by $V(x) = 1 + \norm{x} $. Then since
    \Cref{assum:lip_op}($\msx^2$) with
    $\sup_{\gamma \in \ocint{0,\bgamma}} \kappa(\gamma) \leq \kappa_-
    < 0$ holds, using the triangle inequality and the Cauchy-Schwarz
    inequality, we have for any $\gamma \in \ocint{0, \bgamma}$ and
    $x \in \msx$
  \begin{equation}
    \Rcoupling_{\gamma} V(x) \leq \norm{\Tg(x)} + \sqrt{\gamma d} \leq (1+\kappa_-\gamma)\norm{x} + \norm{\Tg(0)} + \sqrt{\gamma d} +1 \leq \lambda V(x) + A \eqsp,
  \end{equation}
  with $\lambda \in \ooint{0,1}$ and $A \geq 0$.  As a result, since
  for any $\gamma \in \ocint{0, \bgamma}$, $\Rker_{\gamma}$ is a
  Feller kernel and the level sets of $V$ are compact,
  $\Rker_{\gamma}$ admits a unique invariant probability measure
  $\pi_{\gamma}$ for any $\gamma \in \ocint{0, \bgamma}$ by
  \cite[Theorem 12.3.3]{douc:moulines:priouret:soulier:2018}. Then the
  last result is a straightforward consequence of
  \Cref{propo:doeb}-\ref{item:kappa_neg}, \eqref{eq:distrib_coupling}
  and the fact that for any $\ell \in \nsets$ and
  $\gamma \in \ocint{0,\bgamma}$,
  $\alpha_-(\kappa,\gamma,\ell) \geq -(\exp(-\ell
  \kappa_-)-1)/\kappa_-$ since $t \mapsto (\exp(\ell t)-1)/t$ is
  increasing on $\rset$.
  
\item This result is a direct consequence of
  \Cref{propo:doeb}-\ref{item:kappa_0}, \eqref{eq:distrib_coupling}
  and the fact that $\Rker_{\gamma}$ admits an invariant probability
  measure $\pi_{\gamma}$.
  \end{enumerate}

  \subsection{Proof of \Cref{coro:doeblin_lemme_2}}
\label{coro:doeblin_lemme_2:proof}

  \begin{enumerate}[label= (\alph*),  wide, labelwidth=!, labelindent=0pt]
  \item The proof is a direct application of
    \Cref{propo:doeb}-\ref{item:kappa_0}, the fact that
    $(X_k,Y_k) \in \msx^2$ for any $k \in \nset$ and that
    $\Kker_{\gamma}$ is the Markov kernel associated with
    $(X_k,Y_k)_{k \in \nset}$.
  \item 
Consider the case where \Cref{assum:lip_op}($\msx^2$)-\ref{assum:lip_op_non_convex} holds. Using that for any $t \geq 0$, $1- \rme^{-t} \geq t/(t+1)$ we obtain that for any $\gamma \in \ocint{0, \bgamma}$ and $\ell \in \nsets$
\begin{equation}
  \label{eq:alpha_ineg}
  \alpha_+(\kappa, \gamma, \ell)
  \geq  \ell / (1 + (\ell  + \bgamma) \kappa(\gamma) ) \geq (1 + (1 + \bgamma) \kappa_+)^{-1} \geq (1 + \bgamma)^{-1}(1 + \kappa_+)^{-1} \eqsp,
\end{equation}
where $\alpha_+$ is given in
\Cref{lemma:technical_doeb_0}-\ref{item:kappa_pos}.  Then, combining
this result and \Cref{propo:doeb}-\ref{item:kappa_pos} complete the
proof.
\end{enumerate}

\subsection{Proof of \Cref{theo:discrete_contrac_wass_D_v2}}
\label{sec:theo:discrete_contrac_wass_D_v2:proof}

We start with the following theorem.

\begin{theorem}
  \label{thm:uno}
  Under the assumptions of \Cref{theo:discrete_contrac_wass_D_v2}, we have 
for any $\gamma \in \ocint{0, \bgamma}$,  $x, y \in \msx$ and $k \in \N$
\begin{equation}
  \label{eq:theo:discrete_contrac_wass_D_v2_a}
  \distV(\updelta_x \Rcoupling_{\gamma}^k, \updelta_y \Rcoupling_{\gamma}^k) \leq \KkerD_{\gamma}^k \bfc(x,y) \leq \lambda^{k\gamma/4} [D_{\gamma,1} 
 \bfc(x,y) + D_{\gamma,2}\1_{\Deltar^{\complementary}}(x,y)] +  \tC_{\gamma} \trho_{\gamma}^{k\gamma/4}\1_{\Deltar^{\complementary}}(x,y)  \eqsp , 
  \end{equation}
  where $\distV$ is the Wasserstein metric associated with $\bfc$
  defined by \eqref{eq:def_wbf},
  \begin{equation}
    \begin{aligned}
      D_{\gamma,1} &= 1+ 4A  [\log(1/\lambda)\lambda^{\gamma}]^{-1} \eqsp,    \qquad    D_{\gamma,2} =       D_{\gamma,1} \parentheseDeux{ A \lambda^{-\gamma \step \ell} \gamma  \step \ell}  \eqsp, \\ \tC_{\gamma} &= 8 A \log^{-1}(1/\trho_{\gamma})/ \trho_{\gamma}^{\gamma} \eqsp ,\\ 
      \log(\trho_{\gamma}) \eqsp &= \defEns{\log(\lambda) \log(1 - \tvareps_{\discrete, \gamma}) } / \defEns{-\log(\tc_{\gamma})  + \log(1 - \tvareps_{\discrete, \gamma}) } \eqsp ,
\\
\tBdisc &= \sup_{(x,y) \in \msc }\lyap(x,y) \eqsp, \quad   \tc_{\gamma}  = \tBdisc  + A \lambda^{-\gamma \step \ell} \gamma \step \ell \eqsp ,      \\ \tvareps_{\discrete, \gamma} &=  \inf_{(x,y) \in \Delta_{\msx,\tM_{\discrete}}} \Psibf(\gamma, \ell, \norm{x-y}) \eqsp.
  \end{aligned}
\end{equation}  
\end{theorem}

\begin{proof}
  The proof of this proposition is an application of
  \Cref{theo:quanti_v_alain_v_norm} in \Cref{sec:quant-bounds-geom}
  with $\distance \leftarrow \1_{\Delta_{\msx}^{\complementary}}$
  which satisfies \Cref{ass:metric}.  Let
  $\gamma \in \ocint{0, \bgamma}$. Then, since $\KkerD_{\gamma}$ and
  $\Psibf$ satisfy
  \hyperlink{ass:drift_discrete}{$\bfDd(\VlyapD,\lambda^{\gamma},
    A\gamma, \msc)$} and \eqref{eq:minorization_condition_v2}
  respectively, and $\Delta_{\msx}$ is absorbing for
  $\KkerD_{\gamma}$, \Cref{ass:kernel_coupling}($\Kcoupling_{\gamma}$)
  and \Cref{assum:drift_d}($\Kcoupling_{\gamma}$) are satisfied.  More
  precisely, for any $\gamma \in \ocint{0, \bgamma}$ setting
  $\tvareps_{\discrete, \gamma} = \inf_{(x,y) \in
    \Delta_{\msx,\tM_{\discrete}}} \Psibf(\gamma, \ell, \norm{x-y})$,
  then
  \Cref{ass:kernel_coupling}($\Kcoupling_{\gamma}$)-\ref{ass:kernel_coupling_a}
  is satisfied since for any
  $x,y \in \msc \subset \Delta_{\msx,\tM_{\discrete}}$,
  \begin{align} \KkerD_{\gamma}^{\step
      \ell}\1_{\Deltar^{\complementary}}(x,y) &\leq \defEns{1 -
      \inf_{(x,y) \in \Delta_{\msx,\tM_{\discrete}}}\Psibf(\gamma,
      \ell, \norm{x-y})} \1_{\Deltar^{\complementary}}(x,y) \\ &\leq
    (1 - \tvareps_{\discrete,
      \gamma})\1_{\Deltar^{\complementary}}(x,y) \eqsp. \end{align}
  \Cref{ass:kernel_coupling}($\Kcoupling_{\gamma}$)-\ref{ass:kernel_coupling_b}
  is satisfied since for any $\gamma \in \ocint{0, \bgamma}$ and
  $x,y \in \msx$,
  $\Kcoupling_{\gamma} \1_{\Deltar^{\complementary}}(x,y) \leq
  \1_{\Deltar^{\complementary}}(x,y)$.  Finally, the conditions
  \Cref{ass:kernel_coupling}($\Kcoupling_{\gamma}$)-\ref{ass:kernel_coupling_c}
  and \Cref{assum:drift_d}($\Kcoupling_{\gamma}$) hold using
  \hyperlink{ass:drift_discrete}{$\bfDd(\VlyapD,\lambda^{\gamma},
    A\gamma, \msc)$} with $\lyap_1 \leftarrow \lyap$,
  $\lyap_2 \leftarrow \lyap \distY$,
  $\lambda_1 = \lambda_2 \leftarrow \lambda^{\gamma}$,
  $A_1=A_2 \leftarrow A \gamma$, $\ntt \leftarrow \ell \step$.
  Applying \Cref{theo:quanti_v_alain_v_norm}, we obtain that for any
  $k \in \nset$, $\gamma \in \ocint{0, \bgamma}$ and $x,y \in \msx$
  \begin{align}
     &\distV(\updelta_x \Rcoupling_{\gamma}^k, \updelta_y \Rcoupling_{\gamma}^{k})   \\ & \leq \lambda^{k \gamma} \lyap(x,y) + A\gamma  \parentheseDeux{ \trho_{\gamma}^{k \gamma/4} r_1 (1+\1_{\Deltar^{\complementary}}(x,y))  +  \lambda^{k \gamma /4} r_2 \Xibf(x,y,\ell\step)}  \\
     & \leq \lambda^{k \gamma /4}\lyap(x,y) + 2r_1 A\gamma \trho_{\gamma}^{k \gamma /4} + A\gamma r_2 \lambda^{k \gamma /4}\Xibf(x,y,\ell \step) \\
    &\leq \lambda^{k \gamma /4}(1 + A\gamma r_2 ) \parentheseDeux{ \lyap(x,y) + A\gamma \lambda^{-\ell \step \gamma} \ell \step \gamma }+ 2 r_1 A \gamma \trho_{\gamma}^{k \gamma /4} \eqsp ,
  \end{align}
  where
  \begin{equation}
       r_1 = 4 \log^{-1}(1/\trho_{\gamma})/(\gamma \trho_{\gamma}^{\gamma}) \eqsp, \quad r_2 = 4 \log^{-1}(1/\lambda)/(\gamma \lambda^{\gamma})  \eqsp. 
  \end{equation}
  This concludes the proof of \eqref{eq:theo:discrete_contrac_wass_D_v2_a} upon using that $\Delta_{\msx}$ is absorbing for $\KkerD_{\gamma}$.
  \end{proof}

  \begin{proof}[Proof of \Cref{theo:discrete_contrac_wass_D_v2}]
  The first part of the proof is straightforward using \Cref{thm:uno} and that $\lambda^{\gamma} \geq \lambda^{\bgamma}$.
  
By assumption on $\bgamma$ and $\lambda$, we have $\lambda^{-\gamma \step \ell}\gamma \step \ell \leq \lambda^{-(1+\bgamma)\ell}(1+ \bgamma)\ell$.  
As a result and using  the fact that $\log(1-t) \leq -t$ for any $t\in \ooint{0,1}$, $\log((1 - \bvareps_{\discrete, 1})^{-1}) \leq 1$ and $\lyap(x,y) \geq 1$ for any $x,y \in \msx$,  we obtain that
\begin{align}
  &\log^{-1}(\brho_1^{-1}) \leq [ \log(\lambda^{-1}) \log((1 - \bvareps_{\discrete, \bgamma})^{-1})]^{-1} \parentheseDeux{1 + \log(\bc_{2})} \\
    &\phantom{aaaa} \leq \parentheseDeux{\log(\lambda^{-1}) \bvareps_{\discrete, 1}}^{-1} \parentheseDeux{1 + \log\parentheseLigne{\tBdisc} + \log(1 + 2A\ell \lambda^{-2\ell})} \eqsp, \\ 
              &\phantom{aaaa}\leq \parentheseDeux{\log(\lambda^{-1}) \bvareps_{\discrete, 1}}^{-1} \parentheseDeux{1 + \log\parentheseLigne{\tBdisc} + \log(1 + 2A\ell) + 2\ell \log(\lambda^{-1})} \eqsp, 
\end{align}
which completes the proof. 
  \end{proof}

  \subsection{Proof of \Cref{prop:w1_contrac_iterees}}
\label{sec:prop:w1_contrac_iterees:proof}

  Let $\gamma \in \ocint{0, \bgamma}$, $x,y \in \msx$ and $k \in \nset$.
  We divide the proof into two parts.
  \begin{enumerate}[label= (\alph*),  wide, labelwidth=!, labelindent=0pt]
  \item If $k \leq \step$. Then using 
    we get that
    \begin{equation}
      \KkerD_{\gamma}^k \norm{x-y} \leq (1 + \gamma \varkappa)^k \norm{x-y} \leq (1 + \gamma \varkappa)^{\step} \norm{x-y} \leq \exp[\varkappa(1 + \bgamma)] \norm{x-y} \eqsp .
    \end{equation}
  \item If $k>\step$ then using
    \Cref{theo:discrete_contrac_wass_D_v2},
    \eqref{eq:minorization_condition_v2} and $\rho_1 \geq
    \lambda$ 
    we get that
    \begin{align}
      &\distV(\updelta_x \Rcoupling_{\gamma}^k, \updelta_y \Rcoupling_{\gamma}^k) \leq \KkerD_{\gamma}^{\step} \Kker_{\gamma}^{k - \step} \bfc(x,y) \\
                                                                                 & \leq \KkerD_{\gamma}^{\step} \defEns{ \lambda^{(k-\step)\gamma/4} [\bD_{1} \lyap(x,y) + \bD_{2} \1_{\Delta^{\complementary}}(x,y)] +  \bC_{2} \brho_1^{(k-\step)\gamma/4} \1_{\Delta^{\complementary}}(x,y)}  \eqsp \\
                                                                                 & \leq \brho_1^{(k-\step)\gamma/4}\KkerD_{\gamma}^{\step} \defEns{( \bD_1 + \bD_2 + \bC_1)\1_{\Delta^{\complementary}}(x,y) + \vartheta \bD_1 \norm{x-y}} \\
                                                                                 & \leq \brho_1^{k \gamma / 4} \defEns{( \bD_1 + \bD_2 + \bC_1)(1 - \Psibf(\gamma, 1, \norm{x-y})) + \vartheta \bD_1 \exp[\varkappa(1 + \bgamma)] \norm{x-y}} / \brho_1^{(1+\bgamma)/4} \\
      & \leq -\mathbf{a} ( \bD_1 + \bD_2 + \bC_1)\brho_1^{k\gamma/4} \norm{x-y}/\brho_1^{1/4} + \vartheta \bD_1 \exp[\varkappa(1 + \bgamma)] \brho_1^{k\gamma/4} \norm{x-y} / \brho_1^{(1+\bgamma)/4} \eqsp ,
    \end{align}
    which concludes the proof upon noting that
    $\distV(\updelta_x \Rcoupling_{\gamma}^k, \updelta_y \Rcoupling_{\gamma}^k)
    \geq \vartheta \wassersteinD[1](\updelta_x \Rcoupling_{\gamma}^k, \updelta_y
    \Rcoupling_{\gamma}^k)$. 
  \end{enumerate}

  \subsection{Proof of \Cref{prop:from_1_to_p}}
\label{sec:prop:from_1_to_p:proof}

  Let $q \in \nset$ and $\gamma \in \ocint{0,\bgamma}$. Using that $\KkerD_{\gamma}$ satisfies
  \hyperlink{ass:drift_discrete}{$\bfDd((x,y) \mapsto
    \normLigne{x-y}^q,\tilde{\lambda}_q^{\gamma}, \tilde{A}_q\gamma)$}, we get that
  for any $x, y \in \msx$ and $k \in \nset$ we
  have
  \begin{equation}
    \label{eq:unif_q}
    \KkerD_{\gamma}^k\norm{x-y}^q \leq \norm{x-y}^q + \tilde{A}_q\gamma \sum_{\ell=0}^{k-1} \tilde{\lambda}_q^{\ell \gamma} \leq \norm{x-y}^q + \tilde{A}_q\log^{-1}(1/\tilde{\lambda}_q)\tilde{\lambda}_q^{-\bgamma} \eqsp .
  \end{equation}
  Let $p \geq 1$, $\upalpha \in \ooint{p, +\infty}$, $x,y \in \msx$ and $k \in \nset$ and consider  $q = p(\upalpha -1)/(\upalpha -p)$. Note that we have
  \begin{equation}
    (1-1/\upalpha)p/\ceil{q} \leq (1-1/\upalpha)p/q \leq 1 - p/\upalpha \leq 1 \eqsp .
  \end{equation}
  Using this result, 
  \eqref{eq:unif_q}, Hölder's inequality, Jensen's inequality and that for any $a,b \geq 0$ and $r \geq 1$, $(a+b)^{1/r} \leq a^{1/r} + b^{1/r}$, we have
  \begin{align}
    &\KkerD_{\gamma}^k \norm{x-y}^p \leq \KkerD_{\gamma}^k \defEns{\norm{x-y}^{p(1-1/\upalpha)}\norm{x-y}^{p/\upalpha}} \\
                                   & \qquad \leq \parenthese{\KkerD_{\gamma}^k \norm{x-y}^{p(1-1/\upalpha)/(1-p/\upalpha)}}^{1-p/\upalpha} \parenthese{\KkerD_{\gamma}^k \norm{x-y}}^{p/\upalpha} \\
                                   & \qquad \leq \parenthese{\KkerD_{\gamma}^k \norm{x-y}^q}^{1-p/\upalpha} \bD^{p/\upalpha} \brho^{k\gamma p/\upalpha} \norm{x-y}^{p/\upalpha} \\
                                   & \qquad \leq \parenthese{\KkerD_{\gamma}^k \norm{x-y}^{\ceil{q}}}^{(1-p/\upalpha)q/\ceil{q}} \bD^{p/\upalpha} \brho^{k\gamma p/\upalpha} \norm{x-y}^{p/\upalpha} \\
                                   & \qquad \leq \parenthese{\norm{x-y}^{\ceil{q}} + \tilde{A}_{\ceil{q}}\log^{-1}(1/\tilde{\lambda}_{\ceil{q}})\tilde{\lambda}_{\ceil{q}}^{-\bgamma}}^{(1-p/\upalpha)q/\ceil{q}} \bD^{p/\upalpha} \brho^{k\gamma p/\upalpha} \norm{x-y}^{p/\upalpha}    \\
                                   & \qquad \leq \parenthese{\norm{x-y}^{\ceil{q}} + \tilde{A}_{\ceil{q}}\log^{-1}(1/\tilde{\lambda}_{\ceil{q}})\tilde{\lambda}_{\ceil{q}}^{-\bgamma}}^{(1-1/\upalpha)p/\ceil{q}} \bD^{p/\upalpha} \brho^{k\gamma p/\upalpha} \norm{x-y}^{p/\upalpha} \\
    & \qquad \leq \parenthese{\norm{x-y}^{(1-1/\upalpha)p} + \defEns{\tilde{A}_{\ceil{q}}\log^{-1}(1/\tilde{\lambda}_{\ceil{q}})\tilde{\lambda}_{\ceil{q}}^{-\bgamma}}^{(1-1/\upalpha)p/\ceil{q}}} \bD^{p/\upalpha} \brho^{k\gamma p/\upalpha} \norm{x-y}^{p/\upalpha} \\
    & \qquad \leq \parenthese{\norm{x-y}^p + \defEns{\tilde{A}_{\ceil{q}}\log^{-1}(1/\tilde{\lambda}_{\ceil{q}})\tilde{\lambda}_{\ceil{q}}^{-\bgamma}}^{(1-1/\upalpha)p/\ceil{q}} \norm{x-y}^{p/\upalpha}} \bD^{p/\upalpha} \brho^{k\gamma p/\upalpha}  \\    
    & \qquad \leq \bD_{4, \upalpha}^{p} \brho^{k \gamma p /\upalpha}(\norm{x-y}^p + \norm{x-y}^{p/\upalpha}) \eqsp ,
  \end{align}
which completes the proof upon using that for any $a,b \geq 0$ and $p \geq 1$, $(a+b)^{1/p} \leq a^{1/p} + b^{1/p}$.


\section{Proofs of  \Cref{sec:applications}}
\label{sec:applications:proof}

\subsection{Proof of \Cref{prop:a1_type}}

  \label{prop:a1_type:proof}
    \begin{enumerate}[label= (\alph*),  wide, labelwidth=!, labelindent=0pt]    
    \item By \Cref{as:item:lip} and
      \Cref{as:b_min}($\mtt$) we have for any
      $\gamma >0$ and $x,y \in \msx$,
      $ \norm{\Tg(x) - \Tg(y)}^2 \leq (1 - 2\gamma \mtt + \gamma^2 \Lip^2)
      \norm{x-y}^2 \leq (1 + \gamma \kappa(\gamma)) \norm{x-y}^2$, which
      concludes the proof.
  \item 
    We have for any $\gamma >0$ and $x,y \in \msx$,
$      \norm{\Tg(x) - \Tg(y)}^2 \leq  \norm{x-y}^2 + \gamma(-2\Lipb + \gamma)\norm{b(x) - b(y)}^2$.    Then if $\gamma \leq 2 \Lipb$, $      \norm{\Tg(x) - \Tg(y)}^2 \leq  \norm{x-y}^2$, which concludes the proof.
\end{enumerate}

\subsection{Proof of \Cref{propo:drift_strong_convex_wass}}
\label{propo:drift_strong_convex_wass:proof}

Let $\gamma \in \ocint{0, \bgamma}$, $x,y \in \msx$ and set   $\rmE = \Tg(y)  - \Tg(x)$. We divide the proof into three parts.

\begin{enumerate}[label=(\alph*),wide, labelwidth=!, labelindent=0pt]
\item First, we show that
  \Cref{propo:drift_strong_convex_wass}-\ref{item:reflex_ineq_a}
  holds.  If $\rmE = 0$ then the proposition is trivial, therefore we
  suppose that $\rmE \neq 0$ and let $\rme = \rmE / \norm{\rmE}$.
  Consider $Z_1$, a $d$-dimensional Gaussian random variable with zero
  mean and covariance identity. By \eqref{eq:coupling_form} and the
  fact that $\Pi_{\msx}$ is non expansive, we have for any
  $\gamma \in \ocint{0, \bgamma}$
\begin{align}
  \Kcoupling_{\gamma} \norm{x-y} & \leq \expe{\parenthese{1 - p_{\gamma}(x,y,\sqrt{\gamma}Z_1)} \norm{(\Tg(x)+\sqrt{\gamma} Z_1) - (\Tg(y) + \sqrt{\gamma}(\Id - 2\rme \rme^{\transpose})Z_1)}} \\
                                 & =\expe{\norm{\rmE - 2\sqrt{\gamma}\rme \rme^{\transpose}Z_1}\parenthese{1 - p_{\gamma}(x,y,\sqrt{\gamma}Z_1)}} \\
                                 &=  \int_{\rset} \norm{\rmE - 2 z\rme} \defEns{\vphibf_{\gamma}(z) - (\vphibf_{\gamma}(z) \wedge \vphibf_{\gamma}(\| \rmE \| - z))}  \rmd z \\
                                 &= \int_{-\infty}^{\norm{\rmE}/2} (\norm{\rmE} -2z) \defEns{\vphibf_{\gamma}(z) - \vphibf_{\gamma}(\| \rmE \|- z)} \rmd z \leq \norm{\rmE}\eqsp ,  \label{eq:norm_reflec}
\end{align}
where we have used the change of variable
$z \mapsto \normLigne{\rmE} - z$ for the last line.  We conclude this
part of the proof upon using \tup{\Cref{as:item:lip}} and
\tup{\Cref{as:b_min}($\mtt$)}.
\item Second, we show that
  \Cref{propo:drift_strong_convex_wass}-\ref{item:reflex_ineq_b}
  holds. Consider the case
  $(x,y) \in \Delta_{\msx,\Run}^{\complementary}$.  By
  \Cref{as:item:lip}, \Cref{assum:strong_convex_outside_ball}, and
  since for any $t \in \coint{-1,+\infty}$,
  $\sqrt{1 + t} \leq 1 + t/2$, we have that
\begin{equation}
  \label{eq:norm_reflec_22}
  \| \Tg(x) - \Tg(y) \| \leq (1 - 2\gamma \mttplusun + \gamma^2 \Lip^2)^{1/2}\norm{x-y} \leq  (1 - \gamma \mttplusun + \gamma^2 \Lip^2/2)\norm{x-y} \eqsp.
\end{equation}
   Combining 
   \eqref{eq:norm_reflec} and \eqref{eq:norm_reflec_22} and since $\gamma < 2 \mttplusun/\Lip^2$,
we obtain that for any $(x,y) \in \Delta_{\msx,\Run}^{\complementary}$,
\begin{align}
  \Kcoupling_{\gamma} \VlyapDun(x,y) &\leq (1 - \gamma \mttplusun + \gamma^2\Lip^2/2) \norm{x-y}/\Run +1 \\
                                 &\leq (1 - \gamma\mttplusun/2 + \gamma^2 \Lip^2/4)(1 + \norm{x-y}/\Run) \leq \lambda^{\gamma}\VlyapDun(x,y) \eqsp . \label{eq:drift_outside}
\end{align}
Similarly, we obtain using \Cref{prop:a1_type}-\ref{item:a1_1} that for any $(x,y) \in \Delta_{\msx,\Run}$
\begin{align}
  \Kcoupling_{\gamma} \VlyapDun &\leq (1 - \gamma \mtt + \gamma^2\Lip^2/2) \norm{x-y}/\Run +1 \\
  &\leq (1 - \gamma \mttplusun/2 + \gamma^2\Lip^2/4) \norm{x-y}/\Run +1 + \gamma \defEns{ \mttplusun/2 - \mtt + \gamma\Lip^2/4} \\
  &\leq (1 - \gamma \mttplusun/2 + \gamma^2\Lip^2/4) \VlyapDun(x,y) + \gamma \parentheseDeux{\mttplusun - \mtt} \leq \lambda^{\gamma} \VlyapDun(x,y) + A \gamma \eqsp .\label{eq:drift_inside}   
\end{align}
We conclude the proof upon combining \eqref{eq:drift_outside} and \eqref{eq:drift_inside}.
\item Finally we show that
  \Cref{propo:drift_strong_convex_wass}-\ref{item:reflex_ineq_c}
  holds.  Let $ p \in \nset$ with $p \geq 2$.  Similarly to
  \Cref{propo:drift_strong_convex_wass}-\ref{item:reflex_ineq_a}, we
  have
  \begin{equation}
    \label{eq:ineq_p}
    \Kker_{\gamma} \norm{x-y}^p = \int_{\rset} (\| \rmE \| - 2z)^p \vphibf_{\gamma}(z) \rmd z \eqsp .
  \end{equation}
  For any $k \in \nset$, let $c_k = \int_{\rset} z^k \vphibf_1(z) \rmd z$ and
  \begin{equation}
    \label{eq:kappa}
    \kappa_{1, \gamma} = 1 - \gamma \mttplusun +\gamma^2 \Lip^2/2 \eqsp , \qquad \kappa_{2, \gamma} = \max(1, 1 - \gamma \mtt +\gamma^2 \Lip^2/2) \eqsp , \qquad \bR = \max(1, R_1) \eqsp .
  \end{equation}
  Note that for any $k \in \nset$, $c_{2k+1} = 0$.  Consider the case
  $\norm{x-y} \geq \bR$. Using \eqref{eq:ineq_p}, \eqref{eq:kappa},
  \tup{\Cref{as:item:lip}},
  \tup{\Cref{assum:strong_convex_outside_ball}} we have
  \begin{align}
    \Kker_{\gamma} \norm{x-y}^p &\leq \norm{\rmE}^p + \sum_{k=2}^p {p \choose k} \norm{\rmE}^{p-k} (2\gamma)^k c_k \\
                                &\leq \kappa_{1, \gamma} \norm{x-y}^p + \sum_{k=2}^p {p \choose k} \norm{x-y}^{p-k} (2\gamma)^k c_k \\
                                &\leq \kappa_{1, \gamma} \norm{x-y}^p + \gamma c_{2p}2^{2p} \max(1, \bgamma)^p  \norm{x-y}^{p-2}\\
                                &\leq \kappa_{2, \gamma/2} \norm{x-y}^p + \gamma \defEns{c_{2p}2^{2p} \max(1, \bgamma)^p  \norm{x-y}^{p-2}  - \mttplusun \norm{x-y}^p/2} \\
                                &\leq \kappa_{2, \gamma/2} \norm{x-y}^p + \gamma \sup_{t \in \coint{0,+\infty}} \defEns{c_{2p}2^{2p} \max(1, \bgamma)^p  t^{p-2}  - \mttplusun t^p/2} \eqsp .
  \end{align}
Note that we have for any $a \geq b \geq 0$ and $t \geq 0$
\begin{equation}
  \label{eq:ineq_a_b}
  (1+ta)^p - (1+tb) \leq t \defEns{-b + \max(1, t)^p \sum_{k=1}^p {p \choose k} a^k } \leq t \defEns{\max(1, t)^p (1+ a)^p -b } \eqsp .
\end{equation}
Now, consider the case $\norm{x-y} \leq \bR$. Using \eqref{eq:ineq_p},
\eqref{eq:kappa}, \eqref{eq:ineq_a_b}, \tup{\Cref{as:item:lip}},
\tup{\Cref{as:b_min}($\mtt$)} we have
\begin{align}
  &\Kker_{\gamma} \norm{x-y}^p - \kappa_{1, \gamma/2} \norm{x-y}^p\leq (\kappa_{2, \gamma}^p - \kappa_{1, \gamma/2})  \norm{x-y}^p + \gamma c_{2p}2^{2p} \max(1, \bgamma)^p \kappa_{2, \gamma}^p  \bR^{p-2} \\
                              &\leq  \gamma c_{2p}2^{2p} \max(1, \bgamma)^p \kappa_{2, \gamma}^p  \bR^{p-2} + (\kappa_{2, \gamma}^p - \kappa_{1, \gamma/2}) \bR^p \\
&\leq \gamma c_{2p}2^{2p} \max(1, \bgamma)^p \kappa_{2, \gamma}^p  \bR^{p-2} +   \gamma \bR^p   \defEns{\max(1, \bgamma)^p (1 - \mtt/2 + \Lip^2 \bgamma /4)^p + \mttplusun} \eqsp ,
\end{align}
which concludes the proof upon setting
\begin{equation}
  \begin{aligned}
    \lambda_p &= \exp[-\mttplusun/2 + \bgamma \Lip^2/4] \eqsp , \\
    A_p &= \max \defEns{A_{p,1}, A_{p,2}} \eqsp ,\\
    A_{p,1} &= \sup_{t \in \coint{0,+\infty}} \defEns{c_{2p}2^{2p} \max(1, \bgamma)^p  t^{p-2}  - \mttplusun t^p/2} \eqsp ,\\
    A_{p,2} &= c_{2p}2^{2p} \max(1, \bgamma)^p \kappa_{2, \gamma}^p  \bR^{p-2} +  \bR^p   \defEns{\max(1, \bgamma)^p (1 - \mtt/2 + \Lip^2 \bgamma /4)^p + \mttplusun} \eqsp .
  \end{aligned}
\end{equation}

\end{enumerate}

\subsection{Proof of \Cref{coro:w1_wp}}
\label{coro:w1_wp:proof}
    Let $\bgamma >0$. Then for any $\gamma \in \ocint{0, \bgamma}$, $\Psibf: \ t \mapsto 2\Phibf\defEnsLigne{-t/(2\Xi_{
        \step}^{1/2}(\kappa))}$ is convex on $\coint{0, +\infty}$, differentiable on $\rset$, and for any $\gamma \in \ocint{0, \bgamma}$
    \begin{equation}
      \label{eq:bound}
      \Psibf'(0) \geq -(\uppi \inf_{\gamma \in \ocint{0, \bgamma}} \Xi_{\step}(\kappa))^{-1/2} \eqsp .
    \end{equation}
    We divide the rest of the proof into two parts.
    \begin{enumerate}[label=(\alph*),wide, labelwidth=!, labelindent=0pt]      
    \item First combining \eqref{eq:bound},
      \Cref{propo:drift_strong_convex_wass}-\ref{item:reflex_ineq_a},
      \eqref{eq:wc1_convergence}, \Cref{propo:cvx_outside_bounds} and
      \Cref{prop:w1_contrac_iterees} shows
      that \begin{equation}\wassersteinD[1](\updelta_x
        \Rker_{\gamma}^k, \updelta_y \Rker_{\gamma}^k ) \leq D_{3,
          \bgamma,a} \rho_{\bgamma, a}^{k \gamma/4} \norm{x-y} \eqsp
        . \end{equation}
    \item Second, combining
      \Cref{propo:drift_strong_convex_wass}-\ref{item:reflex_ineq_c}
      and \Cref{prop:from_1_to_p} shows that
      \begin{equation}\wassersteinD[p](\updelta_x \Rker_{\gamma}^k, \updelta_y
        \Rker_{\gamma}^k ) \leq D_{\upalpha, \bgamma,a} \rho_{\bgamma, a}^{k
          \gamma/(4\upalpha)} \defEns{ \norm{x-y} + \norm{x-y}^{1/\upalpha}}
        \eqsp .\end{equation}
    \end{enumerate}

\subsection{Proof of \Cref{propo:drift_strong_convex}}
  \label{propo:drift_strong_convex:proof}
  We preface the proof by a technical result.
  \begin{lemma}
  \label{lemma:drift_2_drift_1}
  Let $\bgamma >0$, such that for any $\gamma \in \ocint{0, \bgamma}$, $\Pker_{\gamma}$ is a Markov kernel and $\Qker_{\gamma}$ is a Markov coupling kernel for $\Pker_{\gamma}$. Assume that there exist $V : \ \msx \to \coint{1,+\infty}$ measurable, $\lambda \in \ooint{0,1}$ and $A \geq 0$ such that for any $\gamma \in \ocint{0, \bgamma}$, $\Pker_{\gamma}$ satisfies \hyperlink{ass:drift_discrete}{$\bfDd(V,\lambda^{\gamma}, A\gamma,\msx)$}. Let $\lyap: \ \msx^2 \to \coint{1, +\infty}$ given for any $x,y \in \msx$ by  $\lyap(x,y) = \defEns{V(x) + V(y)}/2$. The following properties hold.
  \begin{enumerate}[label= (\alph*)]
  \item $\Qker_{\gamma}$ satisfies \hyperlink{ass:drift_discrete}{$\bfDd(\lyap,\lambda^{\gamma}, A\gamma,\msx^2)$} \label{item:a_drift}
  \item if $\lim_{\normLigne{x} \to +\infty} V(x) = +\infty$, $\Qker_{\gamma}$ satisfies \hyperlink{ass:drift_discrete}{$\bfDd(\lyap,\lambda^{\gamma/2}, A\gamma, \cballdeux{0}{R})$} where $R = \inf \ensembleLigne{r \geq 0}{\text{for any $x \in \cball{0}{r}^{\complementary}$, } V(x) \geq 2 (\lambda^{1/2})^{-2\bgamma} A \log^{-1}(1/\lambda^{1/2})}$\label{item:b_drift} and $\cballdeux{0}{R} \subset \Delta_{\msx, 2R}$.
  \end{enumerate}
\end{lemma}
\begin{proof}
  Let $\gamma \in \ocint{0, \bgamma}$ and $x,y \in \msx$.
  \begin{enumerate}[label= (\alph*),  wide, labelwidth=!, labelindent=0pt]
  \item Since $\updelta_{(x,y)} \Qker_{\gamma}$ is a transference plan between $\updelta_x \Pker_{\gamma}$ and $\updelta_y \Pker_{\gamma}$ we have
  \begin{equation}
    \Qker_{\gamma}\lyap(x,y) = \Qker_{\gamma}\defEns{V(x) + V(y)}/2 = \Pker_{\gamma}V(x)/2 + \Pker_{\gamma}V(y)/2 \leq \lambda^{\gamma} \lyap(x,y) + A \gamma \eqsp .
  \end{equation}
\item Let $x,y \in \msx$. If $(x,y) \in \cballdeux{0}{R}$ then the result is immediate using \Cref{lemma:drift_2_drift_1}-\ref{item:a_drift}. Now, assume that $(x,y) \notin \cballdeux{0}{R}$.
  By definition of $R$, $\max(V(x), V(y)) \geq 4 \lambda^{-\bgamma} A \log^{-1}(1/\lambda)$. Without loss of generality assume that $V(x) \geq V(y)$. Using this result, \Cref{lemma:drift_2_drift_1}-\ref{item:a_drift} and that for any $b \geq a$, $(\rme^{b} - \rme^{a}) \geq \rme^{a}(b-a)$, we have
  \begin{align}
    \Qker_{\gamma} \lyap(x,y) &\leq \lambda^{\gamma} \lyap(x,y) + A \gamma \\
                              &\leq \lambda^{\gamma/2} \lyap(x,y) + \gamma \parentheseDeux{A + \lambda^{\gamma} \defEns{\log(\lambda) - \log(\lambda)/2}\lyap(x,y)} \\
                              &\leq \lambda^{\gamma/2} \lyap(x,y) + \gamma \parentheseDeux{A - \lambda^{\bgamma}  \log(\lambda^{-1}) V(x)/4}  \leq \lambda^{\gamma/2} \lyap(x,y) \eqsp .
  \end{align}

  \end{enumerate}
\end{proof}

\begin{proof}[Proof of \Cref{propo:drift_strong_convex}]
  Let $\gamma \in \ocint{0, \bgamma}$ and $x \in \msx$.  We divide the
  proof into two parts.  Using \eqref{eq:langevin_discrete},
  \Cref{as:item:lip}, \Cref{as:b_min}($\mtt$),
  \Cref{assum:drift_strong}, that the projection $\Pi_{\msx}$ is non
  expansive and $\gamma < 2 \mttplusdeux/\Lip^2$, we obtain for any
  $x \in \msx$
  \begin{align}
    \Rcoupling_{\gamma}V(x)
                              &\leq 1 + \norm{x + \gamma b(x)}^2 + \gamma d \\
                     &\leq1 + \norm{x}^2 + 2\gamma \langle x, b(x) \rangle + \gamma^2 \norm{b(x)}^2 + \gamma d \\
    &\leq (1 + \norm{x}^2) \parentheseDeux{1 - \gamma(2\mttplusdeux - \bgamma \Lip^2)} + \gamma \parenthese{d + 2\Rdeux^2(\mttplusdeux - \mtt)_+ + 2\mttplusdeux} \eqsp .
  \end{align}
In addition, for any $x \in \msx$, such that $\norm{x} \geq 2 A^{1/2}\log^{-1/2}(1/\lambda)$, we have $V(x) \geq 4 A \log^{-1}(1/\lambda)$. We conclude the proof using \Cref{lemma:drift_2_drift_1}-\ref{item:b_drift}.

\end{proof}


\subsection{Proof of \Cref{propo:drift_convex}}
\label{propo:drift_convex:proof}
   Let $\gamma \in \ocint{0, \bgamma}$. Using the fact that $\Pi_{\msx}$ is non expansive, the Log-Sobolev inequality, the fact that $\pi$ is $1$-Lipschitz, \cite[Theorem 5.5]{boucheron2013concentration} and the Jensen inequality we obtain for any $x \in \rset^d$
   \begin{align}
     \Rker_{\gamma} V(x) &\leq \exp \parentheseDeux{\mtttrois \Rker_{\gamma}\phi(x) + \gamma (\mtttrois)^2/2} \leq \exp \parentheseDeux{\mtttrois \sqrt{1 + \Rker_{\gamma}\norm{x}^2} + \gamma (\mtttrois)^2/2 } \\ &\leq  \exp \parentheseDeux{\mtttrois \sqrt{1 + \norm{\Tg(x)}^2 + \gamma d} + \gamma (\mtttrois)^2/2 } \eqsp . \numberthis \label{eq:ineq_1_curv}
   \end{align}
   Let $x \in \rset^d$. The rest of the proof is divided in two parts.
  \begin{enumerate}[label= (\alph*),  wide, labelwidth=!, labelindent=0pt]
  \item    
   In the first case, $\norm{x} \geq \Rquatre$.
   Since $\norm{x} \geq \Rtrois$ and $\gamma \leq 2\mttdeux$, we have using \Cref{assum:curvature}
   \begin{equation}
  \norm{\Tg(x)}^2 \leq \norm{x}^2 - 2\gamma \mttun  \norm{x} + \gamma(\gamma - 2\mttdeux) \norm{b(x)}^2 + \gamma \cconst \leq \norm{x}^2 - 2\gamma \mttun \norm{x} + \gamma \cconst \eqsp . \label{eq:ineq_2_curv}
\end{equation}
Since $\norm{x} \geq 1$ we have $2 \norm{x} \geq \phi(x)$ and therefore, using that $\norm{x} \geq (d+\cconst)/\mttun$, $2 \mttun \norm{x} \geq 2 \mtttrois \phi(x) + d +\cconst$. This inequality, combined with the fact that for any $t \in \ooint{-1,+\infty}$, $\sqrt{1 + t} \leq 1 + t/2$, yields
\begin{align}
  &\sqrt{1 + \norm{x}^2 + \gamma (-2\mttun \norm{x} + d +\cconst)} - \phi(x) \\ & \qquad \qquad \qquad \leq \gamma (-2 \mttun \norm{x} + d+\cconst)/(2\phi(x)) \leq -  \gamma \mtttrois \eqsp . \label{eq:ineq_3_curv}
\end{align}
Combining \eqref{eq:ineq_1_curv}, \eqref{eq:ineq_2_curv} and \eqref{eq:ineq_3_curv} we get
\begin{equation}
  \Rker_{\gamma} V(x) \leq \lambda^{\gamma} V(x) \eqsp .
\end{equation}
\item In the second case $\norm{x} \leq \Rquatre$. We have the following inequality using \tup{\Cref{assum:curvature}} and that $\gamma \leq 2 \mttdeux$
  \begin{equation}
    \norm{\Tg(x)}^2 \leq \norm{x}^2  + \gamma(\gamma - 2\mttdeux) \norm{b(x)}^2 + \gamma c \leq \norm{x}^2  + \gamma \cconst \eqsp . \label{eq:ineq_4_curv}
  \end{equation}
 Combining \eqref{eq:ineq_1_curv}, \eqref{eq:ineq_4_curv} and the fact that for any $t \in \ooint{-1,+\infty}$, $\sqrt{1 + t} \leq 1 + t/2$ we get
 \begin{align}
   \label{eq:ineq_5_curv} \Rker_{\gamma} V(x) &\leq \exp \parentheseDeux{ \gamma \mtttrois (d+\cconst)/(2\phi(x)) + \gamma (\mtttrois)^2/2 }V(x) \\ &\leq \exp\parentheseDeux{\gamma (\mtttrois (d+\cconst) + (\mtttrois)^2)/2}V(x) \eqsp . 
 \end{align}
Note that for any $c_1 \geq c_2$ and $t \in \ccintLigne{0, \bar{t}}$  we have the following inequality
\begin{equation}
  \rme^{c_1 t} \leq \rme^{c_2 t} + \rme^{c_1 \bar{t}}(c_1 - c_2)t \eqsp . \label{eq:ineq_6_curv}
\end{equation}
Combining \eqref{eq:ineq_5_curv} and \eqref{eq:ineq_6_curv} we get
\begin{equation}
\Rker_{\gamma} V(x) \leq \lambda^{\gamma}V(x) + \exp\parentheseDeux{\bgamma (\mtttrois (d+\cconst) + (\mtttrois)^2)/2}C_{\cconst} \gamma \eqsp ,
\end{equation}
with $C_{\cconst} = (\mtttrois (d+\cconst)/2 + (\mtttrois)^2)\exp(\mtttrois (1 + \Rquatre^2)^{1/2})$, 
which concludes the proof using \Cref{lemma:drift_2_drift_1}.
  \end{enumerate}




\section{Proofs of \Cref{sec:quant-conv-bounds}}
\label{sec:quant-conv-bounds:proof}

\subsection{Proof of \Cref{thm:limit_lip}}
\label{thm:limit_lip:proof}

Combining \Cref{prop:L_ass_check_lip}, \Cref{prop:integrability_R_lip}
and \Cref{propo:majo_cont} in \Cref{prop:general_bound} concludes the
proof.

\subsection{Proof of \Cref{thm:limit_loc_lip}}
\label{thm:limit_loc_lip:proof}

Combining \Cref{prop:L_ass_check_lip}, \Cref{prop:integrability_R_lip}
and \Cref{propo:majo_cont} in \Cref{prop:general_bound} concludes the
proof.

\subsection{Proof of \Cref{theo:c_ergo_v_norm}}
\label{sec:theo:c_ergo_v_norm:proof}

Let $T \geq 0$ and $x \in \rset^{\dim}$.  First, using
\Cref{prop:L_ass_check_loc_lip} we have that
\tup{\Cref{ass:loc_lip_b}}, \tup{\Cref{ass:diff_disc}},
\tup{\Cref{ass:seq_k_n}} and \tup{\Cref{ass:lyap_majo}} are
satisfied. In addition, using \Cref{propo:majo_cont} we get
  \begin{equation}
    \Pker_T V_{\mttun}(x) < +\infty \eqsp ,
  \end{equation}
  where $V_{\mttun} = V_M$ with $M \leftarrow \mttun$ and $V_M$ given
  in \eqref{eq:lyap_m}. Since,
  $\sup_{x\in \rset^d} \norm{b(x)}^{2(1 + \vareps_b)} \rme^{-\mttun (1
    + \norm{x})^{1/2}} <+\infty$, we get that \Cref{assum:unif_integr}
  is satisfied.

  In addition, using that $2 \mtttrois \leq \mttun$ we have
  \begin{equation}
    \Pker_T V^2(x) < +\infty \eqsp .
  \end{equation}
  Thus, the first part of \eqref{eq:integr_condition_3} is
  satisfied. Second using \Cref{propo:drift_convex} and replacing
  $\mtttrois \leftarrow 2\mtttrois$ (which is valid since
  $\mtttrois \leq \mttun /4$), we obtain that for any
  $n, m \in \nset$, with $T/m < 2 \mttdeux$, $\Rker_{T/m, n}$ and
  $\tRker_{T/m}$ satisfy
  \hyperlink{ass:drift_discrete}{$\bfDd(V^2,\lambda^{T/m},AT/m,
    \msx^2)$}. Hence, for any $m \in \nset$ with $T/m < 2 \mttdeux$ we
  have
  \begin{equation}
    \Rker_{\Time/m, n}^{m} + \tRker_{\Time/m, n}^{m} V^2(x) \leq V^2(x) + A T m^{-1} \sum_{k \in \nset} \lambda^{T/m} \leq V^2(x) + A \log^{-1}(1/\lambda) \lambda^{-\bgamma } \eqsp .
  \end{equation}
  Therefore, the second part of \eqref{eq:integr_condition_3} is
  satisfied and we can apply \Cref{prop:general_bound}.  Using
  \Cref{propo:weak_outside_bounds} and we
  get that for any $m,n \in \nset$ with $x,y \in \rset^{\dim}$ and
  $\Time/m \in \ooint{0,2\mttdeux}$
    \begin{equation}
      \Vnorm{\updelta_x \Rker_{\Time/m, n}^{m}- \updelta_y \Rker_{\Time/m, n}^{m} } \leq C_{1/m,c} \rho_{1/m,c}^{\Time} \defEns{V(x) + V(y)}\eqsp,
    \end{equation}
where $C_{1/m,c} \geq 0$ and $\rho_{1/m,c} \in \ooint{0,1}$, see \Cref{sec:rates-crefpr}. We conclude upon noting that $C_{1/m,c}$ and $\rho_{1/m,c}$ admit limits $C_c$ and $\rho_c$ when $m \to +\infty$ which do not depend on $n$.

\subsection{Proof of \Cref{lemma:disclyapfromc}}
\label{sec:proof-crefl-1}

\begin{enumerate}[label= (\alph*),  wide, labelwidth=!, labelindent=0pt]
\item Let $x \in \rset^d$ and let $(\bfX_t)_{t \geq 0}$ a solution of
  \eqref{eq:diff} starting from $x$. Define for any $k \in \nsets$,
  $\tau_k = \inf\{ t \geq 0 \, : \, \norm{\bfX_t} \geq k\}$ and for
  any $t \geq 0$, $\bfM_t = \int_0^t \ps{\nabla V(\bfX_s)}{\rmd B_s}$.
  Using the Itô formula we obtain that for every $t \geq 0$ and
  $k \in \nsets$
     \begin{align}
&       V(\bfX_{t \wedge \tau_k})\rme^{\zeta (t \wedge \tau_k)} =
                                                                 \int_0^{t \wedge \tau_k}  \parentheseDeux{\rme^{\zeta (t \wedge \tau_k)} \generator V (\bfX_u) + \zeta \rme^{\zeta u} V(\bfX_u)} \rmd u   + \bfM_{t \wedge \tau_k} + V(x) \\
                                                                 & \qquad = V(\bfX_{s \wedge \tau_k})\rme^{\zeta (s \wedge \tau_k)} + \bfM_{t \wedge \tau_k} - \bfM_{s \wedge \tau_k} + \int_{s \wedge \tau_k}^{t \wedge \tau_k}  \parentheseDeux{\rme^{\zeta (t \wedge \tau_k)} \generator V (\bfX_u) + \zeta \rme^{\zeta u} V(\bfX_u)} \rmd u  \\
       &\qquad \leq V(\bfX_{s \wedge \tau_k})\rme^{\zeta (s \wedge \tau_k)} + \bfM_{t \wedge \tau_k} - \bfM_{s \wedge \tau_k} \eqsp.
     \end{align}
     Therefore since for any $k \in \nsets$, $(\bfM_{t \wedge \tau_k})_{t \geq 0}$ is a $(\mcf_t)_{t \geq 0}$-martingale, we get for every $t \geq s \geq 0$ and $k \in \nsets$
     \begin{equation}
\CPE{V(\bfX_{t \wedge \tau_k})\rme^{\zeta (t \wedge \tau_k)}}{\mathcal{F}_s} \leq V(\bfX_{s \wedge \tau_k})\rme^{\zeta (s \wedge \tau_k)} \eqsp ,
\end{equation}
which concludes the first part of the proof taking  $k \to \plusinfty$ and using Fatou's lemma.
   \item Similarly we have that $( V(\bfX_t)\rme^{\zeta t} - B(1 - \exp(-\zeta t ))/\zeta )_{t \geq 0}$ is a $(\mcf_t)_{t \geq 0}$-supermartingale which concludes the proof  upon taking the expectation of $V(\bfX_t)\rme^{\zeta t} - B(1 - \exp(-\zeta t ))/\zeta$.
     \end{enumerate}

     \subsection{Proof of \Cref{lemma:V-norm_control}}
\label{lemma:V-norm_control:proof}

Let $\Time \geq 0$, $x \in \rset^d$, $n \in \nset$ and $m \in \nsets$ with $\Time/m \leq \bgamma$.
Using  \cite[Lemma 24]{durmus2017nonasymp}, we obtain
\begin{align}
  &\Vnorm{\updelta_x \Pker_{\Time} - \updelta_x \tRker_{\Time/m, n}^{m} } \\ & \qquad \leq (1/\sqrt{2}) \parenthese{\updelta_x \Pker_{\Time} V^2(x) + \updelta_x \tRker_{\Time/m, n}^mV^2(x)}^{1/2}\KL{\updelta_x \Pker_{\Time}}{\updelta_x \tRker_{\Time/m, n}^{m} }^{1/2} \eqsp .
\end{align}
Let $M \geq 0$, $n \in \nsets$ with $n^{-1} < \bgamma$, $x \in \rset^d$ and
$k \in \nset$.  Therefore, we only need to show that
$\lim_{m \to +\infty} \KLLigne{\updelta_x \Pker_{\Time}}{\updelta_x
  \tRker_{\Time/m, n}^{m} }= 0$.  Consider the two processes
$(\bfX_t)_{t \in \ccint{0,\Time}}$ and $(\bbfX_t)_{t \in \ccint{0,\Time}}$
defined by \eqref{eq:diff} with $\bfX_0 = \bbfX_0 = x$ and
\begin{equation}
  \rmd \bbfX_t =  \tb_{\Time/m, n}(t, (\bbfX_s)_{s \in \ccint{0,\Time}}) \rmd t +  \rmd \bfB_t \eqsp , \qquad \bbfX_0 = x \eqsp ,
\end{equation}
where for any $(\rmw_s)_{s \in \ccint{0,\Time}} \in \rmC(\ccint{0,\Time}, \rset^d)$, $t \in \ccint{0,\Time}$,
\begin{equation}
  \label{eq:def_bbar_langevin_girsanov_V_norm}
  \tb_{\Time/m, n}(t, (\rmw_s)_{s \in \ccint{0,\Time}}) = \sum_{i=0}^{m-1}b_{\Time/m, n}(\rmw_{i\Time/n}) \1_{\coint{i \Time/m, (i+1)\Time/m}}(t) \eqsp.
\end{equation}
Note for any $i \in \lbrace 0, \dots, m \rbrace$, the distribution of $\bbfX_{i \Time/ m}$ is $\updelta_x \tRker_{\Time/m,n}^{i}$. Using that $b$ and $b_{T/m, n}$ are continuous and that $(\bfX_t)_{t\in \ccint{0,\Time}}$ and $(\bbfX_t)_{t \in \ccint{0,\Time}}$ take their values in $\rmC(\ccint{0,\Time}, \rset^d)$, we obtain that
\begin{equation}
  \begin{aligned}
    &\proba{\int_0^{\Time} \| b (\bfX_t) \|^2 \rmd t  < +\infty} = 1 \eqsp , \\
    &\proba{\int_0^{\Time} \| \tb_{\Time/m, n}(t, (\bbfX_s)_{s \in \ccint{0,T}}) \|^2 \rmd t  < +\infty } = 1 \eqsp ,
    \end{aligned}
\end{equation}
and
\begin{equation}
  \begin{aligned}
    &\proba{\int_0^{\Time} \| b(\bfB_t) \|^2 \rmd t  < +\infty} = 1 \eqsp , \\
    & \proba{\int_0^{\Time} \| \tb_{\Time/m, n}(t, (\bfB_s)_{s \in \ccint{0,\Time}}) \|^2 \rmd t  < +\infty } = 1 \eqsp , \end{aligned}
\end{equation}
where $(\bfB_t)_{t \in \ccint{0,\Time}}$ is the $d$-dimensional
Brownian motion associated with \eqref{eq:diff}.  Therefore by
\cite[Theorem 7.7]{liptser2013statistics} the distributions of
$(\bfX_t)_{t \in \ccint{0,\Time}}$ and
$(\bbfX_t)_{t \in \ccint{0,\Time}}$, denoted by $\mu^x$ and
$\tilde{\mu}^x$ respectively, are equivalent to the distribution of
the Brownian motion $\mu_B^x$ starting at $x$. In addition, $\mu^x$
admits a Radon-Nikodym density \wrt \ to $\mu_{B}^x$ and $\mu_{B}^x$
admits a Radon-Nikodym density \wrt \ to $\tilde{\mu}^x$, given
$\mu_B^x$-almost surely for any
$(\rmw_t)_{t \in \ccint{0,\Time}} \in \rmC(\ccint{0,\Time}, \rset^d)$
by
\begin{align}
  \label{eq:expo_density}
   \frac{\rmd \mu^x}{\rmd \mu_B^x}((\rmw_t)_{t \in \ccint{0,\Time}}) &= \exp \left( (1/2) \int_0^{\Time} \langle b(\rmw_s),  \rmd \rmw_s \rangle  - (1/4) \int_0^{\Time} \| b (\rmw_s) \|^2 \rmd s \right)\eqsp , \\
    \frac{\rmd \mu_B^x}{\rmd \tilde{\mu}^x}((\rmw_t)_{t \in \ccint{0,\Time}}) &= \exp \left( -(1/2) \int_0^{\Time} \langle \tb_{\Time/m, n}(s, (\rmw_u)_{u \in \ccint{0,\Time}}),  \rmd \rmw_s \rangle \right. \\ &\phantom{aaaaaaaaaaaa} \left.  + (1/4) \int_0^{\Time} \| \tb_{\Time/m, n} (s, (\rmw_u)_{u \in \ccint{0, \Time}}) \|^2 \rmd s \right)\eqsp .
\end{align}
Finally we obtain that $\mu_B^x$-almost surely for any $(\rmw_s)_{s \in \ccint{0,\Time}} \in \rmC(\ccint{0, \Time}, \rset^d)$ 
\begin{align}
  \frac{\rmd \mu^x}{\rmd \tilde{\mu}^x}((\rmw_t)_{t \in \ccint{0,\Time}}) &= \exp \left( (1/2) \int_0^{\Time} \langle   b(\rmw_s) - \tb_{\Time/m, n}(s, (\rmw_u)_{u \in \ccint{0, \Time}}),  \rmd \rmw_s \rangle \right . \\
\label{eq:density_exp_girsanov}
                                                                & \phantom{aaaaaa}\left. + (1/4) \int_0^{\Time} \| \tb_{\Time/m, n} (s, (\rmw_u)_{u \in \ccint{0, \Time}}) \|^2 - \| b(\rmw_s)  \|^2 \rmd s  \right) \eqsp .
\end{align}
Now define  for any $(\rmw_s)_{s \in \ccint{0,\Time}} \in \rmC(\ccint{0, \Time}, \rset^d)$ and $t \in \ccint{0,\Time}$
\begin{equation}
  \label{eq:def_bdisc}
\bdisc_{\Time/m}(t,(\rmw_s)_{s \in \ccint{0,\Time}}) = \sum_{i=0}^{m-1} b(\rmw_{i\Time/m}) \1_{\coint{i\Time/m,(i+1)\Time/m}} (t) \eqsp.
\end{equation}
Using \eqref{eq:diff}, \eqref{eq:def_bbar_langevin_girsanov_V_norm},
\eqref{eq:density_exp_girsanov}, \Cref{ass:diff_disc}, and for any
$a_1,a_1 \in \rset^d$,
$\norm[2]{a_1-a_2} \leq 2 (\norm[2]{a_1}+ \norm[2]{a_2})$, we obtain
that
\begin{align}
      \label{eq:kl_error_expli}
  &\quad  \quad 2\KL{\updelta_x \Pker_{\Time}}{\updelta_x \tRker_{\Time/m, n}^{m} }  
  \leq 2^{-1} \expe{  \int_0^{\Time} \| b(\bfX_s) - \tb_{\Time/m, n} (s, (\bfX_u)_{u \in \ccint{0, \Time}}) \|^2 \rmd s} \\
  &\leq   \expe{  \int_0^{\Time} \| b(\bfX_s) - \bdisc_{\Time/m}(s,(\bfX_u)_{u \in \ccint{0,\Time}}) \|^2 \rmd s }\\
  & \qquad +  \sum_{i=0}^{m-1}
\expe{  \int_{i\Time/m}^{(i+1)\Time/m} \| b(\bfX_{i\Time/m}) - b_{\Time/m,n}(\bfX_{i\Time/m}) \|^2 \rmd s } \\
  &\leq    \expe{  \int_0^{\Time} \| b(\bfX_s) - \bdisc_{\Time/m}(s,(\bfX_u)_{u \in \ccint{0,\Time}}) \|^2 \rmd s }  \\ & \qquad +  C_1 \Time^{1 + \beta} m^{-\beta}  \sup_{s \in \ccint{0,\Time}} \expe{\norm[2]{b(\bfX_s)}}  \eqsp .
\end{align}
It only remains to show that the first term goes to $0$ as $m \to \plusinfty$. Note that since $(\bfX_s)_{s \in \ccint{0,\Time}}$ is almost surely continuous and $b$ is continuous on $\rset^d$,  $\lim_{m \to \plusinfty} \| b(\bfX_s) - \bdisc_{\Time/m} (s, (\bfX_u)_{u \in \ccint{0,\Time}}) \|^2 = 0$ for any $s \in \ccint{0,\Time}$ almost surely. Then, using the Lebesgue dominated convergence theorem and the continuity of $b$, we obtain that for any $M \geq 0$,
\begin{equation}
  \lim_{m \to +\infty} \expe{ \1_{\ccint{0,M}}\parenthese{ \sup_{s \in \ccint{0,\Time}} \norm{\bfX_s}}  \int_0^{\Time} \| b(\bfX_s) - \bdisc_{\Time/m} (s, (\bfX_u)_{u \in \ccint{0, \Time}}) \|^2 \rmd s} = 0 \eqsp . \label{eq:limite_un}
\end{equation}
On the other hand, using Hölder's inequality and the definition of $\bdisc_{\Time/m}$ \eqref{eq:def_bdisc}, we obtain that for any $M \geq 0$,
\begin{align}
  &\expe{    \1_{\ooint{M,\plusinfty}}\parenthese{ \sup_{s \in \ccint{0,\Time}} \norm{\bfX_s}} \int_0^{\Time} \| b(\bfX_s) - \bdisc_{\Time/m} (s, (\bfX_s)_{s \in \ccint{0, \Time}}) \|^2 \rmd s}\\
  & \phantom{} \leq  
    2   \parenthese{\proba{\sup_{s \in \ccint{0,\Time}} \norm{\bfX_s} >M}}^{\vareps_b/(1+\vareps_b)}\\
  &  \phantom{} \int_0^{\Time} \parenthese{  \expeExpo{1/(1+\vareps_b)}{\norm[2(1+\vareps_b)]{ b(\bfX_s)}}  +\expeExpo{1/(1+\vareps_b)}{ \norm[2(1+\vareps_b)]{\bdisc_{\Time/m}(s,(\bfX_u)_{u \in \ccint{0,\Time}})}}} \rmd s \\
  &\phantom{} \leq 4\Time \parenthese{\proba{\sup_{s \in \ccint{0,\Time}} \norm{\bfX_s} >M}}^{\vareps_b/(1+\vareps_b)}\parenthese{  \sup_{s \in \ccint{0,\Time}} \expe{\norm{b (\bfX_s) }^{2(1+\vareps_b)}}}^{1/(1+\vareps_b)} \eqsp.
\end{align}
Combining this result, \Cref{assum:unif_integr}, and \eqref{eq:limite_un} in \eqref{eq:kl_error_expli}, we obtain that for any $M \geq 0$,
\begin{multline}
  \limsup_{m \to \plusinfty} \KL{\updelta_x \Pker_{\Time}}{\updelta_x \tRker_{\Time/m, n}^{m} }   \\ \leq 2 \Time \parenthese{\proba{\sup_{s \in \ccint{0,\Time}} \norm{\bfX_s} >M}}^{\vareps_b/(1+\vareps_b)}\parenthese{  \sup_{s \in \ccint{0,\Time}} \expe{\norm[2(1+\vareps_b)]{b(  \bfX_s)}}}^{1/(1+\vareps_b)}   \eqsp. 
\end{multline}
Since $(\bfX_s)_{s \in \ccint{0,\Time}}$ is \as \ continuous, we get
by the monotone convergence theorem and \Cref{assum:unif_integr},
taking $M \to \plusinfty$, that
$\lim_{m \to \plusinfty} \KLLigne{\updelta_x \Pker_{\Time
  }}{\updelta_x \tRker_{\Time/m, n}^{m} } = 0$, which concludes the
proof.

\subsection{Proof of \Cref{propo:compare}}
\label{propo:compare:proof}

  For any $n \in \nset$ and $ \gamma \in \ocint{0, \bgamma}$, we consider the synchronous Markov coupling $\Qker_{\gamma, n}$ for $\Rker_{\gamma, n}$ and $\tRker_{\gamma, n}$ defined for any $(x, y) \in \rset^d \times \rset^d$ and $\msa \in \mcb{\rset^d}$ by
 \begin{align}
    \label{eq:synchronous_coupling} &\Qker_{\gamma, n}((x,y), \msa) \\ &= \frac{1}{(2 \uppi \gamma)^{d/2}}\int_{\rset^d} \1_{(\Id, \Pi_{\msk_n})^{\inv}(\msa)}(\Tg(x) + \sqrt{\gamma} z, \Tg(y) + \sqrt{\gamma} z) \rme^{-\norm[2]{z}/2} \rmd z \eqsp , 
 \end{align}
 with $\Tg(x) = x + \gamma b(x)$. Let $\Time \geq 0$, $n \in \nset$,
 $m \in \nsets$ such that $T/m \leq \bgamma$.  Consider
 $(X_j,\tilde{X}_j)_{j \in \nset}$ a Markov chain with Markov kernel
 $\Qker_{\Time/m,n}$ and started from $X_0 = \tilde{X}_0 = x$ for a
 fixed $x \in \rset^d$. Note that by definition and
 \Cref{ass:seq_k_n}, we have that for $k < \tau$, $X_k = \tilde{X}_k$
 where
 $\tau = \inf\{j \in \nset \, : \, \tilde{X}_j \not \in
 \cball{0}{n}\}$. Using \Cref{ass:lyap_majo},
 $\parenthese{\tilde{V}(\tilde{X}_j)
   \exp\parentheseDeux{-j\log(\tilde{A})(\Time/m) (1 + \En
     (\Time/m)^{\varepsn})}}_{j \in \nset}$ is a positive
 supermartingale.  Combining \eqref{eq:synchronous_coupling}, the
 Cauchy-Schwarz inequality, \Cref{ass:lyap_majo} and the Doob maximal
 inequality for positive supermartingale \cite[Proposition
 II-2-7]{neveu1975discrete}
 , we get for any $x \in \rset^d$
  \begin{align}
 &   \Vnorm{\updelta_x \Rker_{\Time/m, n}^{m} - \updelta_x \tRker_{\Time/m, n}^{m} }  
 \leq \expe{\1_{\Delta_{\rset^d}^{\complementary}}(X_{m},\tilde{X}_{m}) (V(X_{m}) + V(\tilde{X}_{m}))/2} \\
&                                                                         \leq  (1/2) \proba{\sup_{j \in \{ 0, \dots, m\}} \norm{\tilde{X}_j} \geq n}\parenthese{\expe{V^2(X_{m})}^{1/2} + \expe{V^2(\tilde{X}_{m})}^{1/2}}  \\
    & \leq (1/2)
  \proba{\sup_{ j \in  \{ 0, \dots, m\}} \tilde{V}(\tilde{X}_j) \geq n}\parenthese{\expe{V^2(X_{m})}^{1/2} + \expe{V^2(\tilde{X}_{m})}^{1/2}} 
 \\
    & \leq (2n)^{-1} \exp\parentheseDeux{\log(\tilde{A})(\Time/m)(1 + \En (T/m)^{\vareps_n })} \tilde{V}(x) \\ & \qquad \qquad \qquad \times \parenthese{(\Rker_{\Time/m, n}^{m} V^2(x))^{1/2} + (\tRker_{\Time/m, n}^{m} V^2(x))^{1/2}}\eqsp ,
  \end{align}
which concludes the proof upon taking $m\to +\infty$ then $n\to +\infty$.

\subsection{Proof of \Cref{prop:L_ass_check_lip}}
\label{prop:L_ass_check_lip:proof}
Let $p \in \nsets$ and $V \in \rmC^2(\rset^d, \coint{1,+\infty})$ be
defined for any $x \in \rset^d$ by $V(x) = 1 + \norm{x}^{2p}$.  For
any $x \in \rset^d$, $\nabla V(x) = 2p \norm{x}^{2(p-1)}x$ and
$\Delta V(x) = (4p(p-1) + 2pd)\norm{x}^{2(p-1)}$. Therefore, using
\Cref{as:item:lip} and the definition of $\generator$ we obtain that
for any $x \in \rset^d$
  \begin{equation}
    \generator V(x) \leq \parentheseDeux{2p(p-1) + p(d +2 \Lip)}V(x) \eqsp . \label{eq:drift_0}
  \end{equation}
  Hence, using \eqref{eq:drift_0} and \cite[Theorem
  3.5]{khasminskii2011stochastic}, we obtain that \Cref{ass:loc_lip_b}
  holds.  Using that for any
  $\sup_{x \in \rset^{\dim}} \normLigne{b(x)} (1+\normLigne{x}^2)^{-1}
  < +\infty$, \eqref{eq:drift_0} and
  \Cref{lemma:disclyapfromc}-\ref{lemma:disclyapfromc_b} we obtain
  that \Cref{assum:unif_integr} holds.

\Cref{ass:diff_disc} and \Cref{ass:seq_k_n} are trivially satisfied.
Finally, using once again \Cref{as:item:lip}, we have  that for any  $x \in \rset^{\dim}$ and $\gamma \in \ocint{0, \bgamma}$ we have
  \begin{align}
    \Rker_{\gamma} (1 + \norm{x}^2) &\leq 1 + \norm{x + \gamma b(x)}^2 + \gamma d \\
                                    &\leq 1 + \norm{x}^2 + 2 \gamma \norm{b(x)} \norm{x} + \gamma^2 \norm{b(x)}^2 + \gamma d \\
                                    &\leq 1 + \norm{x}^2 + 2 \gamma \Lip \norm{x}^2  + \gamma^2 \Lip^2 \norm{x}^2 + \gamma d \\
    &\leq (1  +2 \gamma \Lip  + \gamma^2 \Lip^2 + \gamma d) (1 + \norm{x}^2) \eqsp ,
  \end{align}
which implies that \Cref{ass:lyap_majo} holds.

\subsection{Proof of \Cref{prop:existence_integr}}
\label{prop:existence_integr:proof}

  Let $p \in \nsets$ and $V \in \rmC^2(\rset^d, \coint{1,+\infty})$ be defined for any $x \in \rset^d$ by $V(x) = 1 + \norm{x}^{2p}$.
For any $x \in \rset^d$,  $\nabla V(x) = 2p \norm{x}^{2(p-1)}x$ and $\Delta V(x) = (4p(p-1) + 2pd)\norm{x}^{2(p-1)}$. Therefore, using \Cref{as:b_min}($\mtt$) and the definition of $\generator$
  we obtain that for any $x \in \rset^d$
  \begin{equation}
    \generator V(x) \leq \parentheseDeux{2p(p-1) + p(d - 2\mtt)}V(x) \eqsp . \label{eq:drift_1}
  \end{equation}
  Hence, using \eqref{eq:drift_1} and \cite[Theorem
  3.5]{khasminskii2011stochastic}, we obtain that \Cref{ass:loc_lip_b}
  holds.
  \begin{enumerate}[label= (\alph*),  wide, labelwidth=!, labelindent=0pt]    
  \item If there exists $\vareps_b >0$ such that $\sup_{x \in \rset^d} \norm{b(x)}^{2(1+\vareps_b)}(1 + \norm{x}^{2p})^{-1} < +\infty$, using \eqref{eq:drift_1} and \Cref{lemma:disclyapfromc}-\ref{lemma:disclyapfromc_b} we obtain that \Cref{assum:unif_integr} holds.

  \item If there exists $\vareps_b >0$ such that
    $\sup_{x \in \rset^d} \norm{b(x)}^{2(1 +
      \vareps_b)}\rme^{-\mttplusdeux\norm{x}^2} < +\infty$, and
    \Cref{assum:drift_strong} holds, then consider for any
    $x \in \rset^d$, $V(x) = \rme^{\mttplusdeux \norm{x}^2}$. We have
    for any $x \in \rset^d$,
    $\nabla V(x) = 2 \mttplusdeux \rme^{\mttplusdeux \norm{x}^2} x$
    and
    $\Delta V(x) = 4 \mtt_2^{+2} \rme^{\mttplusdeux \norm{x}^2}
    \norm{x}^2 + 2 \mttplusdeux \rme^{\mttplusdeux \norm{x}^2} d$.
    Therefore, using \Cref{assum:drift_strong} we have for
    any $x \in \cball{0}{\Rdeux}^{\complementary}$
    \begin{equation}
      \generator V(x) \leq \mttplusdeux\parentheseDeux{d + (4\mttplusdeux/2 -2\mttplusdeux)\norm{x}^2}V(x) \leq \mttplusdeux d V(x) \eqsp . \label{eq:drift_2}
    \end{equation}
    Setting
    $\zeta = (\mttplusdeux d) \vee \sup_{x \in \cball{0}{\Rdeux}}
    \generator V(x)$, we obtain that $V$ satisfies
    \hyperlink{ass:drift_continuous}{$\bfDc(V,\zeta,0)$}. Therefore
    using \eqref{eq:drift_2} and
    \Cref{lemma:disclyapfromc}-\ref{lemma:disclyapfromc_b}, we obtain
    that \Cref{assum:unif_integr} holds.
  \end{enumerate}

\subsection{Proof of \Cref{prop:fam_prop}}
\label{prop:fam_prop:proof}

We preface the proof by a preliminary computation. Let $n \in \nset$,
$\gamma \in \ocint{0, \bgamma}$, $x \in \rset^d$ and
$X = x + \gamma \rmb_{\gamma, n}(x) + \sqrt{\gamma} Z$, where $Z$ is a
$d$-dimensional Gaussian random variable with zero mean and covariance
identity. We have using \Cref{as:b_min}($\mtt$) and \eqref{eq:def_fam}
  \begin{equation}
    \expe{\norm{X}^2} \leq  \|x \|^2 - 2 \gamma  \mtt \, \Phi_n(x) \norm{x}^2 + \gamma^2 \Phi_n(x)^2 \norm{b(x)}^2 +  \gamma d  \eqsp , \label{eq:numero_uno}
  \end{equation}
  with $\Phi_n(x) = \varphi_n(x) + (1 - \varphi_n(x))(1 + \gamma^{\alpha}\norm{b(x)})^{-1}$. We recall that
\begin{equation}
  \label{eq:varphi_condition}
 \varphi_n(x) \in \ccint{0,1} \quad \text{ and } \quad \varphi_n(x)
  =
  \begin{cases}
    1 & \text{ if $x \in \cball{0}{n}$}, \\
    0 & \text{ if $x \in \cball{0}{n+1}^{\complementary}$ } \eqsp.
  \end{cases}
\end{equation}
Using \Cref{ass:loc_lip_bb} and \eqref{eq:varphi_condition}, we have
  \begin{equation}
    \Phi_n(x) \norm{b(x)} \leq \Lip_{n+1}\norm{x}  + \gamma^{-\alpha} \eqsp . \label{eq:numero_duo}
  \end{equation}
Combining \eqref{eq:numero_uno} and \eqref{eq:numero_duo} and  since $\Phi_n(x) \leq 1$ by \eqref{eq:varphi_condition}, we obtain
\begin{equation}
  \expe{1 + \norm{X}^2} \leq (1 + \norm{x}^2)\parentheseDeux{1 +  2 \gamma \abs{\mtt} + 2 \gamma^2 \Lip_{n+1}^2} + 2 \gamma^{2 - 2\alpha} + \gamma d 
                          \eqsp. \label{eq:control_norm_2}
                        \end{equation}
                        We are now able to complete the proof of
                        \Cref{prop:fam_prop}.  It is easy to check
                        that \Cref{ass:diff_disc} holds with
                        $\beta = 2\alpha$. It only remains to show
                        that \Cref{ass:lyap_majo} holds. Consider for
                        any $x \in \rset^d$,
                        $\tilde{V}(x) = 1+\norm{x}$. By
                        \eqref{eq:control_norm_2}, for any
                        $\gamma \in\ocint{0,\bgamma}$, $n\in \nset$
                        and $x \in \rset^d$, we have using for any
                        $s \geq \rset$, $1+s \leq \rme^{s}$ we obtain
\begin{align}
        \Rker_{\gamma,n} \tilde{V}(x)
  & \leq \tilde{V}(x)\parentheseDeux{1 +  2 \gamma \abs{\mtt} + 2 \gamma^2 \Lip_{n+1}^2 + 2 \gamma^{2 - 2\alpha} + \gamma d} \\
  & \leq \tilde{V}(x) \exp\parentheseDeux{\gamma\defEns{2 \abs{\mtt} + d + 2 \gamma^{1-2\alpha}(\gamma^{2\alpha} \Lip_{n+1} + 1)}} \\
  & \leq \tilde{V}(x) \exp\parentheseDeux{2\gamma\defEns{2 \abs{\mtt} + d} \defEns{1 +  \gamma^{1-2\alpha}(\gamma^{2\alpha} \Lip_{n+1} + 1)}}
                          \eqsp. 
\end{align}
As a result using that $d \geq 1$, \Cref{ass:lyap_majo} holds upon taking $\tilde{A} = \exp(4 \abs{\mtt} + 2d)$,
$\varepsn = 1 - 2\alpha$ and $\En = 2(\Lip_{n+1} \bgamma^{2\alpha} + 1)$.

\subsection{Proof of \Cref{prop:integrability_R_lip}}
\label{sec:prop:integrability_R_lip:proof}

The proof is similar to the one of \Cref{prop:integrability_R_loc_lip}
upon replacing \eqref{eq:control_norm_2} by
 \begin{equation}
   \expe{1 + \norm{X}^2} \leq (1 + \norm{x}^2)(1 + 2 \gamma \Lip + 2 \gamma^2 \Lip^2) + \gamma d \eqsp .
 \end{equation}

\subsection{Proof of \Cref{prop:integrability_R_loc_lip}}
\label{sec:prop:integrability_R_loc_lip:proof}

Let $M \geq 0$, $n \in \nset$ and $\pow \geq 1$. Using the Log-Sobolev inequality \cite[Theorem 5.5]{boucheron2013concentration}, the fact that $\phi$ is $1$-Lipschitz and that $\Pi_{\cball{0}{n}}$ is non expansive, as well as the Jensen inequality we obtain for any $\gamma \in \ocint{0, \bgamma}$ and $x \in \rset^d$,
     \begin{align}
     \Rker_{\gamma,n} V_M^{\pow}(x) &\leq \exp \parentheseDeux{ \pow M\tRker_{\gamma,n}\phi(x) + (\pow M)^2\gamma/2} \\ &\leq \exp \parentheseDeux{ \pow M \sqrt{ \tRker_{\gamma,n}\phi^2(x) } + (\pow M)^2 \gamma/2 } \eqsp .\label{eq:log_sob}
   \end{align}
   Using \eqref{eq:control_norm_2} and that $\sqrt{1 + t}\leq 1 +t/2$ for any $t \in \ooint{-1, +\infty}$ we get for any $\gamma \in \ocint{0, \bgamma}$ and $x \in \cball{0}{n}$
     \begin{align}
       &\Rker_{\gamma,n} V_M^{\pow}(x) \\  &\leq \exp \parentheseDeux{\pow M \defEns{\phi(x)^2(1+ 2\gamma \abs{\mtt} + 2 \gamma^2 \Lip_{n+1}^2) + 2 \gamma^{2 - 2\alpha} + \gamma d }^{1/2} + (\pow M)^2 \gamma /2}  \\
                             &\leq \exp \parentheseDeux{(1 + \gamma \abs{\mtt} + \gamma^2 \Lip_{n+1}^2) \pow M \phi(x)} \exp\parentheseDeux{(1 + \pow M)^2 \defEns{\gamma (d+1)/2 + \gamma^{2 - 2a}}} \\
       &\leq V_M^{\pow(1 + C_1 \gamma+ C_{2,n} \gamma^2)}(x) \exp\parentheseDeux{\pow^2 C_3 \gamma} \eqsp,
     \end{align}
     with $C_1 = \abs{\mtt}$, $C_{2,n} = \Lip_{n+1}^2$ and $C_3 = (1+M)^2 (d + 3) /2$. By recursion, we obtain that for any $m, n \in \nset$ with $m^{-1} \in \ocint{0, \bgamma}$, $\Time \geq 0$ and $x\in \cball{0}{n}$
     \begin{align}
       &\Rker_{\Time/m,n}^m V_M(x) \\ & \leq V_M(x)^{a_m} \exp\parentheseDeux{\Time C_3\sum_{j=0}^{m-1} (1 + \Time C_1/m+C_{2,n}(\Time/m)^2)^{2j}/m} \\
       &\leq V_M(x)^{a_m} \exp\parentheseDeux{\Time C_3(1 + \Time C_1/m+ C_{2,n}(\Time/m)^2)^{2m}} \eqsp ,
     \end{align}
     with $a_m = (1 + \Time C_1/m +  C_{2,n}(\Time/m)^2)^{m}$.
     Since $\lim_{m \to +\infty} (1 + \Time C_1/m + C_{2,n}(\Time/m)^2)^{tm} = \exp(t\Time C_1)$ for any $t, \Time \geq 0$, we get that for any $n \in \nset$, $\Time \geq 0$ and $x \in \cball{0}{n}$
     \begin{equation}
       \limsup_{m \to +\infty} \Rker_{\Time/m,n}^{m}V_M(x) \leq \exp(\Time C_3\exp(2\Time C_1)) V_M^{\exp(\Time C_1)}(x) \label{eq:majo_grossiere} \eqsp .
     \end{equation}
We conclude the proof upon remarking that the right-hand side quantity in \eqref{eq:majo_grossiere} does not depend on $n$ and that the same inequality holds replacing $\Rker_{\Time/m,n}$ by $\tRker_{\Time/m,n}$ in \eqref{eq:majo_grossiere}.

\subsection{Proof of \Cref{propo:majo_cont}}
\label{propo:majo_cont:proof}
  We have for any $x \in \rset^d$,
  \begin{equation}
    \nabla \phi(x) = x /\phi(x) \eqsp , \qquad \nabla^2 \phi(x) = \Id/\phi(x) - x x^{\transpose} / \phi^2(x) \eqsp,
  \end{equation}
and therefore since $V_M(x) = \exp(M\phi(x))$, 
  \begin{equation}
    \begin{aligned} &\nabla V_M(x) = M \nabla \phi(x) V_M(x) \eqsp, \\
      &\nabla^2 V_M(x) = \defEns{M^2 \nabla \phi(x) (\nabla \phi(x))^{\transpose} + M \nabla^2 \phi(x)} V_M(x) \eqsp .
      \end{aligned}
  \end{equation}
  Therefore, for any $x \in \rset^d$,
  \begin{multline}
    (\generator V_M(x))/V_M(x) \\ \leq \left. \parentheseDeux{M^2 \norm{x}^2/\phi^2(x) + M\defEns{d / \phi(x) - \norm{x}^2/\phi^2(x)}}\middle/2 \right. + M \sup_{x \in \rset^d} \langle b(x), x \rangle_+ \eqsp .
  \end{multline}
  Hence, for any $x \in \rset^d$,
  $\generator V_M(x) \leq \zeta V_M(x)$ with
  $\zeta = M\{\sup_{x \in \rset^d} \langle b(x), x \rangle_+ +d/2\} +
  M^2$. We conclude using
  \Cref{lemma:disclyapfromc}-\ref{lemma:disclyapfromc_a} and the Doob
  maximal inequality.


\def\trho{\tilde{\rho}}
\def\tlambda{\tilde{\lambda}}
\def\Pker{\mathrm{P}}
\def\distYl{\mathbf{d}_{\wass}}
\def\X{\mathrm{X}}
\def\Y{\mathrm{Y}}
\def\rmZ{\mathrm{Z}}
\def\T{T}
\def\tcmf{\tilde{\mcf}}
\def\Lambdabf{\boldsymbol{\Lambda}}
\def\lyapU{\lyap_1}
\def\lyapD{\lyap_2}
\def\ttn{\mathtt{n}}
\def\ntt{\mathtt{n}_0}

\section{Quantitative bounds for geometric convergence of Markov chains in Wasserstein distance}
\label{sec:quant-bounds-geom}

In this section, we establish new quantitative bounds for Markov
chains in Wasserstein distance.  We consider a Markov
kernel $\Pker$ on the measurable space $(\msy,\mcy)$ equipped with the
bounded semi-metric $\distY: \msy \times \msy \to \rset_+$, \ie~which
satisfies the following condition.
\begin{assumptionH}
  \label{ass:metric}
For any $x, y \in \msy$, $\distY(x,y) \leq 1$, $\distY(x,y) = \distY(y,x)$ and $\distY(x,y) = 0$ if and only if $x=y$. 
\end{assumptionH}
Let $\Kcoupling$ be a Markov coupling kernel for $\Pker$. In this section, we assume the following condition on $\Kker$.
 \begin{assumptionH}[$\Kcoupling$]
   \label{ass:kernel_coupling}
   There exists $\msc \in \mcy^{\otimes 2}$ such that 
\begin{enumerate}[label=(\roman*)]
\item \label{ass:kernel_coupling_a} there exist $\ntt \in \nsets$ and $\varepsilon >0$ such that for any $x,y \in \msc$, $\Kker^{\ntt}\distY(x,y) \leq (1-\varepsilon) \distY(x,y)$ ;
\item \label{ass:kernel_coupling_b} for any $x,y \in \msy $, $\Kker\distY(x,y) \leq \distY(x,y)$ ;
\item \label{ass:kernel_coupling_c} there exist
  $\lyapU : \msy^2 \to \coint{1,\plusinfty}$ measurable,
  $\lambda_1 \in \ooint{0,1}$ and $A_1 \geq 0$ such that $\Kker$ satisfies
  \hyperlink{ass:drift_discrete}{$\bfDd(\lyapU,\lambda_1,A_1, \msc)$}.
\end{enumerate}
\end{assumptionH}

We consider the Markov chain $(\X_n,\Y_n)_{n\in\nset}$ associated with the Markov
kernel $\Kker$ defined on the canonical space
$((\msy\times \msy)^{\nset},(\mcy^{\otimes 2})^{\nset})$ and denote
by $\probaMarkovDD{(x,y)}$ and $\expeMarkovDD{(x,y)}$ the
corresponding probability distribution and expectation respectively
when $(\X_0,\Y_0) = (x,y)$. Denote by $(\mcg_n)_{n\in\nset}$ the
canonical filtration associated with $(\X_n,\Y_n)_{n\in\nset}$.  Note that for any
$n \in \nset$ and $x,y \in \msy$, under $\probaMarkovDD{(x,y)}$,
$(\X_n,\Y_n)$ is by definition a coupling of $\updelta_x \Pker^n$ and
$\updelta_y \Pker^n$. The main result of this section is the following
which by the previous observation implies quantitative bounds on
$\wassersteinD[\distY](\updelta_x \Pker^n, \updelta_y \Pker^n)$.
\begin{theorem}
  \label{theo:quanti_v_alain}
  Let $\Kcoupling$ be a Markov coupling kernel for $\Pker$ and assume \tup{\Cref{ass:metric}} and \tup{\Cref{ass:kernel_coupling}($\Kcoupling$)}. Then for any $n \in \nset$ and $x,y \in \msy$,
  \begin{multline}
    \label{theo:quanti_v_alain_1}
    \expeMarkov{(x,y)}{\distY(\X_n,\Y_n)} 
    \leq \min\parentheseDeux{\rho^{n} ( M_{\msc,\ntt}\Xibf(x,y,\ntt) + \distY(x,y)) ,\rho^{n/2} (1+\distY(x,y))  +  \lambda_1^{n/2} \Xibf(x,y,\ntt)}  \eqsp,
  \end{multline}
  where
  \begin{equation}
    \label{eq:def_const_thm3}
  \begin{aligned}
    \Xibf(x,y, \ntt) &=   \lyapU(x,y) + A_1 \lambda_1^{-\ntt}\ntt  \\
    \log(\rho) &= \left. \log(1-\vareps) \log(\lambda_1) \middle/ \parentheseDeux{-\log(M_{\msc,\ntt}) + \log(1-\vareps)} \right. \eqsp,\\
    M_{\msc, \ntt} &= \sup_{(x,y) \in \msc} \Xibf(x,y, \ntt) = \sup_{(x,y) \in \msc}  \parentheseDeux{ \lyapU(x,y)} +A_1\lambda_1^{-\ntt}\ntt \eqsp.  
  \end{aligned}
  \end{equation}
\end{theorem}
In \Cref{theo:quanti_v_alain}, we obtain geometric contraction for $\Pker$ in bounded Wasserstein metric $\wassersteinD[\distY]$ since $\distY$ is assumed to be bounded. To obtain convergence associated with unbounded Wasserstein metric associated with  $\lyapD : \msy^2 \to \coint{0,+\infty}$, we consider the next assumption which is a generalized drift condition linking $\lyapD$ and the bounded semi-metric $\distY$. 
\begin{assumptionH}[$\Kcoupling$]
  \label{assum:drift_d}
  There exist $\lyapD : \msy^2 \to \coint{0,\plusinfty}$ measurable, $\lambda_2 \in \ooint{0,1}$ and $A_2 \geq 0$ such that for any $x, y \in \msy$, 
  \begin{equation}
    \Kcoupling \lyapD(x,y) \leq \lambda_2 \lyapD(x,y) + A_2 \distY(x,y) \eqsp .
  \end{equation}
\end{assumptionH}
In the special case where $\distance(x,y)  =\1_{\Delta_{\msy}^{\complementary}}(x,y)$, $\lyapD(x,y) =  \1_{\Delta_{\msy}^{\complementary}}(x,y)\lyapU(x,y) $ and for any $x \in \msy$, $\Kcoupling((x,x), \Delta_{\msy}) = 1$, we obtain that \hyperlink{ass:drift_discrete}{$\bfDd(\lyapU,\lambda_1,A_1, \msy)$} implies \Cref{assum:drift_d}($\Kcoupling$). The following result implies quantitative bounds on the Wasserstein distance $\distVdeux(\updelta_x \Pker^n, \updelta_y \Pker^n)$ for any $x,y\in \msy$ and $n \in\nsets$.

\begin{theorem}
  \label{theo:quanti_v_alain_v_norm}
  Let $\Kcoupling$ be a Markov coupling kernel for $\Pker$ and  assume \tup{\Cref{ass:metric}}, \tup{\Cref{ass:kernel_coupling}($\Kcoupling$)} and \tup{\Cref{assum:drift_d}($\Kcoupling$)}. Then for any $n \in \nset$ and $x,y \in \msy$,
  \begin{multline}
    \label{theo:quanti_v_alain_2}
    \expeMarkov{(x,y)}{ \lyapD(\X_n,\Y_n)} \leq  \lambda_2^n \lyapD(x,y) \\ +
    A_2  \min \parentheseDeux{ \trho^{n/4} r_{\rho} (\distY(x,y) + \Xibf(x,y, \ntt)), \trho^{n/4} r_{\rho} (1+\distY(x,y))  +  \tlambda^{n/4} r_{\lambda} \Xibf(x,y,\ntt)}  \eqsp ,
  \end{multline}
where
\begin{equation}
  \label{eq:def_const_thm15}
    \begin{aligned}
      \trho &= \max(\lambda_2, \rho) \in \ooint{0,1} \eqsp, \
      \tlambda &= \max(\lambda_{1}, \lambda_{2}) \in \ooint{0,1}
      \eqsp, \ r_{\rho}& = 4\log^{-1}(1/\trho)/\trho \eqsp ,& \
      r_{\lambda} & = 4\log^{-1}(1/\tlambda)/\tlambda \eqsp ,
    \end{aligned}
  \end{equation}
and   $\Xibf(x,y,\ntt)$, $ M_{\msc,\ntt}$ and $\rho$ are given in \eqref{eq:def_const_thm3}.
  \end{theorem}
\Cref{theo:quanti_v_alain} and \Cref{theo:quanti_v_alain_v_norm} share
some connections with \cite[Theorem 5]{rosenthal:1995},
\cite{hairer2011asymptotic} and \cite{durmus2015quantitative} but hold
under milder assumptions than the ones considered in these
works. Compared to \cite{hairer2011asymptotic} and
\cite{durmus2015quantitative}, the main difference is that we allow
here only a contraction for the $\ntt$-th iterate of the Markov chain
(condition \Cref{ass:kernel_coupling}-\ref{ass:kernel_coupling_a})
which is necessary if we want to use \Cref{theo:minorization_general}
to obtain sharp quantitative convergence bounds for
\eqref{eq:langevin_discrete}. Finally, \cite[Theorem
5]{rosenthal:1995} also considers minorization condition for the the
$\ntt$-th iterate, however our results compared favourably for large
$\ntt$. Indeed, \Cref{theo:quanti_v_alain} implies that the rate of
convergence $\min(\rho,\lambda_1)$ is of the form $C \ntt^{-1}$ for
$C \geq 0$ independent of $\ntt$. Applying \cite[Theorem
5]{rosenthal:1995}, we found a rate of convergence of the form
$C \ntt^{-2}$. Finally, a recent work \cite{qin:hobert:2019} has
established new results based on the technique used in
\cite{hairer2011asymptotic}. However, we were not able to apply them since they  assume  as in \cite{hairer2011asymptotic}, a contraction for $\ntt=1$ which does not imply sharp bounds on the situations we consider. 

The rest of this section is devoted to the proof of \Cref{theo:quanti_v_alain} and \Cref{theo:quanti_v_alain_v_norm}. Denote by $\theta : (\msy \times \msy)^{\nset} \to (\msy \times \msy)^{\nset}$ the shift operator defined for any $(x_n,y_n)_{n\in\nset} \in (\msy \times \msy)^{\nset}$ by $\theta((x_n,y_n)_{n\in\nset}) = (x_{n+1},y_{n+1})_{n \in \nset}$.  Define by induction, for any $m \in \nset$, the sequence of $(\mcg_n)_{n \in\nset}$-stopping times $(T_{\msc, \ntt}^{(m)})_{m \in \nset}$, with $  \T_{\msc,\ntt}^{(0)}  = 0$ and for any $m \in \nsets$
\begin{equation}
  \label{eq:def_stopping}
  \begin{aligned}
    \T_{\msc,\ntt}^{(m)} &= \inf\ensemble{ k \geq   \T_{\msc,\ntt}^{(m-1)} + \ntt}{(\X_k,\Y_k) \in \msc}\\ &= \T_{\msc,\ntt}^{(m-1)} + \ntt + \tT_{\msc} \circ \theta^{\T_{\msc}^{(m-1)} + \ntt}  
     = \T_{\msc,\ntt}^{(1)} + \sum_{i=1}^{(m-1)} \T_{\msc,\ntt}^{(1)}  \circ \theta^{\T_{\msc,\ntt}^{(i)}} \eqsp,\\
\tT_{\msc}& = \inf\ensemble{ k \geq  0}{(\X_k,\Y_k) \in \msc}\eqsp.
\end{aligned}
\end{equation}
Note that $( \T_{\msc,\ntt}^{(m)})_{m \in \nsets}$ are the
successive return times to $\msc$ delayed by $\ntt-1$ and
$\tT_{\msc}$ is the first hitting time to $\msc$.
We will use the following lemma which borrows from \cite{durmus2016subgeometric} and \cite[Lemma 3.1]{jarner2001locally}.
\begin{lemma}[\protect{\cite[Proposition 14]{durmus2016subgeometric}}]
  \label{propo:adapt_ahip}
  Let $\Kcoupling$ be a Markov coupling kernel for $\Pker$ and  assume  
  \tup{\Cref{ass:kernel_coupling}($\Kcoupling$)-\ref{ass:kernel_coupling_a}-\ref{ass:kernel_coupling_b}}. Then for any $n,m\in \nset$, $x,y \in \msy$,
  \begin{equation}
    \expeMarkov{(x,y)}{\distY(\X_n,\Y_n)} \leq (1-\vareps)^{m}\distY(x,y)  +   \expeMarkov{(x,y)}{\distY(\X_n,\Y_n) \1_{\ccint{n,\plusinfty}}(\T_{\msc,\ntt}^{(m)})} \eqsp.
  \end{equation}
\end{lemma}

\begin{proof}
  Using \Cref{ass:kernel_coupling}($\Kcoupling$)-\ref{ass:kernel_coupling_b}, we have that $(\distY(X_n, Y_n))_{n \in \nset}$ is a $(\mcg_n)_{n \in \nset}$-supermartingale and therefore using the strong Markov property and \Cref{ass:kernel_coupling}($\Kcoupling$)-\ref{ass:kernel_coupling_a} we obtain for any $m \in \nset$ and $x,y \in \msy$ that
  \begin{align}    
    \expeMarkov{(x,y)}{\distY(\X_{\stopping{m+1}},\Y_{\stopping{m+1}})} &\leq \expeMarkov{(x,y)}{\CPE{\distY(\X_{\stopping{m} + \ntt },\Y_{\stopping{m} + \ntt})}{\mcg_{\stopping{m}}}}  \\ & \leq (1 - \vareps) \expeMarkov{(x,y)}{\distY(\X_{\stopping{m}},\Y_{\stopping{m}})} \eqsp .  \label{eq:recursion_step}
  \end{align}
Therefore by recursion and using \eqref{eq:recursion_step} we obtain that for any $m \in \nset$ and $x,y \in \msy$
\begin{equation}
  \expeMarkov{(x,y)}{\distY(\X_{\stopping{m}},\Y_{\stopping{m}})} \leq (1 - \vareps)^m \distY(x,y) \eqsp . \label{eq:ineq_stopping}
\end{equation}
For any $n, m \in \nset$ we have using \eqref{eq:ineq_stopping} and  that $(\distY(X_n, Y_n))_{n \in \nset}$ is a supermartingale, 
\begin{align}
  &\expeMarkov{(x,y)}{\distY(\X_n, \Y_n)} \leq \expeMarkov{(x,y)}{\distY(\X_{n \wedge \stopping{m}}, \Y_{n \wedge \stopping{m}})} \\ &\leq \expeMarkov{(x,y)}{\distY(\X_{\stopping{m}}, \Y_{\stopping{m}}) \1_{\ccint{0,n}}(\stopping{m})} + \expeMarkov{(x,y)}{\distY(\X_n,\Y_n) \1_{\ccint{n,\plusinfty}}(\T_{\msc,\ntt}^{(m)})} \\
  &\leq (1-\vareps)^{m}\distY(x,y)  +   \expeMarkov{(x,y)}{\distY(\X_n,\Y_n) \1_{\ccint{n,\plusinfty}}(\T_{\msc,\ntt}^{(m)})} \eqsp.
\end{align}
\end{proof}

By \Cref{propo:adapt_ahip} and  since $\distY$ is bounded by $1$, we need to
obtain a bound on $\PP_{(x,y)}(\stopping{m} \geq n)$ for
$x,y \in \msy$ and $m,n \in \nsets$. To this end, we will use  the following proposition which gives an
upper bound on exponential moment of the hitting times
$(\stopping{m})_{m \in \nsets}$.

\begin{lemma}
  \label{lem:bound_T_m}
  Let $\Kcoupling$ be a Markov coupling kernel for $\Pker$ and  assume \tup{\Cref{ass:kernel_coupling}($\Kcoupling$)-\ref{ass:kernel_coupling_c}}. Then for any $x,y \in \msy$ and $m \in \nsets$,
  \begin{equation}
    \label{eq:bound_T_1_T_m}
    \begin{aligned}
      &\expeMarkov{(x,y)}{\lambda_1^{-T_{\msc,\ntt}^{(1)}} } \leq
      \Xibf(x,y, \ntt) \eqsp, \
      \expeMarkov{(x,y)}{\lambda_1^{-T_{\msc,\ntt}^{(m)}+T_{\msc,\ntt}^{(1)}}
      } \leq M_{\msc,\ntt}^{m-1}\eqsp, \
      \expeMarkov{(x,y)}{\lambda_1^{-T_{\msc,\ntt}^{(m)}} } \leq
      \Xibf(x,y, \ntt) M_{\msc,\ntt}^{m-1} \eqsp,
    \end{aligned}
  \end{equation}
  where $\Xibf(x,y, \ntt)$ and $M_{\msc,\ntt}$ are defined in \eqref{eq:def_const_thm3}.
\end{lemma}

\begin{proof}
  We first show that for any $x,y \in \msy$ we have that $\PP_{(x,y)}(\tT_{\msc}) < +\infty$. Let $x,y \in \msy$.  For any $n \in \nset$ we have using \Cref{ass:kernel_coupling}($\Kcoupling$)-\ref{ass:kernel_coupling_c} and the Markov property
  \begin{equation}
    \CPE[(x,y)]{\lyapU(\X_{n+1}, \Y_{n+1})}{\mcg_n} \leq \lambda_1 \lyapU(\X_{n}, \Y_{n}) + A_1 \1_{\msc}(\X_n, \Y_n)\eqsp .
  \end{equation}
  Therefore applying the comparison theorem \cite[Theorem 4.3.1]{douc:moulines:priouret:soulier:2018} we get that
  \begin{equation}
    (1 - \lambda_1) \expeMarkov{(x,y)}{\sum_{k=0}^{\tT_{\msc}-1} \lyapU(\X_k, \Y_k)} + \expeMarkov{(x,y)}{\1_{\coint{0,+\infty}}(\tT_{\msc})\lyap(\X_{\tT_{\msc}}, \Y_{\tT_{\msc}})} \leq \lyap(x,y) \eqsp .
  \end{equation}
  Since for any $\tilde{x},\tilde{y} \in \msy$, $1 \leq \lyapU(\tilde{x},\tilde{y}) < +\infty$ we obtain that $(1-\lambda_1)\expeMarkovLigne{(x,y)}{\tT_{\msc}} \leq \lyap(x,y)$ which implies $\PP_{(x,y)}(\tT_{\msc}) < +\infty$ since $\lambda_1 \in \ooint{0,1}$.
We now show the stated result. Let $x,y \in \msy$ and $(\rmS_n)_{n \in \nset}$ be defined for any $n \in \nset$ by $\rmS_n = \lambda_1^{-n}\lyapU(\X_n, \Y_n)$.
  For any $n \in \nset$ we have using \Cref{ass:kernel_coupling}($\Kcoupling$)-\ref{ass:kernel_coupling_c} and the Markov property
    \begin{align}
      \CPE{\rmS_{n+1}}{\mcg_n} &\leq \lambda_1^{-n} \lyapU(\X_n, \Y_n) + A_1\lambda_1^{-(n+1)} \1_{\msc}(\X_n, \Y_n)  \\ &\leq \rmS_n + A_1 \lambda_1^{-(n+1)} \1_{\msc}(\X_n, \Y_n) \eqsp. \label{eq:ineq_compa}
    \end{align}    
    Using the Markov property,  the definition of $\stopping{1}$ given in  \eqref{eq:def_stopping}, the comparison theorem \cite[Theorem 4.3.1]{douc:moulines:priouret:soulier:2018}, \eqref{eq:ineq_compa} and \Cref{ass:kernel_coupling}($\Kcoupling$)-\ref{ass:kernel_coupling_c}  we obtain that
    \begin{align}
      &\expeMarkov{(x,y)}{\rmS_{\stopping{1}}} =  \expeMarkov{(x,y)}{\CPE[(x,y)]{\rmS_{\stopping{1}}}{\mcg_{\ntt}}} = \expeMarkov{(x,y)}{\CPE[(x,y)]{\rmS_{\ntt + \tT_{\msc} \circ \theta^{\ntt}}}{\mcg_{\ntt}}}  \\
      &= \expeMarkov{(x,y)}{\lambda_1^{-\ntt} \,  \, \CPE[(x,y)]{\lyapU(X_{\ntt + \stopping{1}}, Y_{\ntt + \stopping{1}})\lambda_1^{-\stopping{1}}}{\mcg_{\ntt}}} \\ &\leq \expeMarkov{(x,y)}{\lambda_1^{-\ntt} \, \, \expeMarkov{(X_{\ntt}, Y_{\ntt})}{\rmS_{\tT_{\msc}}}} \\
      &\leq \expeMarkov{(x,y)}{\lambda_1^{-\ntt} \, \, \expeMarkov{(X_{\ntt}, Y_{\ntt})}{\rmS_{\tT_{\msc}}\1_{\coint{0,+\infty}}(\tT_{\msc})}} \\ 
      &\leq  \expeMarkov{(x,y)}{\lambda_1^{-\ntt} \,  \,   \expeMarkov{(X_{\ntt}, Y_{\ntt})}{\rmS_0 + A_1 \sum_{k=0}^{\tT_{\msc} - 1} \lambda_1^{-(k+1)}\1_{\msc}(X_k, Y_k) } } \\
      &\leq \expeMarkov{(x,y)}{\lambda_1^{-\ntt}\lyapU(X_{\ntt}, Y_{\ntt}) } \leq \lyapU(x,y) + A_1 \lambda_1^{-\ntt} \ntt  \eqsp . \label{eq:comparison_theorem}
    \end{align}
Combining \eqref{eq:comparison_theorem} and the fact that for any $x,y \in \msy$, $\lyapU(x,y) \geq 1$, we obtain that 
\begin{equation}
  \expeMarkov{(x,y)}{\lambda_1^{-\stopping{1}}} \leq  \lyapU(x,y) + A_1 \lambda_1^{-\ntt} \ntt  \eqsp . \label{eq:expo_t1}
\end{equation}
We conclude by a straightforward recursion and using \eqref{eq:expo_t1}, the definition of $\stopping{m}$ \eqref{eq:def_stopping} for $m \geq 1$, the strong Markov property and the fact that for any $m \in \nsets$, $(X_{\stopping{m}}, Y_{\stopping{m}}) \in \msc$.
\end{proof}
\begin{proof}[Proof of \Cref{theo:quanti_v_alain}]
Let $x,y \in \msy$ and $n \in \nset$. By \Cref{propo:adapt_ahip}, \Cref{lem:bound_T_m}, \Cref{ass:metric}, the fact that $M_{\msc, \ntt} \geq 1$ and the Markov inequality, we have
  for any $m \in \nset$,
  \begin{align}
    \expeMarkov{(x,y)}{\distY(\X_n,\Y_n)} &\leq (1-\vareps)^{m}\distY(x,y)  + \probaMarkov{(x,y)}{\T_{\msc,\ntt}^{(m)} \geq n} \\ 
    &\leq (1-\vareps)^{m}\distY(x,y)  + \lambda_1^{n} \, \expeMarkov{(x,y)}{ \lambda_1^{-\T_{\msc,\ntt}^{(m)}}} \\
    &\leq (1-\vareps)^{m}\distY(x,y) + \lambda_1^n M_{\msc,\ntt}^{m} \Xibf(x,y,\ntt)\eqsp ,
  \end{align}
  where $\Xibf(x,y,\ntt)$ is given in \Cref{theo:quanti_v_alain}.
  Combining this result and \Cref{lem:bound_T_m}, we can conclude that   $    \expeMarkov{(x,y)}{\distY(\X_n,\Y_n)} \leq \rho^{n} ( M_{\msc,\ntt} \Xibf(x,y,\ntt) + \distY(x,y) )$
setting
\begin{equation}
  \label{eq:def_m_quant_bounds}
  m = \ceil{n  \log(\lambda_1)/\{\log(1-\vareps) -\log(M_{\msc,\ntt})\}} \eqsp. 
  \end{equation}  
  To show that $    \expeMarkov{(x,y)}{\distY(\X_n,\Y_n)} \leq \rho^{n/2} (1+\distY(x,y))  +  \lambda_1^{n/2} \Xibf(x,y,\ntt) $, first note that \Cref{propo:adapt_ahip} and \Cref{ass:metric} imply that for any $m \in \nset$,
  \begin{align}
    &\expeMarkov{(x,y)}{\distY(\X_n,\Y_n)} \\ &\leq (1-\vareps)^{m}\distY(x,y)  + \probaMarkov{(x,y)}{\T_{\msc,\ntt}^{(m)} - \T_{\msc,\ntt}^{(1)} \geq n/2} + \probaMarkov{(x,y)}{\T_{\msc,\ntt}^{(1)} \geq n/2} \\
    & \leq (1-\vareps)^{m}\distY(x,y)  + \lambda_1^{n/2}  \expeMarkov{(x,y)}{ \lambda_1^{-\T_{\msc,\ntt}^{(m)} +\T_{\msc,\ntt}^{(1)} }} + \lambda_1^{n/2}\expeMarkov{(x,y)}{ \lambda_1^{-\T_{\msc,\ntt}^{(1)} }}  \eqsp,
  \end{align}
  where we have used the Markov inequality in the last line. Combining this result and \Cref{lem:bound_T_m}, we can conclude that   $    \expeMarkov{(x,y)}{\distY(\X_n,\Y_n)} \leq \rho^{n/2} (1+\distY(x,y))  +  \lambda_1^{n/2}\,  \Xibf(x,y,\ntt) $
setting \begin{equation}
  \label{eq:def_m_quant_bounds_2}
  m = \ceil{n  \log(\lambda_1)/\{2\log(1-\vareps) -2\log(M_{\msc,\ntt})\}} \eqsp. 
  \end{equation}  

\end{proof}

\begin{proof}[Proof of \Cref{theo:quanti_v_alain_v_norm}]
  Let $x, y \in \msy$ and $n \in \nset$. Using \Cref{assum:drift_d}($\Kcoupling$), we obtain by recursion
  \begin{equation}
    \expeMarkov{(x,y)}{\lyapD(X_n, Y_n)} \leq \lambda_2^n \lyapD(x,y) + A_2 \sum_{k=0}^{n-1} \lambda_2^{n - 1 - k} \, \, \expeMarkov{(x,y)}{\distY(X_k,Y_k)} \eqsp . \label{eq:wass_1}
  \end{equation}
  Applying \Cref{theo:quanti_v_alain} we obtain
  \begin{align}
    &\sum_{k=0}^{n-1} \lambda_2^{n - 1 - k} \, \, \expeMarkov{(x,y)}{\distY(X_k,Y_k)}  \\
    &   \leq  \sum_{k=0}^{n-1} \lambda_2^{n - 1 - k} \min\parentheseDeux{\rho^{k} ( M_{\msc,\ntt}\Xibf(x,y,\ntt) + \distY(x,y)) ,\rho^{k/2} (1+\distY(x,y))  +  \lambda^{k/2} \Xibf(x,y,\ntt)}  \\
& \leq  \min\parentheseDeux{ n \trho^{n-1} (\distY(x,y) + \Xibf(x,y,\ntt)) ,n\trho^{n/2-1} (1+\distY(x,y))  +  n\tlambda^{n/2-1} \Xibf(x,y,\ntt)} \eqsp .
  \end{align}
  We conclude plugging this result in \eqref{eq:wass_1} and using that for any $n \in \nset$ and $t \in (0,1)$, $n t^{n/2} \leq 4 \log^{-1}(1/t) t^{n/4}$.
\end{proof}


\section{Minorization conditions for  functional autoregressive models}
\label{sec:mino_condition_AR}
In this section, we extend and complete the results of  \cite[Section 6]{durmus:moulines:2016} on 
functional autoregressive models. Let $\msx \in \mcbb(\rset^d)$ equipped with its trace $\sigma$-field $\mcx = \{ \msa \cap \msx \, : \, \msa \in \mcbb(\rset^d)\}$.   In fact, we consider a slightly more general class of models than \cite{durmus:moulines:2016}  which is associated with non-homogeneous Markov chains $(\Xar_k)_{k \in \nset}$ with state space $(\msx,\mcx)$ defined for $k \geq 0$ by 
\begin{equation}
\label{eq:def:functio_AR}
  \Xar_{k+1} = \Pi\parenthese{\funreg_{k+1}(\Xar_k) + \sigma_{k+1} Z_{k+1}} \eqsp,
\end{equation}
where $\Pi$ is a measurable function from $\rset^d$ to $\msx$,
$\sequenceg{\funreg}[k][ 1]$ is a sequence of measurable functions
from $\msx$ to $\rset^d$, $\sequenceg{\sigma}[k][1]$ is a sequence of
positive real numbers and $\sequenceg{Z}[k][ 1]$ is a sequence of
\iid~$d$ dimensional standard Gaussian random variables. We assume
that $\Pi$ satisfies \Cref{ass:non_expansive_Pi}.  We also assume some
Lipschitz regularity on the operator $\funreg_k$ for any
$k \in \nsets$
\begin{assumptionAR}[$\msa$]
  \label{assum:strict_contraction_AR}
  For all $k \geq 1$ there exists $\kappar_k \in \rset$ such that  for all $(x,y) \in \msa$,
  \[\norm[2]{\funreg_k(x) - \funreg_k(y)} \leq (1+\kappar_k)\norm[2]{x-y} \eqsp .\]
\end{assumptionAR}
The sequence $\sequence{\Xar}[k][\nset]$ is an inhomogeneous
Markov chain associated with the family of   Markov kernels
$\sequenceg{\Par}[k][1]$ on $(\rset^d, \mcbb(\rset^d))$ given for all $x
\in \rset^d$ and $\eventA \in \rset^d$ by
\begin{equation}
\label{eq:def_markov_kernel_AR_0}
  \Par_k(x,\eventA) = \frac{1}{( 2 \pi \sigmakD)^{d/2}}\int_{\Pi^{-1}(\eventA)}\exp\parenthese{-\norm[2]{y-\funreg_k(x)}/(2\sigmakD)} \rmd y\eqsp.
\end{equation}
We denote for all $n \geq 1$ by $\Qar_n$ the marginal distribution of
$\Xar_n$ given by $ \Qar_n = \Par_1 \cdots \Par_n$.  To obtain an
upper bound of $\tvnorm{\updelta_x \Qar_n - \updelta_y \Qar_n}$ for
any $x,y \in \rset^d$, $n \in \nsets$, we introduce a Markov coupling
$(\Xar_k,\Yar_k)_{k \in \nset}$ such that for any $n \in \nsets$, the
distribution of $\Xar_n$ and $\Yar_n$ are $\updelta_x \Qar_n$ and
$\updelta_x \Qar_n$ respectively, exactly as we have introduced in the
homogeneous setting the Markov coupling with kernel $\Kker_{\gamma}$
defined by \eqref{eq:coupling_form} for $\Rker_{\gamma}$ defined in
\eqref{eq:kernel_langevin}. For completeness and readability, we
recall the construction of $(\Xar_k,\Yar_k)_{k \in \nset}$.  For all
$k \in \nsets$ and $x,y,z \in \rset^d$, define
\begin{equation} \label{eq:e_k_def}
  \rme_k(x,y) = \begin{cases} \rmE_k(x,y) / \| \rmE_k(x,y) \| & \text{if } \rmE_k(x,y) \neq 0 \\
    0 & \text{otherwise}\end{cases}\eqsp , \qquad \rmE_k(x,y) = \funreg_k(y) - \funreg_k(x) \eqsp ,
\end{equation}
\begin{align}
  \mathcal{S}_{k}(x,y,z) &= \funreg_{k}(y) + (\Id - 2\rme_k(x,y)\rme_k(x,y)^{\top})z \eqsp , \\
  p_{k}(x,y,z) &= 1 \wedge \frac{\vphibf_{\sigma_{k+1}^2}(\| \rmE_k(x,y) \| - \langle \rme_k(x,y),z \rangle)}{\vphibf_{\sigma_{k+1}^2}(\langle \rme_k(x,y), z \rangle)} \eqsp , \label{eq:success_prob_0}
\end{align}
where $\vphibf_{\sigma_{k}^2}$ is the one-dimensional zero mean Gaussian distribution function with variance $\sigma_{k}^2$.
Let $(U_k)_{k \in \nsets}$ be a sequence of \iid~uniform random variables on $\ccint{0,1}$ and define the Markov chain $(\Xar_k,\Yar_k)_{k \in \nset}$ starting from $(\Xar_0, \Yar_0) \in \msx^2$ by the  recursion: for any $k \geq 0$
\begin{equation}
  \label{eq:def_project_x_Y_1_ar}
\begin{aligned}
  &\Xart_{k+1} = \funreg_{k+1}(\Xar_k) + \sigma_{k+1} Z_{k+1} \eqsp , \\
  &\Yart_{k+1} =  \begin{cases} \Xart_{k+1} & \text{if} \ \funreg_{k+1}(\Xar_{k}) = \funreg_{k+1}(\Yar_{k}) \eqsp ; \\ \War_{k+1}\Xart_{k+1} + (1-\War_{k+1})\mathcal{S}_{k+1}(\Xar_k,\Yar_k,\sigma_{k+1}Z_{k+1}) & \text{otherwise} \eqsp , \end{cases} 
\end{aligned}
\end{equation}
where $  \War_{k+1} = \1_{\ocint{-\infty, 0}}(U_{k+1} - p_{k+1}(\Xar_k, \Yar_k,\sigma_{k+1}Z_{k+1}))$ and finally set 
\begin{equation}
  (\Xar_{k+1}, \Yar_{k+1}) = (\Pi(\Xart_{k+1}), \Pi(\Yart_{k+1})) \eqsp .
  \label{eq:mc_def_ar}
\end{equation}
For any $k \in \nsets$, marginally,  the distribution of $\Xar_{k+1}$ given $\Xar_k$ is $\Par_{k+1}(\Xar_k, \cdot)$, and it is well-know (see \eg \  \cite[Section 3.3]{bubley:dyer:jerrum:1998}) that $\Yart_{k+1}$ and $\Tg(\Yar_k) + \sigma_{k+1} Z_{k+1}$ have the same distribution given $Y_k$, and therefore the distribution of $Y_{k+1}$ given $Y_k$ is  $\Par_{k+1}(Y_k, \cdot)$. As a result for any $(x,y) \in \msx^2$ and $n \in \nsets$, $(\Xar_n,\Yar_n)$ with $(\Xar_0,\Yar_0) = (x,y)$ is a coupling between $\updelta_x \Qar_n$ and $\updelta_y \Qar_n$. Therefore, we obtain that $\tvnormLigne{\updelta_x \Qar_n - \updelta_y \Qar_n} \leq \probaLigne{\Xar_n \neq \Yar_n}$. Therefore to get an upper bound on $\tvnormLigne{\updelta_x \Qar_n - \updelta_y \Qar_n}$, it is sufficient to obtain a bound on $\probaLigne{\Xar_n \neq \Yar_n}$ which is a simple consequence of the following more general result. 

\begin{theorem}
  \label{theo:strict_convergence_AR}
  Let $\msa \in \mcbb(\rset^{2d})$ and assume \tup{\Cref{ass:non_expansive_Pi}} and \Cref{assum:strict_contraction_AR}$(\msa)$.  Let $(\Xar_k,\Yar_k)_{k \in \nset}$ be defined by \eqref{eq:mc_def_ar}, with  $(\Xar_0,\Yar_0) = (x,y) \in \msa$. Then for any $n \in \nsets$,
    \begin{multline}
\probaMarkov{}{\Xar_n \neq \Yar_n \text{ and for any $k \in \{1,\ldots,n-1\}$,} \, (\Xar_k,  \Yar_k) \in \msa^2}\\  \leq \1_{\Deltar^{\complem}}(x,y) \defEns{1-2\Phibf\parenthese{-\frac{\norm{x-y} }{2(\Xiar_n)^{1/2}}}} \eqsp,
  \end{multline}
 where $\Phibf$ is the cumulative distribution function of the standard normal distribution on $\rset$ and $\sequenceg{\Xiar}[i][1]$ is defined for all $k \geq 1$ by $\Xiar_{k} = \sum_{i=1}^{k} \{ \sigma_{i}^2  /  \prod_{j=1}^{i}(1+ \kappar_{j})\} $.
\end{theorem}

\begin{proof}
  Let $(\mcfar_k)_{k \in \nset}$ be the filtration associated to $(\Xar_k, \Yar_k)_{k \in \nset}$.  
Denote for any $k \in \nset$, 
\begin{equation}
  \scrA_k = \bigcap_{i=0}^k \{ (\Xar_i,\Yar_i) \in \msa\} \eqsp, \qquad \scrA_{-1} = \scrA_{0} \eqsp,
\end{equation}
 and for all $k_1,k_2 \in \nset^*$,  $k_1 \leq k_2$, $\Xiar_{k_1,k_2} = \sum_{i=k_1}^{k_2} \{ \sigma_{i}^2  /  \prod_{j=k_1}^{i}(1+ \kappar_{j})\}$.
Let $n \geq 1$ and $(x,y) \in \msa^2$. We show by backward induction that for all $k \in \defEns{0,\cdots,n-1}$,
\begin{equation}
\label{eq:proof_convergence_AR_induction_strict}
  \PP(\{\Xar_n \not = \Yar_n\} \cap \scrA_{n-1} )  \leq \expe{
\1_{\Deltar^{\complem}}(\Xar_{k},\Yar_{k}) \1_{\scrA_{k-1}} \parentheseDeux{1-2\Phibf\defEns{-\frac{\norm{\Xar_{k}-\Yar_{k}} }{2 \parenthese{\Xiar_{k+1,n}}^{1/2}}}}
} \eqsp.
\end{equation}
Note that the inequality for $k=0$ will conclude the proof.
Using by \eqref{eq:def_project_x_Y_1_ar} that $\Xart_n =  \Yart_n$ if $\Xar_{n-1} =  \Yar_{n-1}$ or $W_{n} = \1_{\ocint{-\infty, 0}}(U_{n} - p_{n}(\Xar_{n-1}, \Yar_{n-1}, \sigma_{n}Z_{n})) = 1$, where $p_{n}$ is defined by \eqref{eq:success_prob_0}, and $(U_{n},Z_{n})$ are independent random variables independent of $\mcfar_{n-1}$, we obtain on $\{\Xar_{n-1} \neq   \Yar_{n-1}\}$
\begin{multline}
  \CPE{\1_{\Delta_{\msx}}(\Xart_n, \Yart_n)}{\mcfar_{n-1}} =    \CPE{p_{n}(\Xar_{n-1},\Yar_{n-1}, \sigma_n Z_{n})}{\mcfar_{n-1}}\\
  =  2\Phi\defEns{- \norm{(2\sigma_n)^{-1}\rmE_{n}(\Xar_{n-1}, \Yar_{n-1})}} \label{eq:proba_coupling} \eqsp .
\end{multline}
Since $\{\Xar_n \not = \Yar_n\} \subset \{ \Xart_n \neq \Yart_n\} \subset \{\Xar_{n-1} \not = \Yar_{n-1}\} $ by \eqref{eq:mc_def_ar} and \eqref{eq:def_project_x_Y_1_ar}, we get 
\begin{align}
&\probaMarkov{}{\{ \Xar_n \not = \Yar_n \} \cap \scrA_{n-1}} \leq  \expeMarkov{}{
  \1_{\Deltar^{\complem}}(\Xar_{n-1},\Yar_{n-1}) \1_{\scrA_{n-1}} \CPE{\1_{\Deltar^{\complem}}(\Xart_{n},\Yart_{n})}{\mcfar_{n-1}}}
\\
 & \qquad \qquad \qquad  = \expeMarkov{}{
\1_{\Deltar^{\complem}}(\Xar_{n-1},\Yar_{n-1}) \1_{\scrA_{n-1}} \parentheseDeux{1-2\Phibf\defEns{-\norm{(2\sigma_n)^{-1}\rmE_{n}(\Xar_{n-1},\Yar_{n-1})}}}} \eqsp ,
\end{align}
Using that $(\Xar_{n-1}, \Yar_{n-1}) \in \msa^2$ on $\scrA_{n-1}$, \Cref{assum:strict_contraction_AR}($\msa$) and 
\eqref{eq:e_k_def}, we obtain that
\begin{equation}
  \normLigne{\rmE_{n}(\Xar_{n-1},\Yar_{n-1})}^2 \leq (1 + \kappar_{n})
\normLigne{\Xar_{n-1} - \Yar_{n-1}}^2 \eqsp,
\end{equation}
showing \eqref{eq:proof_convergence_AR_induction_strict} holds for
$k=n-1$ since $\scrA_{n-2} \subset \scrA_{n-1}$.
Assume that \eqref{eq:proof_convergence_AR_induction_strict} holds for $k \in \{1,\ldots,n-1\}$.
On $\defEnsLigne{\Xart_{k} \not =
  \Yart_{k}}$, we have
\begin{equation}
  \norm{\Xart_{k} - \Yart_{k}} = \abs{-\norm{\rmE_{k}(\Xar_{k-1},\Yar_{k-1})} + 2 \sigma_{k}
  \rme_{k}(\Xar_{k-1} , \Yar_{k-1})^{\transp} Z_{k}} \eqsp,
\end{equation}
which implies by \eqref{eq:mc_def_ar} and since $\Pi$ is non expansive
by \Cref{ass:non_expansive_Pi}
\begin{align}
 &\1_{\Deltar^{\operatorname{c}}}(\Xar_{k},\Yar_{k})
   \parentheseDeux{1-2\Phibf\defEns{-\frac{\norm{\Xar_{k}-\Yar_{k}}}{2 (\Xiar_{k+1,n})^{1/2}}}} \\
& \phantom{aaaaa} \leq \1_{\Deltar^{\operatorname{c}}}(\Xar_{k},\Yar_{k})
\parentheseDeux{1-2\Phibf\defEns{-\frac{\norm{\Xart_{k}-\Yart_{k}}}{2 (\Xiar_{k+1,n})^{1/2}}}}
\\
\nonumber
& \phantom{aaaaa}\leq  \1_{\Deltar^{\operatorname{c}}}(\Xar_{k},\Yar_{k})
\parentheseDeux{1-2\Phibf\defEns{-\frac{\abs{
 2  \sigma_k
  \rme_{k}( \Xar_{k-1} , \Yar_{k-1})^{\transp} Z_{k}-\norm{\rmE_{k}(\Xar_{k-1},\Yar_{k-1})}
}}{2 (\Xiar_{k+1,n})^{1/2}}}} \eqsp.
\end{align}
Since $Z_{k}$ is independent of $\mcfar_k$, $ \sigma_k \rme_k( \Xar_{k-1} , \Xar_{k-1})^{\transp} Z_{k}$ is a real  Gaussian random variable with zero mean and  variance $\sigma_{k}^2$. Therefore by \cite[Lemma 20]{durmus:moulines:2016} and since $\scrA_{k-1}$ is $\mcfar_{k-1}$-measurable, we get
\begin{multline}
\CPE{
\1_{\Deltar^{\complem}}(\Xar_{k},\Yar_{k}) \1_{\scrA_{k-1}}
\parentheseDeux{1-2\Phibf\defEns{-\frac{\norm{\Xar_{k}-\Yar_{k}} }{2 (\Xiar_{k+1,n})^{1/2}}}}
}{\mcfar_{k-1}}
\\
\leq \1_{\scrA_{k-1}} \1_{\Deltar^{\complem}}(\Xar_{k-1},\Yar_{k-1}) \parentheseDeux{1- 2 \Phibf\defEns{-\frac{\norm{\rmE_k(\Xar_{k-1},\Yar_{k-1})}}{2 \parenthese{ \sigma_{k} +  \Xiar_{k+1,n}}^{1/2}}}} \eqsp.
\end{multline}
Since \Cref{assum:lip_op}$(\msa)$ implies that
$\normLigne{\rmE_k(\Xar_{k-1},\Yar_{k-1})}^2 \leq (1 +
\kappar_{k-1})\norm{\Xar_{k-1}-\Yar_{k-1}}^2$ on $\scrA_{k-1}$ and
$\scrA_{k-2} \subset \scrA_{k-1}$ concludes the induction of
\eqref{eq:proof_convergence_AR_induction_strict}.
\end{proof}


\section{Quantitative convergence results based on \protect{\cite{douc:moulines:priouret:soulier:2018, douc:moulines:rosenthal:2004}}}
\label{sec:quant-conv-results}
We start by recalling the following lemma from \cite{douc:moulines:priouret:soulier:2018} which is inspired from the results of \cite{douc:moulines:rosenthal:2004}.

\begin{lemma}[\protect{\cite[Lemma 19.4.2]{douc:moulines:priouret:soulier:2018}}]
\label{lem:eric_book}
  Let $(\msy, \mcy)$ be a measurable space and $\Rcoupling$ be a Markov kernel over $(\msy, \mcy)$. Let $\Qcoupling$ be a Markov coupling kernel for $\Rcoupling$. Assume there exist $\msc \in \mcy^{\otimes 2}$, $M \geq 0$, a measurable function $\VlyapD : \msy \times \msy \to \coint{1,+\infty}$, $\lambda \in \coint{0,1}$ and $c \geq 0$ such that for any $x,y \in \msy$,
  \begin{equation}
    \label{eq:drift_precise}
    \Qcoupling \VlyapD(x,y) \leq \lambda \VlyapD(x,y) \1_{\msc^{\complementary}}(x,y) + c \1_{\msc}(x,y) \eqsp.
  \end{equation}
  In addition, assume that there exists $\vareps >0$ such that for any $(x,y) \in \msc$,
  \begin{equation}
    \Qcoupling((x, y), \Delta^{\complementary}_{\msy}) \le 1 - \vareps\eqsp,
  \end{equation}
  where $\Delta_{\msy} = \ensemble{(y,y)}{y \in \msy}$. 
  Then there exist $\rho \in \coint{0,1}$ and $C \geq 0$ such that for any $x,y \in \mcy$ and $n \in \nsets$
  \begin{equation}
    \int_{\msy \times \msy} \1_{\Delta_{\msy}}(\tilde{x},\tilde{y})\lyap(\tilde{x},\tilde{y}) \Qcoupling^{n}((x,y) , \rmd (\tilde{x}, \tilde{y})) \leq C \rho^n \VlyapD(x,y) \eqsp ,
  \end{equation}
  where 
    \begin{equation}
      \begin{aligned}        
         C &= 2(1 + \left. c \middle/ \defEns{(1-\vareps)(1- \lambda)}) \right. \eqsp , \\
           \log(\rho) &= \left. \defEns{\log(1-\vareps) \log(\lambda)} \middle/ \defEns{\log(1-\vareps) + \log(\lambda) - \log(c) } \right .\eqsp .
          \end{aligned} 
          \end{equation}
\end{lemma}

\begin{theorem}
  \label{thm:v_ergo_gene}
  Let $(\msy, \mcy)$ be a measurable space and $\Rcoupling$ be a Markov kernel over $(\msy, \mcy)$. Let $\Qcoupling$ be a Markov coupling kernel of $\Rcoupling$. Assume that there exist $\lambda \in \coint{0, 1}$, $A \geq 0$ and a measurable function $\VlyapD: \msy \times \msy \to [1,+\infty)$, such that $\Qcoupling$ satisfies \hyperlink{ass:drift_discrete}{$\bfDd(\VlyapD,\lambda,A,\msy)$}. In addition, assume that there exist $\ell \in \nsets$, $ \vareps >0$ and $M \geq 1$ such that for any $(x, y) \in \msc_M = \lbrace (x,y) \in \msy \times \msy, \ \VlyapD(x,y) \leq M \rbrace$,
  \begin{equation}
\label{eq:minoration_eric_Q}
    \Qcoupling^\ell((x, y), \Delta^{\complementary}_{\msy}) \le 1 - \vareps\eqsp ,
  \end{equation}
  with $\Delta_{\msy} = \{(x,y) \in \msy^2 \, : \, x = y \}$ and  $M \geq 2A /(1 - \lambda)$. Then, there exist $\rho \in \coint{0,1}$ and $C \geq 0$ such that for any $n \in \nset$ and $x,y \in \msy$ 
  \begin{equation}
    \label{eq:true_geome_eric}
    \distV(\updelta_x \Rcoupling^n, \updelta_y \Rcoupling^n)  \le C \rho^{\floor{n/\ell}} \VlyapD(x,y)  \eqsp , 
  \end{equation}
  with
 \begin{align}
         C &= 2(1 +  A_\ell)(1 + \left. c_\ell \middle/ \defEns{(1-\vareps)(1- \lambda_\ell)}) \right. \eqsp , \\
         \lambda_\ell &= (\lambda^\ell + 1)/2 \eqsp, \quad c_\ell = \lambda^\ell M + A_\ell \eqsp, \quad A_\ell = A (1 - \lambda^{\ell})/(1 - \lambda)  \eqsp,
   \\
   \label{eq:def_rho_ell_theo_eric_et_al}
         \log(\rho_\ell) &= \left. \defEns{\log(1-\vareps) \log(\lambda_\ell)} \middle/ \defEns{\log(1-\vareps) + \log(\lambda_\ell) - \log(c_\ell) } \right. \eqsp.
      \end{align} 
\end{theorem}

\begin{proof}
  We first show that  for any $(x,y) \in \msc_M $,
  \begin{equation}
    \label{eq:drift_Q_k_eric}
    \Qcoupling^\ell(x,y) \leq \lambda_\ell \VlyapD(x,y) \1_{\msc_{M}^{\complementary}}(x,y) + c_\ell \1_{\msc_{M}}(x,y) \eqsp,
  \end{equation}
in order to apply \Cref{lem:eric_book} to $\Rcoupling^{\ell}$ with the  Markov coupling kernel $\Qcoupling^{\ell}$.
By a straightforward induction,  for any $x, y \in \msy$ we have
  \begin{equation}
    \Qcoupling^\ell \VlyapD(x,y) \leq \lambda^\ell \VlyapD(x,y) + A (1 - \lambda^{\ell})/(1 - \lambda) \eqsp . \label{eq:drift_iterated_Q}
  \end{equation}
We distinguish two cases. If $(x,y) \notin \msc_{M}$, using that $A/M \geq (1 - \lambda)/2$ we have that 
\begin{equation}
  \Qcoupling^\ell \VlyapD(x,y) \leq \lambda^\ell \VlyapD(x,y) + A (1 - \lambda^{\ell}) \VlyapD(x,y) / (M(1 - \lambda)) \leq  2^{-1}(\lambda^\ell +1)\VlyapD(x,y) \eqsp . \label{eq:exte}
\end{equation}
If $(x,y) \in \msc_{M}$, we have 
\begin{equation}
  \Qcoupling^\ell \VlyapD(x,y) \leq \lambda^\ell M + A (1 - \lambda^{\ell}) /(1 - \lambda) \eqsp . \label{eq:inte}
\end{equation}
Therefore \eqref{eq:drift_Q_k_eric} holds.
As a result and since by assumption we have \eqref{eq:minoration_eric_Q}, we can apply \Cref{lem:eric_book} to $\Rker^\ell$. Then, we obtain that for any $i \in \nset$ and $x, y \in \msy$
\begin{equation}
    \int_{\msy \times \msy} \1_{\Delta_{\msy}}(\tilde{x},\tilde{y})\lyap(\tilde{x},\tilde{y}) \Qcoupling^{\ell i}((x,y) , \rmd (\tilde{x}, \tilde{y})) \leq C_{\ell} \rho^{\ell i}_{\ell} \VlyapD(x,y) \eqsp ,
 \label{eq:m_step}
\end{equation}
with $\rho_\ell$ defined by \eqref{eq:def_rho_ell_theo_eric_et_al} and $\tilde{C}_{\ell} = 2\defEns{1 + c_\ell\parentheseDeux{(1-\lambda_{\ell})(1-\vareps)}^{-1}}$.
In addition, using \eqref{eq:drift_iterated_Q}, for any $k \in \lbrace 0, \dots, \ell-1 \rbrace$ and $x,y \in \msy$, $\Qcoupling^k \lyap(x,y) \leq (1+ A_{\ell}) \VlyapD(x,y)$. Therefore,   for any  $n \in \nset$, since $n = i_n \ell + k_n$ with $i_n = \floor{n/\ell}$ and  $k_n \in \lbrace 0, \dots, \ell-1 \rbrace$, we obtain for any $x,y \in \msy$ that
\begin{multline}
  \distV(\updelta_x \Rcoupling^{n}, \updelta_y \Rcoupling^{n}) \leq \tilde{C}_{\ell} \rho_{\ell}^{\ell i_n}     \int_{\msy \times \msy} \1_{\Delta_{\msy}}(\tilde{x},\tilde{y})\lyap(\tilde{x},\tilde{y})\Qcoupling^{k_n}((x,y) , \rmd (\tilde{x}, \tilde{y}))  \\ 
  \leq (1+A_{\ell})\tilde{C}_{\ell} \rho_\ell^{\floor{n/\ell}} \VlyapD(x,y) \eqsp ,
\end{multline}
which concludes the proof.
\end{proof}

We now state an important consequence of \Cref{thm:v_ergo_gene}. The
comparison between \Cref{theo:discrete_contrac_wass} and
\Cref{theo:discrete_contrac_wass_D_v2} is conducted in the remarks
which follow \Cref{theo:discrete_contrac_wass_D_v2}.

\begin{theorem}
  \label{theo:discrete_contrac_wass}
  Assume that there exists a measurable function $\VlyapD : \msx \times \msx \to \coint{1,\plusinfty}$ such that for any $C \geq 0$,
  \begin{equation}
    \diam \ensemble{ (x,y) \in \msx^2}{  \VlyapD(x,y) \leq C } < +\infty \eqsp.
  \end{equation}
  Assume in addition that there exist $\lambda \in \coint{0, 1}$,
  $A \geq 0$ such that for any $\gamma \in \ocint{0,\bgamma}$, there
  exists $\KkerD_{\gamma}$, a Markov coupling kernel for
  $\Rcoupling_{\gamma}$, satisfying
  \hyperlink{ass:drift_discrete}{$\bfDd(\VlyapD,\lambda^{\gamma},
    A\gamma,\msx^2)$}.  Further, assume that there exists
  $\Psibf : \ocint{0,\bgamma} \times \nsets \times \rset_+ \to
  \ccint{0,1}$ such that for any $\gamma \in \ocint{0, \bgamma}$,
  $\ell \in \nsets$ and $x,y \in \msx$,
  \eqref{eq:minorization_condition_v2} is satisfied.  Then the
  following results hold.
\begin{enumerate}[wide, labelwidth=!, labelindent=0pt, label=(\alph*)]
\item   \label{propo:discrete_contrac_wass_a}
  For any $\gamma \in \ocint{0, \bgamma}$, $M_{\discrete} \geq \diameter \parenthese{\ensemble{(x,y) \in \msx^2}{\lyap(x,y) \leq \Bdisc}}$ with $ \Bdisc = 2 A(1 + \bgamma) \{1 + \log^{-1}(1/\lambda)\}$, $\ell \in \nsets$, $x, y \in \msx$ and $k \in \nset$
  \begin{equation}
   \distV(\updelta_x \Rcoupling_{\gamma}^k, \updelta_y \Rcoupling_{\gamma}^k) \leq C_{\gamma} \rho_{\gamma}^{\floor{ k (\ell\ceil{1/\gamma})^{-1}}} \VlyapD(x,y) \eqsp , \label{eq:brute_vcontrac}
  \end{equation}
  where $\distV$ is the Wasserstein metric associated with $\bfc$
  defined by \eqref{eq:def_wbf},
  \begin{equation}
  \begin{aligned}
    C_{\gamma} &= 2[1 + A_{\gamma}][1 + \left. c_{\gamma} \middle/\defEns{(1-\bvareps_{\discrete, 2})(1-\lambda_{\gamma})} \right.] \eqsp ,\\
\log(\rho_{\gamma}) &= \left. \defEns{\log(1-\bvareps_{\discrete, 2}) \log(\lambda_{\gamma}) } \middle/ \defEns{\log(1- \bvareps_{\discrete, 2}) + \log(\lambda_{\gamma}) - \log(c_{\gamma}) } \right. <0  \eqsp,\\
    A_{\gamma} &= A\gamma(1 - \lambda^{\gamma \ell \step})/(1 - \lambda^{\gamma}) \eqsp ,   \quad         c_{\gamma} = \lambda^{\gamma \ell \ceil{1/\gamma}} A_{\gamma} + \Bdisc \eqsp , \\       
        \bvareps_{\discrete, 2}  &= \inf_{\gamma \in \ocint{0, \bgamma}, \ (x,y) \in \Delta_{\msx,M_{\discrete}}} \Psibf(\gamma, \ell, \norm{x-y}) \eqsp , \quad
\lambda_{\gamma} = (\lambda^{\gamma \ell \ceil{1/\gamma}} + 1)/2 \eqsp .
\end{aligned}
\end{equation}
\item \label{propo:discrete_contrac_wass_b} For any $\gamma \in \ocint{0,\bgamma}$, $M_{\discrete} \geq \diameter \parenthese{\ensemble{(x,y) \in \msx^2}{\lyap(x,y) \leq \Bdisc}}$ with $ \Bdisc = 2 A(1 + \bgamma) \{1 + \log^{-1}(1/\lambda)\}$ and $\ell \in \nsets$, it holds that
  \begin{equation}
    \hspace{-1.5cm}
    \begin{aligned}
      C_{\gamma} & \leq \bC_{1} \eqsp, \quad \log(\rho_{\gamma}) \leq \log(\brho_2) \leq 0 \eqsp , \\
\bC_1  &= 2[1 + \bA_{1}][1 + \left. \bc_{1} \middle/\defEns{(1-\bvareps_{\discrete, 2})(1-\blambda_{1})} \right.] \eqsp, \\
        \log(\brho_{2}) &= \left. \defEns{\log(1-\bvareps_{\discrete, 2}) \log(\lambdab_{1}) } \middle/ \defEns{\log(1- \bvareps_{\discrete, 2}) + \log(\lambdab_{1}) - \log(\bc_{1}) } \right. < 0 \eqsp, \\
\bA_{1} &= A(1+\bgamma)\min(\ell, 1 + \log^{-1}(1/\lambda))  \eqsp, \quad        \bc_{1} =    \bA_1 + \Bdisc \eqsp, \quad   \lambdab_{1}= (\lambda+1)/2 \eqsp,\\
        \end{aligned}
      \end{equation}
    \item \label{propo:discrete_contrac_wass_c} In addition, if $\bgamma \leq 1$, $-\log(\lambda) \in \ccint{0,\log(2)}$, $A \geq 1$ and $0 < \bvareps_{\discrete, 2} \leq 1 - \rme^{-1}$, then
      \begin{equation}
        \log^{-1}(1/\brho_{2}) \leq 12 \log(2)  \left. \log\parentheseDeux{6A\defEns{1 + \log^{-1}(1/\lambda)}} \middle/ \parenthese{ \log(1/\lambda) \bvareps_{\discrete, \bgamma}} \right. \eqsp . \label{eq:majo_rho_1} 
      \end{equation}

\end{enumerate}
  
\end{theorem}

\begin{proof}
  First, note that 
$    1-\lambda^t = -\int_0^t \log(\lambda) \rme^{s \log(\lambda)} \rmd s \geq -\log(\lambda)\,  t\, \rme^{t \log(\lambda)}$  for any $t \in \ocintLigne{0,\bt}$, for $\bt >0$, and therefore
  \begin{equation}
    \label{eq:borne_discret_control_dirft_log}
    t/(1-\lambda^t) = t + t \lambda^t/(1-\lambda^t) \leq  \bt +  \log^{-1}(\lambda^{-1}) \eqsp.
  \end{equation}
  \begin{enumerate}[wide, labelwidth=!, labelindent=0pt, label=(\alph*)]
\item 
  To establish \eqref{eq:brute_vcontrac}, we apply \Cref{thm:v_ergo_gene}.
For any $x,y \in \msx$ such that $\VlyapD(x,y) \leq \Bdisc$ we have
\begin{equation}
  \KkerD_{\gamma}^{\ell \step}((x,y), \Deltar^{\complementary}) \leq 1 - \bvareps_{\discrete, 2} \eqsp . 
\end{equation}
 Using 
 that $\KkerD_{\gamma}$ satisfies \hyperlink{ass:drift_discrete}{$\bfDd(\lyap,\lambda^{\gamma}, A\gamma, \msx^2)$},  we can apply
\Cref{thm:v_ergo_gene} with $M \leftarrow  \Bdisc \geq 2 A \gamma /(1-\lambda^{\gamma })$ by \eqref{eq:borne_discret_control_dirft_log}, which completes the proof of \ref{propo:discrete_contrac_wass_a}.

\item  We now provide upper bounds for $C_{\gamma}$ and $\rho_{\gamma}$. These constants are non-decreasing in $c_{\gamma}$ and $ \lambda_{\gamma}$. Therefore it suffices to give upper bounds on $c_{\gamma}, \vareps_{\discrete, \gamma}$ and $ \lambda_{\gamma}$. The result is then  straightforward using that $(1 - \lambda^{\gamma \ell \step})/(1 - \lambda^{\gamma}) \leq \ell \step$, $\gamma(1 - \lambda^{\gamma \ell \step})/(1 - \lambda^{\gamma}) \leq \bgamma + \log^{-1}(1/\lambda)$ and $\lambda^{\gamma \ell \step } \leq \lambda$.

\item
By assumption on $\bgamma$, $\lambda$ and $\bvareps_{\discrete,1}$  we have that $\log((1 - \bvareps_{\discrete, 2})^{-1}) \leq 1$ and
  \begin{equation}
    \log(\blambda_{1}^{-1}) \leq \log(\lambda^{-1}) \leq \log(2) \eqsp , \quad \rme \leq 2(1 + 1/\log(2)) \leq \Bdisc \leq \bc_{1}  \eqsp .
  \end{equation}
  As a result, we obtain that $\log(\blambda_{1}^{-1})/\log(\bc_{1}) \leq 1$, $\log((1 - \bvareps_{\discrete, 2})^{-1})/\log(\bc_{1}) \leq 1$. Therefore we have 
 \begin{align}
   \log^{-1}(1/\brho_{2}) &= \parentheseDeux{\log(\blambda_{1}^{-1}) + \log((1-\bvareps_{\discrete, 2})^{-1}) + \log(\bc_{1})} \\ & \qquad \qquad \qquad \qquad /\parentheseDeux{\log(\blambda_{1}^{-1})\log((1-\bvareps_{\discrete, 2})^{-1})} \\
   &\leq 3 \log[6A(1+\log^{-1}(1/\lambda))] / \parentheseDeux{\log(\blambda_{1}^{-1})\log((1-\bvareps_{\discrete, 2})^{-1})} \eqsp .
 \end{align}
Using that $\log(1- t) \leq -t$ for any $t \in \ocint{0,1}$ and the definition of $\blambda_1$, we obtain that
\begin{align}
  \log^{-1}(\brho_{2}^{-1}) &\leq 6 \bvareps_{\discrete, 2}^{\, -1} (1 - \lambda)^{-1}\log[6A(1 +\log^{-1}(1/\lambda))] \eqsp .
\end{align}
Finally, we get \eqref{eq:majo_rho_1}   using that for any $t \in \ccint{0,\log(2)}$, $1 - e^{-t} \geq (2\log(2))^{-1}t$.
\end{enumerate}
\end{proof}
Note that \Cref{theo:discrete_contrac_wass} gives an upper bound on the rate of convergence $\brho_{2}$ in the worst case scenario for which the minorization constant $\bvareps_{\discrete, 2}$ is small and the constant $\lambda$ in \hyperlink{ass:drift_discrete}{$\bfDd(V,\lambda^{\gamma}, A\gamma, \msx^2)$} is close to one.

Some remarks are in order here concerning the bounds obtained in
\Cref{theo:discrete_contrac_wass} and
\Cref{theo:discrete_contrac_wass_D_v2}. Assume that $\ell =1$, we will
see in \Cref{sec:applications} that the leading term in the upper
bound in \Cref{theo:discrete_contrac_wass_D_v2}, respectively
\Cref{theo:discrete_contrac_wass}, is given by
$\log(A)/(\log(\lambda^{-1}) \bvareps_{\discrete,1})$, respectively
$\log(A)/(\log(\lambda^{-1}) \bvareps_{\discrete,2})$. In addition, in
our applications, $\bvareps_{\discrete,1}$ is larger than
$\bvareps_{\discrete,2}$. Therefore, in these cases the bounds
provided in \Cref{theo:discrete_contrac_wass_D_v2} yield better rates
than the ones in
\Cref{theo:discrete_contrac_wass}-\ref{propo:discrete_contrac_wass_c}.
The main difference between the two results is that in the proof of
\Cref{theo:discrete_contrac_wass} a drift condition on the
\textit{iterated} coupling kernel $\KkerD_{\gamma}^{\step}$ is
required. However, even if such drift conditions can be derived from a
drift condition on $\KkerD_{\gamma}$, the constants obtained using
this technique are not sharp in general. On the contrary, the proof of
\Cref{theo:discrete_contrac_wass_D_v2} uses the iterated minorization
condition and a drift condition on the \textit{original} coupling
$\KkerD_{\gamma}$.



\section{Tamed Euler-Maruyama discretization}
\label{sec:tamed-unadj-lang}

In this subsection we consider the following assumption.
\begin{assumptionT}
  \label{ass:def_tamed}
  $\msx = \rset^d$ and $\Pi = \Id$ and
  \begin{equation}
    \label{eq:def_tamed_euler}
   \Tg(x) = x + \gamma b(x)/(1 + \gamma \norm{b(x)} )  \text{ for any $\gamma >0$ and $x \in \rset^d$} \eqsp.   
  \end{equation}
\end{assumptionT}
Here, we focus on drift $b$ which is  no longer  assumed to be  Lipschitz.
Therefore, the ergodicity results obtained in \Cref{sec:applications} no longer hold since the minorization condition we derived
relied heavily on one-sided Lipschitz condition or Lipschitz regularity for $b$.  We now consider the following assumption on $b$. 

\begin{assumptionT}
  \label{assum:glip_tula}
  There exists $\tLip, \tell \geq 0$ such that for any $x,y \in \rset^d$
  \begin{equation}
    \norm{b(x) - b(y)} \leq \tLip (1 + \norm{x}^{\tell} + \norm{y}^{\tell}) \norm{x - y} \eqsp .
  \end{equation}
  In addition, assume that $b(0) = 0$ and $M_{\tell} = \sup_{x \in \rset^d} \ (1 + \norm{x}^{\tell})(1 + \norm{b(x)})^{-1} <+\infty$.
\end{assumptionT}

\begin{proposition}
  \label{prop:lip_tamed}
  Assume \tup{\Cref{ass:def_tamed}} and \tup{\Cref{assum:glip_tula}}
  then
  \tup{\Cref{assum:lip_op}($\rset^{2d}$)-\ref{assum:lip_op_non_convex}}
  holds with $\bgamma > 0$ and for any
  $\gamma \in \ocint{0, \bgamma}$,
  $\kappa(\gamma) = 2 \tLip_{\gamma} + \gamma \tLip_{\gamma}^2$ where
  \begin{equation}
    \tLip_{\gamma} = 2 \gamma^{-1} M_{\tell} (1 + M_{\tell}) \tLip \eqsp .
  \end{equation}
\end{proposition}

\begin{proof}
  Let $x,y \in \rset^d$ and assume that $\norm{x} \geq \norm{y}$. We have the following inequalities
  \begin{align}
    \norm{\frac{b(x)}{1 + \gamma \norm{b(x)}} - \frac{b(y)}{1 + \gamma \norm{b(y)}}} &\leq \frac{\norm{b(x) - b(y)}}{1 + \gamma \norm{b(x)}} +   \abs{\frac{\norm{b(y)}}{1 + \gamma \norm{b(x)}} - \frac{\norm{b(y)}}{1 + \gamma \norm{b(y)}}} \\
                                                                                     &\leq \gamma^{-1} 2 M_{\tell} \tLip \norm{x-y} + \gamma \frac{\norm{b(y)} \norm{b(x) - b(y)}}{(1 + \gamma \norm{b(x)})(1 + \gamma \norm{b(y)})} \\
    &\leq \gamma^{-1} M_{\tell} (1 + M_{\tell}) \tLip \norm{x-y} \eqsp .
  \end{align}
  The same inequality holds with $\norm{y} \geq \norm{x}$. Therefore, we have 
  \begin{equation}
    \norm{\Tg(x) - \Tg(y)}^2 \leq \parenthese{1 + 2\gamma\tLip_{\gamma} + \gamma^2\tLip_{\gamma}^2}\norm{x-y}^2 \eqsp ,
  \end{equation}
which concludes the proof.
\end{proof}
\Cref{prop:lip_tamed} implies that the conclusions of \Cref{propo:doeb}-\ref{item:kappa_pos} hold.
Note that contrary to the conclusion of \Cref{prop:a1_type}, we do not get that $\sup_{\gamma \in \ocint{0,\bgamma}} \kappa(\gamma) < +\infty$. Hence we have for any $\tell \in \nsets$, $\inf_{\gamma \in \ocint{0, \bgamma}} \alpha_+(\kappa, \gamma, \tell) = 0$.
\begin{assumptionT}
  \label{assum:curvature_tula}
  There exist $\tR$ and $\tmttplus
  $ such that for any $x \in \cball{0}{\tR}^{\complementary}$,
  \[ \langle b(x), x \rangle \leq - \tmttplus \| x \| \norm{b(x)} \eqsp.\]
\end{assumptionT}
Under \Cref{assum:glip_tula} and \Cref{assum:curvature_tula} it is shown in \cite{brosse2018tamed} that there exists $\bgamma >0$, $\lambda \in (0,1)$ and $A \geq 0$ such that for any $\gamma \in \ocint{0,\bgamma}$,  $\Rcoupling_{\gamma}$ satisfies  \hyperlink{ass:drift_discrete}{$\bfDd(V,\lambda^{\gamma}, A\gamma, \msx)$} with $V(x) = \exp(a(1+\|x\|^2)^{1/2})$ for some fixed $a$.

\begin{theorem}
Assume  \tup{\Cref{assum:glip_tula}}
  and \tup{\Cref{assum:curvature_tula}} then there exists $\bgamma >0$ such that for any $\gamma \in \ocint{0, \bgamma}$ there exist $C_{\gamma} \geq 0$ and $\rho_{\gamma} \in \ooint{0,1}$ with for any $\gamma \in \ocint{0,\bgamma}$, $x,y \in \rset^d$ and $k \in \nset$
    \begin{equation}
    \label{eq:contrac_TULA}
    \Vnorm{ \updelta_x \Rcoupling_{\gamma}^k - \updelta_y \Rcoupling_{\gamma}^k} \leq C_{\gamma} \rho_{\gamma}^{k \gamma} \defEns{V(x) + V(y)} \eqsp . 
  \end{equation}
\end{theorem}
\begin{proof}
  The proof is a direct application of \Cref{theo:discrete_contrac_wass}-\ref{propo:discrete_contrac_wass_a}.
\end{proof}
It is shown in \cite[Theorem 4]{brosse2018tamed} that the following result holds: there exists $V:\rset^d \to \coint{1,+\infty}$, $\bgamma >0$, $C, D \geq 0$ and $\rho \in (0,1)$ such that for any $k \in \nset$, $\gamma \in \ocint{0, \bgamma}$ and $x \in \rset^d$
\begin{equation}
  \Vnorm{\updelta_x \Rcoupling_{\gamma}^{k} - \pi} \leq C \rho^{k \gamma}V(x) + D \sqrt{\gamma} \eqsp ,
\end{equation}
where $\pi$ is the invariant measure for the diffusion with drift $b$ and diffusion coefficient $\Id$.


\section{Explicit rates and asymptotics in
  \Cref{propo:weak_outside_bounds}}
\label{sec:rates-crefpr}

We recall that $b$ satisfies
\[ \langle b(x), x \rangle \leq - \mttun \| x \| \1_{\cball{0}{\Rtrois}^{\complementary}}(x) -\mttdeux \| b (x) \|^2 + \cconst/2 \eqsp,\]
with $\mttun, \mttdeux >0$ and $\Rtrois, \cconst \geq 0$ and that
  We recall that
    \begin{equation}
    \label{eq:def_V_2}
    V(x) = \exp(\mtttrois\phi(x)) \eqsp , \qquad \phi(x) =  \sqrt{1 + \norm{x}^2} \eqsp , \qquad \mtttrois \in \ocint{0, \mttun / 2} \eqsp .
  \end{equation}
Let $\VlyapDtrois(x,y) = (V(x) + V(y))/2$ with $V(x) = \exp [\mtttrois \sqrt{1 + \norm{x}^2}]$ and $\mtttrois \in \ocint{0, \mttun/2}$.
Therefore, by \Cref{propo:drift_convex}, $\Kcoupling_{\gamma}$ 
satisfies \hyperlink{ass:drift_discrete}{$\bfDd(\VlyapDtrois,\lambda^{\gamma},A \gamma, \msx^2)$} for any $\gamma \in \ocint{0, \bgamma}$ where $\bgamma \in (0, 2\mttdeux)$, $\Rquatre = \max(1, \Rtrois, (d+\cconst)/\mttun)$ and
\begin{equation}
  \begin{aligned}
    &\lambda = e^{-(\mtttrois)^2/2} \eqsp , \quad  A = \exp\parentheseDeux{\bgamma (\mtttrois (d+\cconst)+ (\mtttrois)^2)/2 + \mtttrois (1 + \Rquatre^2)^{1/2}}(\mtttrois (d+\cconst)/2 + (\mtttrois)^2) \eqsp , \\
    & R = \log(2\lambda^{-2\bgamma} A \log^{-1}(1/\lambda)) \eqsp .\end{aligned}      \label{eq:const_drift_vexp}
 \end{equation}
  Let $\bgamma \in \oointLigne{0, 2\mttdeux}$, $\ell \in \nsets$ specified below, $\lambda_{\bgamma, c}, \rho_{\bgamma, c} \in \ooint{0,1}$ and $D_{\bgamma, 1, c}$, $D_{\bgamma, 2, c}$, $C_{\bgamma, c} \geq 0$ the constants given by \Cref{propo:drift_strong_convex_bounds}, such that for any $k \in \nset$, $\gamma \in \ocint{0, \bgamma}$ and $x,y \in \msx$
  \begin{equation}
    \wassctrois(\updelta_x \Rcoupling_{\gamma}^k, \updelta_y \Rcoupling_{\gamma}^k) \leq \Kker_{\gamma}^k \bfc_3(x,y) \leq \lambda_{\bgamma, c}^{k\gamma/4} [D_{\bgamma,1,c} \bfc_3(x,y) + D_{\bgamma,2,c}\1_{\Deltar^{\complementary}}] +  C_{\bgamma, c} \rho_{\bgamma, c}^{k\gamma/4}  \eqsp ,  \end{equation}
  with $\bfc_3(x,y) = \1_{\Deltar^{\complementary}}(x,y) \defEnsLigne{V(x) + V(y)}/2$ for any $x,y \in \msx$.
Note that by \eqref{eq:def_V_2},  this result implies that for any $k \in \nset$, $\gamma \in \ocint{0, \bgamma}$ and $x,y \in \msx$
  \begin{equation}
    \label{eq:wc3_convergence_true}
    \Vnorm{\updelta_x \Rcoupling_{\gamma}^k - \updelta_y \Rcoupling_{\gamma}^k} \leq \defEns{D_{\bgamma, 1, c} + D_{\bgamma, 2, c} + C_{\bgamma, c} } \rho_{\bgamma, c}^{k \gamma} \bfc_3(x,y) \eqsp .
  \end{equation}
  Note that using \Cref{theo:discrete_contrac_wass_D_v2}, we obtain that the following limits exist and do not depend on~$\Lip$
\begin{equation}
  \begin{cases} &D_{1,c} = \lim_{\bgamma \to 0} D_{\bgamma, 1, c}\eqsp , \quad D_{2,c} = \lim_{\bgamma \to 0} D_{\bgamma, 2, c}\eqsp , \quad C_c = \lim_{\bgamma \to 0}C_{\bgamma, c}\eqsp , \\ &\lambda_c = \lim_{\bgamma \to 0}\lambda_{\bgamma, c} \eqsp , \quad \rho_c = \lim_{\bgamma \to 0}\rho_{\bgamma, c} \eqsp . \end{cases}\end{equation}

We now discuss the dependency of $\rho_b$ with respect to the
introduced parameters, depending on the sign of $\mtt$ and based on
\Cref{theo:discrete_contrac_wass_D_v2}.
\begin{enumerate}[label= (\alph*), wide, labelwidth=!, labelindent=0pt]
\item If \Cref{as:b_lemme_descente} holds, set $\ell = \ceil{\tM_{\discrete}^2}$. Then, if  we consider $\mttun, \mttdeux$ sufficiently small and $\cconst$ sufficiently large such that the conditions of \Cref{theo:discrete_contrac_wass_D_v2} hold, we have
    \begin{align}
   &\log^{-1}(\rho_c^{-1}) \leq \left. 2\parentheseDeux{1 + \mttplustrois (1+R)/4  + \log(1+2A) + (1 + 4R^2) \mttplustrois} \middle/ \parentheseDeux{\mttplustrois\Phibf(-1/2)}  \right. \eqsp.
  \end{align}
    Note that the leading term on the right hand side of this equation is of order $R^2$.
\item If \Cref{as:b_min}($\mtt$) with $\mtt \in \rset_-$, set
  $\ell = \ceil{\tM_{\discrete}^2}$. Then, if  we consider $\mttun, \mttdeux$ sufficiently small and $\cconst$ sufficiently large such that the conditions of \Cref{theo:discrete_contrac_wass_D_v2} hold, we have 
  \begin{align}
    \log^{-1}(\rho_b^{-1}) &\leq \left. 2\parentheseDeux{1 + \mttplustrois (1+R)/4 + \log(1+2A) + (1 + 4R^2) \mttplustrois} \right .\\ & \qquad \qquad \left . \middle/ \parentheseDeux{\mttplustrois \Phibf\defEnsLigne{-2(-\mtt)^{1/2}R/(2 - 2\rme^{2\mtt R^2})^{1/2}}}  \label{eq:rho_bgamma_3_majo_lim}\right. \eqsp ,
  \end{align}
Note that the right hand side of \eqref{eq:rho_bgamma_3_majo_lim}  is exponential in $-\mtt R^2$.
  \end{enumerate}
A similar result was already obtained in \cite[Theorem 10]{durmus2017nonasymp} but the scheme of the proof was different as the authors compared the discretization scheme to the associated diffusion process and used the contraction of the continuous process.



\end{document}